%% file: twist.tex
\documentclass[11pt,a4paper]{amsart}
\usepackage{amssymb}
\usepackage{pgfplots}
\usepackage{csvsimple}
\usepackage[final]{hyperref}
\usepackage{graphicx,color}
\usepackage[normalem]{ulem}
\usepackage{mathrsfs,bbm}
\graphicspath{{pics/}}
\DeclareMathAlphabet{\mathbbx} {U}{bbold}{b}{n}
\newcommand{\zero}{\mathbbx 0}
\newcommand{\one}{\mathbbx 1}
\include{sb_macros}
\include{pr_macros}
\def\TP{{\rm TP}}

\newcommand{\korr}[1]{#1}
\allowdisplaybreaks

\begin{document}
\title[Numerical solution of a bending-torsion model]
{Numerical solution \\ of a bending-torsion model \\  for elastic rods}
\author[S. Bartels]{S\"oren~Bartels}
 \address{Abteilung f\"ur Angewandte Mathematik,  
 Albert-Ludwigs-Universit\"at Freiburg, Hermann-Herder-Str.~10, 
 79104 Freiburg i.~Br., Germany}
 \email{bartels@mathematik.uni-freiburg.de}
\author[Ph. Reiter]{Philipp~Reiter}
\address{Institut f\"ur Mathematik, Martin-Luther-Universit\"at Halle-Wittenberg, 06099 Halle (Saale), Germany}
 \email{philipp.reiter@mathematik.uni-halle.de}
\date{\today}
\renewcommand{\subjclassname}{%
  \textup{2010} Mathematics Subject Classification}
 \begin{abstract}
  Aiming at simulating elastic rods,
  we discretize a rod model based on {a general theory
  of hyperelasticity}
  for inextensible and unshearable rods.

  After reviewing this model
  {and discussing topological effects of periodic rods},
  we prove convergence of the discretized functionals
  and stability of \korr a corresponding discrete flow.
  
  Our experiments numerically confirm thresholds e.g.\@
  for Michell's instability and indicate a complex
  energy landscape, in particular in the presence of impermeability.
 \end{abstract}
\maketitle

\setcounter{tocdepth}1
\tableofcontents

\input{intro}

\section{Elastic rods}\label{sect:rods}

Here we provide a \korr{short} presentation of the geometry of elastic rods
which is inspired by Langer, Singer and Ivey~\cite{langer-singer,ivey-singer}.
It is not essential for the analysis of the numerical scheme
in the subsequent sections
but sheds some light on the interpretation of the experiments in the last section.

\subsection{Framed curves}

A \emph{rod} is modeled by a curve
$\curve:[0,L]\to\R^{3}$ which corresponds to its
centerline and an orthonormal \emph{frame}
$\frame:[0,L]\to\mathrm{SO}(3)$
whose columns $\frame=[t,b,d]$ are called \emph{directors}.
Of course, $d = t\times b$ where $\times$ denotes the vector cross product.
In the following we {consider} $\curve\in H^{2}$ and $F\in H^{1}$.

We will assume that the first column of $\frame$
coincides with the unit tangent $t(x) = \left.\curve'(x)\middle/\abs{\curve'(x)}\right.$, $x\in [0,L]$.
The idea is that the directors $b$ and $d$ track the twisting of
the material about the centerline.

{The assumed energy regime for bending stiffness in~\eqref{eq:P}
imposes inextensibility as a physical property.
Therefore} we can prescribe
arc\-length parametrization which leads \korr{to the unit tangent vector $t=\curve'$
and the curvature $k=\abs{t'}=\abs{\curve''}$}.

Our analysis also covers the case of closed rods where $[0,L]$
is understood to be the periodic interval $\R/L\Z$.
We will realize the latter by imposing suitable
(periodic) boundary conditions at $0$ and~$L$.
In general {a twist-free} frame {(see Section~\ref{sect:reframe} below)} of a closed curve will not close up,
i.e., there can be a discontinuity at one point of $\R/L\Z$.
One has to take care of this fact when defining boundary conditions.

An important frame that will always be well-defined and
continuous for sufficiently smooth (both open and closed) curves $\curve$ with nonvanishing curvature
is the \emph{Fr\'enet frame} where $b_{\mathrm F} = t'/\abs {t'}$.

A rod is assumed to have some small diameter
which can be considered infinitesimal; however, self-penetrations
are not excluded at this stage (see Section~\ref{sect:imper} below
for a discussion on modeling impermeability).

\subsection{Twist rate}

Using the orthonormality of the frame, we may express the {variation} of the director $b$ by
\[ b' = (b'\cdot t)t + {(b'\cdot b)}b + (b'\cdot d)d
= -(b\cdot\curve'')t + (b'\cdot d)d. \]
The first term tracks the change of $b$ that is imposed by the
{spatial behavior} of the curve. It is just a component of the curvature vector
as $y''={(y''\cdot b)}b + (y''\cdot d)d$.
Only the second one actually provides information
about the twisting of the frame about the centerline.
Therefore we will call $b'\cdot d$ the \emph{twist rate} of the frame.

We may also characterize a frame by a $9\times 9$ linear system, namely
\[ \begin{pmatrix} t\\b\\d \end{pmatrix}'
= \begin{pmatrix} \zero&k_{b}\one&k_{d}\one\\-k_{b}\one&\zero&\beta\one\\-k_{d}\one&\beta\one&\zero \end{pmatrix}\begin{pmatrix} t\\b\\d \end{pmatrix}
 \]
for scalar coefficient functions $k_{b}$, $k_{d}$, and $\beta$
where $\zero,\one\in\R^{3\times3}$ denote the zero and identity matrices.
Here $k_{b}=y''\cdot b$ and $k_{d}=y''\cdot d$ are the components of the curvature $k$ of $\curve$
and $\beta=b'\cdot d$ is the twist rate.
For instance, the Fr\'enet frame is characterized by $k_{d}\equiv0$.

A more detailed discussion of the impact of the twist rate
is given in Section~\ref{sect:tot-twist} below.

\subsection{Reference frame}
\label{sect:reframe}

For any curve $\curve$, a point $\xi\in[0,L]$, and $\widehat b\in\mathbb S^{2}$, $\widehat b\perp\curve'(\xi)$,
we obtain by integration a unique frame $F_{0}=[t_{0},b_{0},d_{0}]$ for $\curve$ with
$F_{0}(\xi)=[\curve'(\xi),\widehat b,\curve'(\xi)\times\widehat b]$ whose twist rate
is constantly zero. We call it synonymously a \emph{Bishop frame}, \emph{natural frame},
\emph{reference frame},
or \emph{twist-free} frame for $\curve$
as it is a frame in rest position subject to a fixed curve.
Therefore, up to a rotation of the initial vector $\widehat b$
(which corresponds to an element of $\mathbb S^{1}$) there is a unique twist-free frame
for any given curve.

A twist-free frame $F_{0}$ provides a useful reference configuration.
Denoting the (cumulative) angle between the director $b$ of any other
frame and $b_{0}$ by $\varphi$, we arrive at
$b=(\cos\varphi)b_{0}+(\sin\varphi)d_{0}$ and
$d=-(\sin\varphi)b_{0}+(\cos\varphi)d_{0}$.
Consequently, the rate of change of $\varphi$ is just the twist rate
\begin{equation}\label{eq:twistrate}
 \varphi' = b'\cdot d = \beta.
\end{equation}

Two frames that just differ by a constant angle $\varphi$
may be considered equivalent, in particular when modeling
rods with a circular diameter where there is no natural choice
of a director. This is of course different for small ribbons
with lateral extension in a particular direction.

 Note that even for closed curves with frames that close up,
$\varphi(L)-\varphi(0)$ {does not need to} be an integer multiple of $2\pi$, {unless the twist-free reference frame closes up. The latter applies in particular to
rods with planar centerline where
the vector being perpendicular to the respective plane
provides a ``canonical'' twist-free frame.}

In general, there is no {direct} correlation between $\varphi(L)-\varphi(0)$ and
the angle {enclosed by} $ b(0)$ and $ b(L)$.
{One can think of a revolute joint that controls
the latter angle.}

\subsection{Total twist}
\label{sect:tot-twist}

An important quantity, in the literature often simply referred to as ``twist'', is the \emph{total twist} (more precisely, total twist rate)
\begin{align*}
 \Tw(\curve,b)
 &= \frac{\varphi(L)-\varphi(0)}{2\pi}
 = \frac1{2\pi}\int_{0}^{L} \varphi'(s) \d s
 = \frac1{2\pi}\int_{0}^{L} \beta(s) \d s \\
 &= \frac1{2\pi}\int_{0}^{L} b'(s)\cdot d(s) \d s
 = \frac1{2\pi}\int_{0}^{L} \det\br{\curve'(s),b(s),b'(s)}\d s
\end{align*}
where $s$ is an arclength parameter and the last expression is
parametrization invariant.

As the first identity suggests,
the total twist can be interpreted as the number of rotations
the director $b$ (or, equivalently, $d$) performs about
the curve, i.e., the centerline of the rod.

{The total twist takes \emph{integer} values
on any closed curve for which both frame and twist-free reference frame close up.
In particular, this holds for any \emph{planar} closed curve with a closed frame.}

\subsection{C\u alug\u areanu's identity}
\label{sect:calu}

A given (sufficiently smooth) embedded closed curve $\curve$
together with a closed frame $b$
(i.e., there are no discontinuities of the frame)
defines a \emph{link} consisting of $\curve$ and $\curve+\eps b$ for some
small $\eps>0$.
For embedded {closed} curves,
C\u alug\u areanu's identity~\cite{calu59,calu61}
{\[ \Lk = \Tw + \Wr \]}%
provides a decomposition of
the Gauss linking number {$\Lk$} (an integer topological invariant,
a special case of the mapping degree)
into the sum of two geometric terms (that can take arbitrary real values each),
namely the total twist {$\Tw$} and the writhe functional~{$\Wr$}.
{See Moffatt and Ricca~\cite{moffatt-ricca} for
an account on the history of this result.

The Gauss linking number amounts to
half of the sum of all signed crossings
of $\curve$ and $\curve+\eps b$ with respect to a
\emph{regular} projection direction.
The latter guarantees that we can identify any crossing either as an {over-} or {undercrossing}, respectively.
{To this end, the intersection of (the images of)
$\curve$ and $\curve+\eps b$ must be empty.}
It was Gauss' seminal discovery that this quantity
can also be expressed by a double integral over $[0,L]$,
see, e.g., the nice review by Ricca and Nipoti~\cite{RN}.
{For two curves $\curve,\widetilde\curve$ forming a link,
such as $\widetilde\curve=\curve+\eps b$, we have
\[
\Lk(\curve,\widetilde\curve) = \tfrac1{4\pi}\iint_{[0,L]^{2}}\frac{\det\br{\curve(x)-\widetilde\curve(\widetilde x),\curve'(x),\widetilde\curve'(\widetilde x)}}{\abs{\curve(x)-\widetilde\curve(\widetilde x)}^{3}}
\d x\d\widetilde x
\]
which is independent under reparametrization.}%

Any crossing of $\curve(x)$ and $\curve(x')+\eps b(x')$
where the preimages $x,x'\in[0,L]$ are close
will be called \emph{local} and \emph{global} otherwise.
Given any regular projection direction,
we can distinguish between local and global
contributions to the linking number, see Figure~\ref{fig:crossings}.
Note that, in contrast to the linking number, its local and global
contributions are not invariant for any regular projection
direction in general.
By Sard's theorem, almost every direction is regular, so
we can just compute the average of local and global
contributions over the sphere of all projection directions.
The first one agrees with the total twist,
the second one is the \emph{writhe} functional.
For details we refer to Dennis and Hannay~\cite{dennis-hannay}
and references therein.

\begin{figure}
 \includegraphics[scale=.2,trim=180 380 140 280,clip]{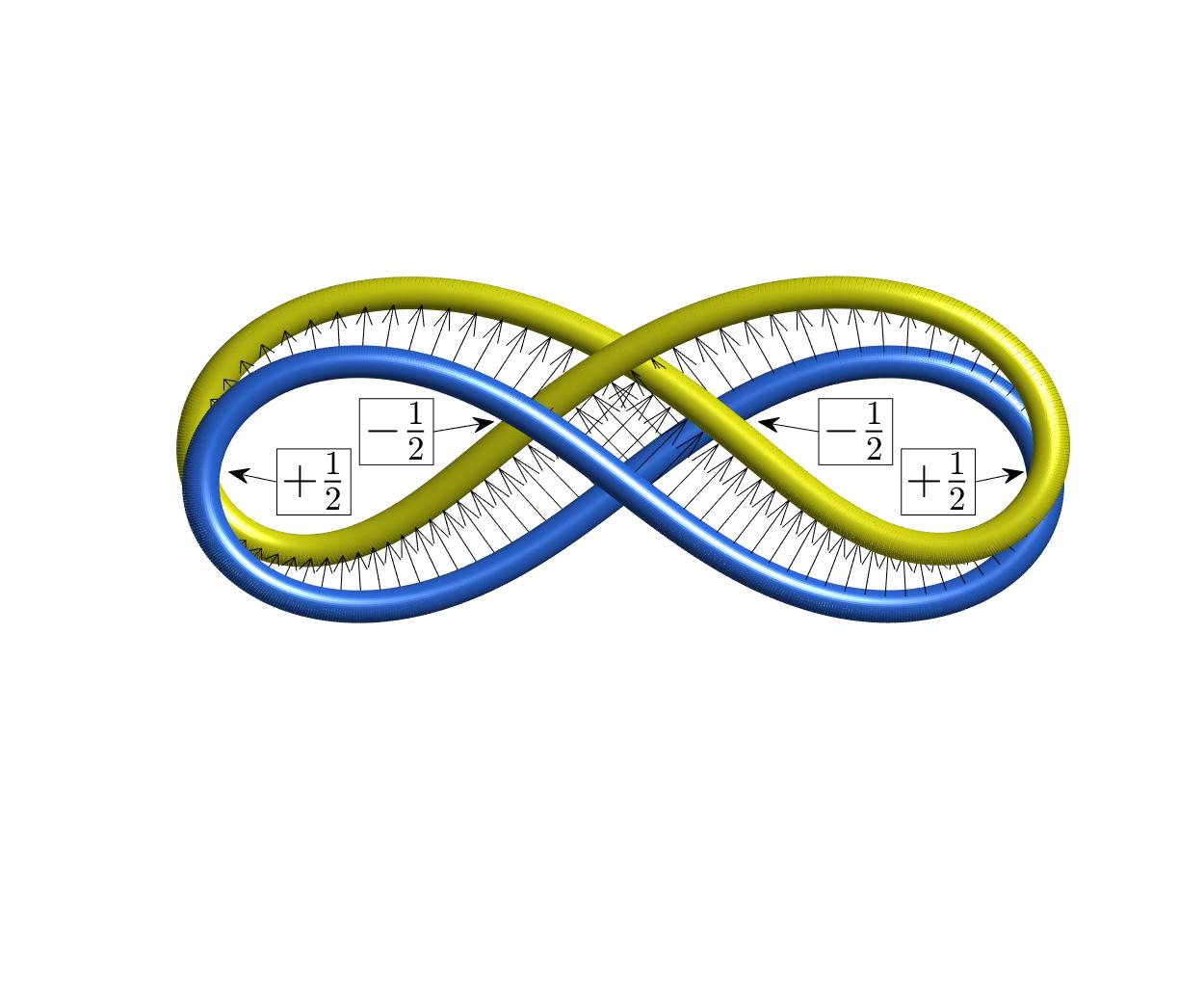}
 \caption{A curve $y$ and a director $b$ produce a link
 consisting of the curves $y$ and $y+\eps b$ for some small $\eps>0$.
 In the projection shown here we see four \emph{mutual} intersections of
 the first with the second curve.
 On the left and right margin we see a local crossing (overcrossing).
 This is a contribution to the total twist which amounts to~$1$ in this case.
 At the center there are two global crossings (undercrossing).
 {These count for the writhe functional.
 In this particular example we have $\Lk=0$,
 $\Tw=-\Wr\approx0.95$.}}
 \label{fig:crossings}
\end{figure}

Writhe can equally be computed by
summing up the signs of \emph{self}-crossings of any projection of the curve and averaging over all projection directions.
In particular, it only depends on the curve, not on the frame,
{but in contrast to linking number and total twist
it requires $\curve$ to be embedded.
It can also} be represented by the Gauss double integral
{via
\[ \Wr(\curve) = \Lk(\curve,\curve). \]
Writhe measures non-planarity (and non-sphericity) of a curve;
in particular, it vanishes on planar (and spherical) embedded curves.

Consequently, the total twist will be close to the Gauss linking number
for embedded curves that are nearly planar
(which occur at several stages of the experiments in Section~\ref{sect:exper}).
So we can check the values of the total twist
by computing the Gauss linking number,
i.e., counting signed crossings.
We can remove the claim of embeddedness of the centerline by
relating crossings of $\curve$ and $\curve+\eps b$
to the number of rotations of the director about the centerline.}

Of course, there is no ``writhe-free'' frame,
but, as suggested by Maddocks, we can geometrically construct a ``writhe frame''
whose linking number is zero, see~\cite{dennis-hannay}.}

Compared to the other two functionals in C\u alug\u areanu's identity,
which can be evaluated using double integrals,
the total twist is much easier to compute.
Furthermore it does not require embeddedness of the curve.

\subsection{Energies}

We assume that the behavior of the rod is driven by a linear combination
of the \emph{bending energy} ({half of the} total squared curvature)
and the \emph{twisting energy} ({half of the} total squared twist rate),
{cf.\@ Mora and M\"uller~\cite{MorMul03}}.
More precisely we consider the functional
\begin{equation}\label{eq:elastic}
 (\curve,b) \mapsto \frac\cbend2\int_{0}^{L} k(s)^{2}\d s + \frac\ctwist2\int_{0}^{L} \beta(s)^{2}\d s
\end{equation}
where $s$ is an arclength parameter, $k(s)$ denotes the curvature of $\curve$ at $\curve(s)$,
and $\cbend,\ctwist{{}>{}}0$ are material constants.
Minimizers are called \emph{elastic rods}.

{As mentioned in the \korr introduction, we rescale the energy functional by $1/\ct$ and define
\(
\k = {\cb}/{\ct}. \)}

At the end of this section, we will briefly discuss two
related minimization issues.

\subsection{Optimal frames}
\label{sect:optframe}

 For a given curve $\curve$, we may consider the problem to
 find a director $b$ minimizing $\energy[y,\cdot]$ subject to
 the boundary condition $L_\bc^\rod[y,b] = \elbc^\rod$.
 
 In first place, if $b$ is a stationary point of $\energy[\curve,\cdot]$
 for some fixed $\curve$ then
 \begin{equation}\label{eq:dubois}
  \beta \equiv \varphi' \equiv \tfrac{2\pi}L\Tw(\curve,b)
 \end{equation}
 is constant due to du Bois-Reymond's lemma.
 {According to the Cauchy--Schwarz inequality,
 the twisting energy is bounded below by $\frac{2\pi^{2}}L\Tw(y,b)^{2}$.
 This minimum is attained if and only if~\eqref{eq:dubois} holds
 such that the twisting energy then amounts to
 $\frac{2\pi^{2}}L\Tw(y,b)^{2} = \tfrac L2\beta^{2}$. 
 In particular, we can check whether a given rod
 has a uniform twist rate by computing the
 quotient of total squared twist rate over squared total twist rate.}
 
 Note that for any global minimizer $(\curve,b)$ of $\energy$,
 the director $b$ is a global minimizer of $\energy[y,\cdot]$ as well.

 In case $L_\bc^\rod[y,b]$ does not affect both points~$0$ and~$L$,
 minimizing $\energy[\curve,\cdot]$ is equivalent to constructing a twist-free frame.
 
 Otherwise we face a \emph{clamped problem}, i.e.,
 $b$ has to satisfy $b(0)=\widehat b_{-}$ and $b(L)=\widehat b_{+}$
 for $\widehat b_{-},\widehat b_{+}\in\mathbb S^{2}$, $\widehat b_{-}\perp\curve'(0)$,
 $\widehat b_{+}\perp\curve'(L)$.
 (If the boundary condition just forces the frame to close up,
 i.e., $b(0)=b(L)$, we may just let $\widehat b_{-}=\widehat b_{+}$
 for an arbitrary vector perpendicular to $\curve'(0)$ and $\curve'(L)$.)
 In this case there is a global minimizer $b_{\min}$
 with constant twist rate~\eqref{eq:dubois}.
 Using a twist-free reference frame $F_{0}=[\curve',b_{0},d_{0}]$
 with $b_{0}=\widehat b_{-}$ we have $\varphi(0)=0$ and $\varphi(L)\in(-\pi,\pi]$.
 If $\varphi(L)\ne\pi$ there is a unique minimizer $b$ with
 twist rate $\varphi' \equiv \frac{\varphi(L)}{L}$.
 If $\varphi(L)=\pi$ there are precisely two minimizers with
 twist rate $\varphi' \equiv \pm\frac\pi L$. 
 
 Topological restrictions can enforce arbitrary angles $\varphi(L)\in\R$,
 however, this does not apply to~\eqref{eq:P} which does not preserve
 this sort of condition throughout the evolution.
 Keeping track of topology enforces modeling impermeability---quite a
 natural feature which we will address in Section~\ref{sect:imper}.

\subsection{Releasing total twist}
\label{sect:mintt}

In light of {Section}~\ref{sect:optframe} we must have
$\abs{\Tw(\curve,b)}\le\frac 12$ for any global minimizer $(\curve,b)$
of~$\energy$.
In general, we have
\[ \abs{\Tw(\curve,b)}^{2}\le \frac{L}{2\pi^{2}}\cdot\frac12 \int_{0}^{L}(b'(s)\cdot d(s))^{2}\d s, \]
however, the absolute value of the total twist does not have to be decreasing
throughout the evolution.

At the final stage of an evolution of a closed curve
(the frame does not have to close up), all we can hope for, however,
is $\abs{\Tw(\curve,b)}\le1$.
We briefly explain how this bound can be realized.

One can change the Gauss linking number of a given (embedded) rod by $\pm2$
by locally forming a small loop, performing a suitable self-penetration and
moving the curve back to the original position. The value of
the writhe functional is not affected as it does not depend on
the frame. So we have changed the total twist by {$\pm2$} as well
according to the C\u alug\u areanu identity (cf.\@ {Section}~\ref{sect:tot-twist}).

A self-penetration of the curve will in general lead to a change
in topology resulting in a discontinuity of the linking number.
While the total twist is continuous throughout the evolution,
the writhe functional is not well-defined on non-embedded curves
and thereby compensates the change of the linking number.

An evolution does not necessarily realize the bound $\abs\Tw\le1$.
First of all, it is in general unclear whether it will in fact
converge to a (local) minimizer at all.
Another obstruction is discussed in the next section.

\subsection{Michell’s instability}
\label{sect:zajac}

Among all closed curves, the round circle fra\-med by its normal vector
is the unique global minimizer of $\energy$
(up to a constant rotation of the frame).

It is a remarkable fact that the round circle remains a minimizer
(at least a local one, cf.\@ {Section}~\ref{sect:mintt}) when we add some twist
by increasing the (constant) twist rate~$\beta$ (which results in a discontinuity of the frame at one point).
This phenomenon which is referred to as \emph{Michell's instability}
has been discovered 130 years ago~\cite{michell} and then been rediscovered several
times, see Goriely~\cite{goriely} for more details.

Zajac~\cite{zajac} has found the threshold $\beta_{*} = 2\pi\sqrt3\kappa/L$
that separates the stable and unstable regime.
As before, $L$ denotes the length of the curve.
More precisely, the circular rod is stable as long as $\abs\beta<\beta_{*}$
and unstable if $\abs\beta>\beta_{*}$.
The dependency on $\kappa=\cb/\ct$ is quite intuitive:
If $\kappa$ is very small, the bending energy dominates
which always prefers the circle.
A proof of Zajac's result adapted to our setting can be found in
Ivey and Singer~\cite[Sect.~6]{ivey-singer}.

Values of $\kappa<\tfrac13\sqrt3\approx0.5774$
lead to an initial twist $\beta_{*}<\frac{2\pi}L$.
Starting an evolution with $\beta_{\ini}\in\br{\beta_{*},\frac{2\pi}L}$,
we have $\abs\Tw=\Tw<1$.
Therefore the rod cannot reduce twist by self-penetration,
so we will merely face some buckling of the rod---which is
{difficult to detect numerically}.
In Experiment~\ref{sect:michell} we chose $\kappa=\tfrac32$,
for which we measure
a drastic change of the twisting energy by self-penetration of the curve.

Interestingly, Michell's instability does not occur for initially
curved curves, see Olsen et al.~\cite{olson} and Hu~\cite{hu}.

\section{Density}\label{sect:approx}

We can smoothly approximate any framed curve in $\class$,
i.e., $\class\cap\br{C^{\infty}\times C^{\infty}}$ is dense in $\class$
with respect to the $H^{2}\times H^{1}$-topology,
{preserving given boundary conditions.}

\begin{lemma}\label{lem:approx}
 For any frame $(\curve,b)\in\class$ and $\eps>0$
 there is another frame $(\curve_{\eps},b_{\eps})\in\class\cap\br{C^{\infty}\times C^{\infty}}$
 with $\norm{\curve_{\eps}-\curve}_{\hzwei}\le\eps$
 and $\norm{b_{\eps}-b}_{\heins}\le\eps$.
\end{lemma}

\begin{proof}
 Our strategy is as follows. We first construct
 smooth approximizers $(\curve_{\delta},b_{\delta})\in C^{\infty}\times C^{\infty}$.
 In a second step we correct the boundary values by adding (smooth)
 functions $v_{\delta}$ and $c_{\delta}$.
 The new curve $\curve_{\delta}+v_{\delta}$ will not have
 length $L$.
 We balance the length by adding another
 smooth function $w_{\delta}$ compactly supported in $(0,L)\setminus\supp v_{\delta}$.
 Now we reparametrize
 the curve $\curve_{\delta}+v_{\delta}+w_{\delta}$ to arclength
 and apply the same reparametrization to the vector field $b_{\delta}+c_{\delta}$.
 Renormalizing it by the usual Gram--Schmidt scheme produces
 the required director.

 To begin with, we perpare the length correction.
 If $\abs{\curve(L)-\curve(0)}=L$, the curve $u$ just parametrizes
 the segment from $\curve(0)$ to $\curve(L)$, so
 $\curve\in C^{\infty}$ and we only have
 to treat the director $b$ (which can be done similarly as outlined
 below).
 If $\abs{\curve(L)-\curve(0)}<L$ we infer from
 $H^{2}\subset C^{1}$ that $\curve\in\class$ cannot only
 move on a straight line.
 So we may assume that there is some constant unit vector $\vec v\in\mathbb S^{2}$, $\vec v\perp\br{\curve(L)-\curve(0)}$
 (this condition is empty for closed curves), such that
 $\curve'\cdot\vec v>0$ on some interval $I_{+}\subset [0,L]$.
 Due to the fact that $\int_{0}^{L}\curve'(x)\cdot\vec v\d x
 =(\curve(L)-\curve(0))\cdot\vec v=0$
 there has to be another interval $I_{-}\subset [0,L]$ on which
 $\curve'\cdot\vec v<0$.
 Diminishing $I_{\pm}$ if necessary, we may assume that
 they are both contained in $(\mu,L-\mu)$ for some $\mu\in(0,L/2)$.
 Moreover we can assume that there is some $\lambda\in(0,\tfrac12]$ such that
 $\pm\curve'\cdot\vec v\ge2\lambda$ on $I_{\pm}$.
 
 We choose $\delta\in(0,\delta_{0}]$ for some $\delta_{0}\in(0,1]$
 which will
 be fixed later on
 only depending on $(\curve,b)$ and $\eps$.
 Using a standard mollifier, we
 obtain $(\curve_{\delta},b_{\delta})\in C^{\infty}\rzd
 \times C^{\infty}\rzd$
 with $\norm{\curve_{\delta}-\curve}_{\hzwei}\le\delta$ and
 $\norm{b_{\delta}-b}_{\heins}\le\delta$.
 {We may assume that
 $\pm\curve_{\delta}'\cdot\vec v\ge\lambda$ on $I_{\pm}$}
 for all $\delta\in(0,\delta_{0}]$.
 
 In order to match the boundary conditions, we subtract
 suitable functions. More precisely, we let
 \begin{align*}
 \bar\curve_{\delta} = \curve_{\delta} + v_{\delta} &= \curve_{\delta}
 - (\curve_{\delta}(0)-\curve(0))\zeta_{0}
 - (\curve_{\delta}'(0)-\curve'(0))\zeta_{1} \\
 &\qquad\quad{}- (\curve_{\delta}(L)-\curve(L))\zeta_{0}(L-\cdot)
 + (\curve_{\delta}'(L)-\curve'(L))\zeta_{1}(L-\cdot), \\
 \bar b_{\delta} = b_{\delta} + c_{\delta} &= b_{\delta}-(b_{\delta}(0)-b(0))\zeta_{0}-(b_{\delta}(L)-b(L))\zeta_{0}(L-\cdot)
 \end{align*}
 where
 $\zeta_{0},\zeta_{1}\in C^{\infty}([0,L])$ fulfill
 $\zeta_{j}(0) = \delta_{j,0}$,
 $\zeta_{j}'(0) = \delta_{j,1}$, $\left.\zeta_{j}\right|_{[\mu,L]}\equiv0$, $j=0,1$.
 By construction we have
 $L_{\mathrm{bc}}^{\rod}[\bar\curve_{\delta},\bar b_{\delta}]
 =L_{\mathrm{bc}}^{\rod}[\curve,b]=\elbc^{\rod}$ as well as
 \begin{align*}
 \norm{\bar\curve_{\delta} - \curve_{\delta}}_{\hzwei}
 &\le C_{\mu}\norm{\curve_{\delta}-\curve}_{C^{1}} \le C_{\mu}\widetilde C\delta, \\
 \norm{\bar b_{\delta} - b_{\delta}}_{\heins}
 &\le C_{\mu}\norm{b_{\delta}-b}_{C^{0}} \le C_{\mu}\widetilde C\delta
 \end{align*}
 where $C_{\mu}$ only depends on~$\mu$ and $\widetilde C>0$ on the embedding $H^{1}\hookrightarrow C^{0}$.

 The length correction function will be defined by
 $w_{\delta}=\omega_{\delta}\phi\vec v$
 for some $\omega_{\delta}\in\R$ to be defined later
 and $\phi\in C^{\infty}([0,L])$ is compactly supported in $(0,L)$
 with $\pm\phi'\ge0$ on $I_{\pm}$ and
 $\phi'\equiv0$ elsewhere, but $\phi\not\equiv0$.

 The idea is that by choosing $\omega_{\delta}$ accordingly,
 we can correct the length of ($\curve_{\delta}$ and) $\bar\curve_{\delta}$ by
 an amount between $[-\alpha,\alpha]$ where $\alpha>0$ does
 not depend on $\delta$ (nor $\delta_{0}$). As $\length[\bar\curve_{\delta}]\to L$
 for $\delta\searrow0$ we can perform the length correction
 if $\delta_{0}$ is small enough. Furthermore, $\omega_{\delta}\to0$
 as $\delta\searrow0$.

 To make this more precise, let $w=\omega\phi\vec v$ for some $\omega\in\R$ with
 \[ \abs{\omega}\le\frac\lambda{\norm{\phi'}_{C^{0}}}. \]
 As $y_{\delta}'\cdot\vec v\phi'$ is
 bounded below by $\lambda\abs{\phi'}$ on $[0,L]$,
 we obtain
 \begin{align*}
  \frac{\abs{y_{\delta}'+w'}^{2} - \abs{y_{\delta}'}^{2}}\omega
  &={2y_{\delta}'\cdot\vec v\phi'+\omega{\phi'}^{2}}
  \ge2\lambda\abs{\phi'}-\abs\omega{\phi'}^{2}
  \ge\lambda\abs{\phi'}, \\
  \abs{y_{\delta}'+w'} + \abs{y_{\delta}'}
  &\le2\abs{y_{\delta}'}+\abs{w'}
  \le 2\br{1+\abs{y_{\delta}'-u'}} + \abs\omega\abs{\phi'} \\
  &\le 2\br{1+\widetilde C\delta_{0}} + \lambda
  \le 2\br{2+\widetilde C}, \\
  \frac{\abs{y_{\delta}'+w'} - \abs{y_{\delta}'}}\omega
  &\ge\frac{\lambda\abs{\phi'}}{2\br{2+\widetilde C}}.
 \end{align*}
 Therefore,
 \[
 \abs{y_{\delta}'+w'}\quad
 \left\{
 \begin{array}{ll}
 \le\abs{y_{\delta}'} - \abs\omega\frac{\lambda\abs{\phi'}}{2\br{2+\widetilde C}}
 & \text{if }\omega\le0, \\
 \ge\abs{y_{\delta}'} + \abs\omega\frac{\lambda\abs{\phi'}}{2\br{2+\widetilde C}}
 & \text{if }\omega\ge0.
 \end{array}
 \right.
 \]
 Recalling that $v_{\delta}$ and $w$ have disjoint support, we infer
 \[
 \abs{y_{\delta}'+v_{\delta}'+w'}\quad
 \left\{
 \begin{array}{ll}
 \le\abs{y_{\delta}'+v_{\delta}'} - \abs\omega\frac{\lambda\abs{\phi'}}{2\br{2+\widetilde C}}
 & \text{if }\omega\le0, \\
 \ge\abs{y_{\delta}'+v_{\delta}'} + \abs\omega\frac{\lambda\abs{\phi'}}{2\br{2+\widetilde C}}
 & \text{if }\omega\ge0,
 \end{array}
 \right.
 \]
 which allows for the desired length correction depending on the
 sign of $\omega$.
 More precisely, we can change the length of $y_{\delta}+v_{\delta}$
 by at least $\pm\alpha$ where
 $\alpha = \frac{\lambda^{2}\norm{\phi'}_{L^{1}}}{2\br{2+\widetilde C}\norm{\phi'}_{C^{0}}}$.
 Diminishing $\delta_{0}$ if necessary,
 we can ensure that $\abs{\length(y_{\delta}+v_{\delta})-L}\le\alpha$.
 So we can find $\omega=\omega_{\delta}$ such that
 the curve $\bar{\bar\curve}_{\delta}=\curve_{\delta}+v_{\delta}+w_{\delta}$
 has length~$L$ with
 $L_{\mathrm{bc}}^{\rod}[\bar{\bar\curve}_{\delta},\bar b_{\delta}]=\elbc^{\rod}$.

 The embedding $H^{1}\hookrightarrow C^{0}$ guarantees that
 the curve $\bar{\bar\curve}_{\delta}$ is immersed
 and $\min\abs{\bar b_{\delta}}\ge\tfrac12$
 if $\delta_{0}$ is small enough.
 So we may apply the reparametrization operator from Lemma~\ref{lem:repara}
 and let $\widecheck y_{\delta}={\bar{\bar\curve}}_{\delta}\circ\psi_{\bar{\bar\curve}_{\delta}}^{-1}$
 and $\widecheck b_{\delta}=\bar b_{\delta}\circ\psi_{\bar{\bar\curve}_{\delta}}^{-1}$.
 We still have $\min\abs{\widecheck b_{\delta}}\ge\tfrac12$.
 Now
 \[ \norm{{\bar{\bar\curve}}_{\delta}-\curve}_{\hzwei}
 \le\norm{\curve_{\delta}-\curve}_{\hzwei}
 +\norm{v_{\delta}}_{\hzwei} + \norm{w_{\delta}}_{\hzwei}
 \le\delta + C_{\mu}\widetilde C\delta + \omega_{\delta}\norm\phi_{H^{2}} \]
 and $\norm{\bar b_{\delta}-b}_{\heins}$
 tend to zero as $\delta\searrow0$.
 Using the continuity of the repara\-metrization and $\abs{\curve'}\equiv1$
 we find that 
 $\norm{{{\widecheck\curve}}_{\delta}-\curve}_{\hzwei}$
 tends to zero as well.
 Choosing $\delta_{0}$ sufficiently small, we may assume that
 $\norm{{{\widecheck\curve}}_{\delta}-\curve}_{\hzwei}\le\eps$ and
 $\norm{\bar b_{\delta}-b}_{\heins}\le\frac\eps4$,
 and additionally
 $\norm{\widecheck b_{\delta}-{\bar b}_{\delta}}_{\heins}\le\frac\eps4$
 since $\psi_{\bar{\bar\curve}_{\delta}}\to\id_{[0,L]}$ with respect to $H^{2}$-convergence.
 Note that $(\widecheck\curve_{\delta},\widecheck b_{\delta})$
 are still $C^{\infty}$-smooth with
 $L_{\mathrm{bc}}^{\rod}[\widecheck\curve_{\delta},\widecheck b_{\delta}]
 =\elbc^{\rod}$.
 
 It remains to correct the director. To this end, we let
 $\widetilde b_{\delta}=\widecheck b_{\delta}-\br{\widecheck b_{\delta}\cdot\widecheck\curve_{\delta}'}\widecheck\curve_{\delta}'$.
 We have
 $\norm{\widetilde b_{\delta}-\widecheck b_{\delta}}_{\heins}\le\tfrac\eps4$
 and
 $\norm{\left.\widetilde b_{\delta}\middle/\abs{\widetilde b_{\delta}}\right.-\widetilde b_{\delta}}_{\heins}\le\frac\eps4$
 if $\delta_{0}$ is sufficiently small. Indeed,
 using \korr{Leibniz rule $\norm{vw}_{\heins}\le\norm v_{\heins}\norm w_{\heins}$ and}
 the fact that both
 $\norm{{{\widecheck\curve}}_{\delta}-\curve}_{\hzwei}$ and
 $\norm{\widecheck b_{\delta}-b}_{\heins}$ get arbitrarily small
 provided $\delta_{0}$ is chosen accordingly, the same applies to
 \begin{align*}
 &\norm{\widetilde b_{\delta}-\widecheck b_{\delta}}_{\heins}
 =\norm{\br{\widecheck b_{\delta}\cdot\widecheck\curve_{\delta}'}\widecheck\curve_{\delta}'}_{\heins}
 =\norm{\br{\widecheck b_{\delta}\cdot\widecheck\curve_{\delta}'-b\cdot\curve'}\widecheck\curve_{\delta}'}_{\heins} \\
 &\le\br{\norm{\widecheck b_{\delta}-b}_{\heins}\norm{\widecheck\curve_{\delta}'}_{\heins}
 +\norm b_{H^{1}}\norm{\widecheck\curve_{\delta}'-\curve'}_{\heins}}\norm{\widecheck\curve_{\delta}'}_{\heins} \\
 &\le\br{\norm{\widecheck b_{\delta}-b}_{\heins}\br{\norm{\curve'}_{\heins}+\eps}
 +\norm b_{H^{1}}\norm{\widecheck\curve_{\delta}'-\curve'}_{\heins}}
 \br{\norm{\curve'}_{\heins}+\eps},
 \end{align*}
 \korr{and
 \begin{align*}
 &\norm{\frac{\widetilde b_{\delta}}{\abs{\widetilde b_{\delta}}}-\widetilde b_{\delta}}_{\heins}
 =\norm{\widetilde b_{\delta}\cdot\frac{1-\abs{\widetilde b_{\delta}}}{\abs{\widetilde b_{\delta}}}}_{\heins}
 \le\norm{\widetilde b_{\delta}}_{\heins}\norm{\frac{\abs b-\abs{\widetilde b_{\delta}}}{\abs{\widetilde b_{\delta}}}}_{\heins} \\
 &\le\br{\norm b_{H^{1}}+\eps}\norm{\frac{\sp{\widetilde b_{\delta}+b,\widetilde b_{\delta}-b}}{\abs{\widetilde b_{\delta}}+\abs b}}_{\heins}
 \norm{\frac1{\abs{\widetilde b_{\delta}}}}_{\heins} \\
 &\le\br{2\norm b_{H^{1}}+\eps}^{2}\norm{\frac{1}{\abs{\widetilde b_{\delta}}+\abs b}}_{\heins}
 \norm{\frac1{\abs{\widetilde b_{\delta}}}}_{\heins}\norm{\widetilde b_{\delta}-b}_{\heins}.
 \end{align*}}%
 Arguing as above, we find that the term $\norm{\left.1\middle/{\abs{\widetilde b_{\delta}}}\right.}_{\heins}$
 is uniformly bounded. In fact, from
 $\abs{\widecheck b_{\delta}\cdot\widecheck\curve_{\delta}'}
 = \abs{\widecheck b_{\delta}\cdot\widecheck\curve_{\delta}' - b\cdot\curve'}
 \xrightarrow{\delta\searrow0}0$
 we infer $\abs{\widetilde b_{\delta}}\ge\tfrac14$
 if $\delta_{0}$ is sufficiently small.
 This allows for bounding $\norm{\br{{\left.1\middle/\abs{\widetilde b_{\delta}}\right.}}'}$
 as well.
 \korr{In the same way we prove boundedness of $\norm{\left.1\middle/\br{\abs{\widetilde b_{\delta}}+\abs b}\right.}_{\heins}$.}
 Letting $(\curve_{\eps},b_{\eps})=\br{\widecheck\curve_{\delta_{0}},\left.\widetilde b_{\delta_{0}}\middle/\abs{\widetilde b_{\delta_{0}}}\right.}
 \in\class\cap\br{C^{\infty}\times C^{\infty}}$ finishes the proof.
\end{proof}

Let $\Hr^{2}\rzd$ denote the (open) subset of regular (i.e., {non-vanishing derivative})
curves in $H^{2}\rzd$.

\begin{lemma}\label{lem:repara}
 The operator $\Hr^{2}\rzd\to\Hr^{2}\rzd$ defined by
  \[ \curve \mapsto \curve\circ\psi_{\curve}^{-1}
 \qquad\text{where}\qquad
 \psi_{\curve}(x)=\frac{L}{\length[\curve]}\length\sq{\left.\curve\right|_{[0,x]}} \]
 that reparametrizes an immersed curve to constant speed
 is continuous with respect to the $H^{2}$-norm.
 Moreover, $\curve\circ\psi_{\curve}^{-1}\in C^{\infty}$ if $u\in C^{\infty}$
 is immersed.
\end{lemma}

A proof can be found in~\cite[Appendix]{reiter};
the argument applies without rescaling and the additional requirement of embeddedness.
It applies to non-periodic intervals $[0,L]$ as well.
The last statement can be derived from the formula.

\section{Discretization}\label{sect:discrete}

Our discretization is based on cubic and linear finite element spaces.
We consider a partition of $[0,L]$ by a set of
nodes $\cN_{h}$ that contains the endpoints $0$ and $L$.
We define nodal bases $\br{\varphi_{z}}_{z\in\cN_{h}}$ and
$\br{\psi_{z,j}}_{z\in\cN_{h}},j=0,1$ with the following properties.
If $z_{\pm}\in\cN_{h}$ are \korr neighboring nodes of $z\in\cN_{h}$
then $\varphi_{z}$ and $\psi_{z,j}$ are supported in $[z_{-},z_{+}]$.
On {the intervals $[z_-,z]$ and $[z,z_+]$}
the functions $\varphi_{z}$ are (affine) linear
with $\varphi_{z}(z)=1$ while $\psi_{z,j}$ are cubic polynomials
satisfying $\psi_{z,j}(z)=\delta_{j,0}$ and $\psi_{z,j}'(z)=\delta_{j,1}$,
$j=0,1$.
We define nodal interpolation operators on $C^{0}([0,L])$
and $C^{1}([0,L])$ respectively by letting
\begin{align*}
 \cI_{h}^{1,0}v &= \sum_{z\in\cN_{h}}v(z)\varphi_{z}, \\
 \cI_{h}^{3,1}w &= \sum_{z\in\cN_{h}}\br{w(z)\psi_{z,0}+w'(z)\psi_{z,1}}.
\end{align*}
Furthermore, we will employ the averaging operator
$\Qh$ which is piecewise defined on any element $[z,z']$
(i.e., $z,z'\in\cN_{h}$ are neighboring) by
\[ \Qh v(x) = \frac1{z'-z}\int_{z}^{z'}v(\xi)\d\xi,
\qquad x\in(z,z'). \]
We have $\norm{v-Q_{h}v}_{L^{\infty}}\le Ch^{\alpha}\sq v_{C^{0,\alpha}}$
for all $v\in C^{0,\alpha}$, $\alpha\in(0,1]$
and $\norm{Q_{h}v_{h}}_{L^{\infty}}\le\norm{v_{h}}_{L^{\infty}}$
for all piecewise
linear functions $v_{h}$ subject to $\cN_{h}$.

Both for the discrete approximation result and the stability of our
numerical scheme presented in Section~\ref{sect:scheme} below it is crucial
to exploit the structure of the {dimensionally reduced} functionals.

To identify convex and concave terms, we observe that
the orthonormality of the frame $\frame=[t,b,d]$ implies
$b'\cdot b=0$ and $b'\cdot t=-b\cdot t'$.
Therefore the integrand of the twisting functional becomes
$(b'\cdot d)^{2} = \abs{b'}^{2} - (b\cdot t')^{2}$.

To obtain a coercivity property (under the restriction $\|b_h\|_{L^\infty}\le 1$)
we set
\begin{equation}\label{eq:theta}
 \theta = \min\big\{\frac{\k}{2},1\big\}\quad\in(0,1]
\end{equation}
which ensures $\k\ge 2 \theta$.
This will allow for controlling (part of) the second term of $I_\rod[y,b]$
by the first one even if $\kappa<2$.
Now we decompose
\begin{align*}
I_\rod[y,b]
&= \frac{\k}{2} \int_0^L |y''|^2 \dv{x}  + \frac{1}{2} \int_0^L (b'\cdot d)^2 \dv{x} \\
&= \frac{\k}{2} \int_0^L |y''|^2 \dv{x} + \frac{\theta}{2} \int_0^L |b'|^2 \dv{x} - \frac{\theta}{2} \int_0^L (b\cdot\curve'')^2 \dv{x} \\
&\quad{}+\frac{1-\theta}{2} \int_0^L (b'\cdot(\curve'\times b))^2 \dv{x}.
\end{align*}

By $\Vh\subset H^{2}\times H^{1}$ we denote the cross product of piecewise cubic and piecewise
linear functions subject to $\cN_{h}$.
With the product finite element space $V_\rod^h$ and the operator $Q_h$ we consider
the following discretization of the minimization problem~\eqref{eq:P} in which
the pointwise orthogonality relation $y'\cdot b =0$ is approximated via a penalty term:
\begin{equation}\label{eq:Ph}
\tag{${\rm P}_\rod^{h,\veps}$} 
\left\{ 
\begin{array}{l}
\text{Minimize}\quad \displaystyle{I_\rod^{h,\veps}[y_h,b_h] =  
\frac{\k}{2} \int_0^L |y_h''|^2 \dv{x} 
+ \frac{\theta}{2} \int_0^L |b_h'|^2 \dv{x}} \\[2mm]
\qquad\qquad \displaystyle{- \frac{\theta}{2} \int_0^L (Q_h b_h \cdot y_h'')^2 \dv{x} 
+ \frac{1}{2\veps} \int_0^L \cI_h^{1,0}[(y_h'\cdot b_h)^2]\dv{x}} \\[2mm]
\qquad\qquad \displaystyle{+ \frac{1-\theta}{2} \int_0^L 
 (b_h'\cdot (y_h'\times Q_h b_h))^2 \dv{x}} \\[2.5mm]
\text{in the set} \ \cA_h = \{(y_h,b_h)\in V_\rod^h: 
L_\bc^\rod[y_h,b_h] = \elbc^\rod, \\[2.5mm]
\qquad\qquad\qquad\qquad\qquad\qquad |y_h'(z)| = |b_h(z)| = 1 \text{ f.a. }z\in \cN_h \big\}.
\end{array}\right.
\end{equation}
{The operators $\Qh$ and $\cI_h^{1,0}$ are {included} in a way
that leads to a \korr simple assembly of the corresponding matrices
and avoids quadrature.}

Note that the constraints are imposed on particular degrees of freedom
which makes the method practical.
In order to ensure $\class_{h}\ne\emptyset$,
the distance between the endpoints of $\curveh$
must be less than $L$ and the spatial mesh
has to be chosen sufficiently fine.

For establishing the Gamma-convergence result, it is useful to write
\begin{align*}
 I_\rod[y,b] &= \frac{\k-\theta}{2} \norm{\curve''}_{\lzwei}^{2} + \frac{\theta}{2} \norm{b'}_{\lzwei}^{2}
 + \frac{\theta}{2} \norm{P_{b}\curve''}_{\lzwei}^{2} + \frac{1-\theta}{2} \norm{b'\cdot(\curve'\times b)}_{\lzwei}^2, \\
 I_\rod^{h,\eps}[y_{h},b_{h}] &= \frac{\k-\theta}{2} \norm{\curveh''}_{\lzwei}^{2} + \frac{\theta}{2} \norm{\bh'}_{\lzwei}^{2}
 + \frac{\theta}{2} \norm{P_{\Qh\bh}\curveh''}_{\lzwei}^{2} \\
 &\quad{} + \frac{1-\theta}{2} \norm{\bh'\cdot(\curveh'\times\Qh\bh)}_{\lzwei}^2
 + \frac{1}{2\veps} \int_0^L \cI_h^{1,0}[(y_h'\cdot b_h)^2]\dv{x}
\end{align*}
where $\proj[b]$ and $\proj[\Qh b_{h}]$ denote the square roots of the positive semidefinite matrices
$\one-b\otimes b$ and $\one-\Qh\bh\otimes\Qh\bh$ respectively.

The only difference to $\energy$ is the penalization term and
the fact that $b$ is replaced by $\Qh\bh$ in the third term
and once in the fourth term.

Our first task is to show that minimizers of $\Eh$ within $\class[h]$
approximate $\energy$-minimizers within $\class$.
In the following statement, we assume that
$\energy$ and $\Eh$ attain the value $+\infty$ outside of $\class$ and $\class[h]$, respectively.

\begin{proposition}\label{prop:gamma}
 As $\eps,h\searrow0$, the functional $I_\rod^{h,\eps}$ Gamma-converges to $I_\rod$
 with respect to the weak $H^{2}\times H^{1}$-topology.
\end{proposition}

{Note that there is no restriction on the ratio of $\eps$ and $h$.}

\begin{proof} 
 To establish the {\emph{Lim-inf inequality}},
 we consider \korr{an arbitrary} sequence
 $\br{(\curve_{h},b_{h})}_{\korr{h>0}}\subset\class[h]$
 and $(\curve,b)\in H^{2}\times H^{1}$
 with $\curve_{h}\rightharpoonup\curve$ in $H^{2}$ and
 $b_{h}\rightharpoonup b$ in $H^{1}$ as $h\searrow0$.
 In particular, we have $\curve_{h}\to\curve$ in $C^{1}$ and
 $b_{h}\to b$ in $C^{0}$.
 
 We have to show that if $\liminf_{{(h,\eps)\searrow0}}\Eh[\curve_{h},b_{h}]<\infty$
 then the limit point $(\curve,b)$ belongs to~$\class$
 (such that the formula for $\energy$ is applicable)
 and the $\liminf$-inequality holds.
 
 As the mesh size tends to zero as $h\searrow0$,
 the condition $\abs{\curve'}=\abs b=1$ is satisfied everywhere
 due to the uniform convergence $\curve_{h}'\to\curve'$,
 $b_{h}\to b$.
 Furthermore $L_{\mathrm{bc}}^{\rod}$ is continuous with respect
 to weak convergence. It remains to verify that $b$ is perpendicular
 to~$\curve'$.
 By the interpolation estimate we have
 \begin{align*}
 &\norm{\br{\curvehk'\cdot\bhk}^{2}-\cI_{h}^{1,0}\sq{\br{\curvehk'\cdot\bhk}^{2}}}_{\lzwei}
 \le Ch\norm{\br{\br{\curvehk'\cdot\bhk}^{2}}'}_{\lzwei} \\
 &\le Ch\norm{\curvehk'}_{L^{\infty}}\norm{\bhk}_{L^{\infty}}
 \norm{\curvehk'}_{\heins}\norm{\bhk}_{\heins}.
 \end{align*}
 As $\curvehk'\cdot\bhk\to\curve'\cdot b$ in $C^{0}$, we infer
 \[ \int_{0}^{L}\cI_{h}^{1,0}\sq{\br{\curvehk'\cdot\bhk}^{2}}\d x
 \to \int_{0}^{L} \br{\curve'\cdot b}^{2}\d x. \]
 As all terms of $I_\rod^{h,\eps}$ are non-negative, the term
 $\eps^{-1}\int_{0}^{L}\cI_{h}^{1,0}\sq{\br{\curvehk'\cdot\bhk}^{2}}\d x$
 is uniformly bounded which implies $\int_{0}^{L} \br{\curve'\cdot b}^{2}\d x = 0$.
 This gives $\curve'\perp b$, so $(\curve,b)\in\class$.

 \korr{We use $\norm{\Qh\bh+b}_{L^\infty} \le 2$ to derive
 \begin{align*}
  \norm{P_{b}\curveh''}^{2}
  &=\norm{P_{\Qh\bh}\curveh''}^{2} - \norm{b\cdot\curveh''}^{2} + \norm{\Qh\bh\cdot\curveh''}^{2} \\
  &\le\norm{P_{\Qh\bh}\curveh''}^{2} + 2\norm{\Qh\bh-b}_{L^{\infty}}\norm{\curveh''}^{2}.
 \end{align*}}
 The second term on the right-hand side tends to zero as $h\searrow0$ due to the
 boundedness of weakly converging sequences and
 \begin{equation}\label{eq:Qhbh-b}
 \norm{\Qh \bhk-b}_{L^{\infty}}
 \le\norm{\Qh \bhk-\bhk}_{L^{\infty}} + \norm{\bhk-b}_{L^{\infty}}
 \le C\sqrt{h} + o(1).
 \end{equation}
 This estimate also yields
 $\curveh'\times\Qh\bh\to \curve'\times b$
 which implies
 $\bh'\cdot(\curveh'\times\Qh\bh)\rightharpoonup b'\cdot(\curve'\times b)$
 as $h\searrow0$.
 
 Now the $\liminf$-inequality follows from the
 lower semicontinuity of the $L^{2}$-norm and the fact that
 the penalization term is non-negative
 (since it is the linear interpolation of a non-negative term).
 
 {\korr We turn to the \emph{Lim-sup inequality}.}
 Let $(\curve,b)\in\class$
 and $\delta\in(0,1]$. Of course, $\energy[\curve,b]<\infty$.
 {We aim at constructing a recovery sequence.
 We apply} Lemma~\ref{lem:approx}
 {to} obtain $(\widetilde\curve_{\delta},\widetilde b_{\delta})\in\class\cap\br{C^{\infty}\times C^{\infty}}$
 with $\norm{\widetilde\curve_{\delta}-\curve}_{\hzwei}\le\delta$
 and $\tnorm{\widetilde b_{\delta}-b}_{\heins}\le\delta$.
 We let $\br{\curve_{h},b_{h}}=\br{\cI_{h}^{3,1}\widetilde\curve_{\delta},\cI_{h}^{1,0}\widetilde b_{\delta}}$.
 Owing to the smoothness of the regularized frame, we have
 $\norm{\curve_{h}-\widetilde\curve_{\delta}}_{\hzwei}\le C_{\delta}h$
 and $\tnorm{b_{h}-\widetilde b_{\delta}}_{\heins}\le C_{\delta}h$.
 In particular, $(\curve_{h},b_{h})\in\class[h]$.
 We have to bound $\Eh[\curveh,\bh]-\energy[\curve,b]$ above
 by an expression that tends to zero.
 {For the first three terms of $\Eh[\curveh,\bh]$, w}e obtain
 \begin{align*}
 \norm{\curveh''}_{\lzwei}^{2} - \norm{\curve''}_{\lzwei}^{2}
 &\le \norm{\curveh''+\curve''}_{\lzwei}\norm{\curveh''-\curve''}_{\lzwei} \\
 &\le \br{2\norm{\curve''}_{\lzwei}+\delta+C_{\delta}h}\br{\delta+C_{\delta}h}, \\
 \norm{\bh'}_{\lzwei}^{2} - \norm{b'}_{\lzwei}^{2}
 &\le \br{2\norm{b'}_{\lzwei}+\delta+C_{\delta}h}\br{\delta+C_{\delta}h}, \\
 \norm{\proj[\Qh\bh]\curveh''}_{\lzwei}^{2} - \norm{\proj[b]\curve''}_{\lzwei}^{2}
 &\le\norm{\curveh''}_{\lzwei}^{2} - \norm{\curve''}_{\lzwei}^{2}
 +\norm{b\cdot\curve''}_{\lzwei}^{2} - \norm{\Qh\bh\cdot\curveh''}_{\lzwei}^{2}.
 \end{align*}
 For the {last two terms in the previous line}, we infer
 \begin{align*}
 &\norm{b\cdot\curve''}_{\lzwei}^{2} - \norm{\Qh\bh\cdot\curveh''}_{\lzwei}^{2}
 \le \norm{\Qh\bh\cdot\curveh''+b\cdot\curve''}_{\lzwei}\norm{\Qh\bh\cdot\curveh''-b\cdot\curve''}_{\lzwei} \\
 &\le \br{2\norm{b\cdot\curve''}_{\lzwei} + \norm{\Qh\bh\cdot\curveh''-b\cdot\curve''}_{\lzwei}} \norm{\Qh\bh\cdot\curveh''-b\cdot\curve''}_{\lzwei}
 \end{align*}
 and, recalling~\eqref{eq:Qhbh-b},
 \begin{align*}
 &\norm{\Qh\bh\cdot\curveh''-b\cdot\curve''}_{\lzwei}
 \le\norm{\Qh\bh-b}_{L^{\infty}}\norm{\curveh''}_{\lzwei}
 + \norm{\curveh''-\curve''}_{\lzwei} \\
 &\le\br{C\sqrt h+\delta+C_{\delta}h}\br{\norm{\curve''}_{\lzwei}+\delta+C_{\delta}h}
 + \delta+C_{\delta}h.
 \end{align*}
 {We treat the fourth term of $\Eh[\curveh,\bh]$ similarly
 as above. More precisely, we infer
 \begin{align*}
 &\norm{\bh'\cdot(\curveh'\times\Qh\bh)}_{\lzwei}^2
 - \norm{b'\cdot(\curve'\times b)}_{\lzwei}^2 \\
 &=\norm{\bh'\cdot(\curveh'\times\Qh\bh)
 + b'\cdot(\curve'\times b)}_{\lzwei}
 \norm{\bh'\cdot(\curveh'\times\Qh\bh)
 - b'\cdot(\curve'\times b)}_{\lzwei} \\
 &\le\br{\norm{\bh'}\norm{\curveh'}_{L^{\infty}}\norm{\bh}_{L^{\infty}}
 + \norm{b'}}_{\lzwei}\cdot{}\\
 &\quad{}\cdot\br{\norm{\bh'-b'}\norm{\curveh'}_{L^{\infty}}\norm{\bh}_{L^{\infty}}
 + \norm{b'}\norm{\curveh'-\curve'}_{L^{\infty}}\norm{\bh}
 + \norm{b'}\norm{\Qh\bh-b}}.
 \end{align*}
 Finally we deal with the penalty term.
 Due to $\widetilde\curve_{\delta}{}'\perp\widetilde b_{\delta}$
 we find that
 $\curveh'(z)\cdot\bh(z)=0$ for all $z\in\cN_{h}$.
 Therefore
 $\cI_{h}^{1,0}\sq{\br{\curveh'\cdot\bh}^{2}}(x)=0$
 for all $x\in I$
 which implies that the penalty term vanishes on the entire sequence.}
\end{proof}

\section{Iterative minimization}
\label{sect:scheme}

\input{bend_tor_iter}

\input{experiments}

\subsection*{Acknowledgments}

S\"oren Bartels acknowledges support via the DFG Research
Unit FOR 3013 {\em Vector- and Tensor-Valued Surface PDEs}.
Philipp Reiter was partially supported by
the German Research Foundation (DFG), Grant RE~3930/1--1.

\bibliographystyle{abbrvhref}
\bibliography{bib_project}

\end{document}

%% file: sb_macros.tex
\def\R{\mathbb{R}}

\def\Z{\mathbb{Z}}

\def\cA{\mathcal{A}}

\def\cE{\mathcal{E}}
\def\cF{\mathcal{F}}

\def\cI{\mathcal{I}}

\def\cN{\mathcal{N}}

\def\cS{\mathcal{S}}
\def\cT{\mathcal{T}}

\def\b{\beta}

\def\d{\delta}

\def\k{\kappa}

\def\p{\partial}

\def\veps{\varepsilon}

\def\Eh{\mathcal{E}_h}

\def\Vh{V_\rod^h}

\def\Qh{{Q}_h}

\newcommand{\dv}[1]{\,{\mathrm d}#1}

\newcommand{\wcheck}[1]{#1\hspace{-.8ex}\mbox{\huge {\lower.45ex \hbox{$\textstyle \check{}$}}} \hspace{.5ex}}

\DeclareMathOperator{\supp}{supp}

\DeclareMathOperator{\id}{id}
\let\oldmarginpar\marginpar
\renewcommand\marginpar[1]{
  \oldmarginpar[\raggedleft\footnotesize #1]
  {\raggedright\footnotesize #1}}
\newtheorem{definition}{Definition}
\newtheorem{lemma}[definition]{Lemma}
\newtheorem{proposition}[definition]{Proposition}

\newtheorem{remark}[definition]{Remark}

\newtheorem{algorithm}[definition]{Algorithm}

\numberwithin{definition}{section}

\def\bc{{\rm bc}}
\def\rod{{\rm rod}}
\def\cb{c_{\rm b}}
\def\ct{c_{\rm t}}
\def\hI{\widehat{I}}

\def\tk{{\widetilde{k}}}

%% file: pr_macros.tex
\usepackage{esint}

\hypersetup{
breaklinks=true,
colorlinks=true,
linkcolor=blue,
citecolor=blue,
urlcolor=blue,
}

\newcommand{\abs}[1]{\ensuremath{\left\lvert#1\right\rvert}}
\newcommand{\br}[1]{\ensuremath{\left(#1\right)}}

\newcommand{\norm}[1]{\ensuremath{\left\lVert#1\right\rVert}}
\newcommand{\tnorm}[1]{\ensuremath{\lVert\mbox{$#1$}\rVert}}

\renewcommand{\sp}[1]{\ensuremath{\left\langle#1\right\rangle}}

\newcommand{\sq}[1]{\ensuremath{\left[#1\right]}}

\definecolor{labelkey}{gray}{.5}

\DeclareMathOperator{\am}{am}

\newcommand{\bh}{b_{h}}
\newcommand{\bhk}{b_{h}}

\DeclareMathOperator{\cn}{cn}
\newcommand{\ccN}{{N}}

\newcommand{\curve}{y}
\newcommand{\curveh}{y_{h}}
\newcommand{\curvehk}{y_{h}}

\renewcommand{\d}{\ensuremath{\,\mathrm{d}}}

\newcommand{\eps}{\ensuremath{\varepsilon}}

\newcommand\elbc{\ell_{\bc}}

\renewcommand{\frame}{F}

\newcommand{\heins}{H^{1}}
\newcommand{\hzwei}{H^{2}}

\renewcommand{\kappa}{\varkappa}
\DeclareMathOperator{\Lk}{Lk}

\newcommand{\lzwei}{}
\renewcommand{\rho}{\ensuremath{\varrho}}
\newcommand\rzd{((0,L),\R^{3})}

\newcommand{\TP}{\mathrm{TP}}

\DeclareMathOperator{\Tw}{Tw}
\DeclareMathOperator{\Wr}{Wr}

\newcommand{\energy}{I_\rod}
\newcommand{\class}[1][\empty]{\cA\ifx#1\empty\else_{#1}\fi}
\newcommand{\cbend}{c_{\mathrm b}}
\newcommand{\ctwist}{c_{\mathrm t}}
\renewcommand{\Eh}{I_\rod^{h,\veps}}
\newcommand{\proj}[1][]{P_{#1}}
\newcommand{\Hr}{H_{\mathrm{r}}}
\newcommand{\length}{\mathscr{L}}
\newcommand{\ini}{\mathrm{ini}}

\makeatletter
\DeclareRobustCommand\widecheck[1]{{\mathpalette\@widecheck{#1}}}
\def\@widecheck#1#2{%
    \setbox\z@\hbox{\m@th$#1#2$}%
    \setbox\tw@\hbox{\m@th$#1%
       \widehat{%
          \vrule\@width\z@\@height\ht\z@
          \vrule\@height\z@\@width\wd\z@}$}%
    \dp\tw@-\ht\z@
    \@tempdima\ht\z@ \advance\@tempdima2\ht\tw@ \divide\@tempdima\thr@@
    \setbox\tw@\hbox{%
       \raise\@tempdima\hbox{\scalebox{1}[-1]{\lower\@tempdima\box
\tw@}}}%
    {\ooalign{\box\tw@ \cr \box\z@}}}
\makeatother


%% file: intro.tex
\section{Introduction}

Long slender objects---such as springy wires made of plastic or metal---can be approximated by curves.
In many cases, equilibrium shapes are characterized in terms of the
bending energy, i.e., {(half of)} the total squared curvature.
The latter has a long history, dating back to Bernoulli,
and can be seen as the starting points of elasticity theory.

The bending energy depends just on the centerline of an object
and does not incorporate other physical effects such as
twisting, friction, or shear. For instance,
only relying on the bending energy one
cannot explain why a telephone cable tends to curl.
It also does not preclude self-penetration.

In this paper we extend the study of inextensible elastic curves
by the first author~\cite{Bar13}
to inextensible and unshearable elastic rods.
To this end we discretize
the minimization problem
\begin{equation}\label{eq:P}
\tag{${\rm P}_\rod$} 
\left\{ 
\begin{array}{l}
\text{Minimize}\quad \displaystyle{I_\rod[y,b] =  
\frac{{\cb}}{2} \int_0^L |y''|^2 \dv{x} 
+ \frac{{\ct}}{2} \int_0^L \br{b'\cdot(\curve'\times b)}^2 \dv{x}} \\[3mm]
\text{in the set} \ \cA = \{(y,b)\in H^{2}\times H^{1}: 
L_\bc^\rod[y,b] = \elbc^\rod, \\[2.5mm]
\qquad\qquad \qquad\qquad\qquad\qquad |y'| = |b| = 1, \curve'\perp b \big\}.
\end{array}\right.\!\!\!\!\!
\end{equation}
and devise a numerical
scheme in order to simulate a suitable $H^{2}$-like gradient flow.

Here
{$\cb,\ct>0$
are bending and torsion rigidities that are determined 
by the Lam\'e coefficients of the material and geometrical properties of the rod.
Furthermore,} $L_{\mathrm{bc}}^{\rod}: H^{2}\times H^{1}\to Y$
encodes the boundary data $\elbc^{\rod}$
in some finite-dimensional linear space $Y$.
We assume that it
only involves linear combinations of boundary points
of $\curve$, $\curve'$, and $b$.
Therefore $L_{\mathrm{bc}}^{\rod}$ is continuous with respect to weak convergence in $H^{2}\times H^{1}$.
In particular $L_{\mathrm{bc}}^{\rod}$
can be used to incorporate periodicity in case of a closed curve~$\curve$.

{For ease of readability, we will rescale $I_{\rod}$ by $1/\ct$
from now on and abbreviate
\[ \k = {\cb}/{\ct}. \]}

We will always assume $\cA$ to be nonempty which
is guaranteed if the boundary data $\elbc^{\rod}$
{is compatible with the frame condition and}
implies that the
distance between the endpoints is strictly less than $L$.
Any boundary data on a frame $F$ can be matched by adjusting a reference frame $F_{0}$ using a suitable cumulative angle function $\varphi$
(see Section~\ref{sect:reframe} below).

\bigskip

\subsection*{Elastic rods}

Based on the work of Mora and M\"uller~\cite{MorMul03}
{for general rods},
the minimization problem~\eqref{eq:P} can be {rigorously} derived from a general three-dimensional
hyperelastic model,
{see~\cite{Bar19} for a short formal derivation}.
{In the situation of rods with circular cross-section,
made of some isotropic and homogeneous material,
we find} {that}
$\cb\ge2\ctwist$.
According to Coleman and Swigon~\cite[p.~195]{CS1}
there is some indication that
values less than $\kappa=\tfrac32$ are appropriate for modeling DNA.

The study of elastic rods has a long history.
It is closely related to elasticae, i.e., stationary points
of the bending energy, see Levien~\cite{levien} and references therein.
A comprehensive presentation on the subject from the perspective
of elasticity theory is provided by Antman~\cite{antman}.

We find applications in different fields such as
the modeling of
coiling and kinking of submarine cables (Zajac~\cite{zajac}; Goyal et al.~\cite{GPL1,GPL2}),
cell filaments (Manhart et al.~\cite{MOSS15}),
and computer graphics (Bergou et al.~\cite{BWRAG}; Spillmann and Teschner~\cite{ST}).
Modeling in molecular biology has stimulated a lot of activity
in this field as well.

A prototypical model for DNA supercoiling which has received considerable attention
is the \emph{twisted elastic ring}
investigated by Maddocks in various collaborations~\cite{DLM,KM,KM,maddocks,MM,MRM}.
The solution of the corresponding minimization problem
leads to an intrinsically straight uniform rod with equal bending stiffness.
The analysis bases on Hamiltonian formulations of rod mechanics.
The general idea is to impose a twist rate $\beta$ on a unit loop.

For small values of $\beta$ the round circle (with a uniform twist) remains an equilibrium.
This phenomenon is known as \emph{Michell's instability}~\cite{michell},
see Goriely~\cite{goriely} and references therein. Larger perturbations
lead to instability and bifurcation phenomena,
cf.\@ Goriely and Tabor~\cite{GT1,GT2}.

Ivey and Singer~\cite{ivey-singer} reconsidered the problem from a variational
point of view, obtaining a complete description of the space of closed and quasiperiodic minimizers.
Recently, a reformulation in terms of symplectic geometry has been given by Needham~\cite{needham}.
Regarding the discretization of elasticae we refer to Scholtes et al.~\cite{SSW}.

\subsection*{Gradient flows}

We aim at numerically detecting configurations of framed curves
with low bending and twisting energy. For this we consider the
gradient flow of \korr the energy functional in~\eqref{eq:P},
a weighted sum of an elastic bending 
energy term and a functional that tracks the twisting of the frame
about its centerline.

{Recently}, gradient flows involving the bending energy have received much attention,
with respect to rigorous analysis, see Dziuk et al.~\cite{DzKuSc02}
as well as regarding discretization aspects,
see Deckelnick and Dziuk~\cite{DeDz09},
Barrett et al.~\cite{BaGaNu12},
Bartels~\cite{Bar13},
Dall'Acqua et al.~\cite{DaLiPo14},
Pozzi and Stinner~\cite{PozSti17}.
Lin and Schwetlick~\cite{LinSchw05} also include frames in their model.

We implement a constraint ensuring that the curves stay close to arclength parametrization
if the initial curve is arclength parametrized.
Moreover the bending energy can be replaced by the squared $L^{2}$ norm
of the second derivative of the curve
which is a crucial point in the analysis of the discretization.

\subsection*{Impermeability}

Based on earlier work aiming at modeling DNA plasmids~\cite{coleman92,CST,coleman96,TSC},
Coleman and Swigon~\cite{CS1} take self-contact phenomena into account and
discuss the interaction between certain topological quantities
such as writhe, excess link, and the number of self-contact points.
In~\cite{CS2} they include the case of (two-bridge) torus knots.
In contrast to a related approach by Starostin~\cite{starostin} they impose a (small) positive thickness.

The corresponding case of open curves with appropriate boundary conditions
has also been studied by several authors.
Van der Heijden et al.~\cite{vdHNGT} provide a
comprehensive study of jump phenomena in clamped rods with and without self-contact.
A more detailed classification of the respective equilibrium configurations
is given by Neukirch and Henderson~\cite{NH}.
Clauvelin, Audoly, and Neukirch~\cite{CAN} modeled
the situation of a small loosely knotted arc with open end-points
and studied the shape of the set of self-contact and the
influence of twist applied to the end-points as well.
Starostin and van der Heijden~\cite{SvdH} model
the situation of so-called two-braids,
i.e., structures formed by two elastic rods winding around each other,
which also covers the case of $(2,b)$-torus knots~\cite{SvdH2,SvdH18}.
The dynamic evolution of intertwined clamped loops
subject to varying loads
has been addressed by Goyal et al.~\cite{GPL2,GPL1}.

\bigskip

Some of the above-mentioned models involve
initial assumptions on the geometry,
especially regarding the contact situation,
focussing on explicit constructions for modeling and simulation.
Our approach of treating~\eqref{eq:P} does not rely on any precondition.

We will redefine~\eqref{eq:P} in Section~\ref{sect:imper}
to incorporate impermeability. To this end, we rely on
the tangent-point energies whose impact on the evolution
of (unframed) curves has been discussed in~\cite{bartels-reiter}.
We thereby extend a regularization ansatz due to von der Mosel~\cite{vdM:meek} with
O'Hara's energies~\cite{oha:en-fam} in place of the tangent-point functional.
In fact, one might conjecture that any self-avoiding functional
will qualitatively produce the same results.

Computationally, this case is particularly challenging since strong
forces related to bending effects have to by compensated by repulsive
forces related to the tangent-point functional to avoid self-intersections.
Regularization approaches guaranteeing global injectivity
have been successfully implemented in different fields,
see Kr\"omer and Valdman~\cite{kroemer-valdman} for an example in
the context of elasticity.

The existence of curves minimizing~\eqref{eq:P}
{follows via the direct method of the calculus of variations
or, equally, from the Gamma-convergence (Proposition~\ref{prop:gamma})
together with the coercivity of the functionals.
In the presence of uniform thickness bounds, however,
we cannot rely on these reasoning, see}
Gonzalez et al.~\cite{gmsm}.
While~\eqref{eq:P} covers the ``uniform symmetric case'' of the Kirchhoff rod,
which constitutes maybe the simplest model that involves both bending and twisting,
the setting discussed in~\cite{gmsm} offers more flexibility
and especially also covers the cases of extensible shearable rods.
Schuricht and von der Mosel~\cite{heiko2} derived the Euler--Lagrange equations
for elastic rods with self-contact.
A similar approach has been followed by Hoffman and Seidman~\cite{HS1,HS2}.

The evolution of impermeable rods preserves isotopy classes,
so topology aspects come into play.
Here we will encounter a more involved picture compared
to the analysis of the twist-free setting, see Langer and Singer~\cite{langer-singer:cs} and Gerlach et al.~\cite{GRvdM}.
A complete characterization of minimizers is wide open.

\subsection*{Outline}

We review the geometry of elastic rods
in Section~\ref{sect:rods}.
In Section~\ref{sect:approx} we derive an approximation result
(Lemma~\ref{lem:approx})
that is used in Section~\ref{sect:discrete}
in order to prove Gamma-convergence of the discrete problem
to the continuous one (Proposition~\ref{prop:gamma}).
We prove stability of the numerical scheme in Section~\ref{sect:scheme}
(Proposition~\ref{prop:stability}).
Several experiments discussed in Section~\ref{sect:exper}
indicate a complex energy landscape.

\subsection*{Notation}
The inner products corresponding to $L^{2}$, $H^{1}$, $H^{2}$ are denoted by
$\br{\cdot,\cdot}_{\lzwei}$,
$\br{\cdot,\cdot}_{\heins}$,
and $\br{\cdot,\cdot}_{\hzwei}$, respectively.
The norms are written accordingly.
Constants may change from line to line.

%% file: bend_tor_iter.tex
We linearize the pointwise constraints in our iterative algorithm for computing
minimizers of $I_\rod^{h,\veps}$. For a vector field $y_h\in \cS^{3,1}(\cT_h)^3$ we set 
\[
\cF_h[y_h] = \{ w_h \in \cS^{3,1}(\cT_h)^3: \, L_{\bc,y}^\rod[w_h] = 0,
\, y_h'(z)\cdot w_h'(z) = 0 \text{ f.a. } z\in \cN_h\}
\]
while for a vector field $b_h\in \cS^{1,0}(\cT_h)^3$ we define
\[\begin{split}
\cE_h[b_h] = \{v_h\in  \cS^{1,0}&(\cT_h)^3: L_{\bc,b}^\rod[v_h]=0, \, v_h(z)\cdot b_h(z) = 0
\text{ f.a. } z\in \cN_h \}.
\end{split}\]
The functionals $L_{\bc,y}^\rod$ and $L_{\bc,b}^\rod$
are the components of $L_{\bc}^\rod$
corresponding to the variables $y$ and $b$, respectively,
{assuming that the boundary conditions can be appropriately separated}.

In what follows we abbreviate the 
partial torsion term weighted by the factor $(1-\theta)$ by the 
nonlinear operator
\[
\ccN_h[y_h,b_h] = \frac{1-\theta}{2} \int_0^L (b_h'\cdot (y_h'\times Q_h b_h))^2 \dv{x}.
\]
The nonlinear term $\ccN_h$ does not occur if $\k\ge 2$. In this case 
the negative contribution in the energy functional can be treated using 
its separate concavity properties. Otherwise, an inductive argument is 
used in the stability analysis which requires a different treatment. We therefore set
\[
\tk = \begin{cases}
k & \text{if } \theta = 1, \\
k-1 & \text{if } \theta<1.
\end{cases}
\]

{We abbreviate
\[ \br{v,w}_{h} = \int_0^L \cI_h^{1,0}[v\cdot w] \dv{x}. \]}

We generate a sequence $(y_h^k,b_h^k)_{k=0,1,\dots}$ that
approximates a stationary configuration for $I_\rod^{h,\veps}$ 
with the following algorithm.
{It approximates a gradient flow that is determined
by the metrics $\br{\cdot,\cdot}_{\star}$ and $\br{\cdot,\cdot}_{\dagger}$.
These can be chosen quite general, however,
our stability result relies on certain embeddings, see~\eqref{eq:metrics} below.}

\begin{algorithm}[Gradient descent for elastic rods]\label{alg:descent_rods}
Choose an initial pair $(y_h^0,b_h^0) \in \cA_h$ and a step size $\tau>0$,
set $k=1$. \\
(1) Compute $d_t y_h^k\in \cF_h[y_h^{k-1}]$ such that for all $w_h \in \cF_h[y_h^{k-1}]$ we have
\[\begin{split}
(d_t y_h^k,w_h)_\star + \k ([y_h^k]'', w_h'') 
+ & \veps^{-1}([y_h^k]'\cdot b_h^{k-1},w_h' \cdot b_h^{k-1})_h \\
&=  \theta \big([Q_h b_h^{k-1}]\cdot [y_h^{k-1}]'',[Q_h b_h^{k-1}] \cdot [w_h]''\big) \\
& \qquad - \p_y \ccN_h [y_h^{k-1},b_h^{k-1};w_h]. 
\end{split}\]
(2) Compute $d_t b_h^k\in \cE_h[b_h^{k-1}]$ such that 
for all $r_h\in \cE_h[b_h^{k-1}]$ we have
\[\begin{split}
(d_t b_h^k, r_h)_\dagger 
+ \theta ([b_h^k]',r_h')+ & \veps^{-1}([y_h^k]'\cdot b_h^k,[y_h^k]' \cdot r_h)_h \\
&= \theta \big([Q_h b_h^{k-1}]\cdot [y_h^\tk]'',Q_h r_h \cdot [y_h^\tk]''\big)\\
& \qquad - \p_b \ccN_h[y_h^{k-1},b_h^{k-1};r_h].  
\end{split}\]
(3) Stop the iteration if 
\[
\|d_t y_h^k\|_\star + \|d_t b_h^k\|_\dagger  \le \veps_{\rm stop};
\]
otherwise, increase $k\to k+1$ and continue with~(1).
\end{algorithm}

{ Note that the $\ccN_{h}$-terms on the right-hand sides vanish in case $\theta=1$}
{which corresponds to $\k\ge2$.}

It is useful to view $d_t y_h^k$ and $d_t b_h^k$ as the unknowns
in Steps~(1) and~(2) instead of $y_h^k=y_h^{k-1}+\tau d_t y_h^k$
and $b_h^k = b_h^{k-1}+\tau d_t b_h^k$. 
The algorithm exploits the fact that the penalty term is separately
convex while the nonquadratic contribution to the torsion term is 
separately concave. Therefore, the decoupled semi-implicit 
treatment of these terms is natural and unconditionally energy stable
if $\theta=1$, i.e., in the bending-dominated case. 

\begin{proposition}[Convergent iteration]\label{prop:stability}
Assume that we have
\begin{equation}\label{eq:metrics}
{
\begin{aligned}
\|w_h''\| &\le c_\star \|w_h\|_\star, &
\|r_h'\| &\le c_\dagger \|r_h\|_\dagger, \\
\|w_h'\|_{L^{\infty}} &\le c_\star \|w_h\|_\star, &
\|r_h\|_{L^{\infty}} &\le c_\dagger \|r_h\|_\dagger
\end{aligned}
}\end{equation}
for all $(w_h,r_h)\in V^h_\rod$ with $L_\bc^\rod[w_h,r_h] = 0$.
Algorithm~\ref{alg:descent_rods} is well defined and produces a sequence
$(y_h^k,b_h^k)_{k=0,1,\dots}$ such that for a constant $c_0\ge 0$ and all $L\ge 0$ we have
\[
I_\rod^{h,\veps}[y_h^L,b_h^L] + {\tau}(1-c_0 \tau) \sum_{k=1}^L \big(\|d_t y_h^k\|_\star^2
+ \|d_t b_h^k\|_\dagger^2 \big) \le I_\rod^{h,\veps}[y_h^0,b_h^0].
\]
If $c_0 \tau \le 1/2$ then the iteration is energy decreasing, convergent, 
and the unit-length violation is controlled via 
\[
\max_{k=0,\dots,L} {\Big(}\||[y_h^k]'|^2-1\|_{L^\infty} + \||b_h^k|^2-1\|_{L^\infty} {\Big)}
\le \tau c_{\star,\dagger} e_{0,h},
\]
where $e_{0,h} = I_\rod^{h,\veps}[y_h^0,b_h^0]{{}<\infty}$
{and $c_{\star,\dagger}>0$ only depends on the metrics}.

If $\theta = 1$ then we may choose $c_0 =0$,
{i.e., stability and convergence hold
unconditionally}.
\end{proposition}

\begin{remark}
In case $\theta=1$, i.e., in case of low torsion rigidity,
{we actually only require the second line of~\eqref{eq:metrics}.
Moreover, if we even replace it by}
$\|w_h'\| \le c_\star \|w_h\|_\star$ and 
$\|r_h\| \le c_\dagger \|r_h\|_\dagger$
the estimates for the constraint
violation {still} hold in $L^1$ instead of $L^\infty$.
\end{remark}

{In general, the condition \eqref{eq:metrics} can be satisfied
if $\norm{\cdot}_{\star}$ and $\norm{\cdot}_{\dagger}$ are $H^{2}$- and $H^{1}$-seminorms
and if $L_\bc^\rod$ imposes suitable Dirichlet conditions on one
endpoint of the interval.
An $L^{2}$-gradient flow, however, requires stronger assumptions on the step size.}

\begin{proof}[Proof of Proposition~\ref{prop:stability}]
(a) We first consider the case $\theta=1$
for which $\tk=k$ and the nonlinearities related to the operator $\ccN_h$ disappear and the
asserted estimate holds with $c_0=0$. For this we note that the functional
\[
G_h[y_h,b_h] =  \frac{\theta}{2} \int_0^L (Q_h b_h \cdot y_h'')^2 \dv{x}
\]
is separately convex, i.e., convex in $y_h$ and in $b_h$. Therefore, we have that
\[\begin{split}
\p_y G_h[y_h^{k-1},b_h^{k-1};y_h^k-y_h^{k-1}] + G_h[y_h^{k-1},b_h^{k-1}] & \le G_h[y_h^k,b_h^{k-1}], \\
\p_b G_h[y_h^k,b_h^{k-1};b_h^k-b_h^{k-1}] + G_h[y_h^k,b_h^{k-1}] &\le G_h[y_h^k,b_h^k],
\end{split}\]
which by summation leads to the inequality 
\[
\p_y G_h[y_h^{k-1},b_h^{k-1};d_t y_h^k] + \p_b G_h[y_h^k,b_h^{k-1};d_t b_h^k]
\le d_t G_h[y_h^k,b_h^k].
\]
Similarly, the functional 
\[
P_{h,\veps}[y_h,b_h] = \frac{1}{2\veps} \int_0^L \cI_h^{1,0}[(y_h' \cdot b_h)^2] \dv{x}
\]
is separately convex and we have
\[
\p_y P_{h,\veps}[y_h^k,b_h^{k-1};d_t y_h^k] + \p_b P_{h,\veps}[y_h^k,b_h^k;d_t b_h^k]
\ge d_t P_{h,\veps}[y_h^k,b_h^k].
\]
We choose $w_h = d_t y_h^k$ and $r_h = d_t b_h^k$ 
in the equations of Steps~({1}) and~({2}) of Algorithm~\ref{alg:descent_rods} and find that
\[\begin{split}
\|d_t y_h^k\|_\star^2& + \|d_t b_h^k\|_\dagger^2  
+  d_t \big( \frac{\k}{2} \|[y_h^k]''\|^2 + \frac{\theta}{2} \|[b_h^k]'\|^2 \big) 
+ d_t P_{h,\veps}[y_h^k,b_h^k] \\
&\qquad \qquad\qquad \qquad  + \tau \big(\frac{\k}{2} \|[d_t y_h^k]''\|^2 + \frac{\theta}{2} \|[d_t b_h^k]'\|^2 \big) \\
& \le \p_y G_h[y_h^{k-1},b_h^{k-1};d_t y_h^k] + \p_b G_h[y_h^k,b_h^{k-1};d_t b_h^k] \le d_t G_h[y_h^k,b_h^k].
\end{split}\]
Since for $\theta=1$ we have that 
\[
I_\rod^{h,\veps}[y_h^k,b_h^k] = \frac{\k}{2} \|[y_h^k]''\|^2 + \frac{1}{2} \|[b_h^k]'\|^2 
- G_h [y_h^k,b_h^k] + P_{h,\veps}[y_h^k,b_h^k], 
\]
we deduce the asserted estimate. The nodal orthogonality conditions encoded in the spaces 
$\cF_h[y_h^{k-1}]$ and $\cE_h[b_h^{k-1}]$ lead to the relations
\[\begin{split}
|[y_h^k]'(z)|^2 & = |[y_h^{k-1}]'(z)|^2 + \tau^2 |[d_t y_h^k]'(z)|^2,  \\
|b_h^k(z)|^2 &= |b_h^{k-1}(z)|^2 + \tau^2 |d_t b_h^k(z)|^2.
\end{split}\]
Repeated application leads to 
\[
\big\||[y_h^k]'(z)|^2-1\big\|_{L^\infty} + \big\||b_h^k(z)|^2-1\big\|_{L^\infty} 
= \tau^2 \sum_{\ell=1}^k {\Big(}\|d_t b_h^\ell(z)\|_{L^\infty}^2+\|[d_t y_h^\ell]'\|_{L^\infty}^2{\Big)}.
\]
{Using~\eqref{eq:metrics} and the previously established bound for the discrete time derivatives}
proves the estimate for the constraint violation if $\theta =1$.

(b) We next turn to the case $\theta<1$ and argue by induction over $L\ge 0$. 
Assume that the estimates have been established for some $\widetilde{L} = L-1\ge 0$ with
{$c_0\ge c_{\star,\dagger} e_{0,h}$ (independent of $L$)}; the estimate trivially holds for $L=0$. Let $0\le k\le L$ and
choose $w_h = d_t y_h^k$ and $r_h = d_t b_h^k$ in Algorithm~\ref{alg:descent_rods}.
{From $c_{0}\tau\le1/2$ we infer
\begin{equation}\label{eq:bbound}
 \|[y_h^{k-1}]'\|_{L^\infty}^{2}\le 3/2,
 \qquad\qquad
 \|b_h^{k-1}\|_{L^\infty}^{2}\le 3/2.
\end{equation}}%
Arguing as above we find that
\[\begin{split}
& \|d_t y_h^k\|_\star^2 + \|d_t b_h^k\|_\dagger^2  
+ d_t \Big\{ \frac{\k}{2} \|[y_h^k]''\|^2 + \frac{\theta}{2} \|[b_h^k]'\|^2 
+ P_{h,\veps}[y_h^k,b_h^k]\Big\} \\
& \le G_h'[y_h^{k-1},b_h^{k-1};d_t y_h^k,d_t b_h^k] - \ccN_h'[y_h^{k-1},b_h^{k-1};d_t y_h^k,d_t b_h^k] \\
& \le 
{
\frac12 \big[c_{G'} \big(\|[y_h^{k-1}]''\|^2 + \|[y_h^{k-1}]''\|^4\big)  
+ c_{\ccN'} \big(\|[b_h^{k-1}]'\|^2 + \|[b_h^{k-1}]'\|^4\big)\big]
} \\
& \qquad +\frac12  \big(\|d_t y_h^k\|_\star^2 + \|d_t b_h^k\|_\dagger^2\big)
\end{split}\]
{where we used~\eqref{eq:metrics} and~\eqref{eq:bbound}
as well as $\norm{Q_{h}r_{h}}\le\norm{r_{h}}$
in the last step. Due to~\eqref{eq:theta}, t}%
he bounds at level $k-1$ imply that the first term on the right-hand side
is bounded by a constant $D_{0,h}'$. We consider the reduced energy functional 
\[
\hI_\rod^{h,\veps}[y_h^k,b_h^k] = \frac{\k}{2} \|[y_h^k]''\|^2 + \frac{\theta}{2} \|[b_h^k]'\|^2 
+ P_{h,\veps}[y_h^k,b_h^k] 
\]
and note that
{according to~\eqref{eq:bbound}}
we have
\[
\hI_\rod^{h,\veps}[y_h^{k-1},b_h^{k-1}] \le 4 I_\rod^{h,\veps}[y_h^{k-1},b_h^{k-1}].
\]
We thus deduce from the estimate
\[
\frac12 \big( \|d_t y_h^k\|_\star^2 + \|d_t b_h^k\|_\dagger^2\big)
+ d_t \hI_\rod^{h,\veps}[y_h^k,b_h^k] \le D_{0,h}'
\]
that, imposing the condition $\tau D_{0,h}' \le e_{0,h}$,  
\[
\frac{\tau}{2} \big( \|d_t y_h^k\|_\star^2 + \|d_t b_h^k\|_\dagger^2\big)
+ \hI_\rod^{h,\veps}[y_h^k,b_h^k] \le \tau D_{0,h}' + 4 e_{0,h} \le 5 e_{0,h} =: D_{0,h}.
\]
This estimate and the (unrestricted) coercivity of $\hI_\rod^{h,\veps}$ imply that
\begin{equation}\label{eq:coercivity}
 \|[y_h^k]''\| + \|[b_h^k]'\| \le D_{1,h}.
\end{equation}
To obtain the asserted full energy law we again choose 
$w_h = d_t y_h^k$ and $r_h = d_t b_h^k$ in Algorithm~\ref{alg:descent_rods} 
and note that as in~(a) we find 
\begin{multline*}
\|d_t y_h^k\|_\star^2 + \|d_t b_h^k\|_\dagger^2  
+  d_t I_\rod^{h,\veps}[y_h^k,b_h^k]  \\
\le d_t\ccN_h[y_h^k,b_h^k]- \ccN_h'[y_h^{k-1},b_h^{k-1};d_t y_h^k,d_t b_h^k] \\
    - d_t G_h[y_h^k,b_h^k]+ G_h'[y_h^{k-1},b_h^{k-1};d_t y_h^k,d_t b_h^k].
\end{multline*}
By the mean value theorem the terms on the right-hand side are equal to
\[
\tau \ccN_h''[\xi_h^{(1)},\eta_h^{(1)};d_t y_h^k, d_t b_h^k;d_t y_h^k,d_t b_h^k]
- \tau G_h''[\xi_h^{(2)},\eta_h^{(2)};d_t y_h^k, d_t b_h^k;d_t y_h^k,d_t b_h^k].
\]
Here, $\xi_h^{(\ell)}$ and $\eta_h^{(\ell)}$, $\ell=1,2$, are convex combinations of 
$y_h^{k-1}$ and $y_{h}^k$ and $b_h^{k-1}$ and $b_h^k$, respectively, and using
their uniform bounds%
{~\eqref{eq:bbound} and~\eqref{eq:coercivity} as well as~\eqref{eq:metrics}}
we find that
\[\begin{split}
&\tau \ccN_h''[\xi_h^{(1)},\eta_h^{(1)};d_t y_h^k, d_t b_h^k;d_t y_h^k,d_t b_h^k]
\le \tau  c_{\ccN''} D_{2,\ccN} \big(\|d_t y_h^k\|_\star^2 + \|d_t b_h^k\|_\dagger^2\big), \\
&\tau G_h''[\xi_h^{(2)},\eta_h^{(2)};d_t y_h^k, d_t b_h^k;d_t y_h^k,d_t b_h^k]
\le \tau  c_{G''} D_{2,G} \big(\|d_t y_h^k\|_\star^2 + \|d_t b_h^k\|_\dagger^2\big).
\end{split}\]
With %
{$c_0 \ge c_{\ccN''} D_{2,\ccN} + c_{G''} D_{2,G}$}
we deduce the asserted estimate after
multiplication by $\tau$ and summation over $k=1,2,\dots,L$.
{The second estimate is established as in (a), provided the
step size $\tau$ is bounded accordingly.}
\end{proof}

%% file: experiments.tex
\section{Experiments}\label{sect:exper}

\begin{table}
\begin{tabular}{lcccccccccc}\hline
 Section & $L$ & $N$ & $\kappa$ & $\beta_{\ini}$ & $h_{\max}$ & $\eps$ & bc & $\rho$
  \\\hline
\ref{sect:uniframe}	& $2\pi$ & $100$ & $2$ & $*$ & $0.0628$ & $0.0628$ & c & --- \\
\ref{sect:michell}	& $2\pi$ & $400$ & $3/2$ & $*$ & $0.0157$ & $10^{-5}$ & p & --- \\
\ref{sect:overtwist}& $2\pi$ & $400$ & $2$ & $5$ & $0.0157$ & $0.0010$ & p &  --- \\
\ref{sect:f8}		& $4K(m)$ & $400$ & $*$ & $0$ & $0.0232$ & $0.0232$ & p &  --- \\
\ref{sect:clamped}	& $*$ & $400$ & $2$ & $4$ & $0.0351$ & $0.0010$ & c &  --- \\
\ref{sect:imper} (a)& $2\pi$ & $800$ & $2$ & $5$ & $0.0079$ & $0.0079$ & p &  $0.1$ \\
\ref{sect:imper} (b)& $*$ & $400$ & $2$ & $4$ & $0.0351$ & $0.0010$ & c &  $0.1$ \\
\hline 
\end{tabular}
\bigskip

\caption{Modeling and discretization parameters {for the experiments presented in Section~\ref{sect:exper}.}
{An asterisk refers to details given in the corresponding text.}%
}\label{tab:param}
\end{table}

Here we report on some numerical experiments.
The parameters that have been used are shown in Table~\ref{tab:param}.
The length of the curve is denoted by~$L$,
the number of nodes by~$N$,
the maximum step size $h_{\max}$ is normally close to $L/N$
where $L$ is the length of the curve.
{Except for Experiment~\ref{sect:uniframe},
t}he initial curve has constant twist rate $\beta_{\ini}$.
{While the director~$b$ will always be clamped
on both ends of the rod,
the boundary conditions~(bc) for the curve are either periodic~(p) or
clamped on both ends~(c).
Unless otherwise stated}
the penalization parameter has been chosen to be $\eps=\tfrac{2\pi}N$.
We always use the time step size $\tau = \tfrac1{8}{h_{\max}}$.
In some cases we also added some small perturbation to the initial curve
which is specified in the text.

We briefly comment on the legend of the corresponding energy plots
where the horizontal axis shows the iteration steps of the evolution.
Of course, ``bending'' refers to the {scaled} bending energy $\tfrac\kappa2\norm{\curveh''}_{L^{2}}^{2}$,
``twisting'' to the {scaled} twisting energy
$\frac{\theta}{2} \norm{\bh'}_{\lzwei}^{2}
 - \frac{\theta}{2} \norm{\curveh''\cdot b_{h}}_{\lzwei}^{2}
  + \frac{1-\theta}{2} \norm{\bh'\cdot(\curveh'\times\Qh\bh)}_{\lzwei}^2$
while ``total twist'' means the functional $\Tw$ defined {in Section~\ref{sect:tot-twist}},
``potential'' to the 
tangent-point energy (see {Section}~\ref{sect:imper}), ``penalizing'' to $P_{h,\veps}$,
and ``total'' is $I_\rod^{h,\veps}$.

In general the total twist is scaled differently, i.e.,
there is a second axis on the right margin of the energy plots
while all other values refer to the axis on the left margin
of the coordinate system.
\korr{A coloring of a curve represents curvature values.}

{For typical discretizations with $400$ nodes the overall runtime
of our implementation in Matlab
on a server (3.1~GHz)
took about $3$ minutes per $10{,}000$ iteration steps.
An impermeable rod with $800$ nodes in Experiment \ref{sect:imper} (b),
with an assembly of the self-avoidance potential in C,
requires about $15$ minutes per $10{,}000$ iteration steps.}

{We would like to stress that we observe energy monotonicity
for all our experiments, confirming the stability features
discussed in Proposition~\ref{prop:stability}.}

\subsection{Uniform twist rate}
\label{sect:uniframe}
\newcommand{\bild}[1]{\fbox{\includegraphics[scale=.2,trim=200 480 250 450,clip]{a-#1.jpg}}\makebox[0ex][r]{\makebox[0ex][r]{\tiny#1}\hspace{0ex} }\,\ignorespaces}

\begin{figure}
\raisebox{1ex}{\begin{minipage}[b]{6cm}
\bild{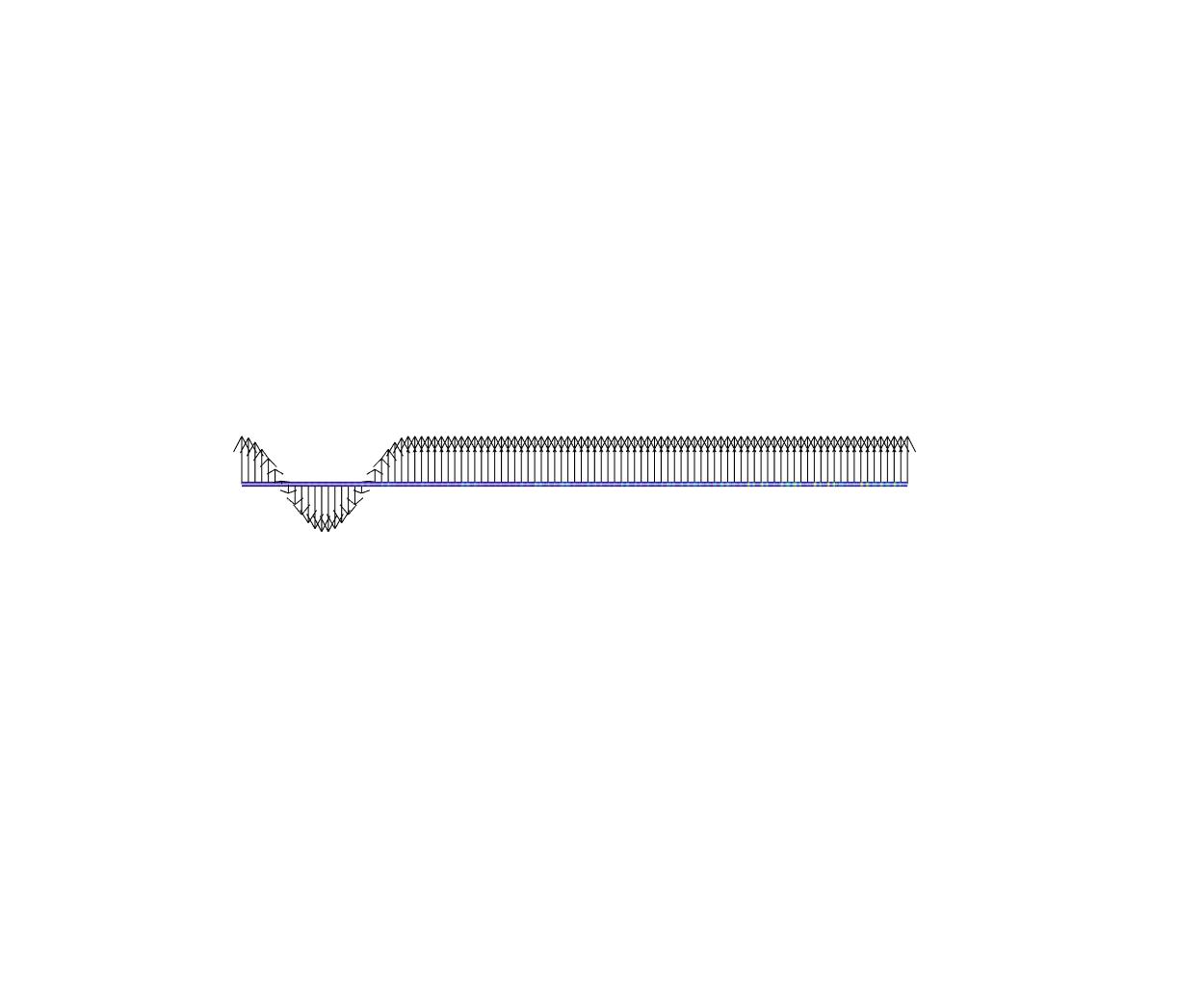}

\bild{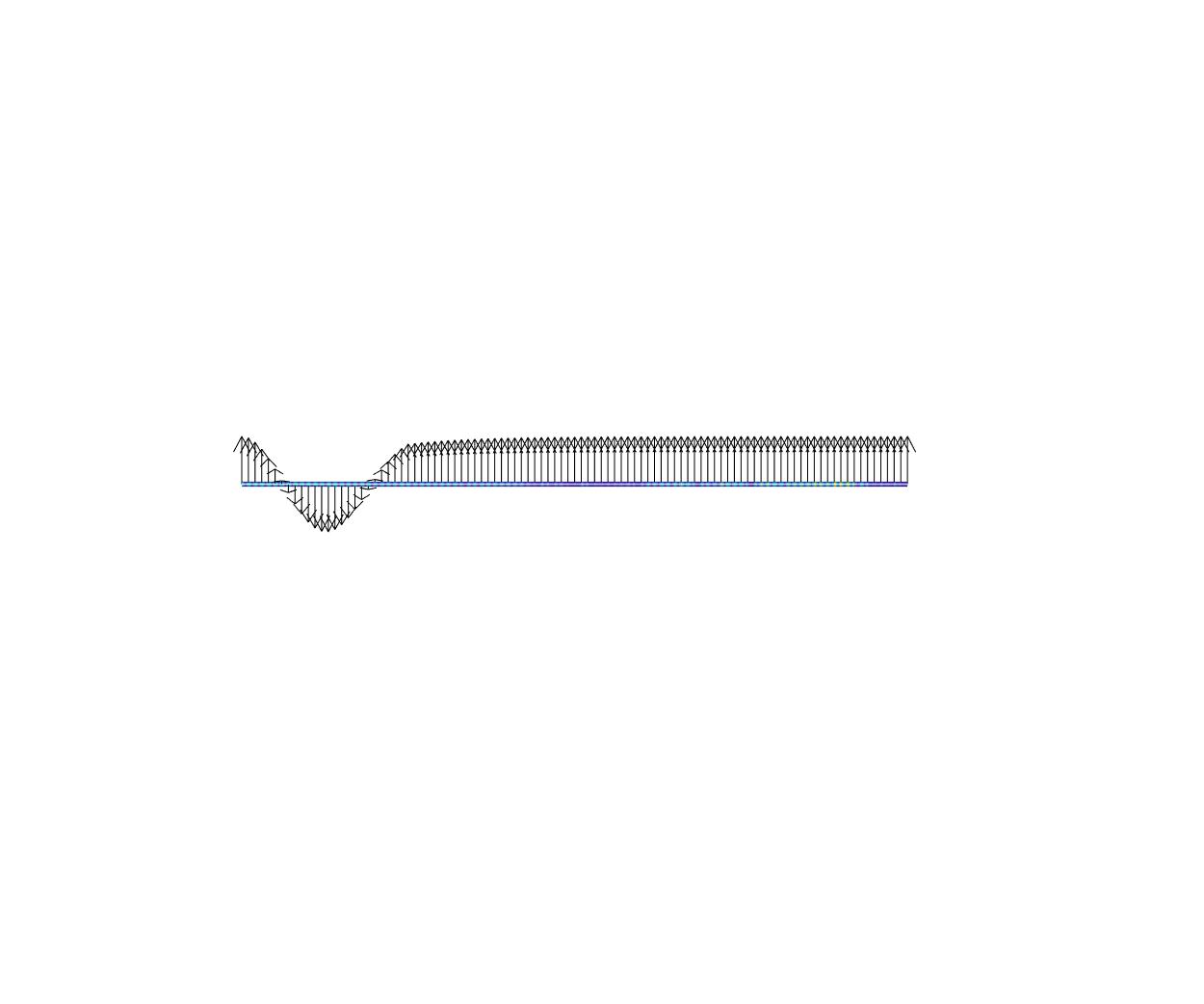}

\bild{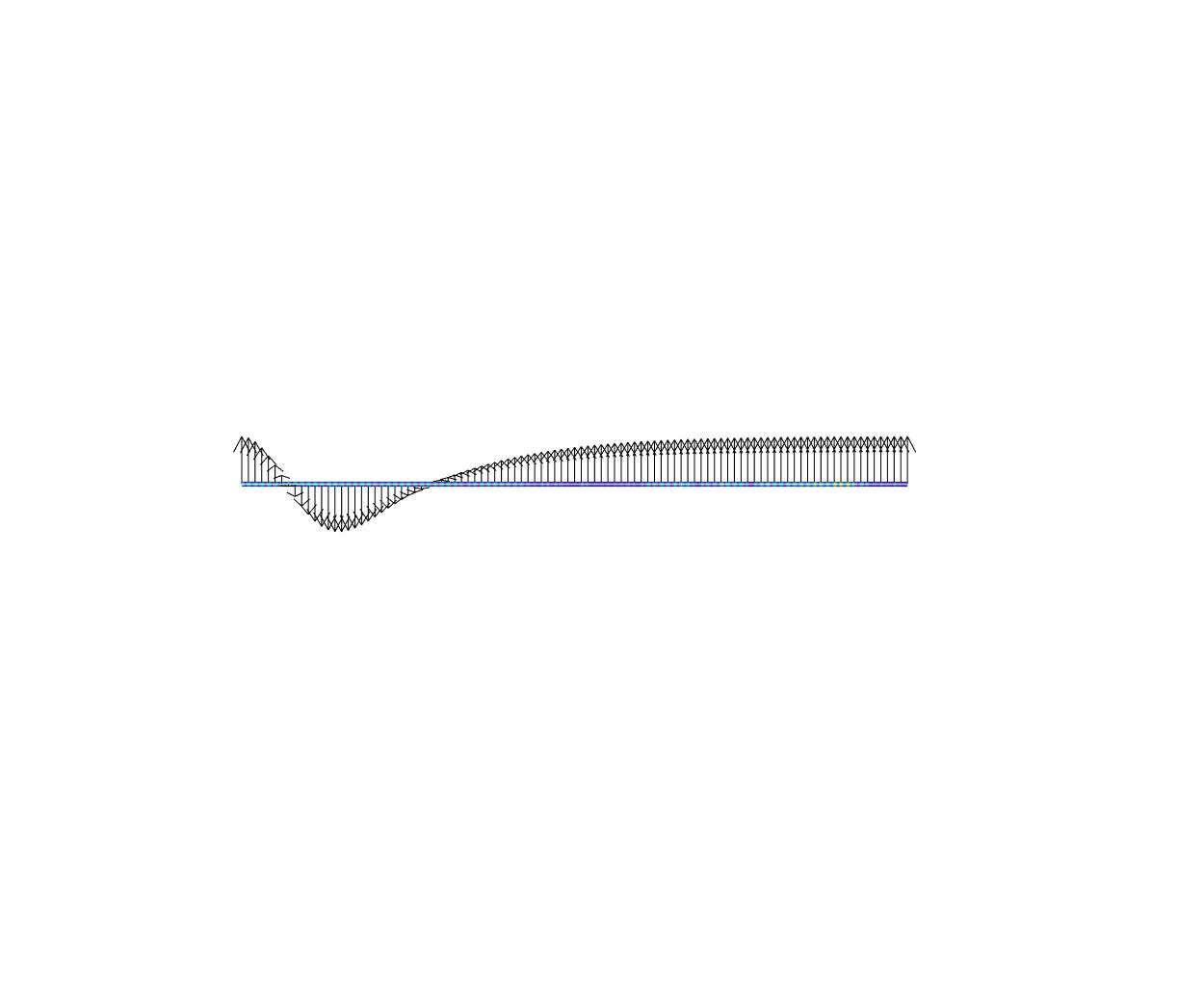}

\bild{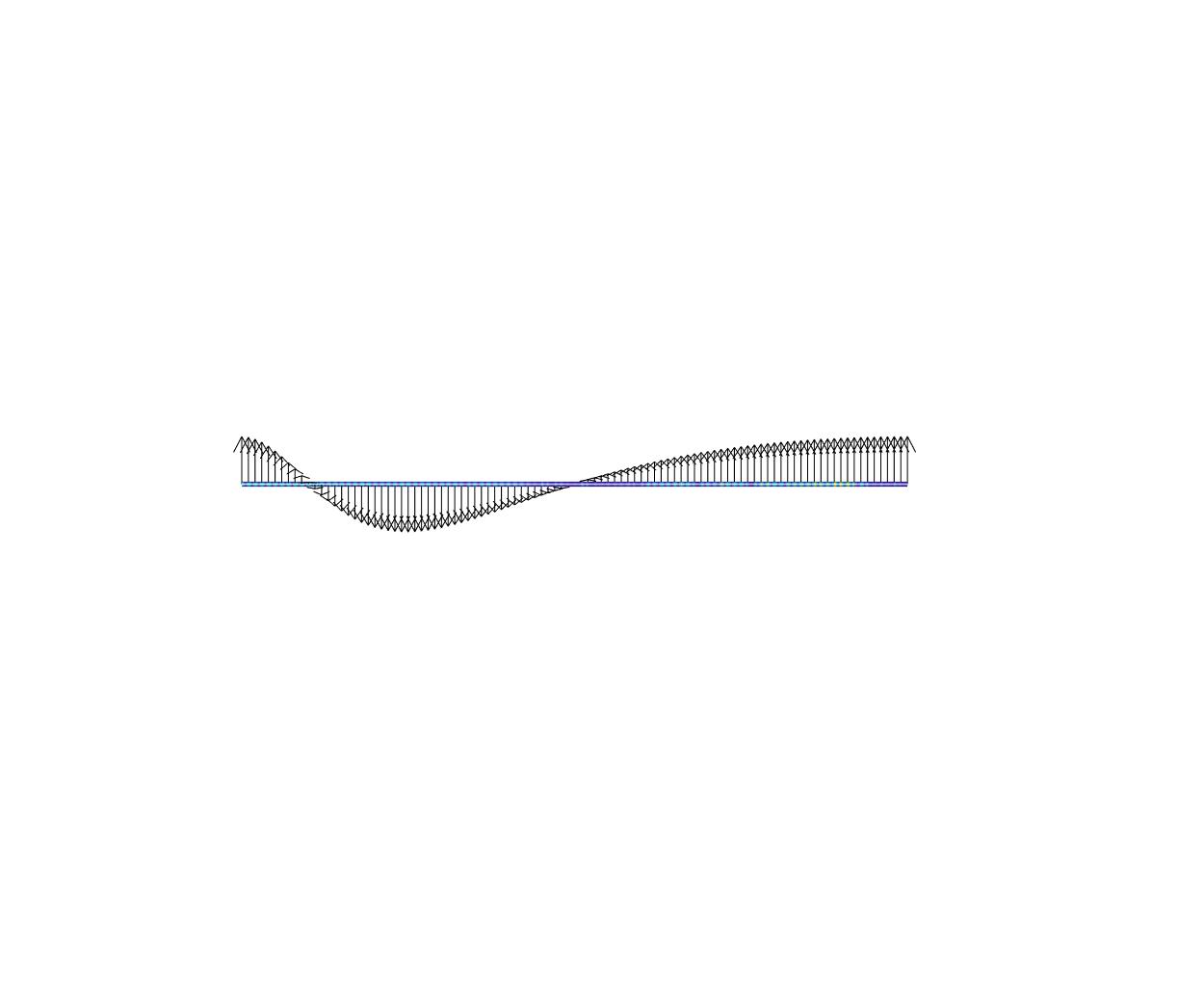}

\bild{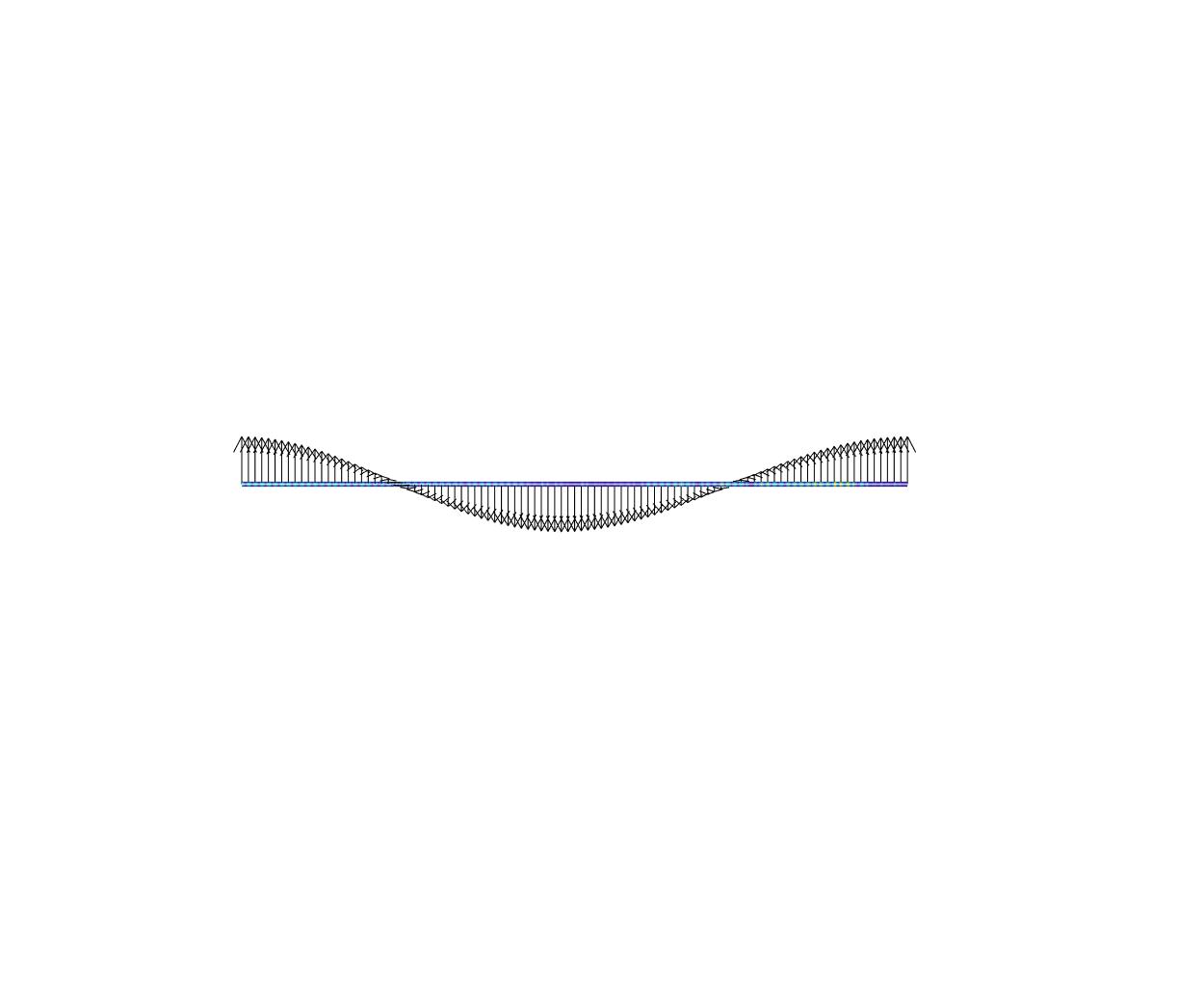}
\end{minipage}}
\hfill
\includegraphics[scale=.33]{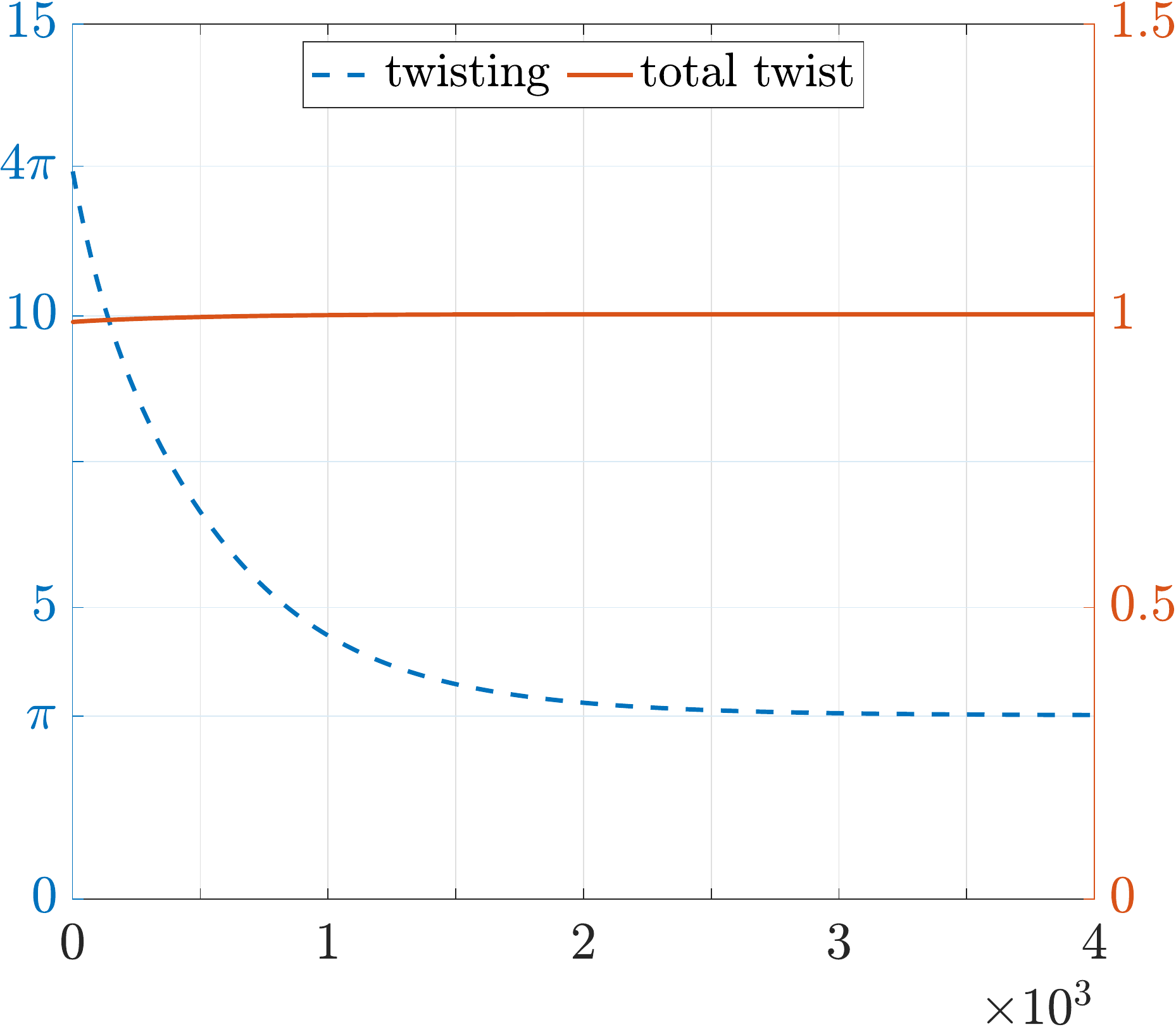}

\caption{{(Experiment~\ref{sect:uniframe})
\newline\emph{Left}:
\korr{States} of the evolution at several iteration steps.
The initial rod consists of a straight line, framed with a non-uniform twist rate.
\newline\emph{Right}: The twisting energy exhibits the dissipation of the twist rate after about 4{,}000 iteration steps
while the total twist is constant.}}
\label{fig:uniframe}
\end{figure}

According to {Section}~\ref{sect:optframe}, stationary frames have
constant twist rate.
We expect that a non-uniform twist configuration will
{become constant within the evolution}.

We start with a straight line curve of length $2\pi$
with uniform twist rate $4$ on $[0,\frac\pi2]$
and $0$ on $[\frac\pi2,2\pi]$.
{No perturbation of the initial rod was used.}

Initial stage and some intermediate steps from the evolution
are visualized in Figure~\ref{fig:uniframe}.
After $4{,}000$ steps the twist rate is nearly constant
and amounts to~$1$.

All energy values are neglectable except for the twisting energy
which virtually coincides with the total energy.
Furthermore, at initial and final step the twisting energy
data matchs quite closely the expected values
of $\tfrac12\int_{0}^{2\pi}\beta(s)^{2}\d s$
which amount to $4\pi$ and $\pi$.

Looking at the following experiments
whose initial configurations all have a uniform twist rate,
{we find this property being violated
in the course of the iteration.
In first place, this is due to the spatial behavior of the curve
which does not seems to allow for an simultaneous reaction
by the director.
Eventually, uniform twist rate is restored, at least when the evolution has reached a stationary configuration,
cf.\@ Section~\ref{sect:optframe}.

We can test for uniform twist rate by computing the quotient
of $\frac{2\pi^{2}}L\Tw^{2}$ over the twisting energy.
As to Experiments~\ref{sect:michell} to~\ref{sect:f8},
throughout the evolutions
this number stays close to~$1$ for most of the time
where we detect} a relative deviation below~$\frac{1}{200}$.
The twist rate for Experiment~\ref{sect:overtwist} is
plotted in Figure~\ref{fig:overplot} (right).
{In Experiment~\ref{sect:f8} we ignore the initial
steps where the twist is zero.
In the other cases,
we also see that the twist rate eventually dissipates 
but it can last a relatively long time}.

\subsection{Michell’s instability}
\label{sect:michell}

\renewcommand{\bild}[1]{\fbox{\includegraphics[scale=.11,trim=280 300 240 300,clip]{b#1.jpg}}\makebox[0ex][r]{\makebox[0ex][r]{\tiny#1}\hspace{0ex} }\,\ignorespaces}

\begin{figure}

\bild{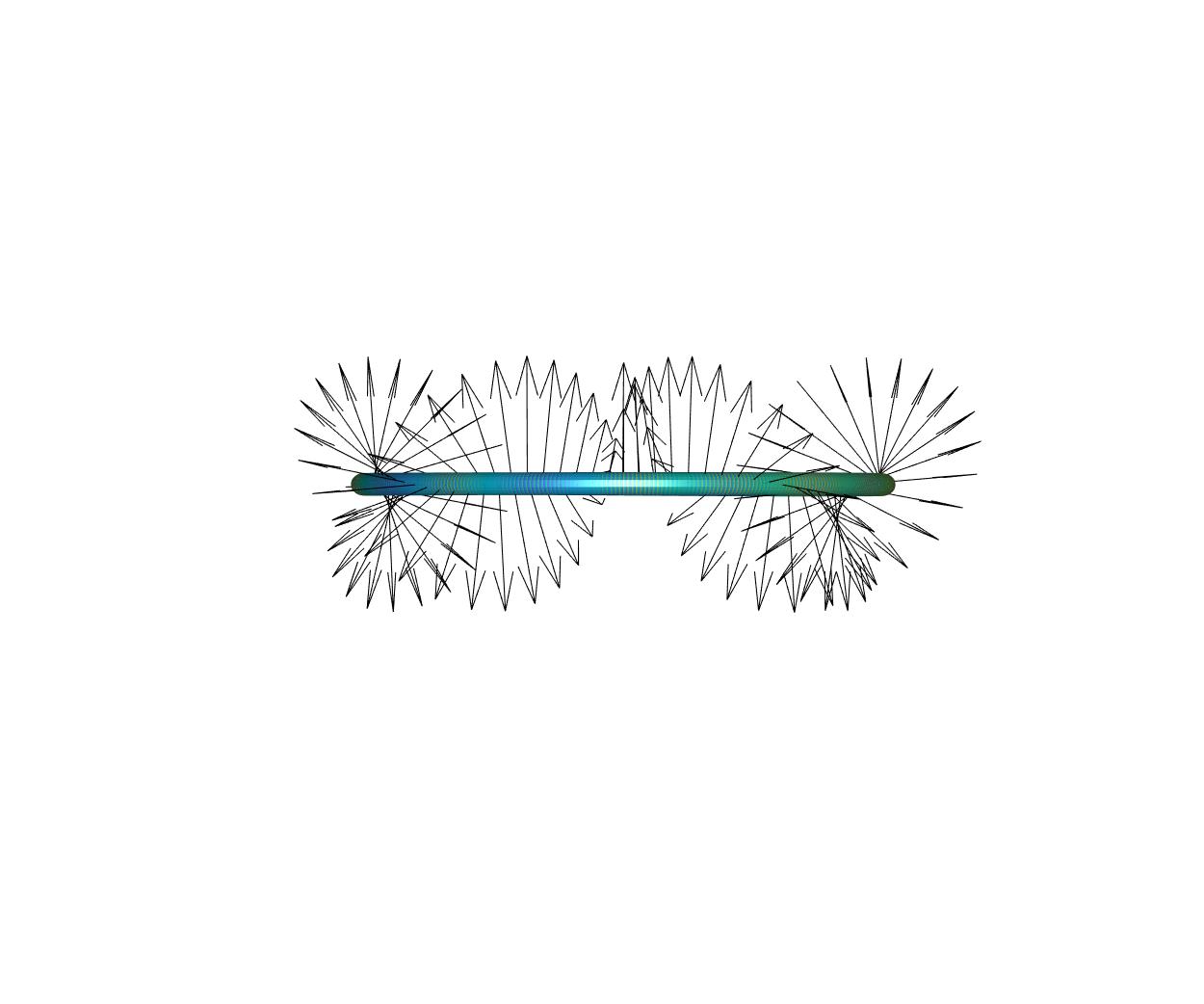}
\bild{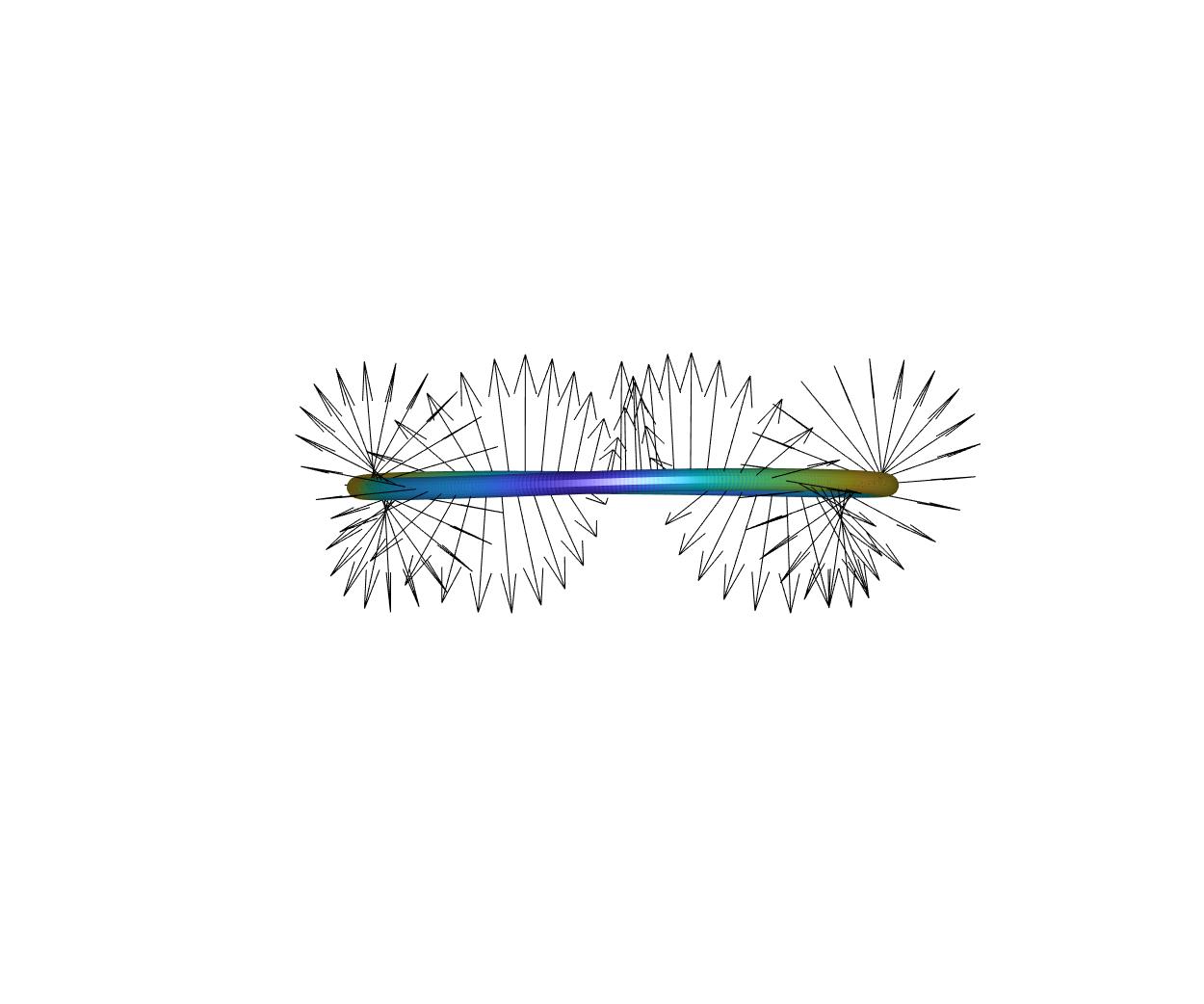}
\bild{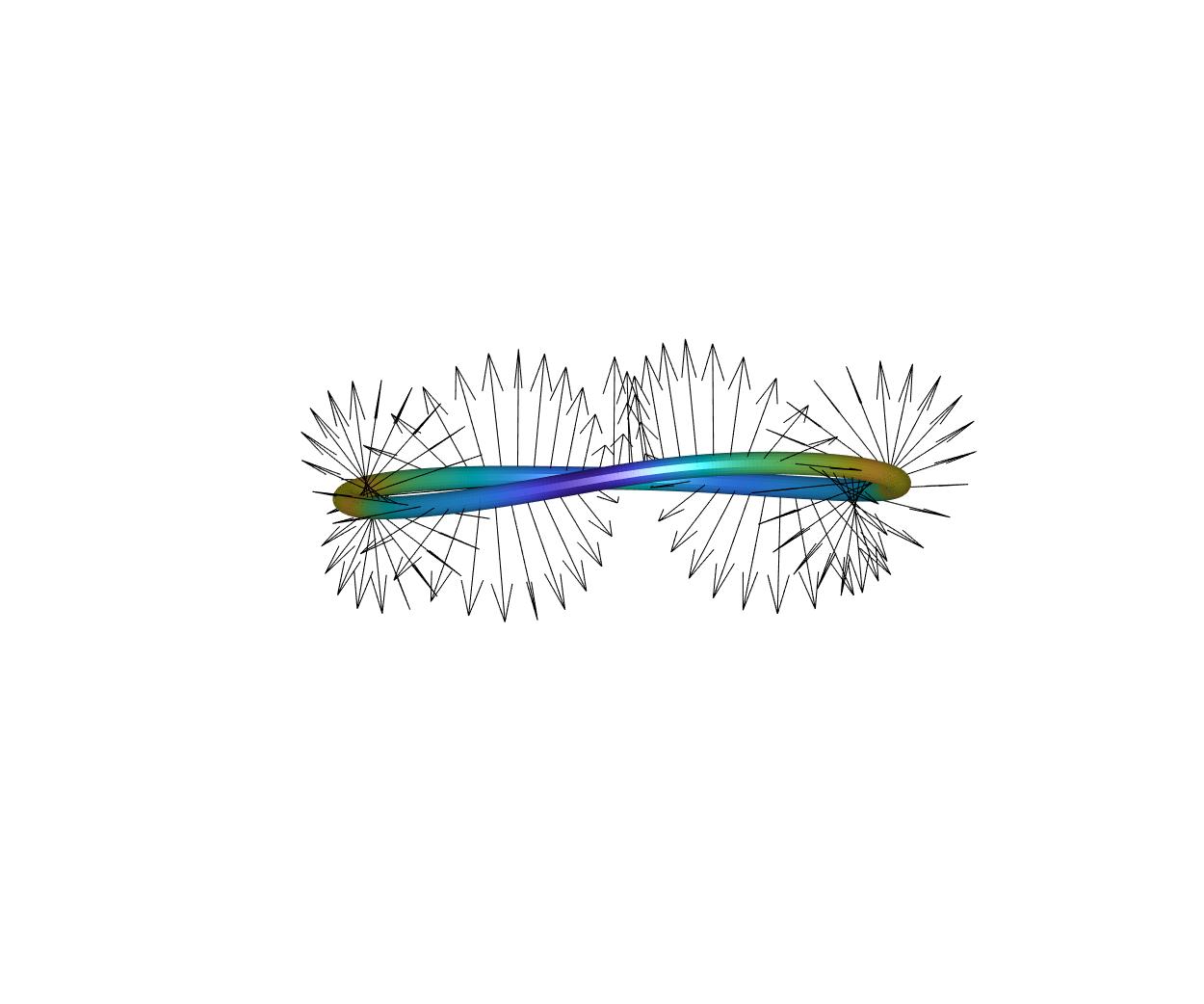}
\bild{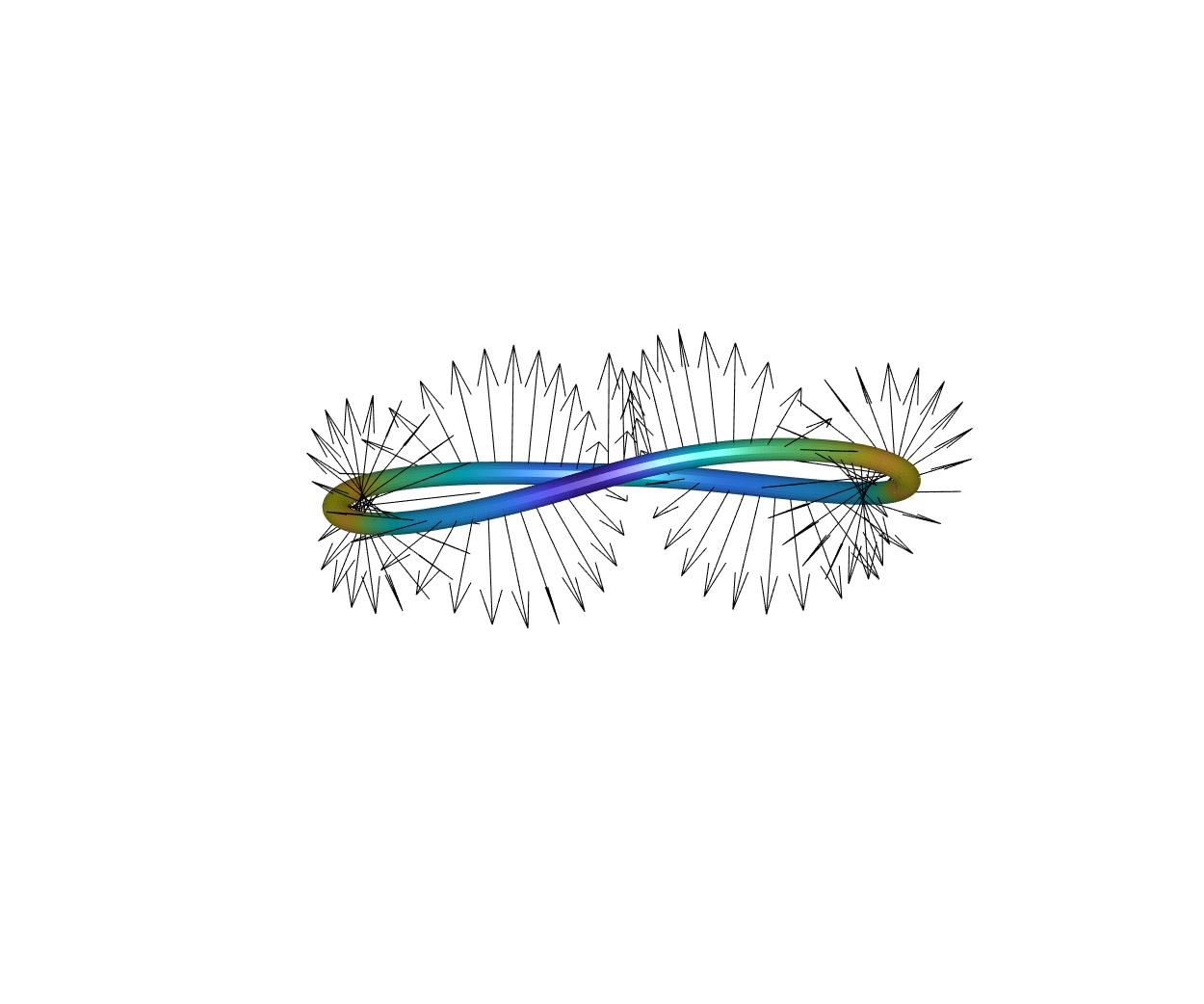}

\bild{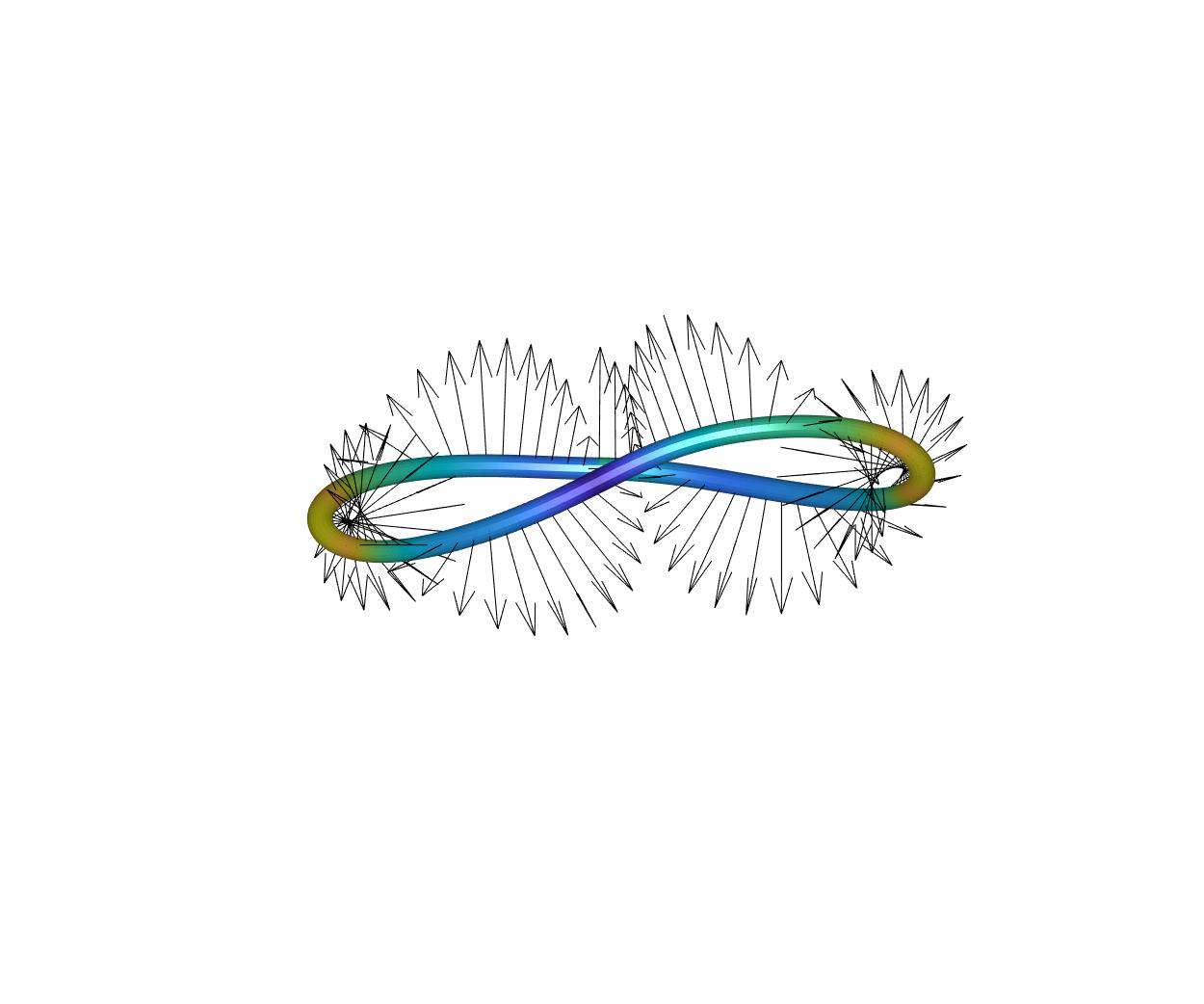}
\bild{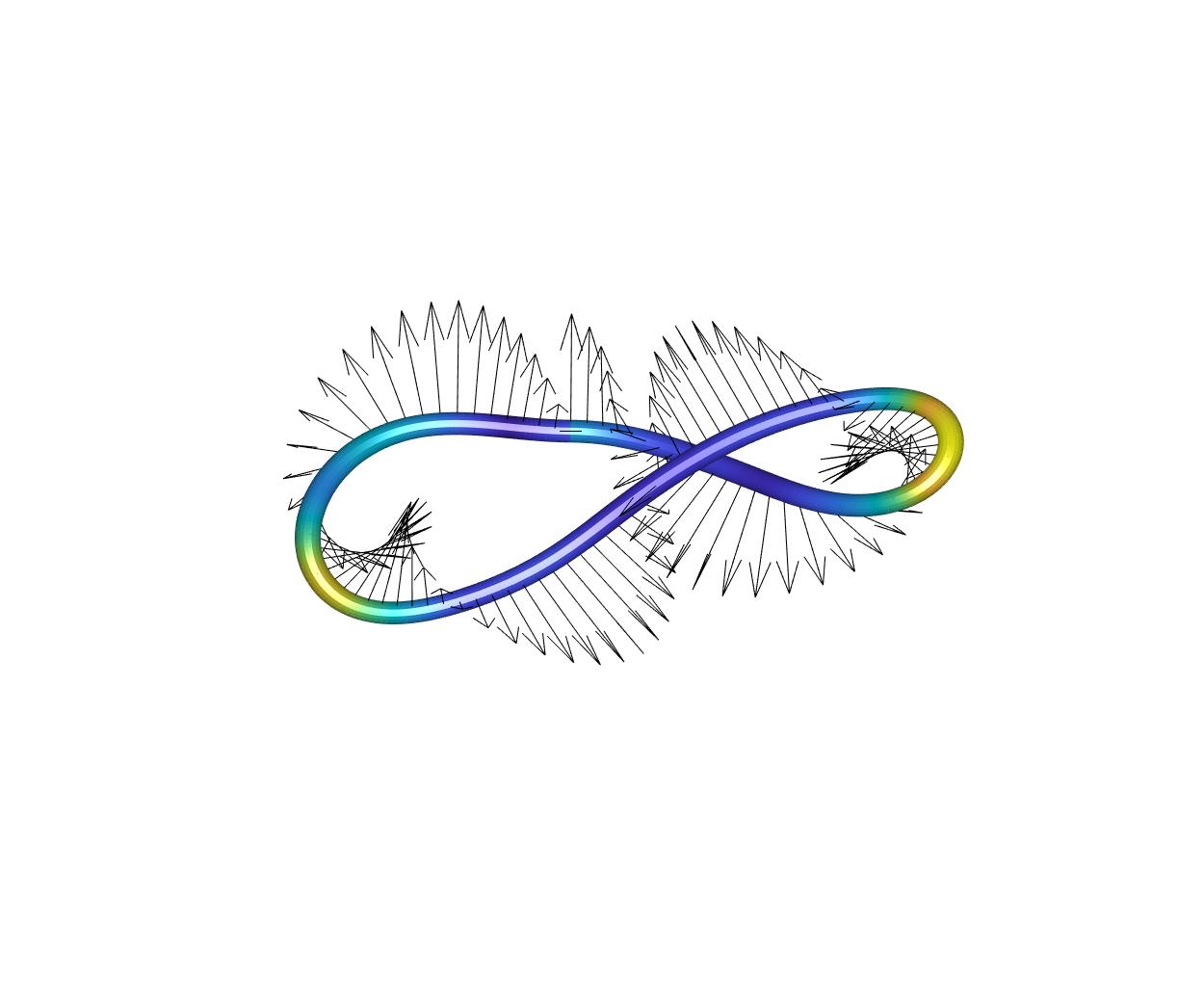}
\bild{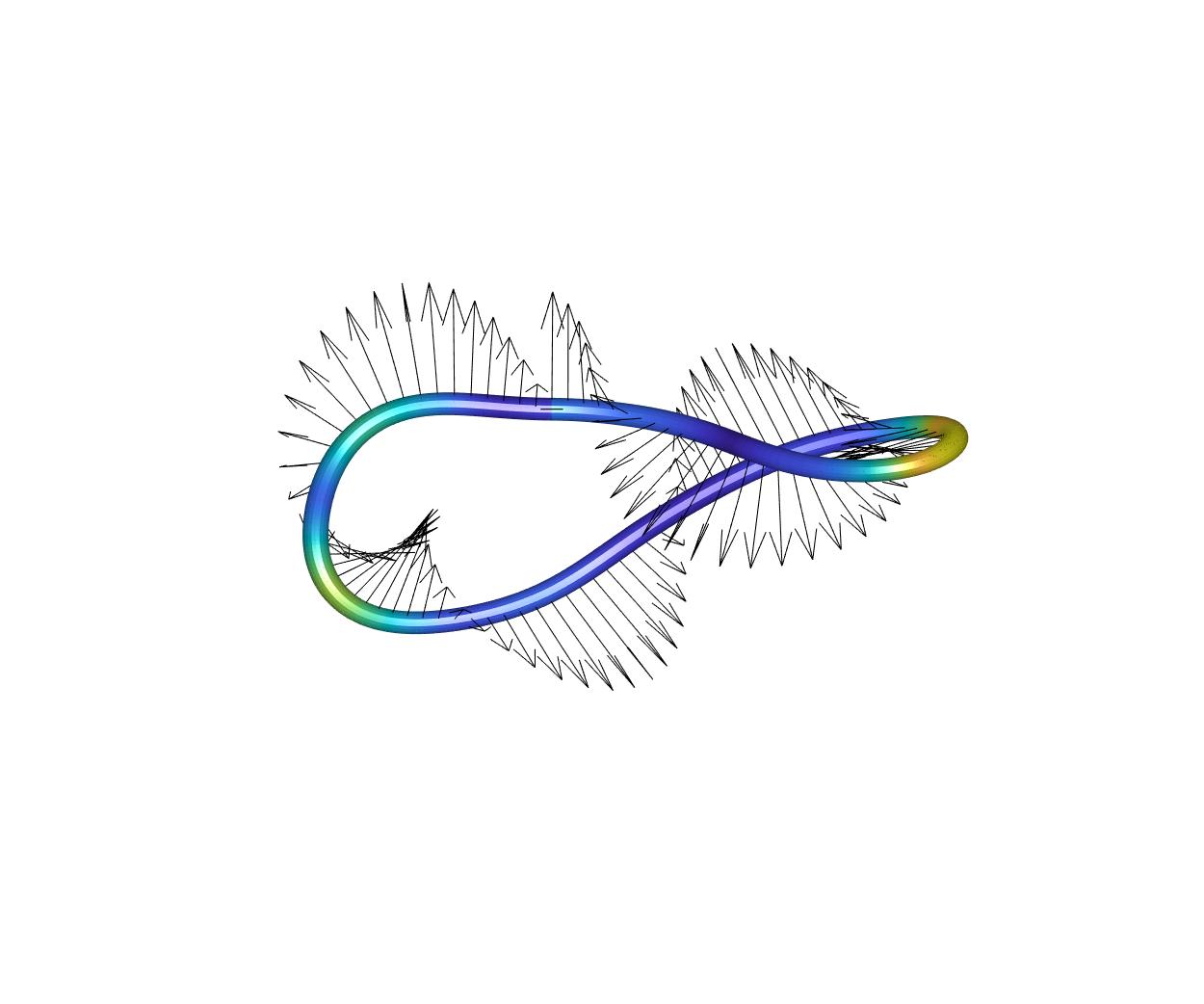}
\bild{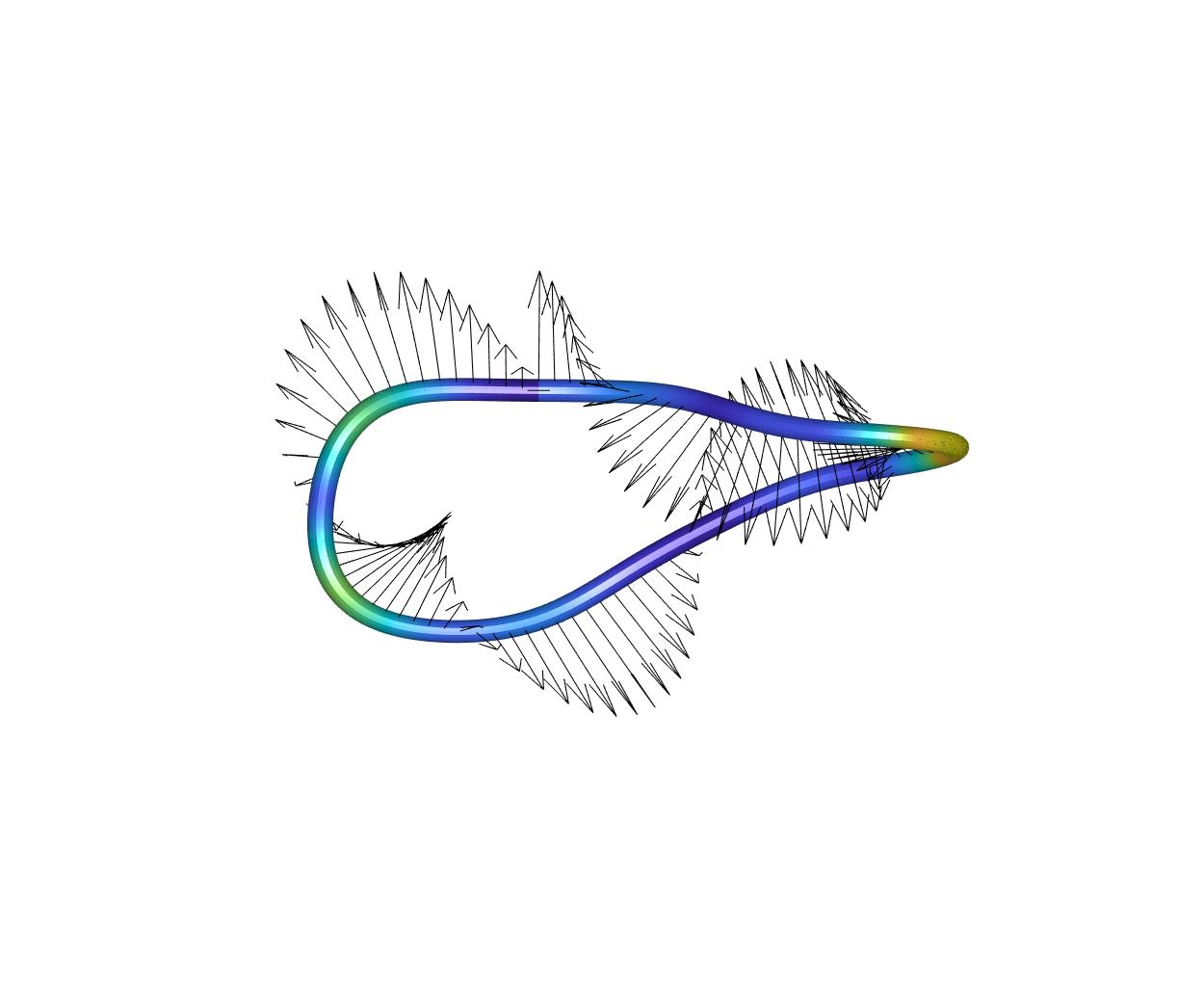}

\renewcommand{\bild}[1]{\fbox{\includegraphics[scale=.11,trim=280 230 240 200,clip]{b#1.jpg}}\makebox[0ex][r]{\makebox[0ex][r]{\tiny#1}\hspace{0ex} }\,\ignorespaces}

\bild{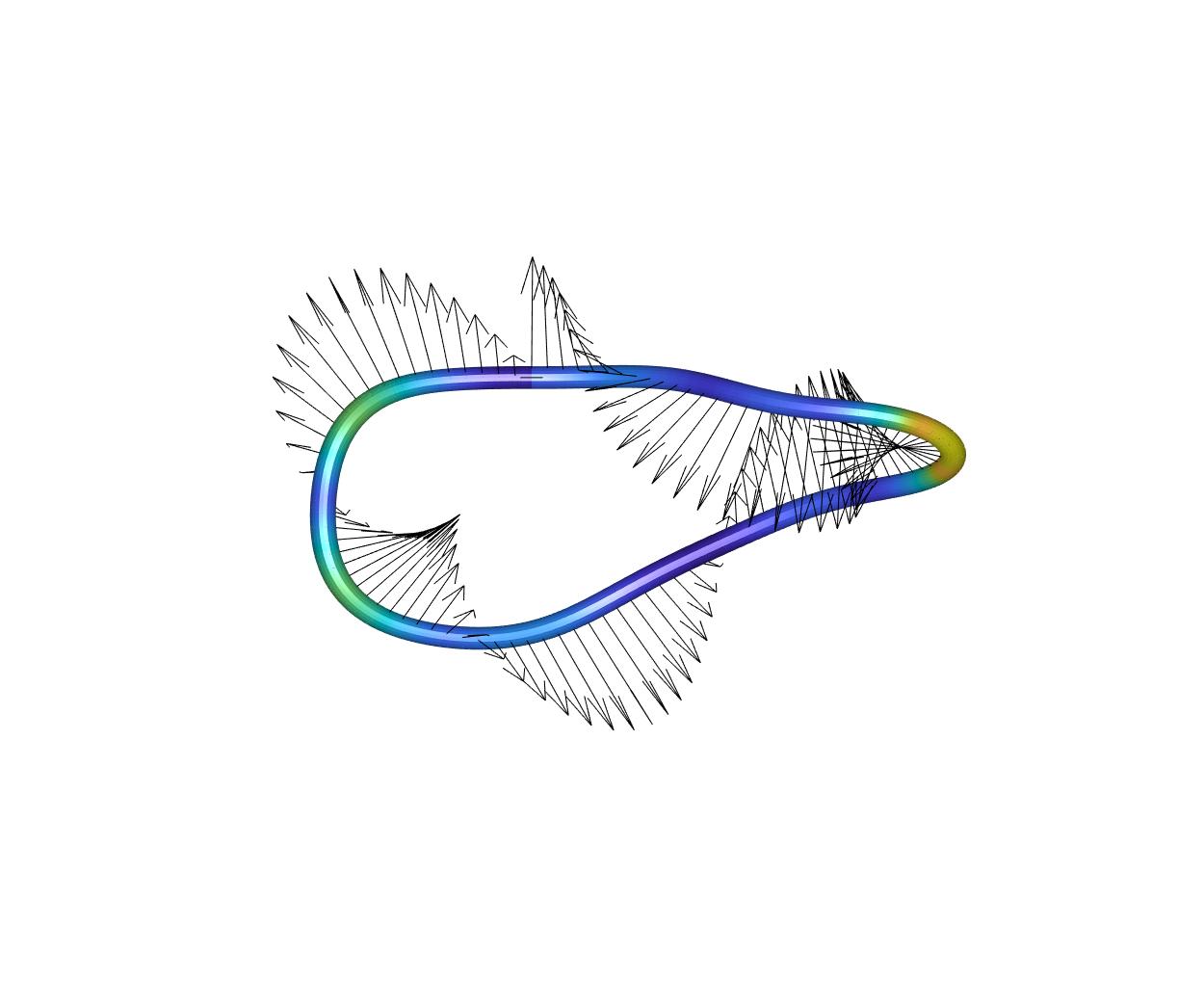}
\bild{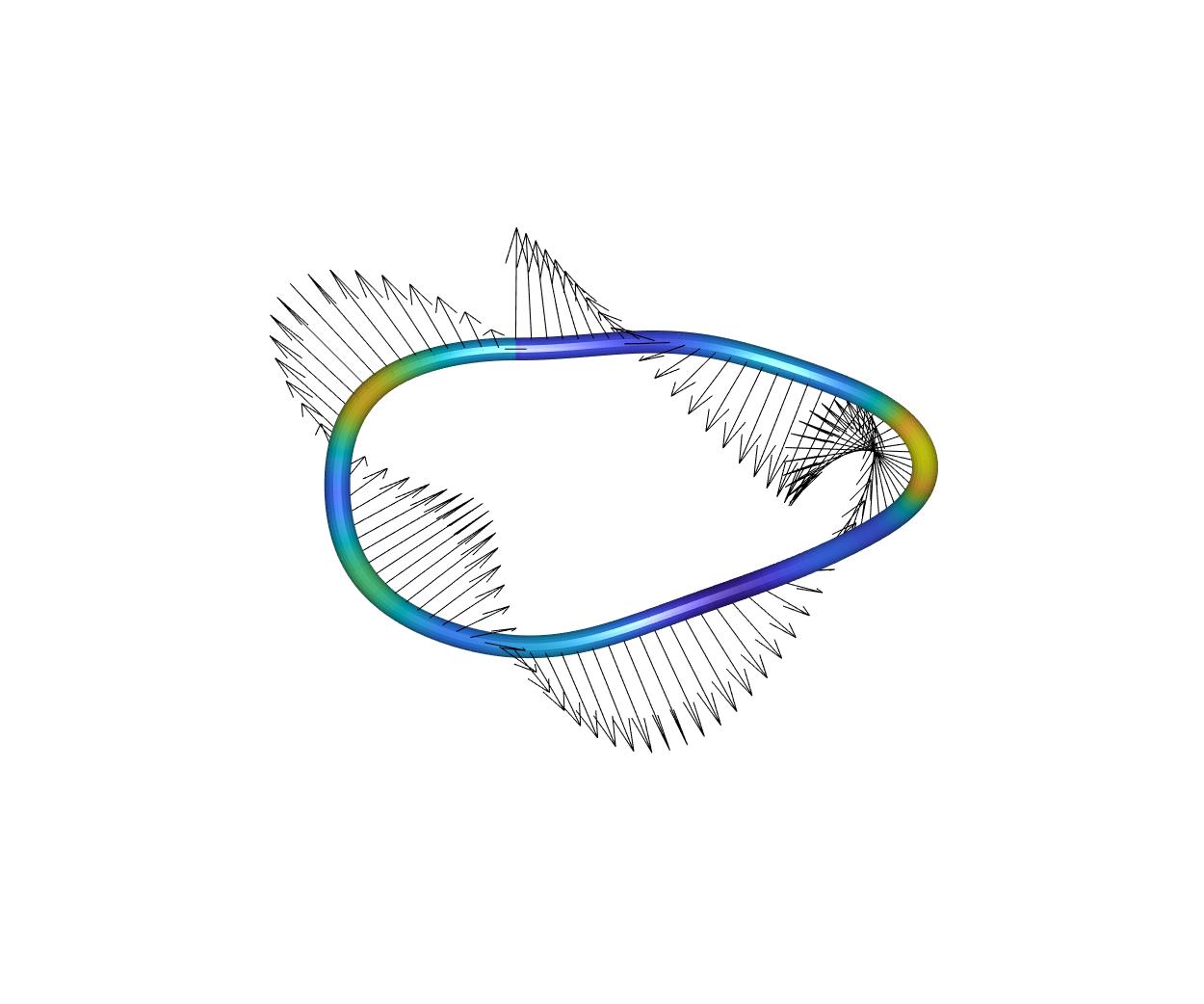}
\bild{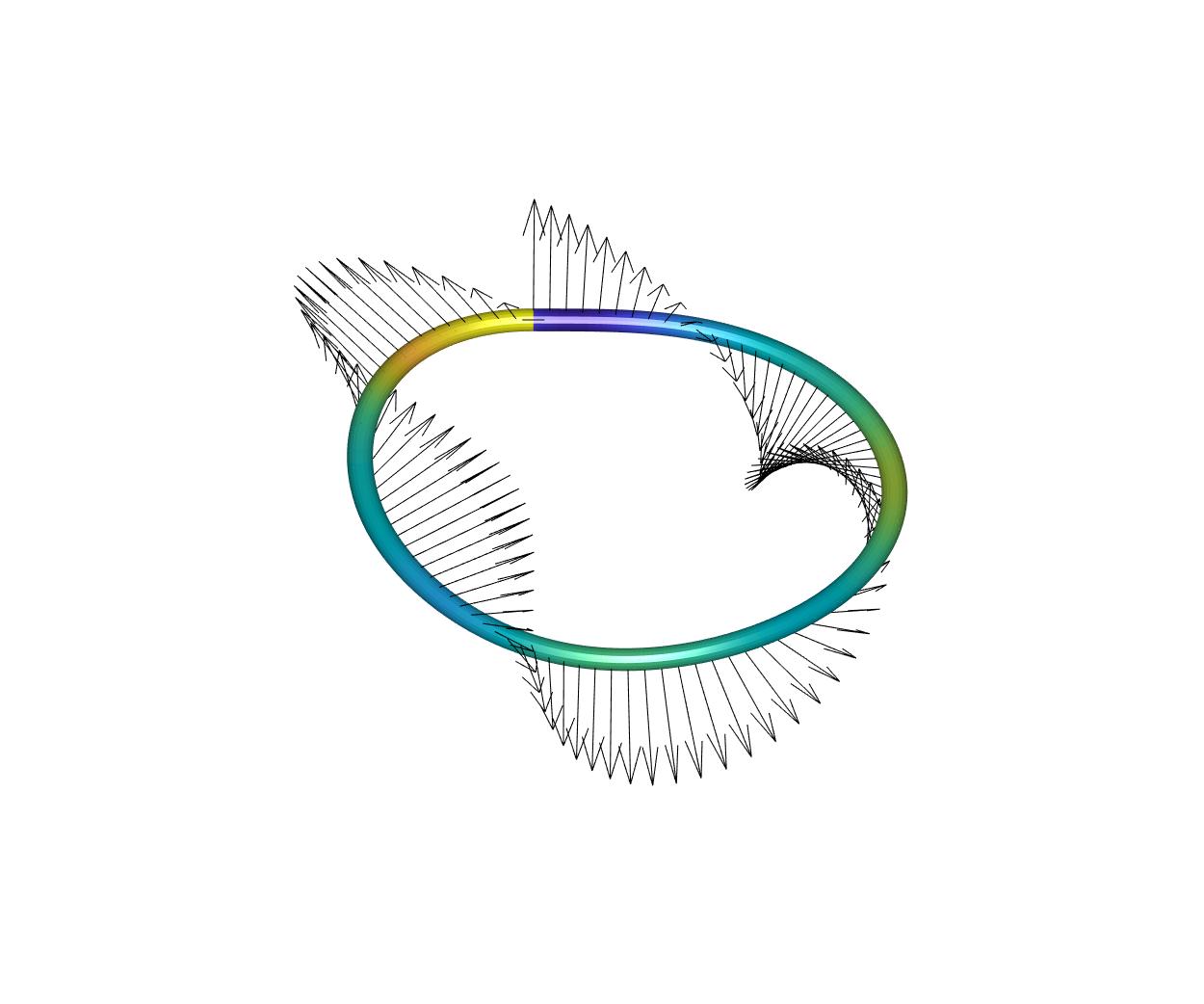}
\bild{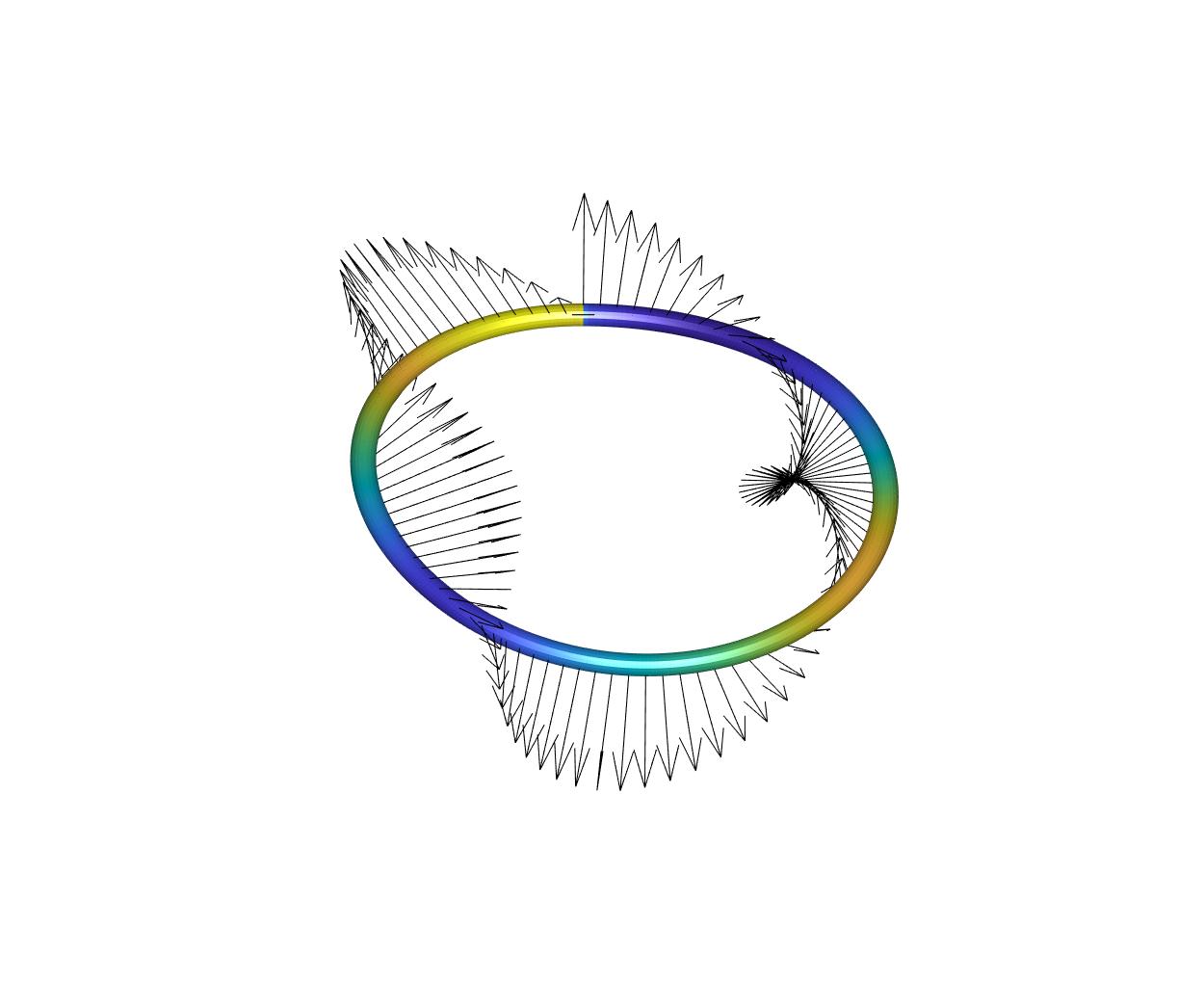}

\caption{{In Experiment~\ref{sect:michell} we start
the evolution with a round circle, framed with
uniform twist rate $\beta_{\ini}=4.2>\beta_{*}$.
The total twist is reduced by self-penetration in the course of the evolution around iteration step~$k=112{,}500$.
It ends with another round circle, situated in a different plane.}}
\label{fig:michell}
\end{figure}

\begin{figure}
\raisebox{1ex}{
\includegraphics[scale=.33]{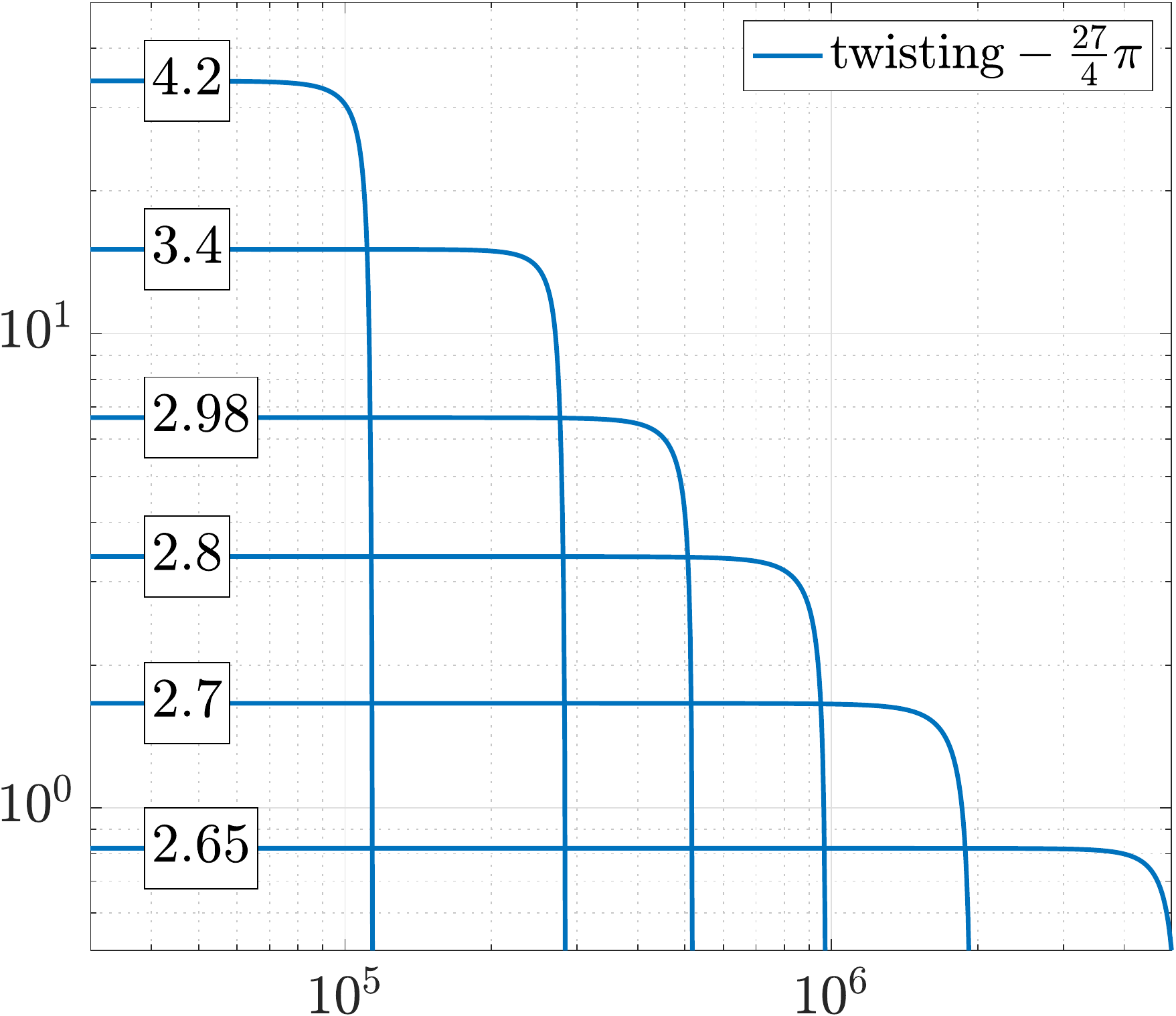}}
\hfill
\includegraphics[scale=.33]{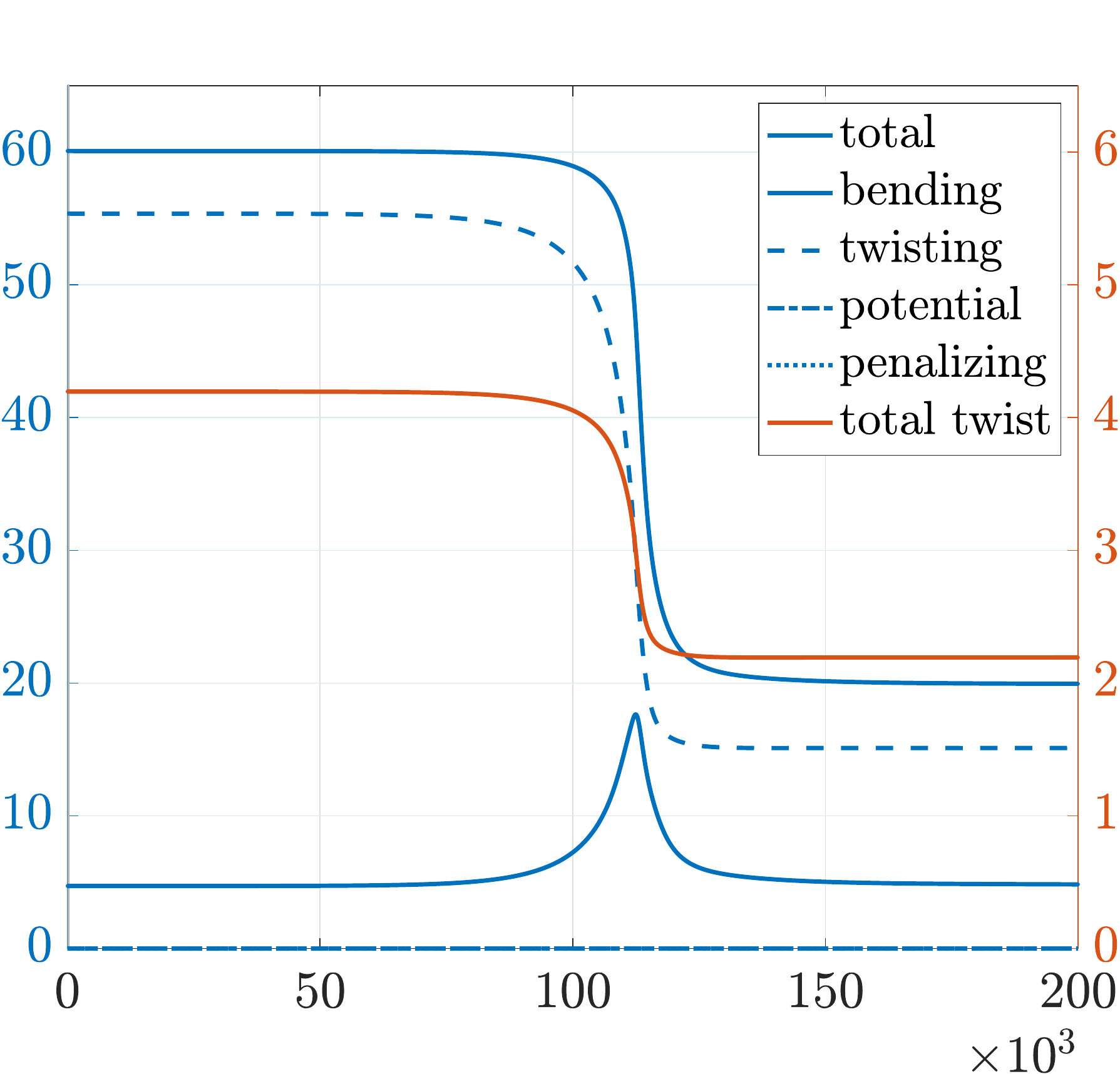}

\caption{{(Experiment~\ref{sect:michell})
\newline\emph{Left}:
In order to experimentally confirm Zajac’s threshold,
we repeat the evolution depicted in Figure~\ref{fig:michell} above
for different twist rates $\beta_{\ini}>\beta_{*}$.
The evolutions turn out to be very similar; essentially they only differ in speed.
The (logarithmically scaled) twisting profiles
reveal a drastic reduction of the twisting energy
(due to the self-penetration of the curve).
The smallest iteration step at which the twisting energy is below
$\tfrac{27}4\pi$ serves as a threshold for the speed of the evolution.
\newline\emph{Right}:
The energy plot for the evolution from Figure~\ref{fig:michell} where $\beta_{\ini}=4.2$.
Here the twisting energy decay occurs around iteration step $114{,}200$.}}
\label{fig:michplot}
\end{figure}

We aim at experimentally confirming Zajac's threshold
$\beta_{*} = 2\pi\sqrt3\kappa/L$, see {Section}~\ref{sect:zajac}.

We consider the initial curve
${y(s) =}\br{\cos s,\sin s,0}^{\top}$
with the frame
\[ {b(s) = }
\cos(\beta_{\ini} s) \begin{pmatrix} -\cos s\\ -\sin s\\0 \end{pmatrix}
+\sin(\beta_{\ini} s) \begin{pmatrix} 0\\0\\1 \end{pmatrix}, \qquad s\in[0,2\pi]. \]
In order to break symmetry which seems to prevent
rod configurations from leaving even energetically unfavorable states,
a slight perturbation {has been applied} to the initial curve, namely
\begin{equation}\label{eq:perturb}
 s\mapsto \tfrac1{1000}\sin(7s)
\end{equation}
perpendicularly to the plane (i.e., in $z$-direction);
the frame is corrected accordingly.

In order to quantitatively verify Zajac's threshold,
we had to choose a rather high penalty coefficient,
namely $\eps=10^{-5}$.
For $\kappa=3/2$, we obtain
\[ \beta_{*}=\tfrac32\sqrt3\approx2.5981. \]

In Figure~\ref{fig:michplot} (left) we plot the
twisting energy of several evolutions using logarithmic scales
on both axes.
From top to bottom, the {profiles} correspond to initial values of
{$\beta_{\ini}=\beta_{*}+\tfrac{2^{\ell}}{10}$ for $\ell=-1,0,1,2,3,4$.}
More precisely, we plot
the twisting energy of the evolutions corresponding
to different values of $\beta_{\ini}$ minus
the twisting energy of the configuration
at $\beta_{*}$, i.e.,
$\tfrac12\int_{0}^{2\pi}\beta(s)^{2}\d s - \tfrac{27}4\pi$,
{which seems to be stationary}.
Of course, values less than $\tfrac{27}4\pi\approx21.2058$
are ignored.

Experimentally we find that evolutions for different initial values of $\beta_{\ini}$ seem to
be very similar and essentially only vary in speed.
The region of iteration steps where the twisting energy drastically falls
is a good indication for the latter.
There is one caveat---probably due to symmetry it turned out
that the evolution corresponding to the initial values $\beta_{\ini}=3$
is remarkably slower than expected.
Therefore we chose $\beta_{\ini}=2.98$ instead.

The plot in Figure~\ref{fig:michplot} (left) indicates a reciprocal dependence of the evolution speed on $\beta_{\ini}-\beta_{*}$.
{For values $\beta_{\ini}\le\beta_{*}$
the initial configuration remained unchanged (at least until step $k=5\cdot10^{6}$).}
This confirms Zajac's threshold as desired.

We show a typical full energy profile in Figure~\ref{fig:michplot} (right)
for the initial value $\beta_{\ini}=4.2$.
Some iteration steps are visualized in Figure~\ref{fig:michell}.
The corresponding plots for the other initial values of $\beta_{\ini}$
essentially differ by the scaling of the horizontal axis
and the simulations looks very similar.
Initial and final stage are round circles which correspond
to the (unique) global bending energy minimum of $L=2\pi$
among all closed curves.

Note that the total twist is reduced by precisely~$2$
from $\beta_{\ini}=4.2$ to $\beta=2.2$.
In light of {Section}~\ref{sect:mintt} a second reduction would be
possible as well.
 However,
in contrast to {Experiment}~\ref{sect:overtwist} below,
the {gradient of the energy does not seem to be steep enough}
to invest the amount of additional bending required for another self-penetration.
As $2.2<\beta_{*}$ the evolution becomes stationary
due to Michell's instability.

\subsection{Reducing twist by self-penetration}
\label{sect:overtwist}

\renewcommand{\bild}[1]{\fbox{\includegraphics[scale=.1,trim=260 300 220 300,clip]{c#1.jpg}}\makebox[0ex][r]{\makebox[0ex][r]{\tiny#1}\hspace{0ex} }\,\ignorespaces}

\begin{figure}

\bild{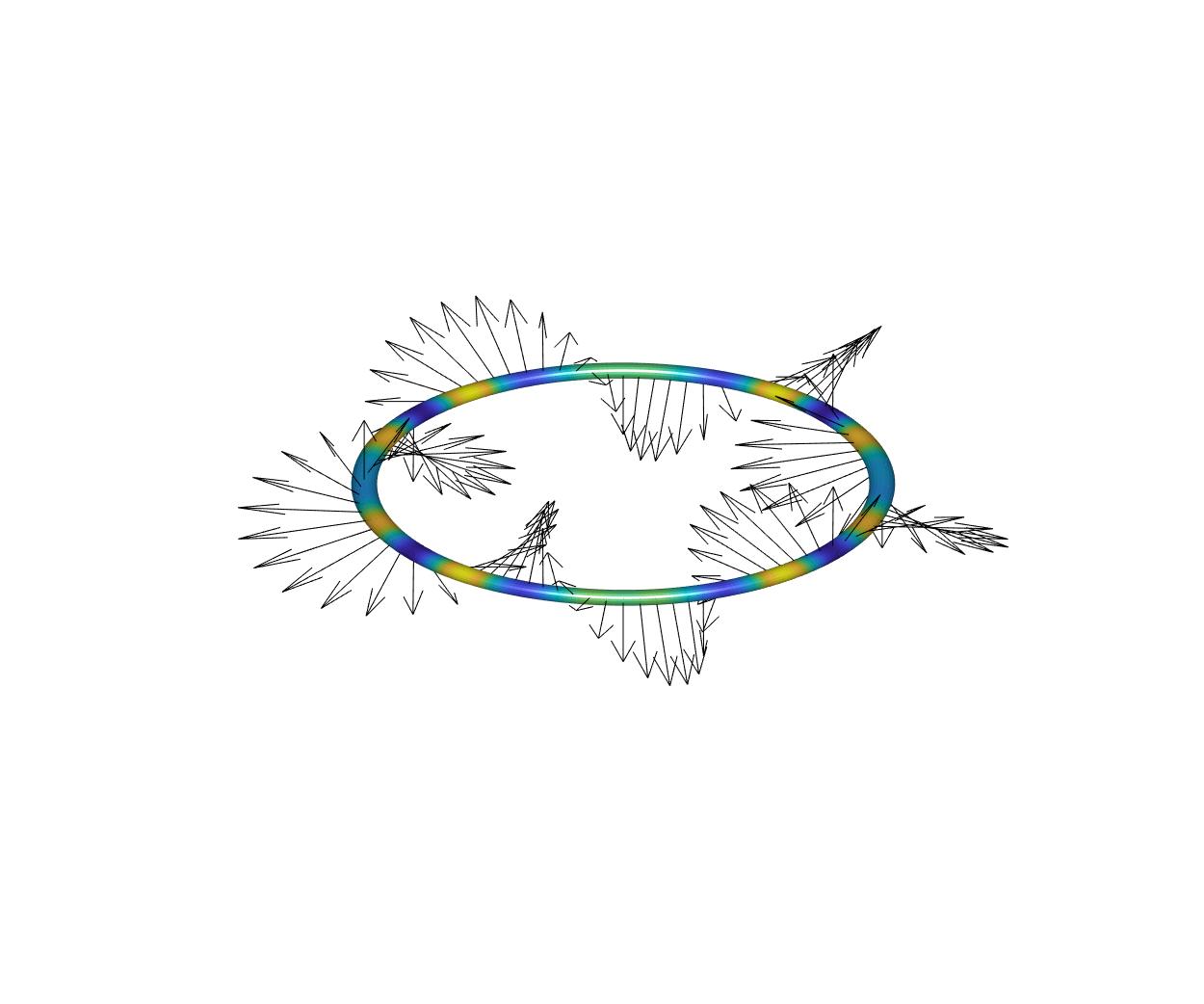}
\bild{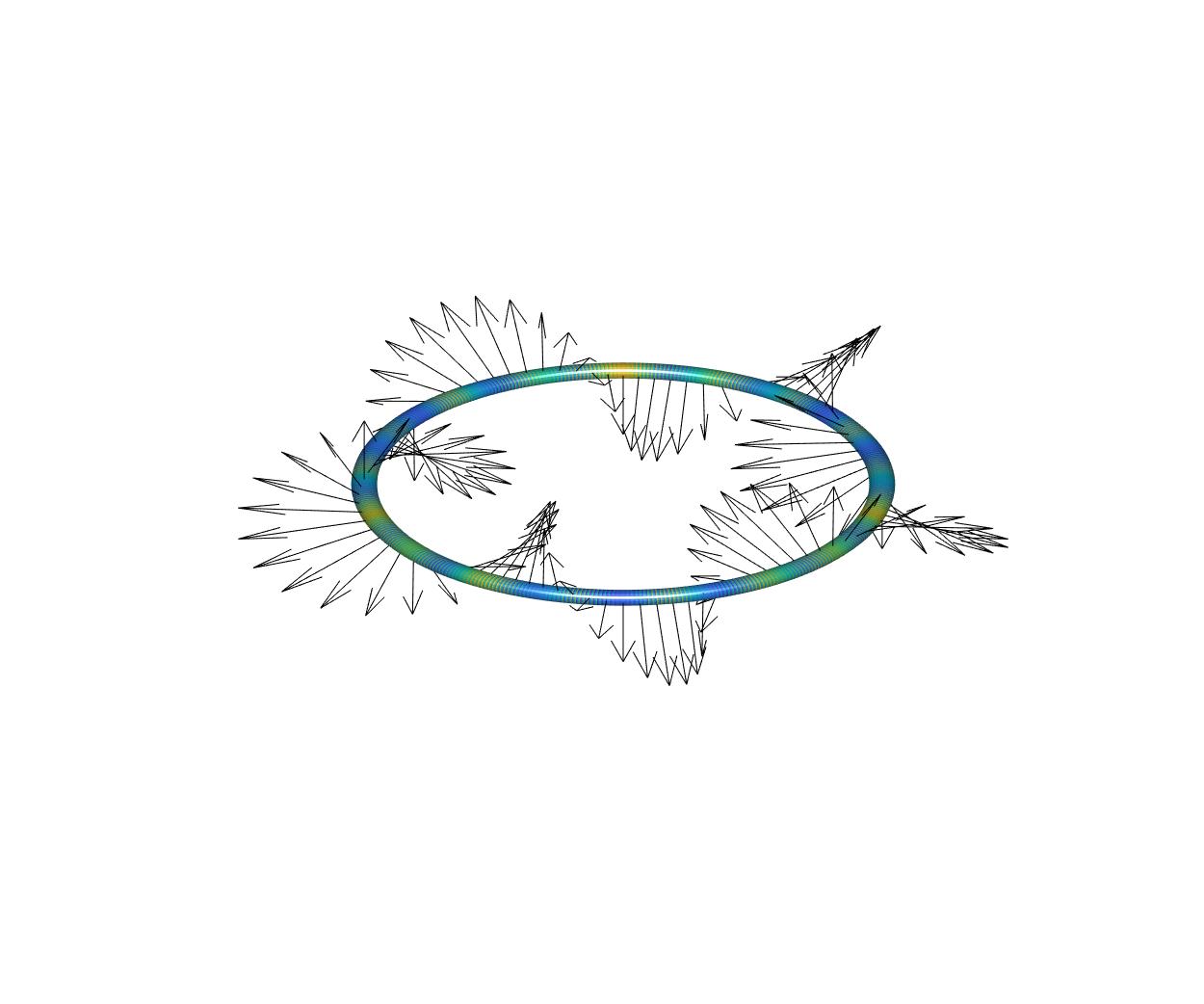}
\bild{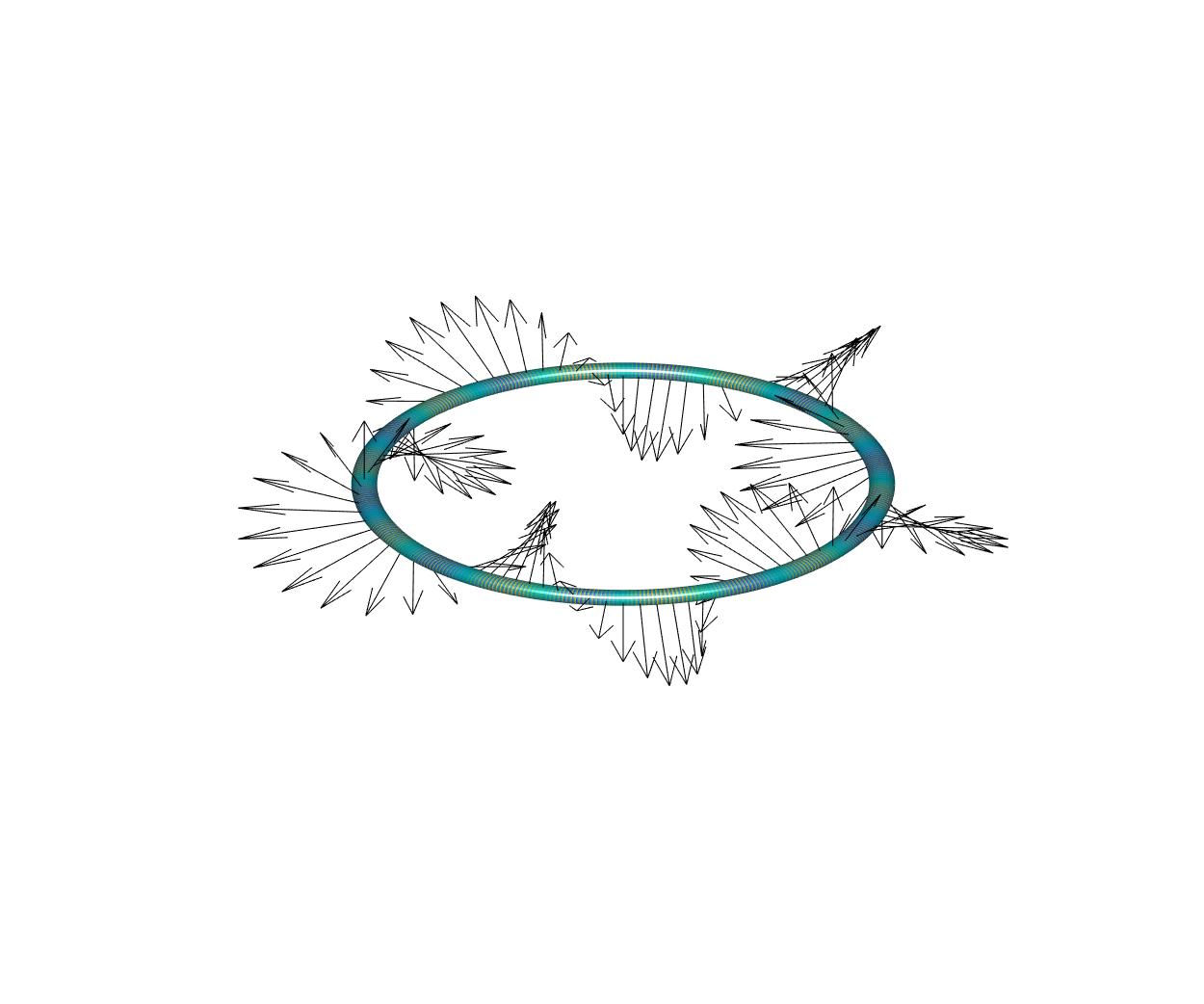}
\bild{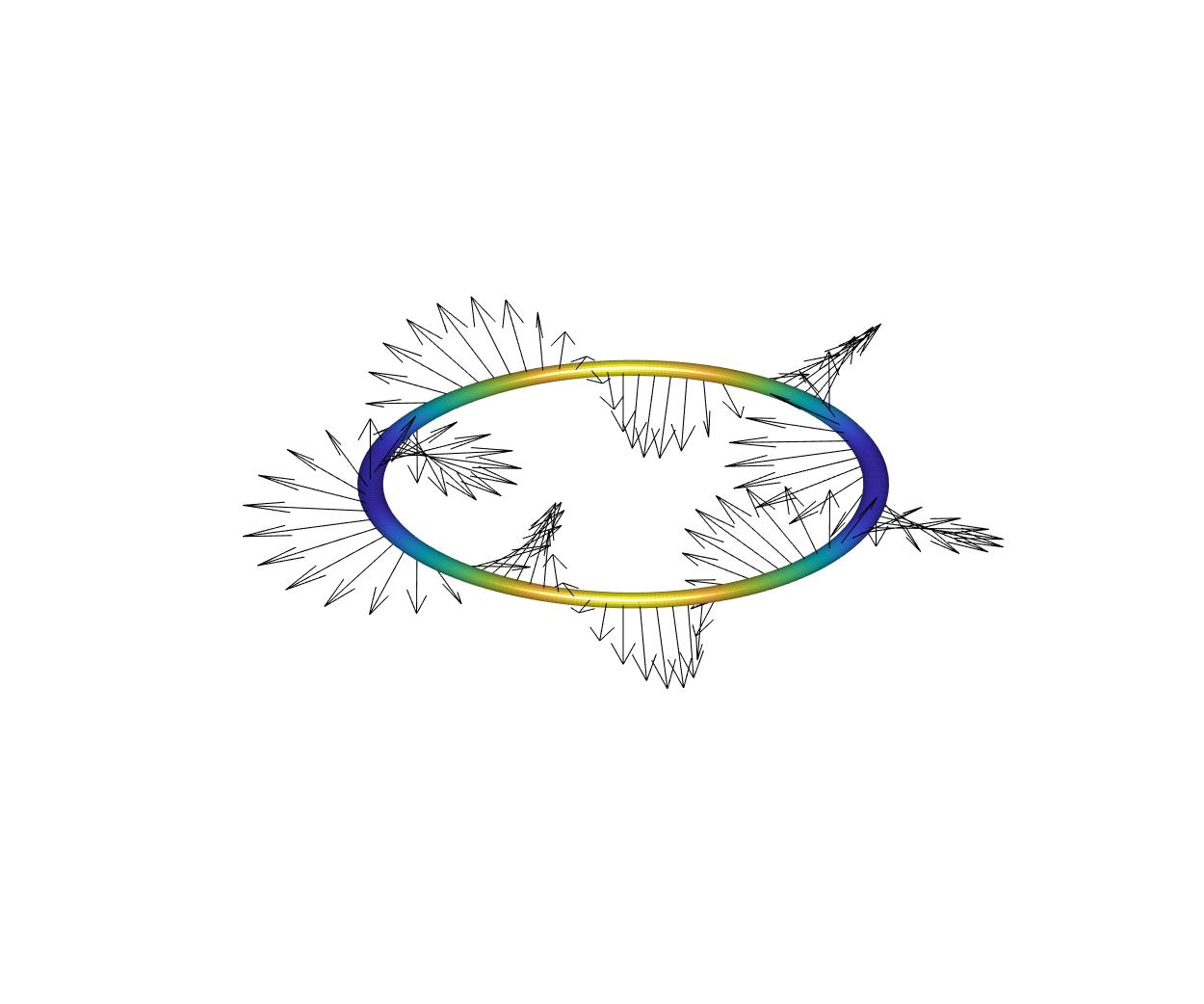}

\bild{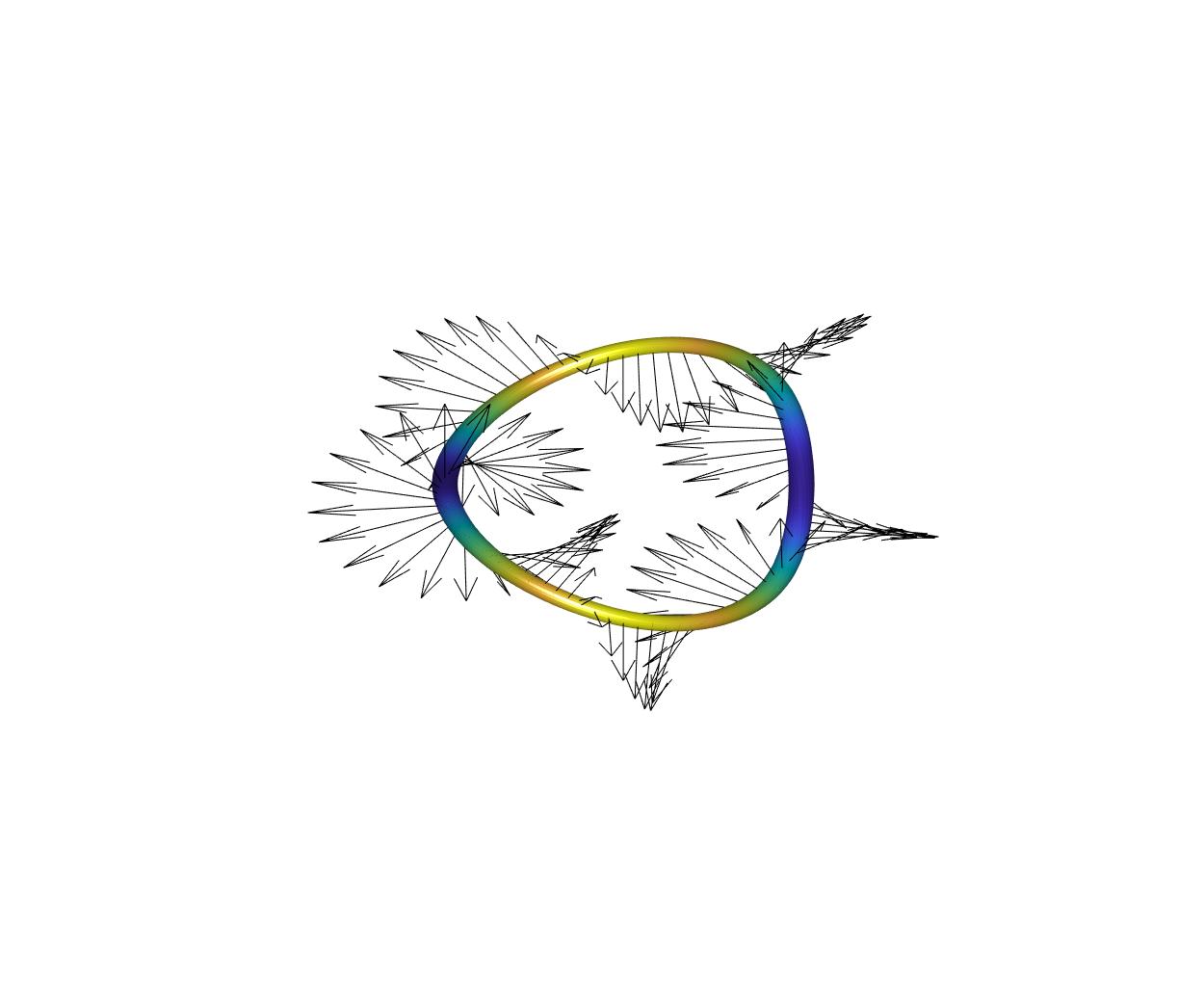}
\bild{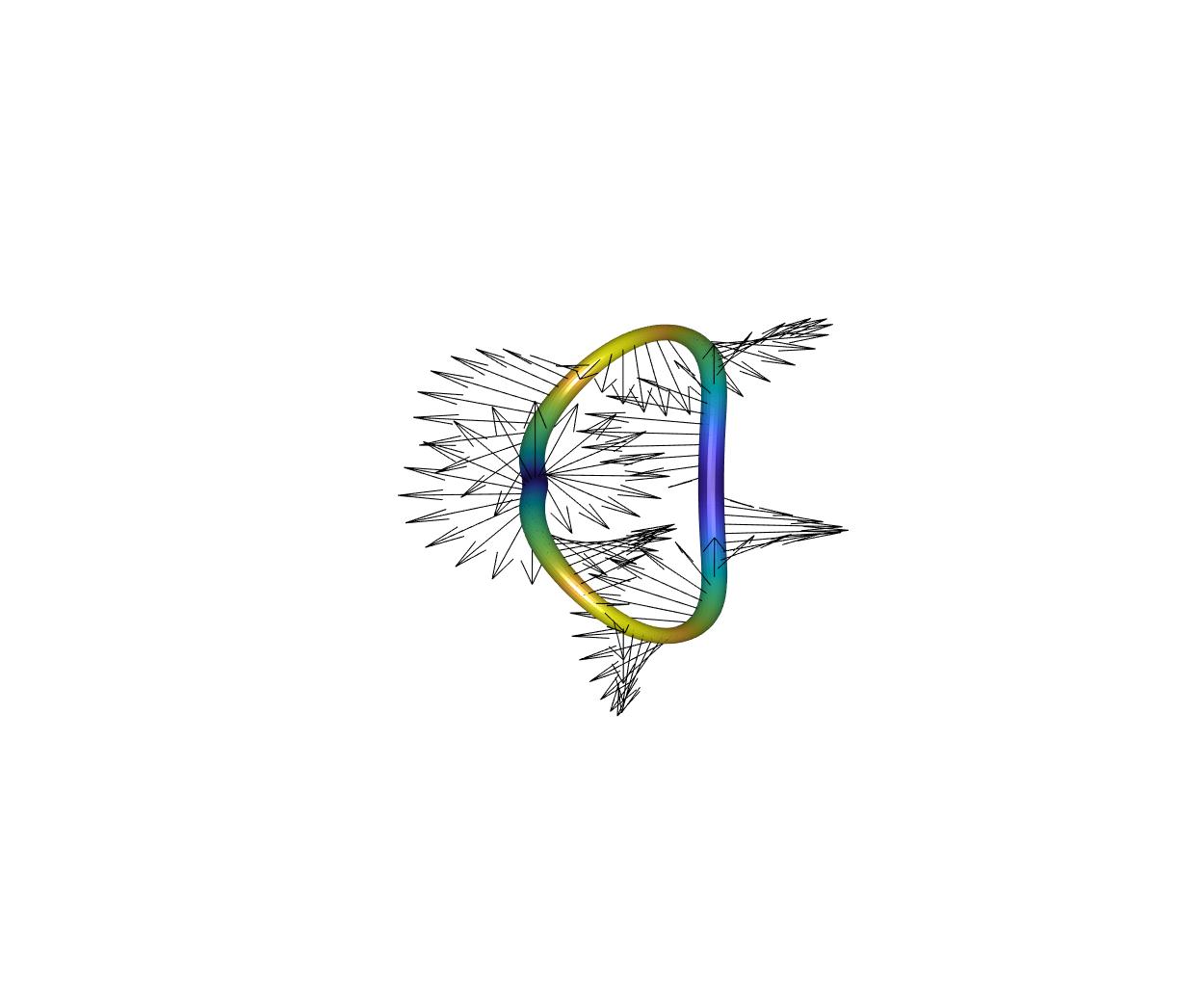}
\bild{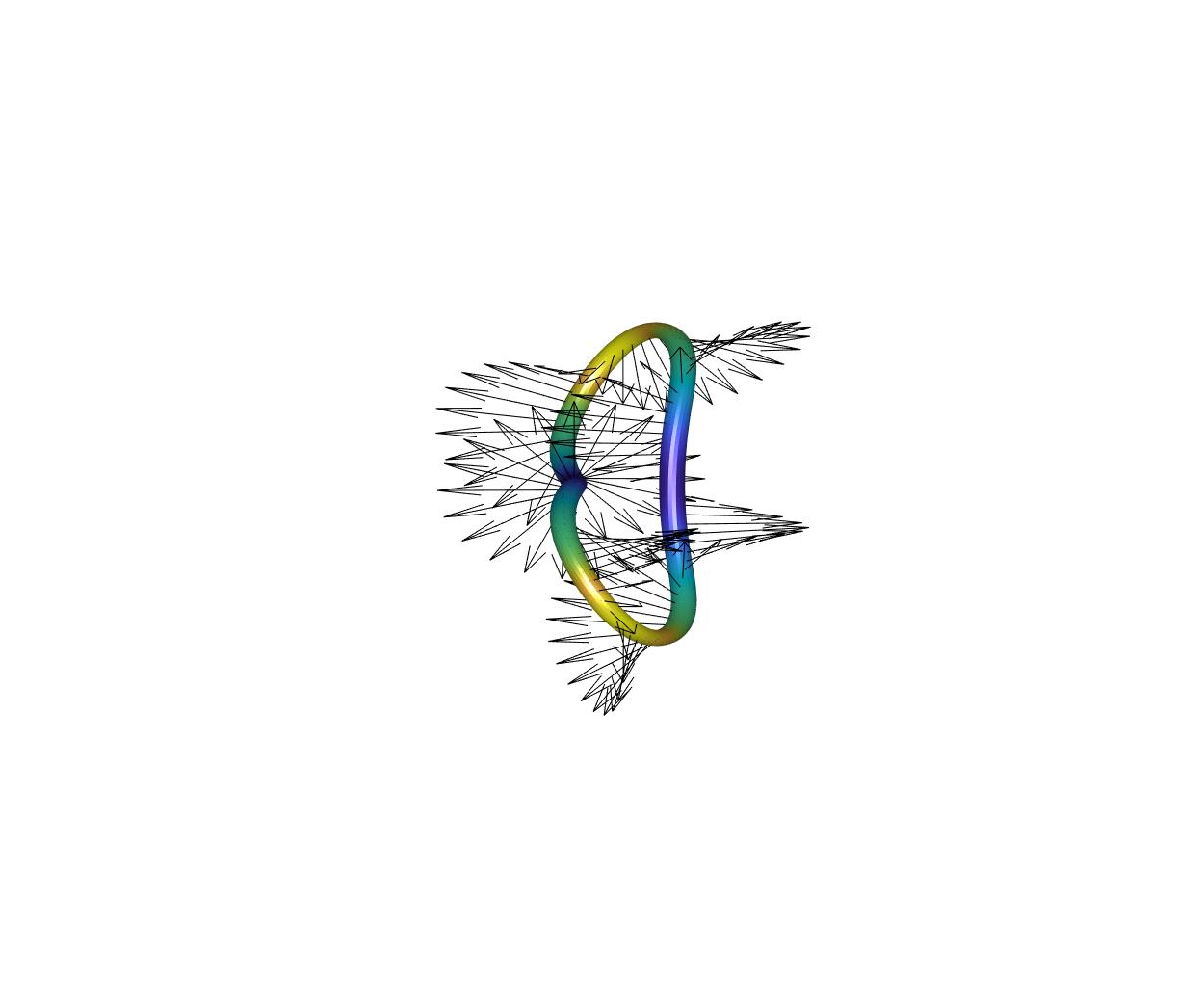}
\bild{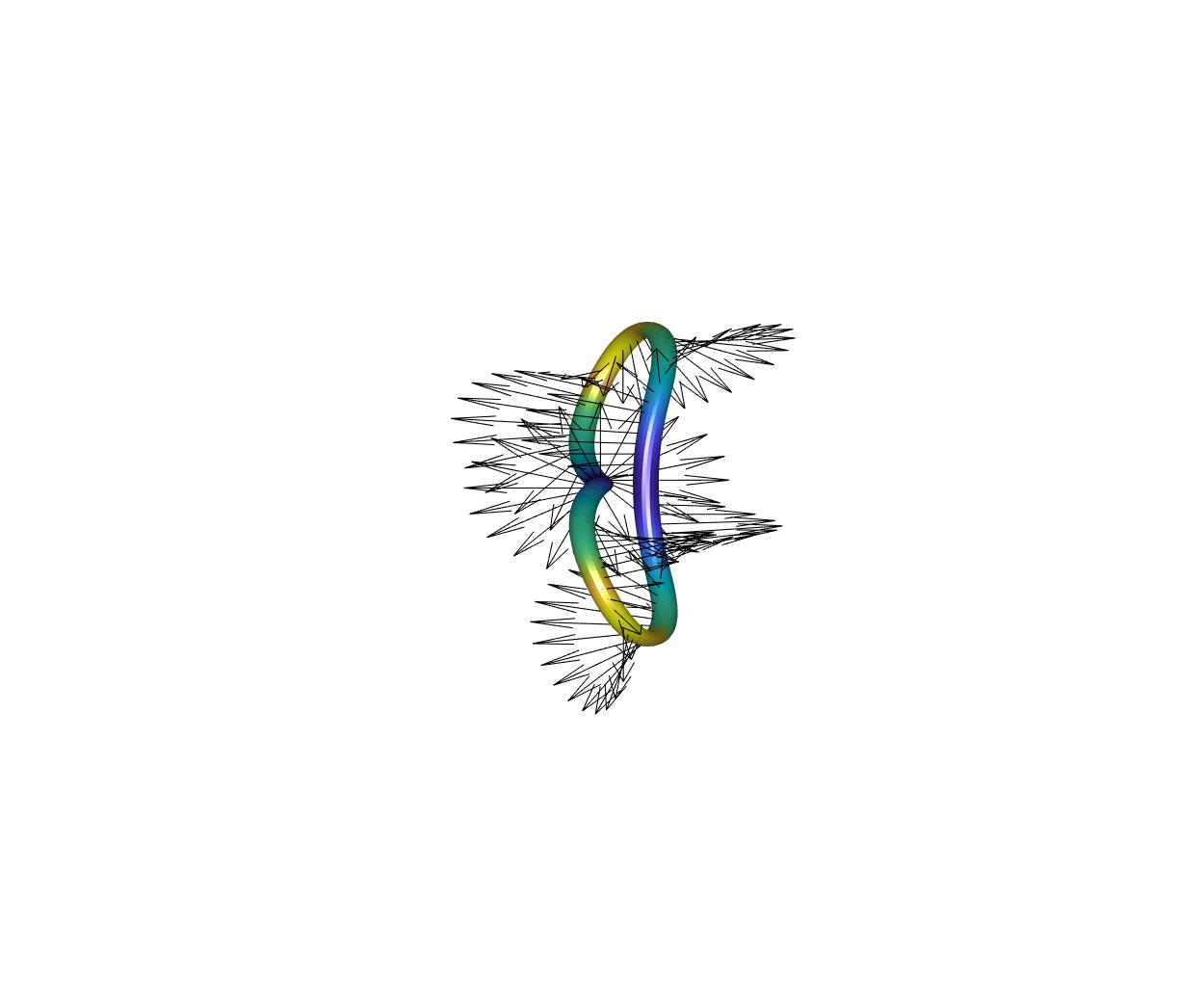}

\bild{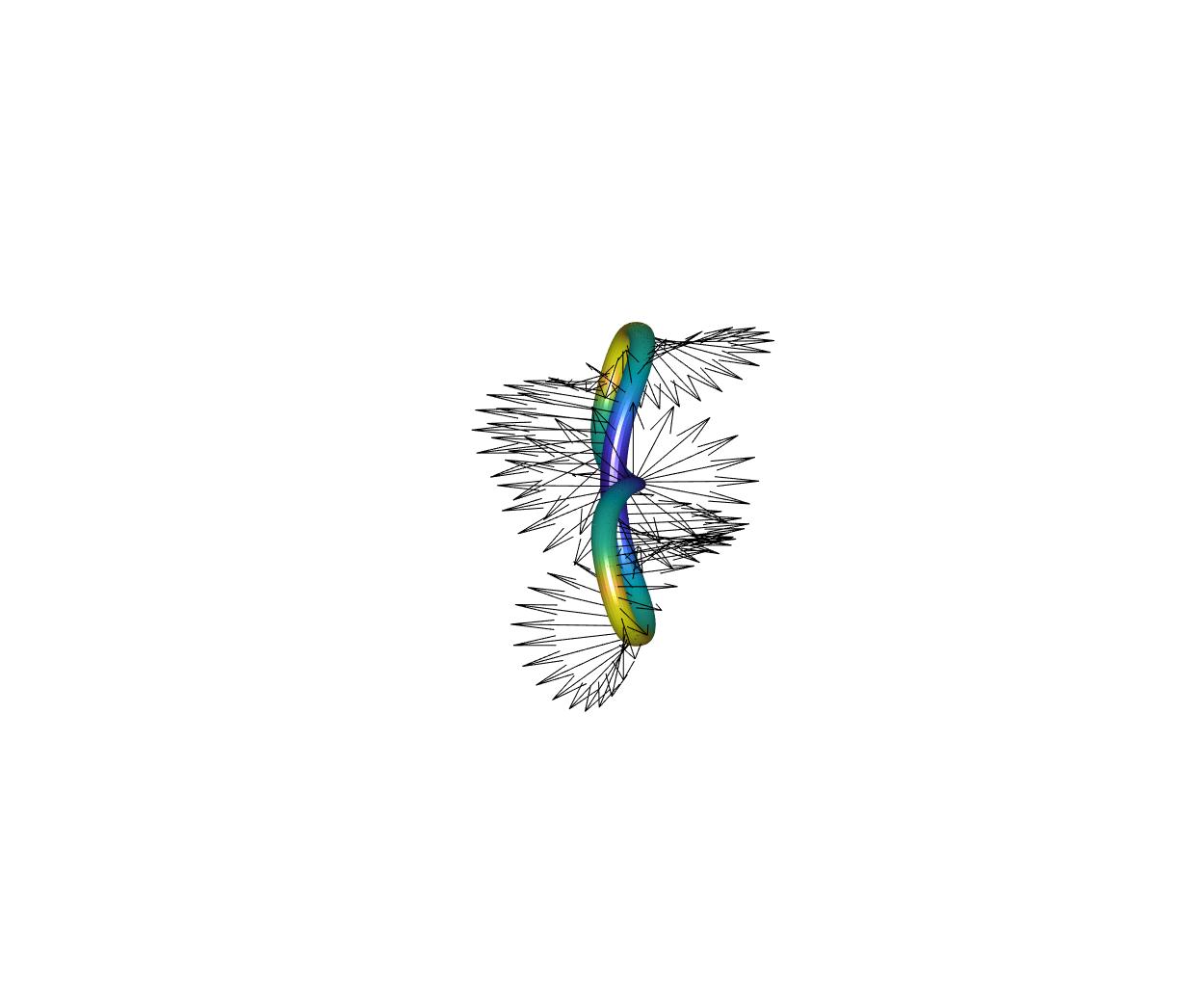}
\bild{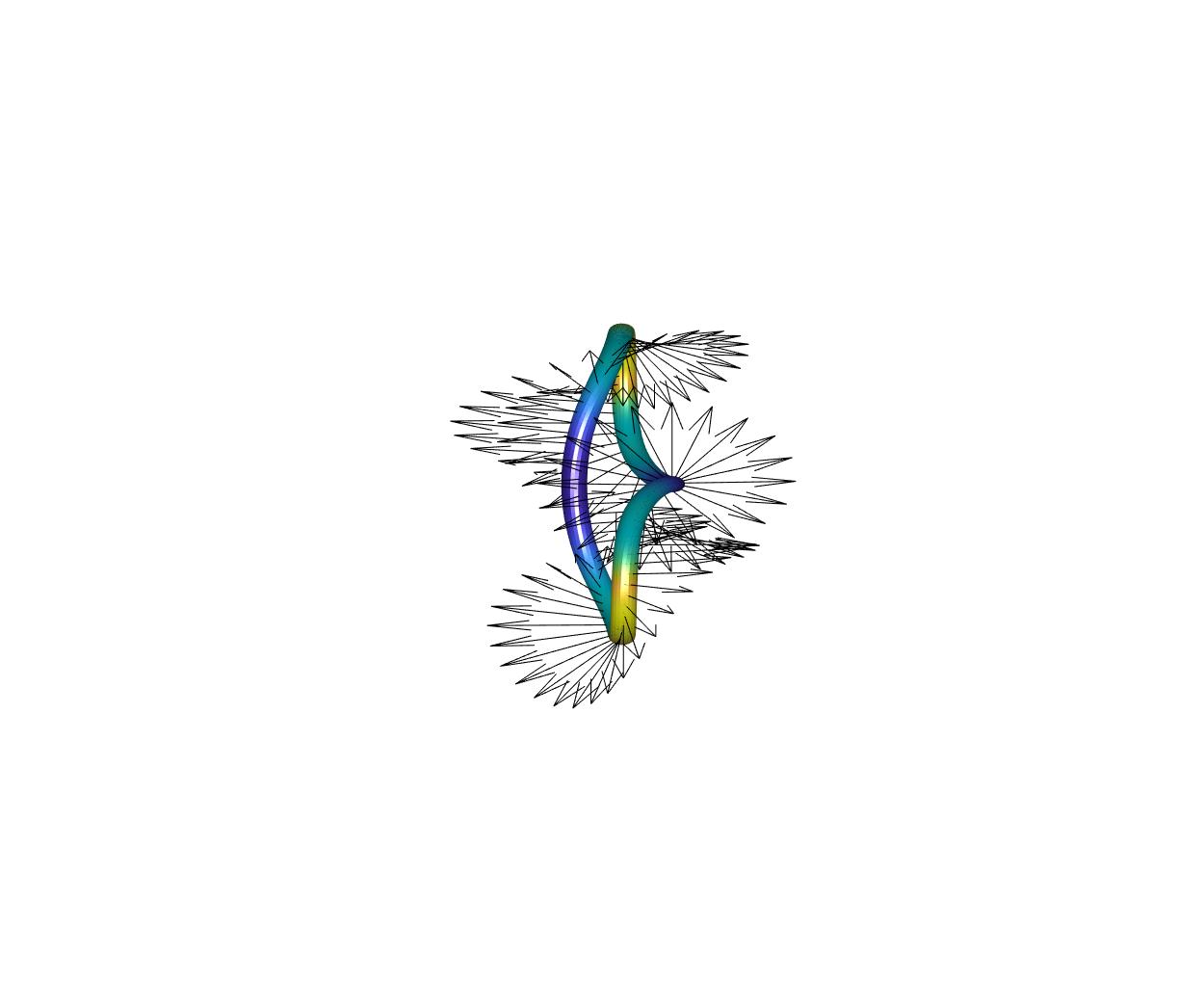}
\bild{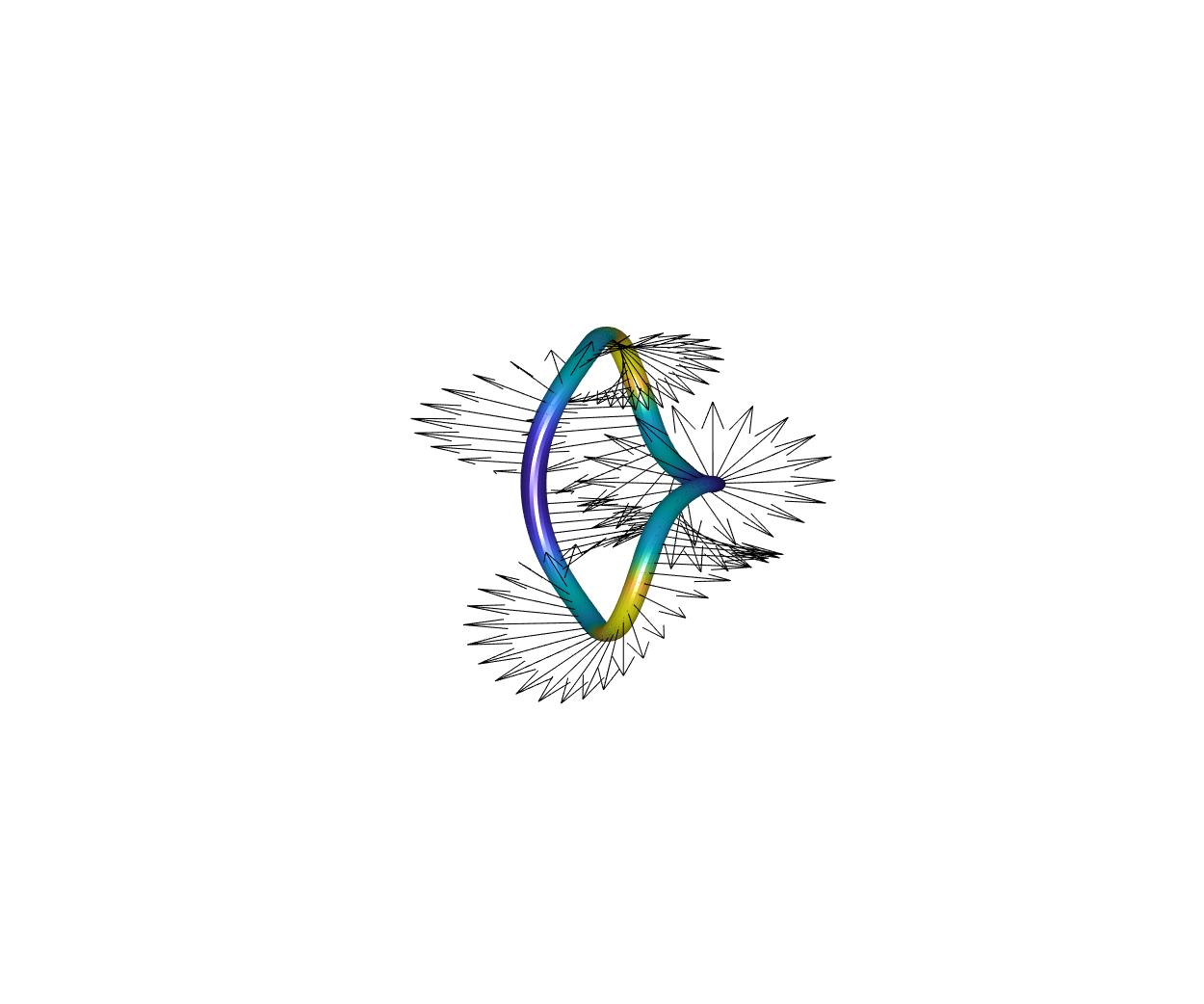}
\bild{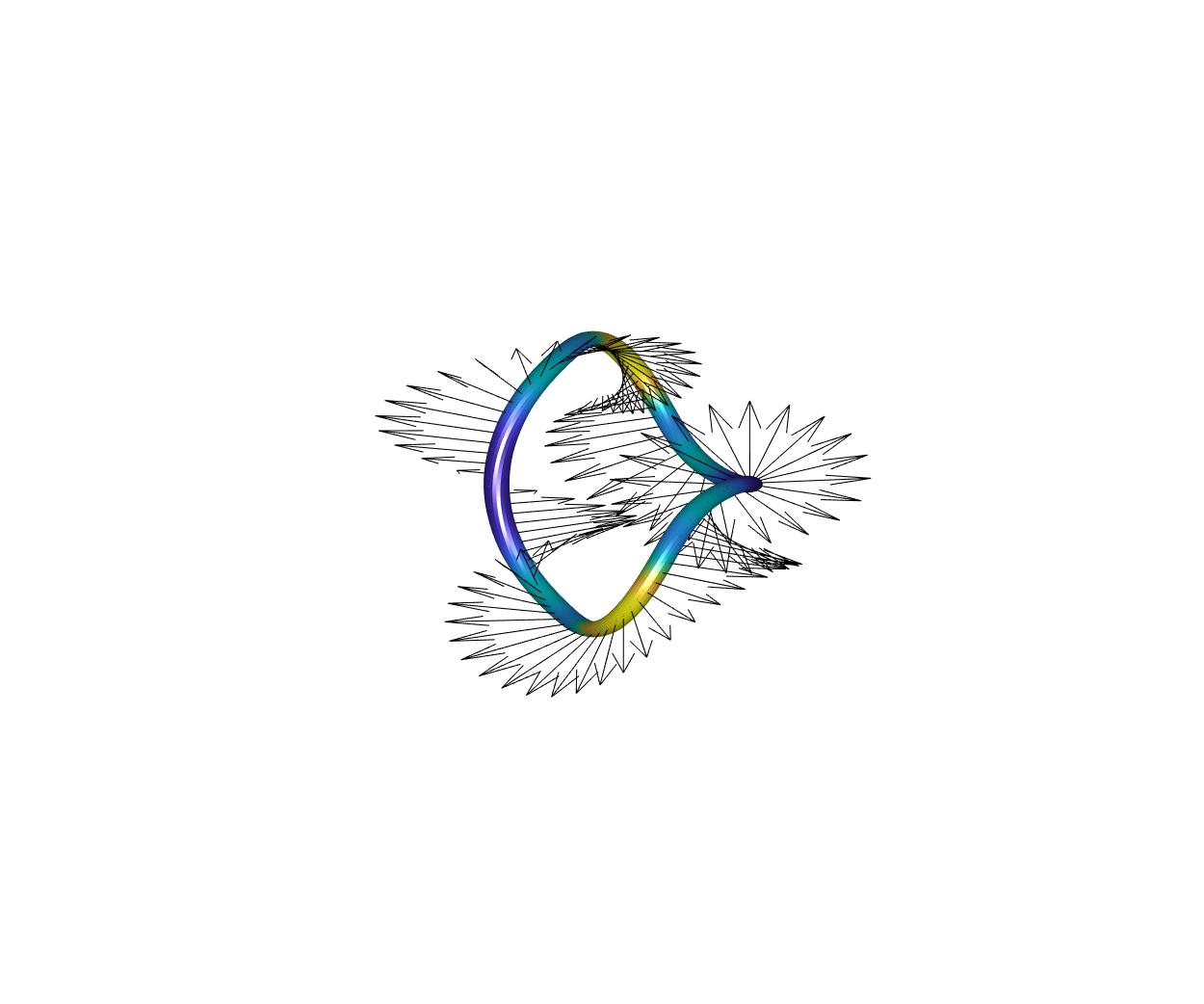}

\bild{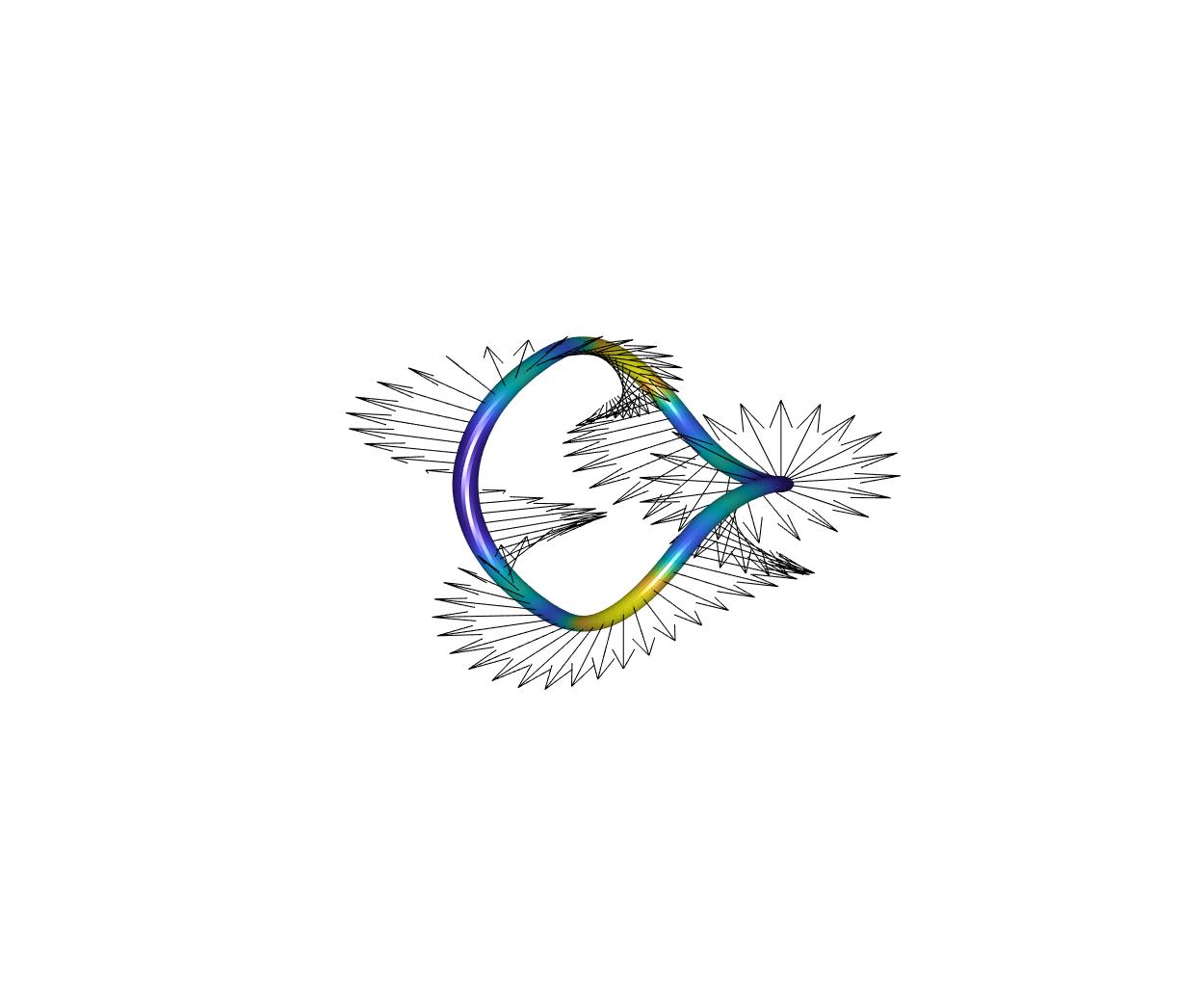}
\bild{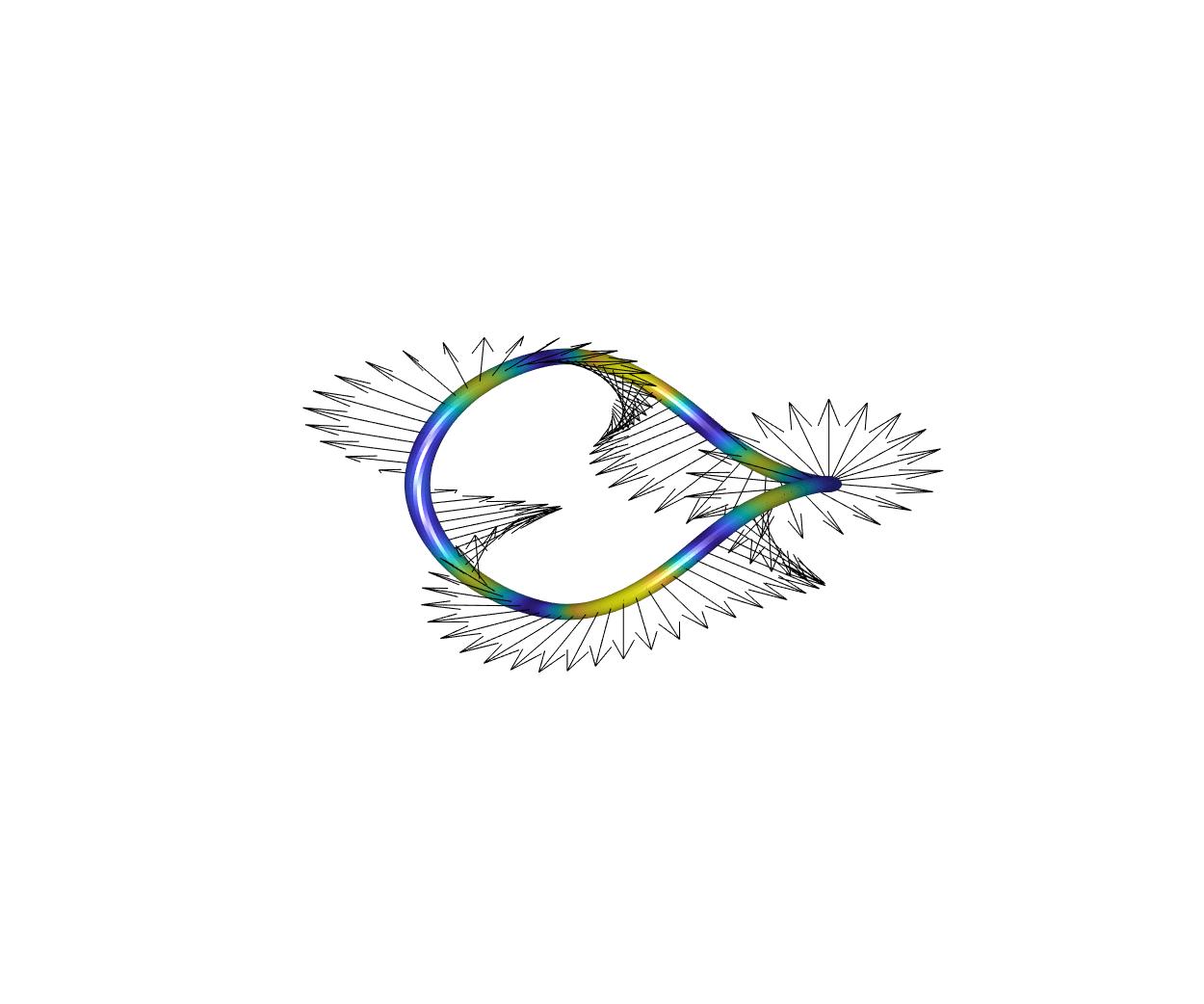}
\bild{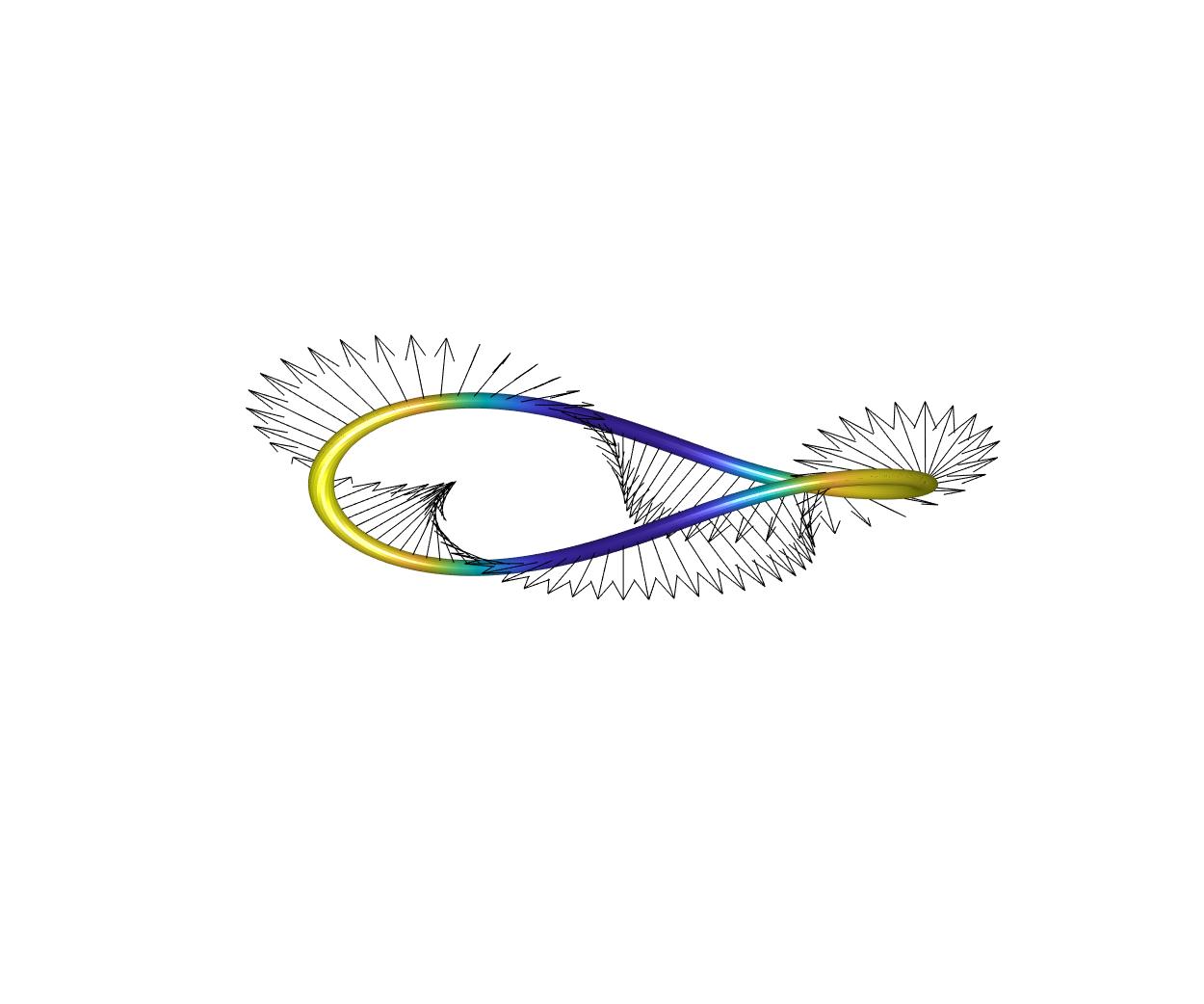}
\bild{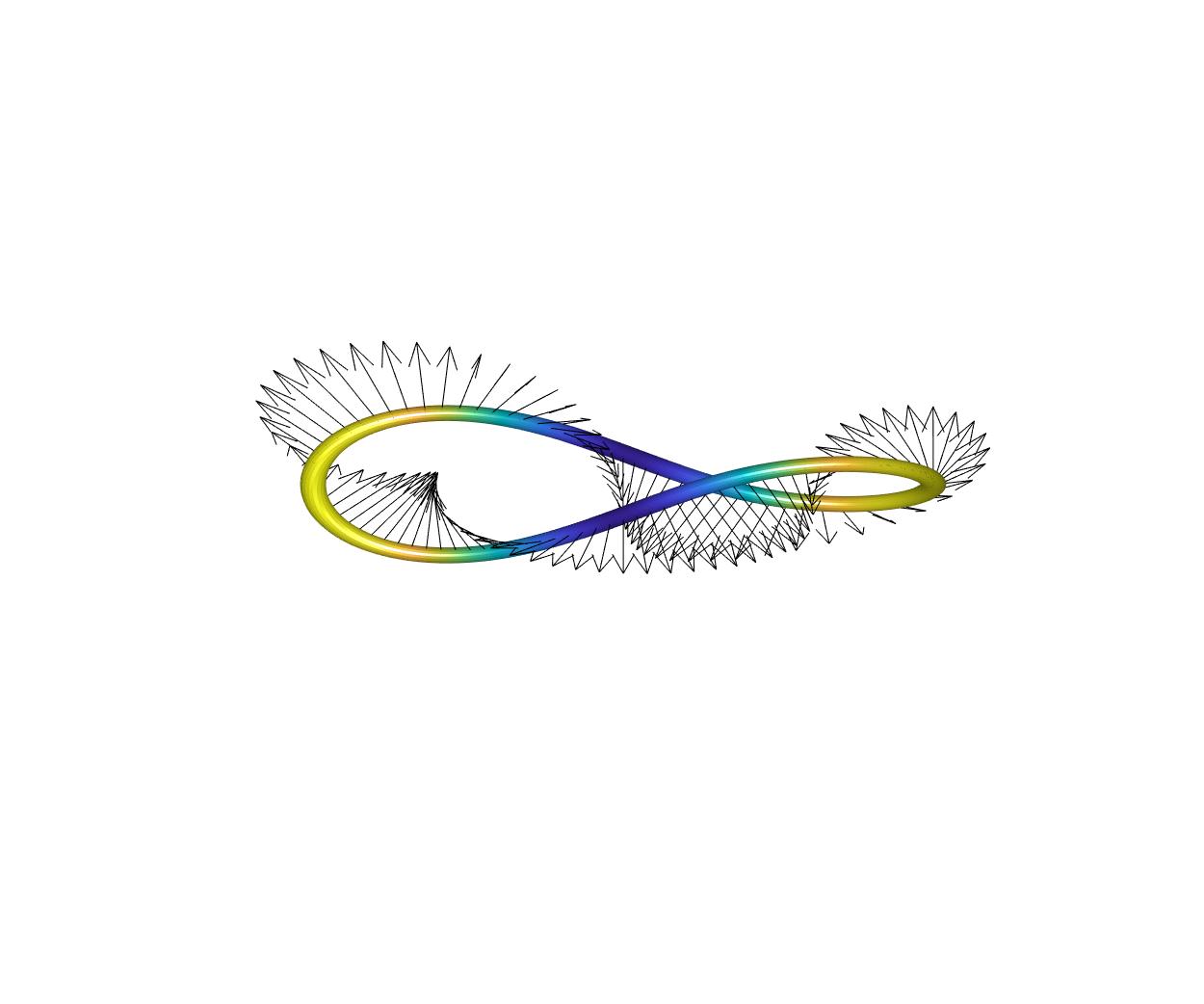}

\bild{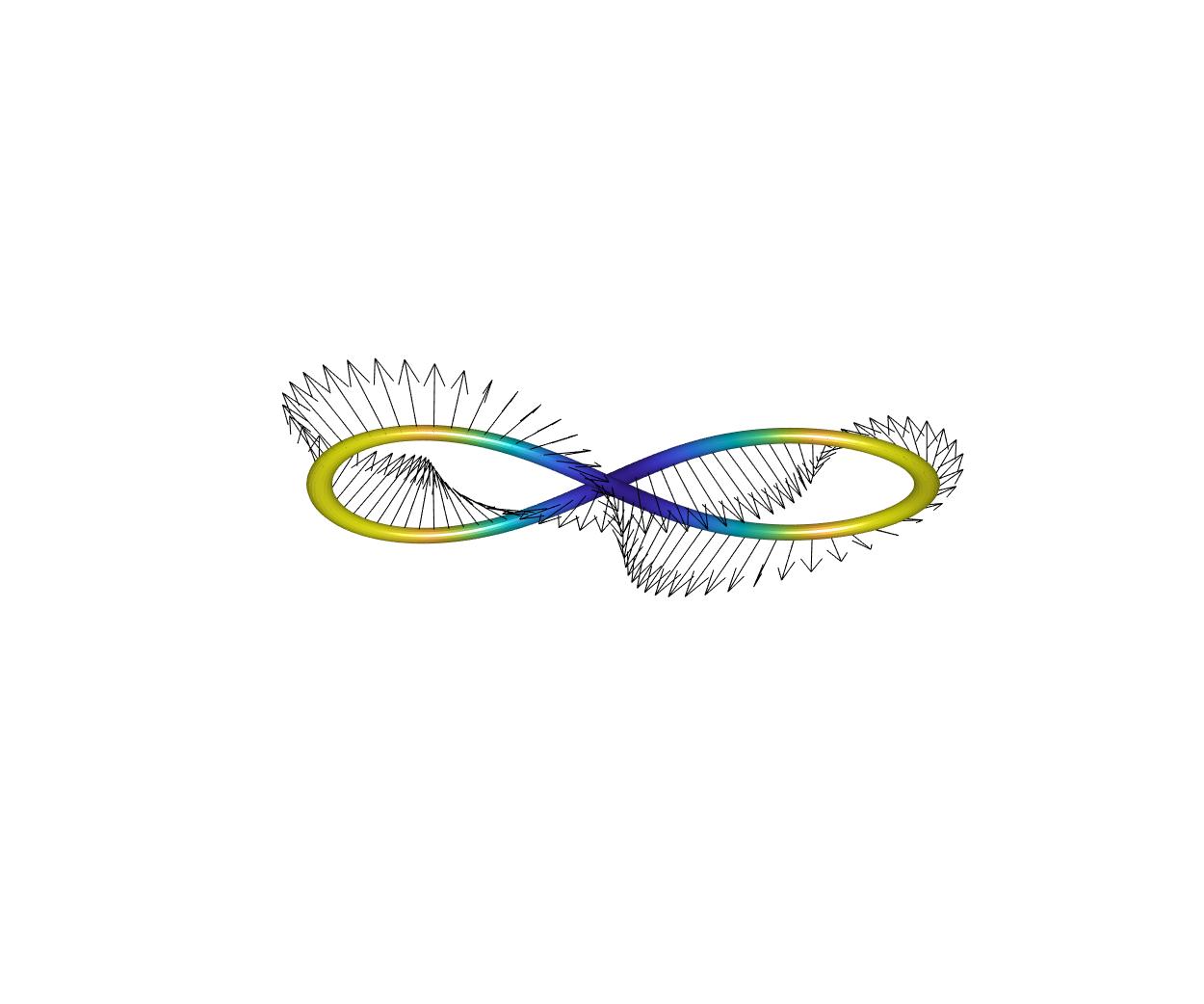}
\bild{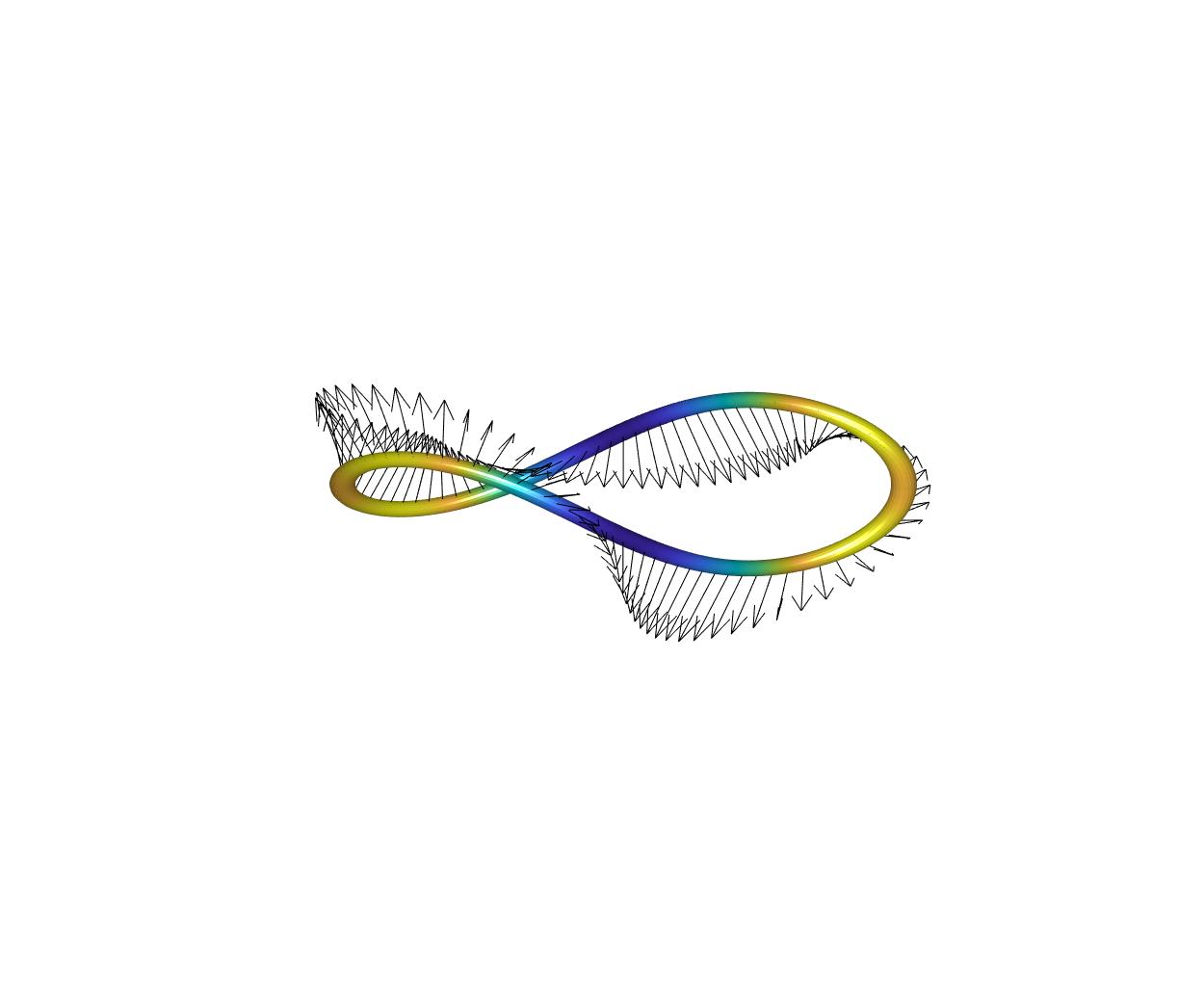}
\bild{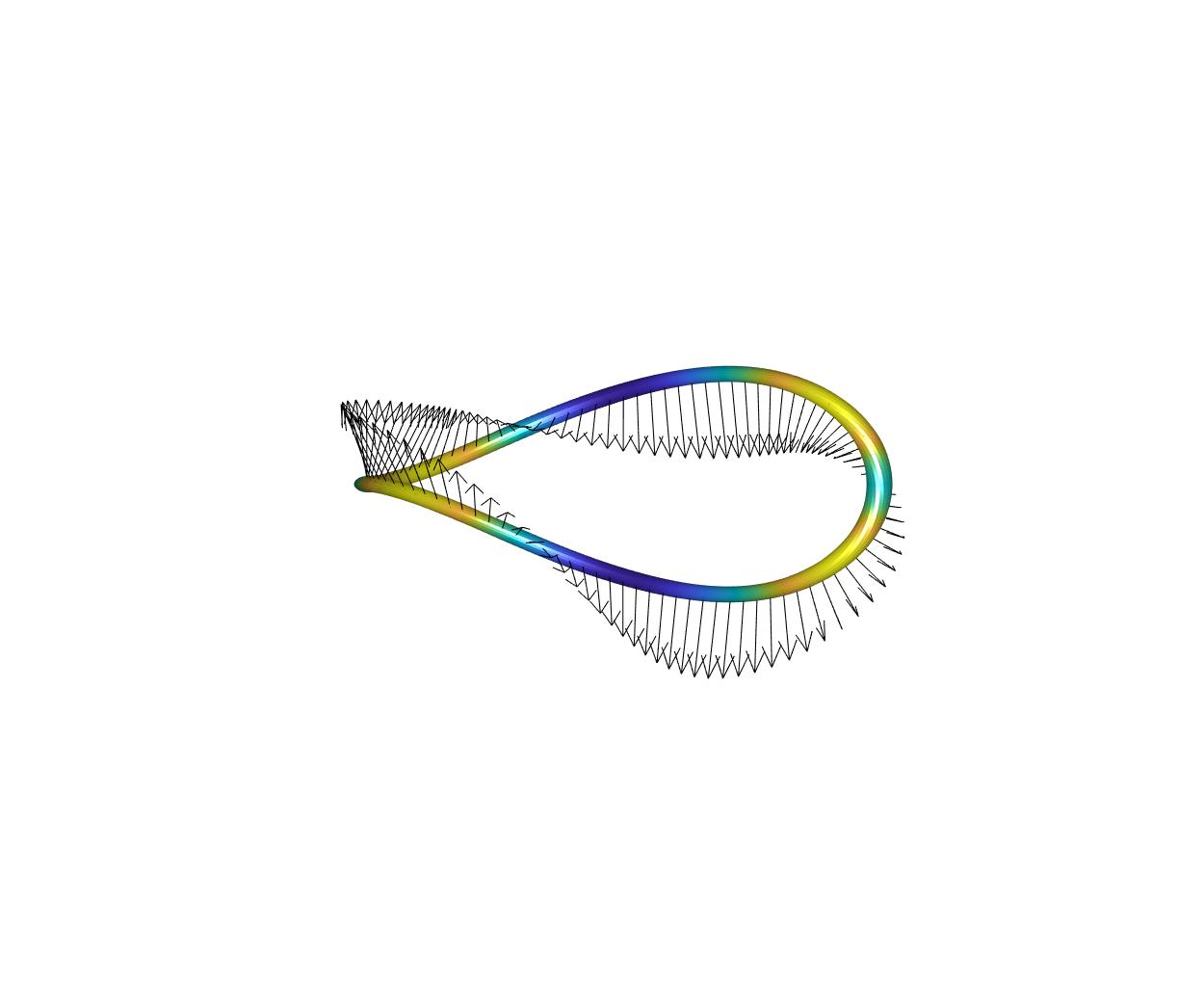}
\bild{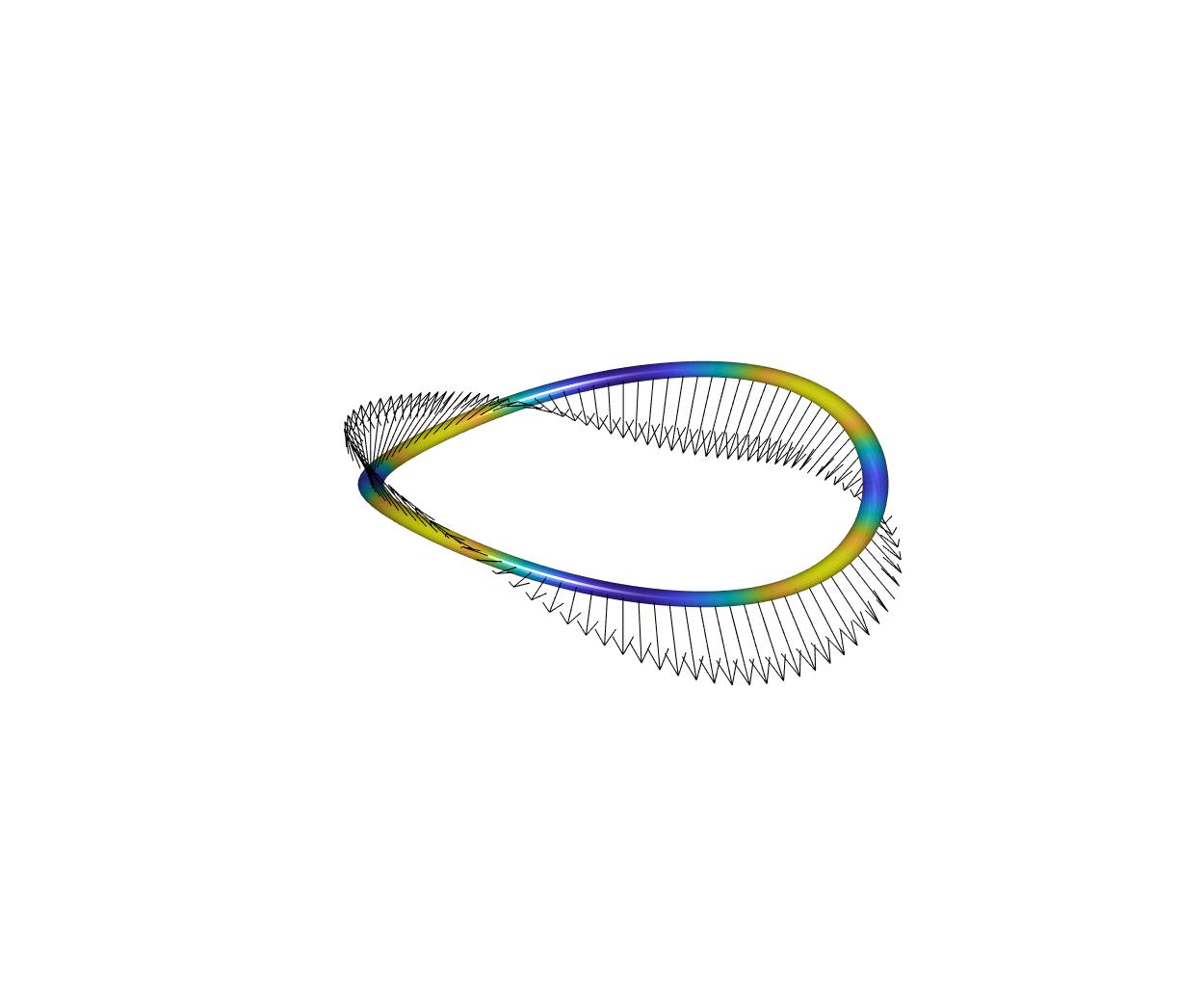}

\bild{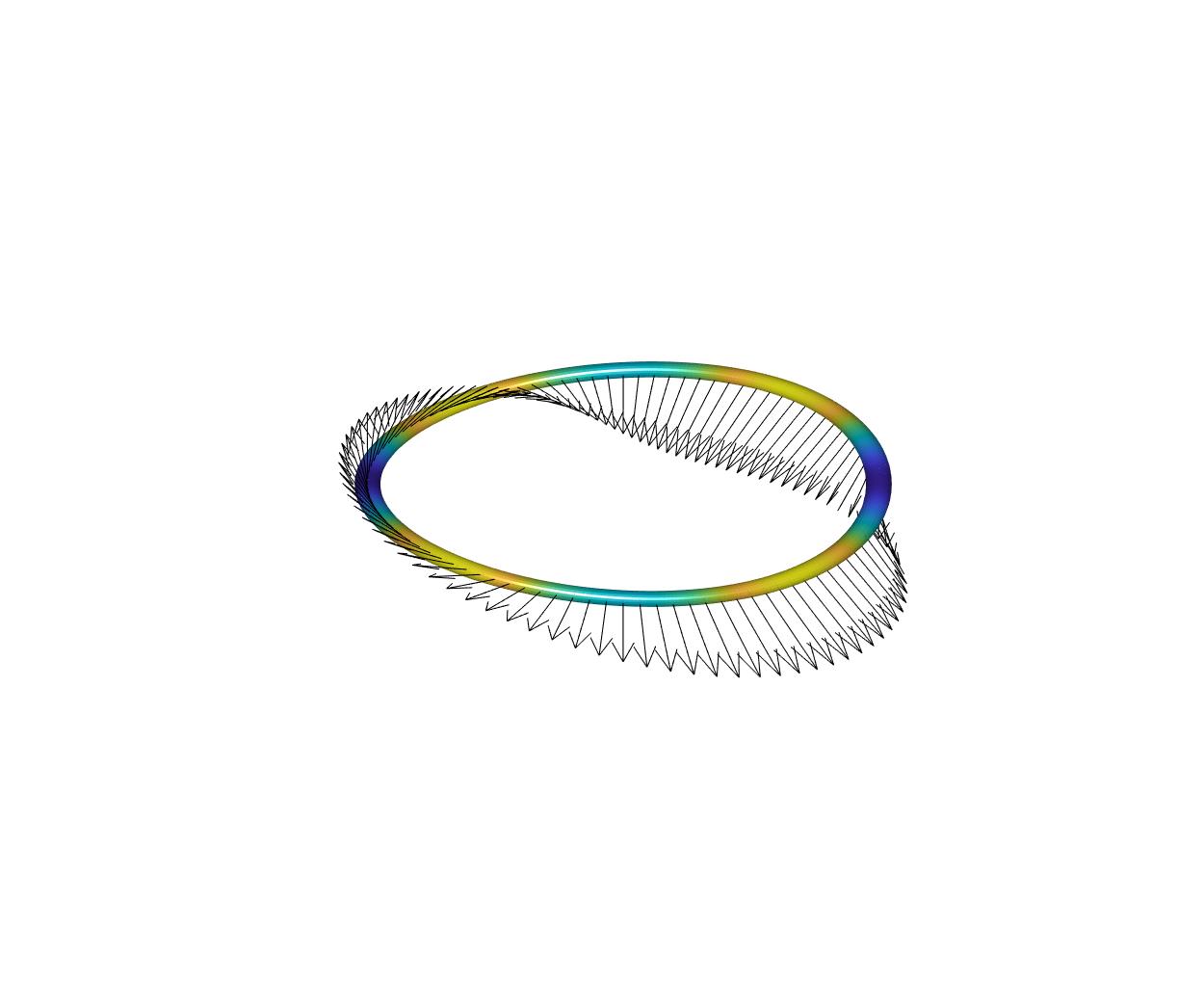}
\bild{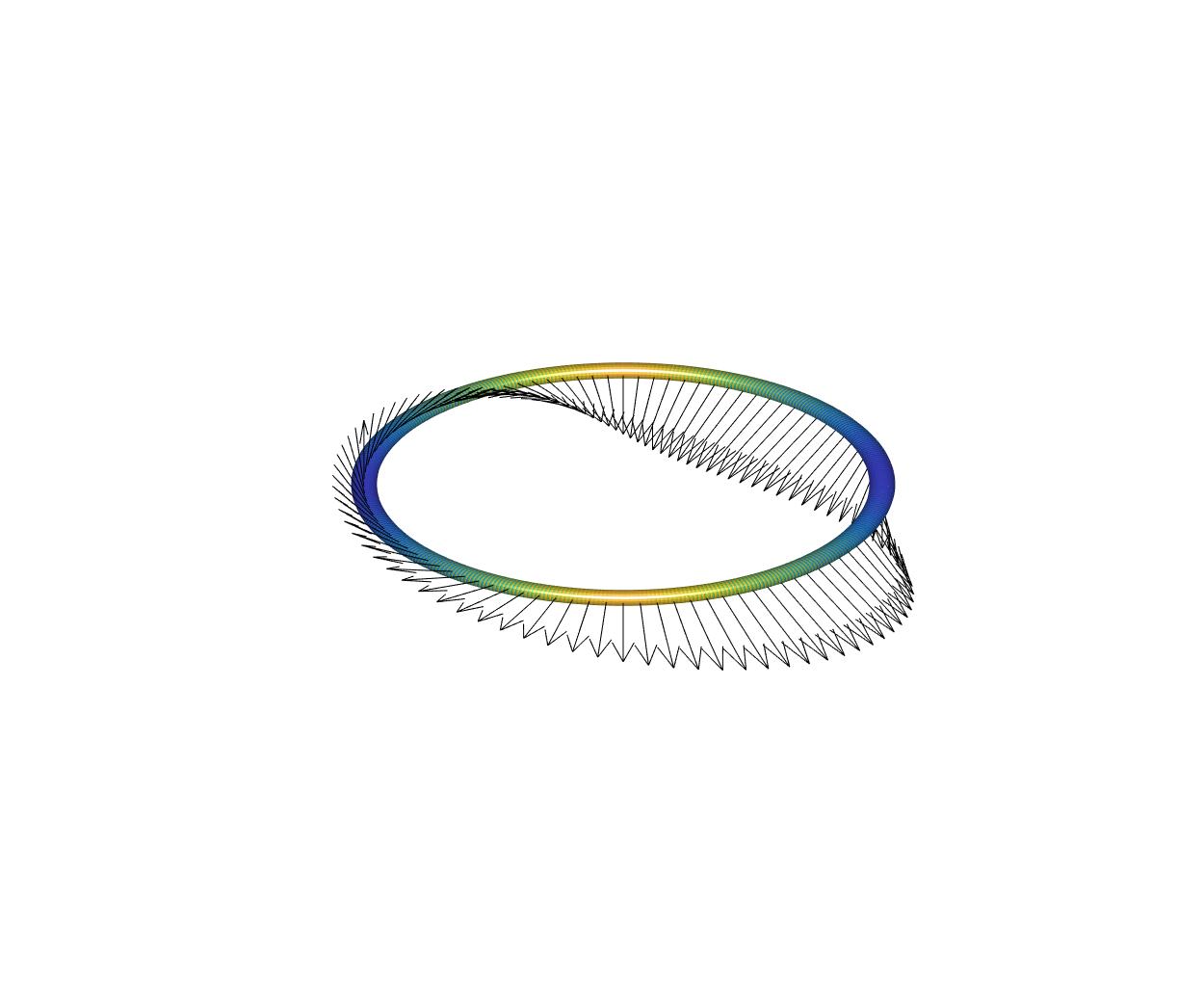}
\bild{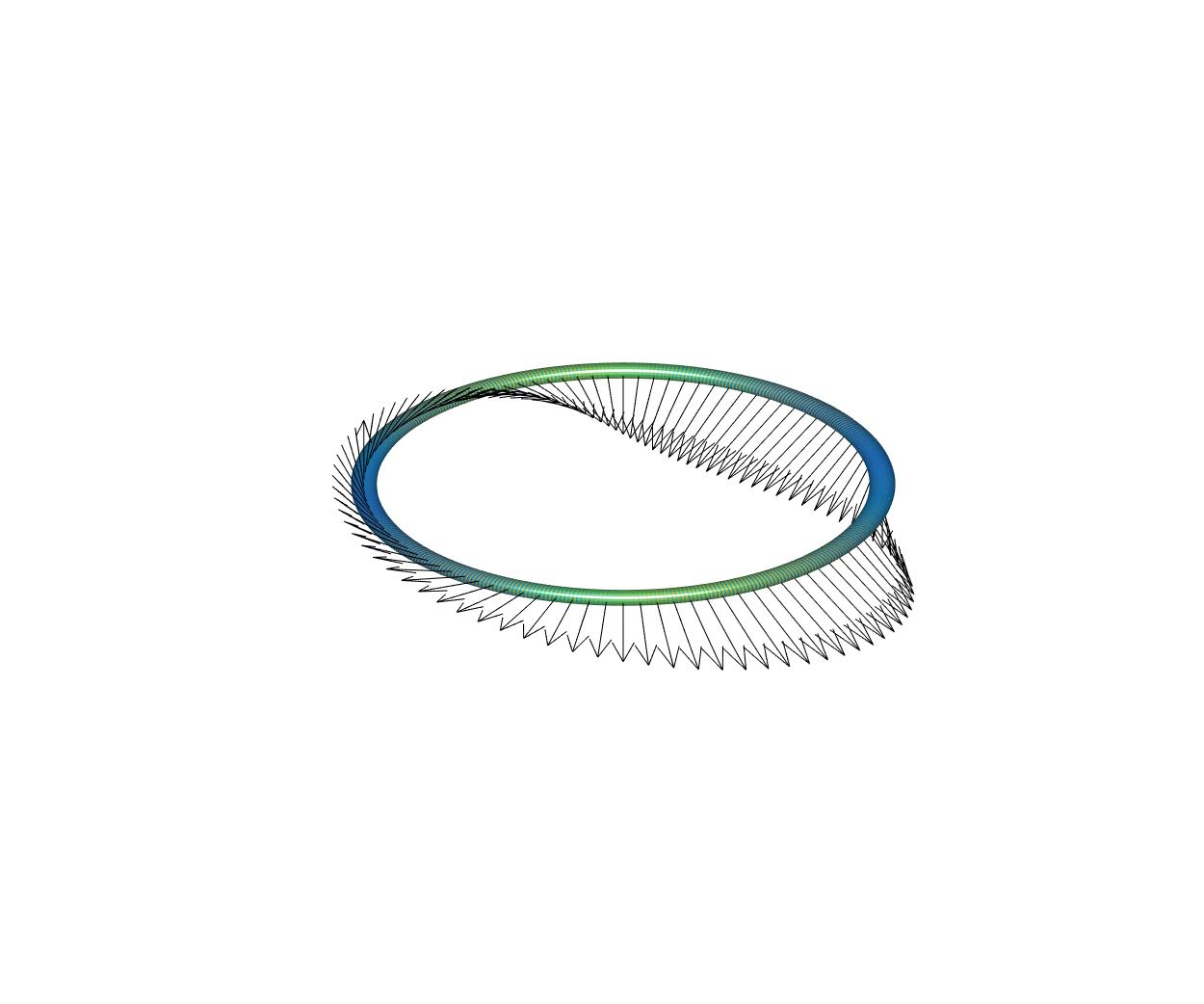}
\bild{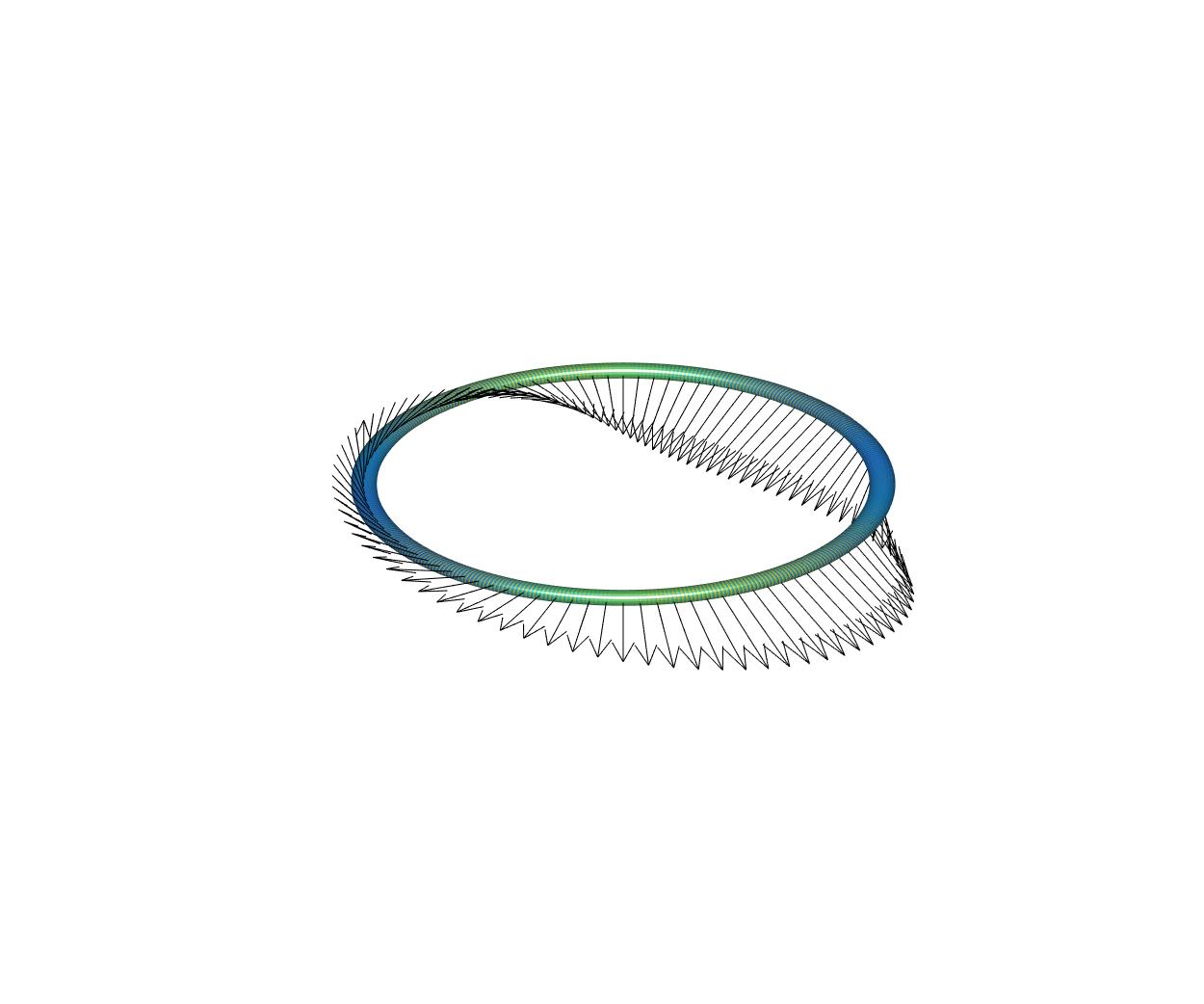}

\caption{Evolution of a round circle twisted by five full rotations
{in Experiment~\ref{sect:overtwist}}.
Twist is reduced due to self-penetrations of the rod
{around iteration steps $k=46{,}400$ and $k=48{,}800$}.
Initial and final curves are round circles {which appear to lie} in the same plane.
{Plots of energy profile and twist rate
are shown in Figure~\ref{fig:overplot} below.}}
\label{fig:overtwist}
\end{figure}

\begin{figure}
 \includegraphics[scale=.29]{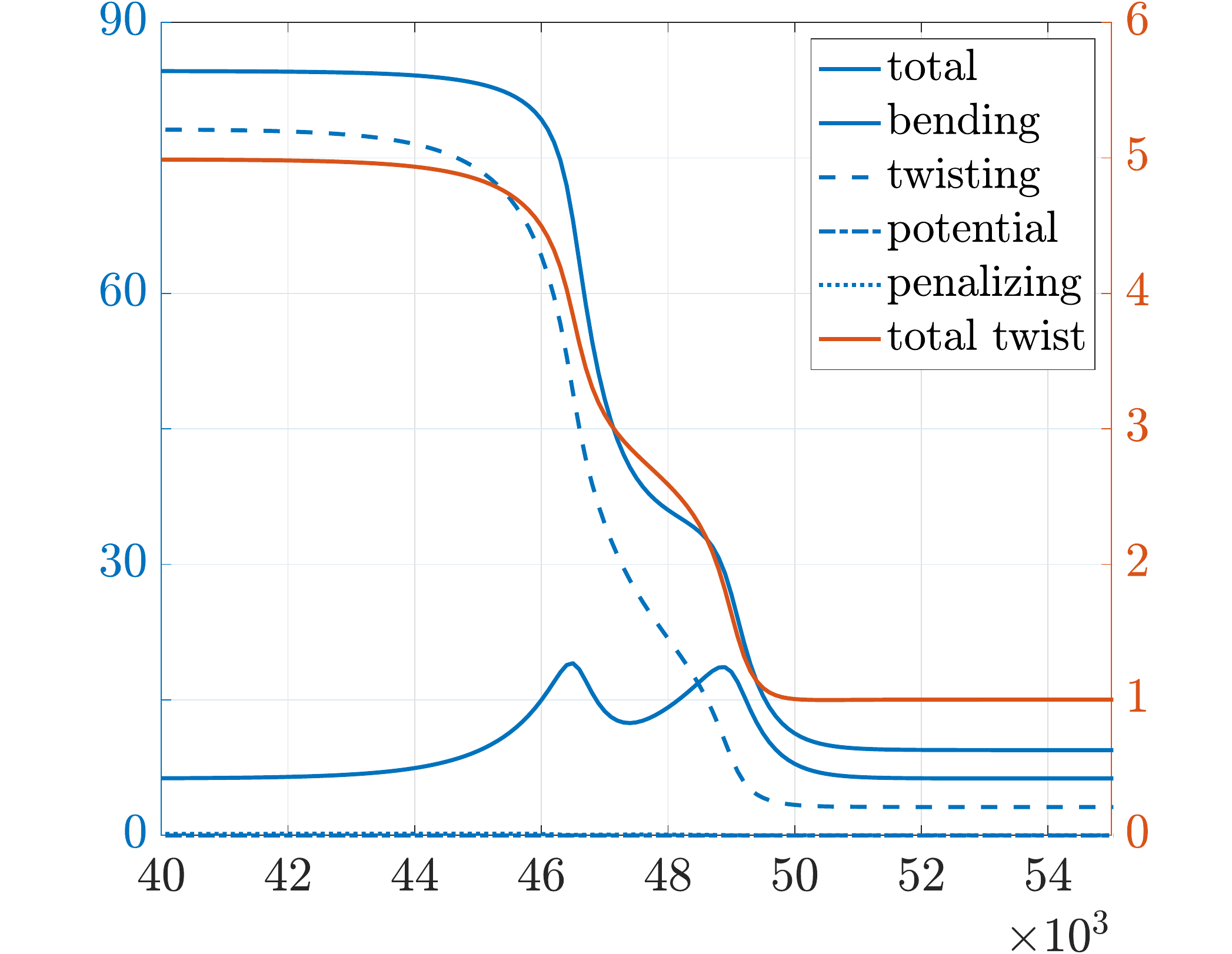}
 \hfill
 \includegraphics[scale=.29]{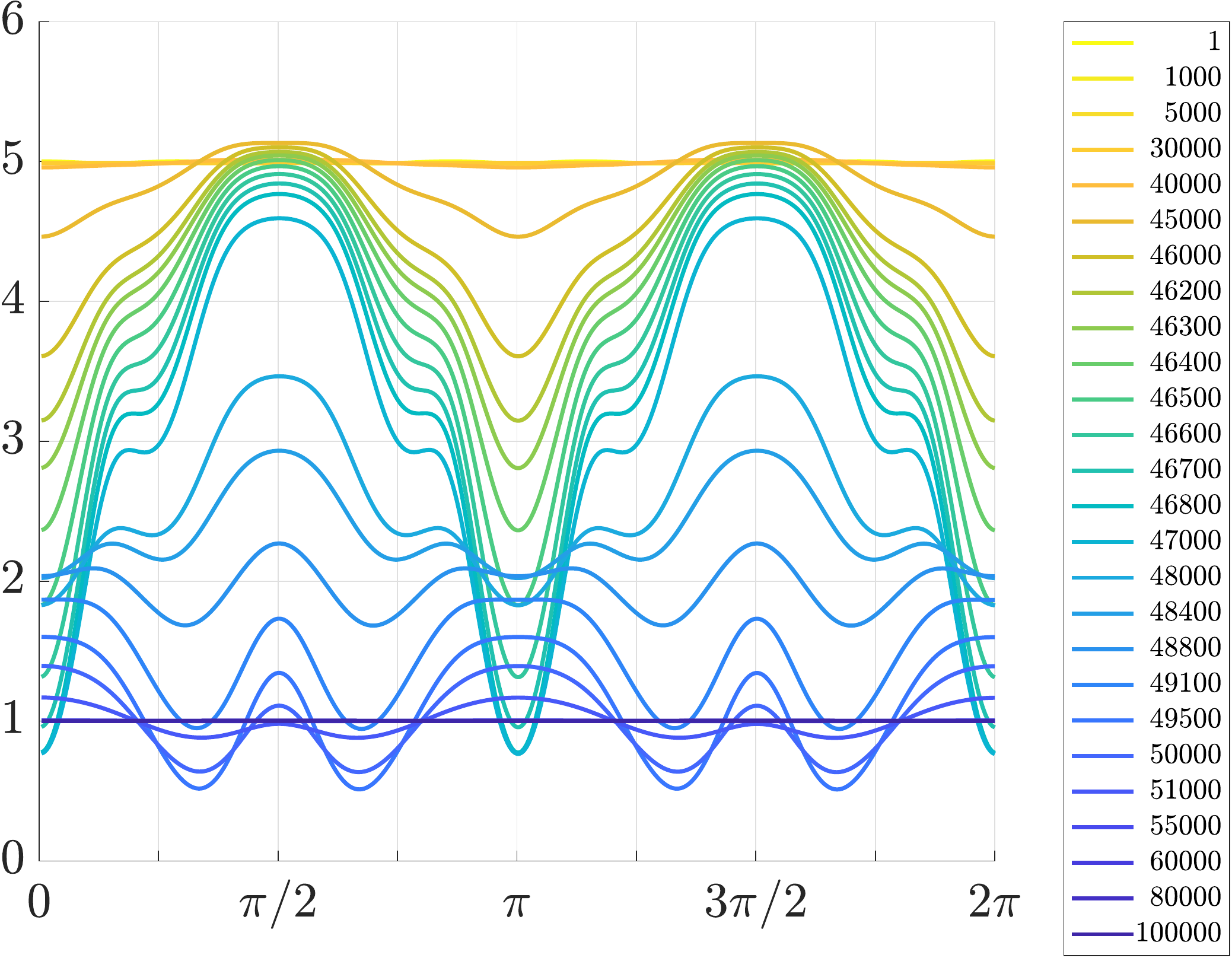}

 \caption{{(Experiment~\ref{sect:overtwist})
 \newline\emph{Left}:
 The energy plot reveals two peaks of the bending
 profile corresponding to the two self-penetrations.
 The energy profile is virtually constant for the iteration steps
 outside the range shown here.
 \newline\emph{Right}:
 Plot of the twist rate~$\beta$ for the iteration
 steps visualized in Figure~\ref{fig:overtwist}.
 Apart from the initial configuration,
 the twist rate is non-uniform throughout the evolution.
 It eventually dissipates around iteration step $100{,}000$.}}
 \label{fig:overplot}
\end{figure}

We repeat the experiment from Section~\ref{sect:michell} with
$\kappa=2$ and $\beta_{\ini}=5$, i.e., for a continuous frame.
Here we have $\beta_{*}\approx3.4641$.
The same slight perturbation {has been added} to the initial curve as before in~\eqref{eq:perturb}.

In this case we observe twist reduction by two consecutive self-penetrations
although the evolution could possibly stop after the first one
since the twist value is then already below the threshold $\beta_{*}$.
Obviously,
{the gradient of the energy
{of the noncircular configuration around iteration step $k=48{,}000$} is so steep}
that Michell's instability does not play any role here.

The evolution is visualized in Figure~\ref{fig:overtwist} and
the energy values are plotted in Figure~\ref{fig:overplot} (left).
As the initial and final configurations are circular
and the frame closes up (because of $\beta_{\ini}\in\Z$),
the total twist is integer at the beginning and end of the
evolution (cf.\@ {Section}~\ref{sect:calu}).

The twist rate, however, does not stay uniform throughout the
evolution as can be seen from Figure~\ref{fig:overplot} (right).
Eventually the twist will be balanced similarly to {Experiment}~\ref{sect:uniframe}.

\subsection{Planar figure eight}
\label{sect:f8}

\renewcommand{\bild}[1]{\fbox{%
\scalebox{-1}[1]{%
\includegraphics[scale=.1,trim=260 300 140 270,clip]{d#1.jpg}}}\makebox[0ex][r]{\makebox[0ex][r]{\tiny#1}\hspace{0ex} }\,\ignorespaces}

\begin{figure}

\bild{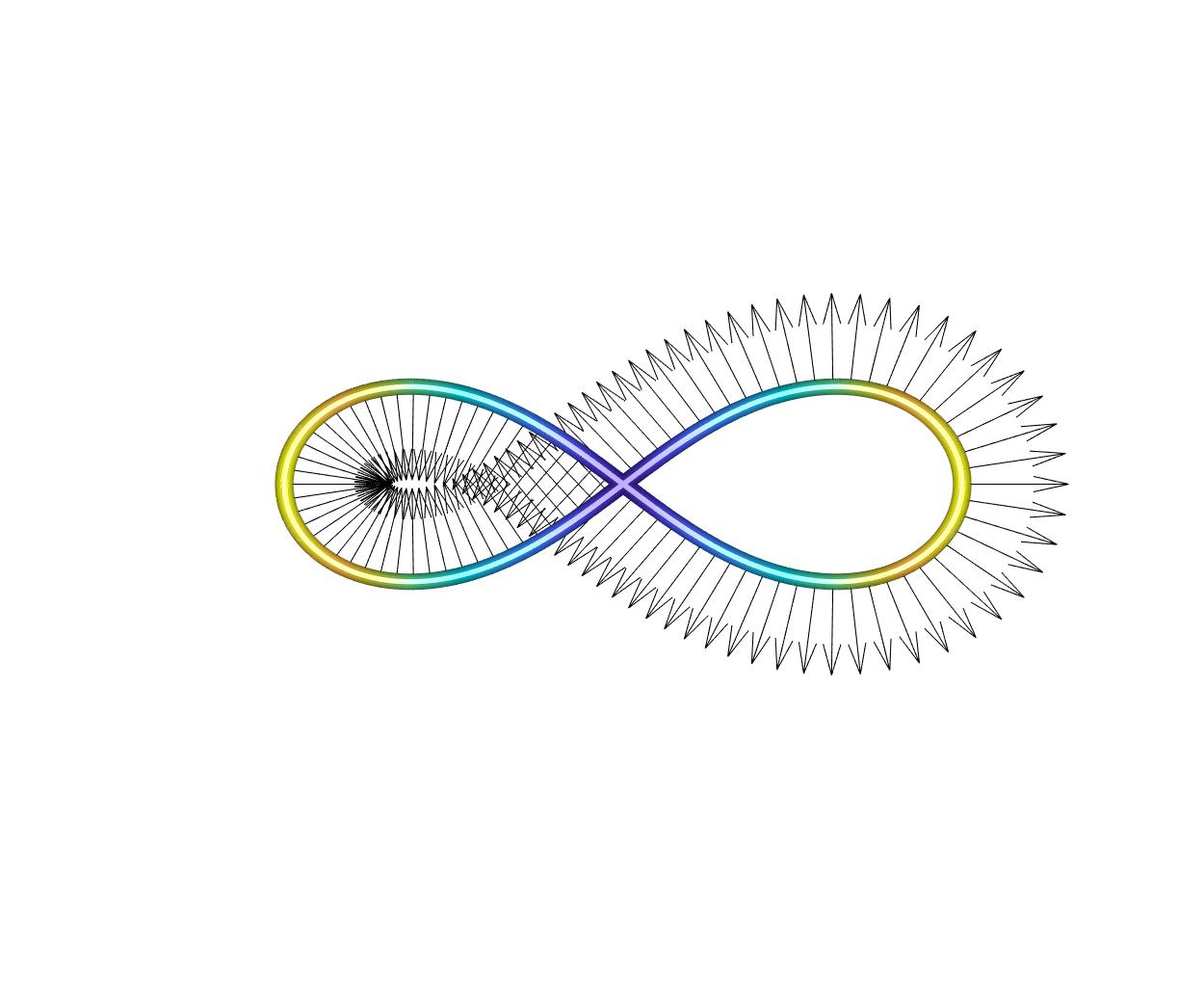}
\bild{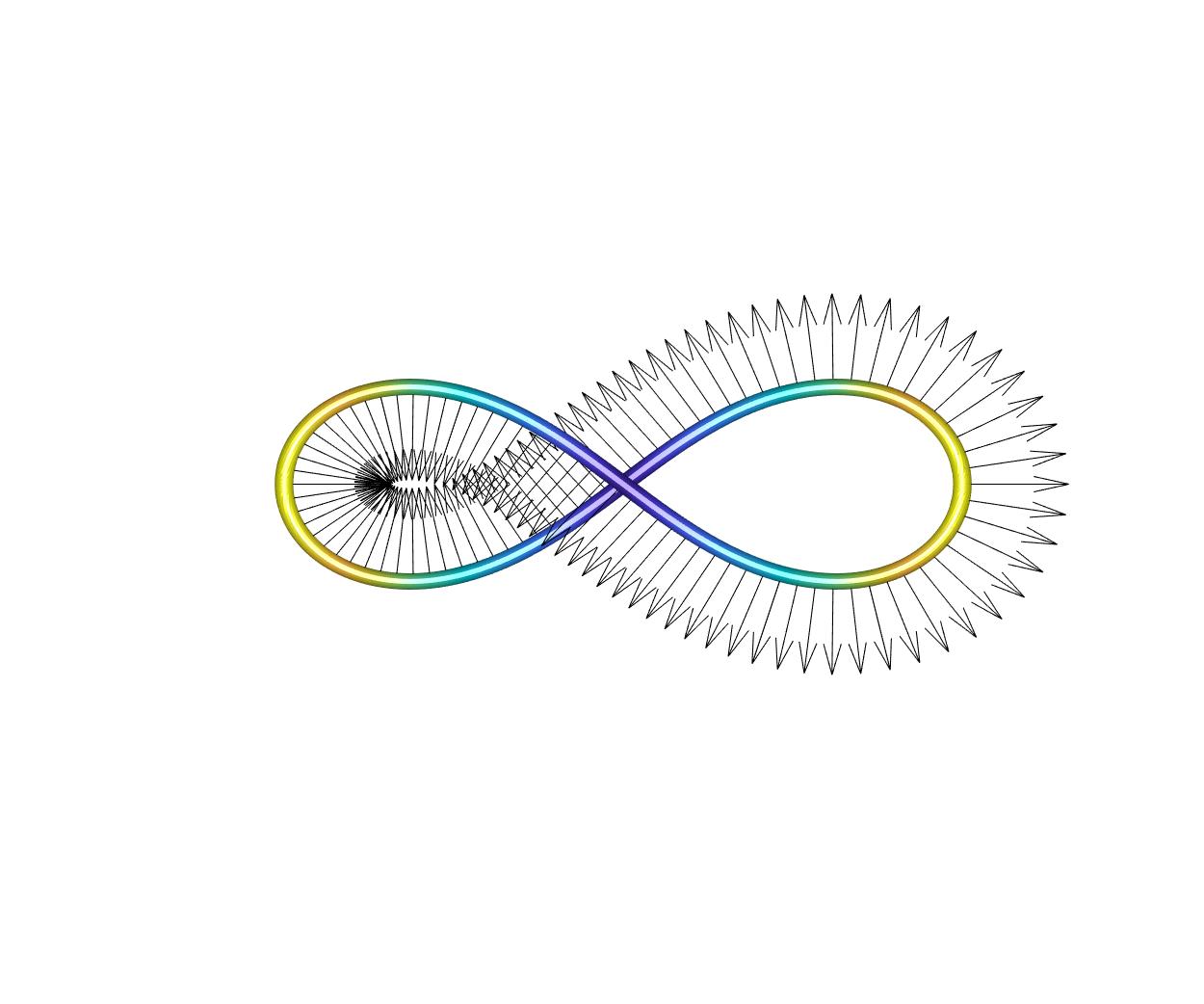}
\bild{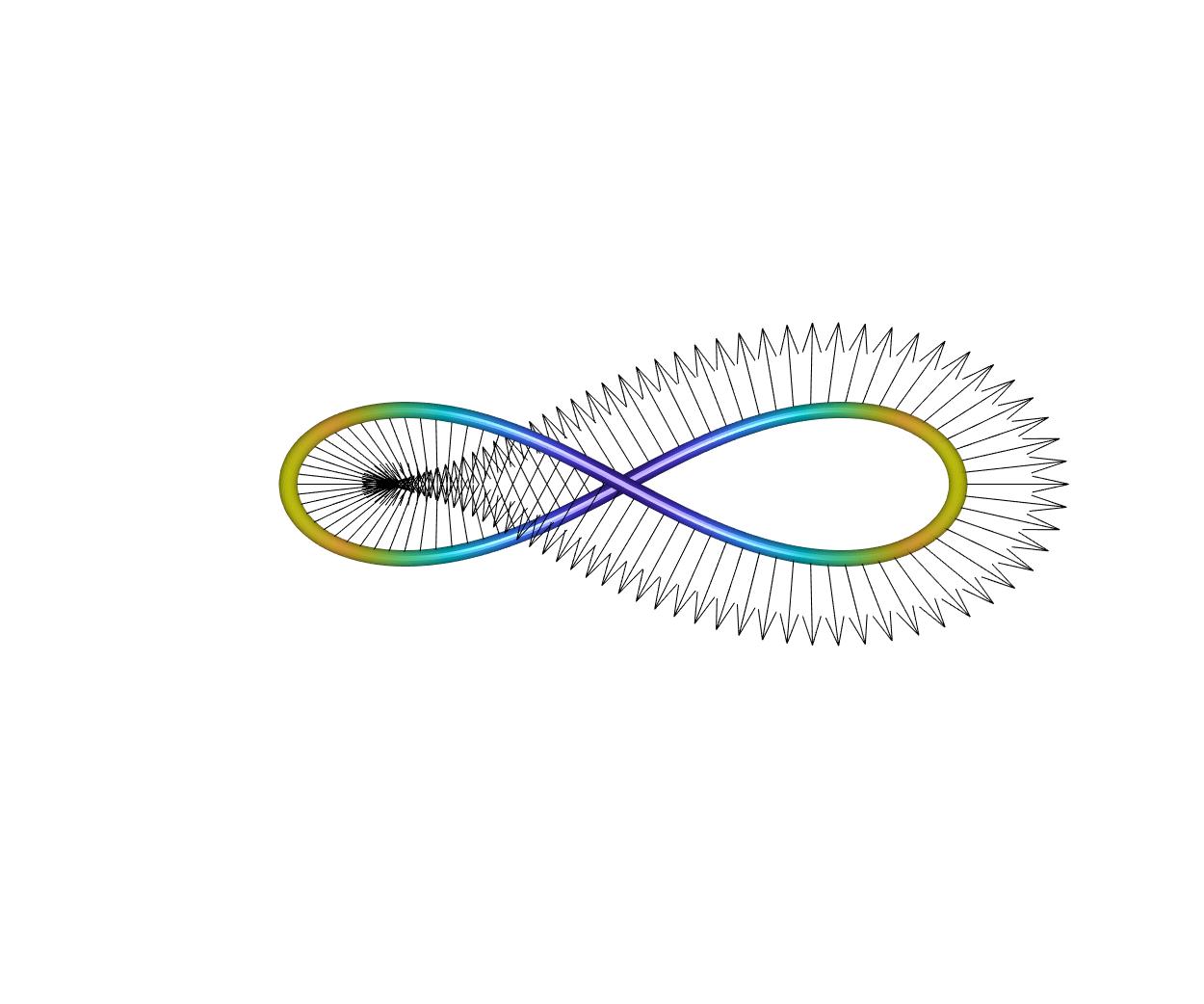}
\bild{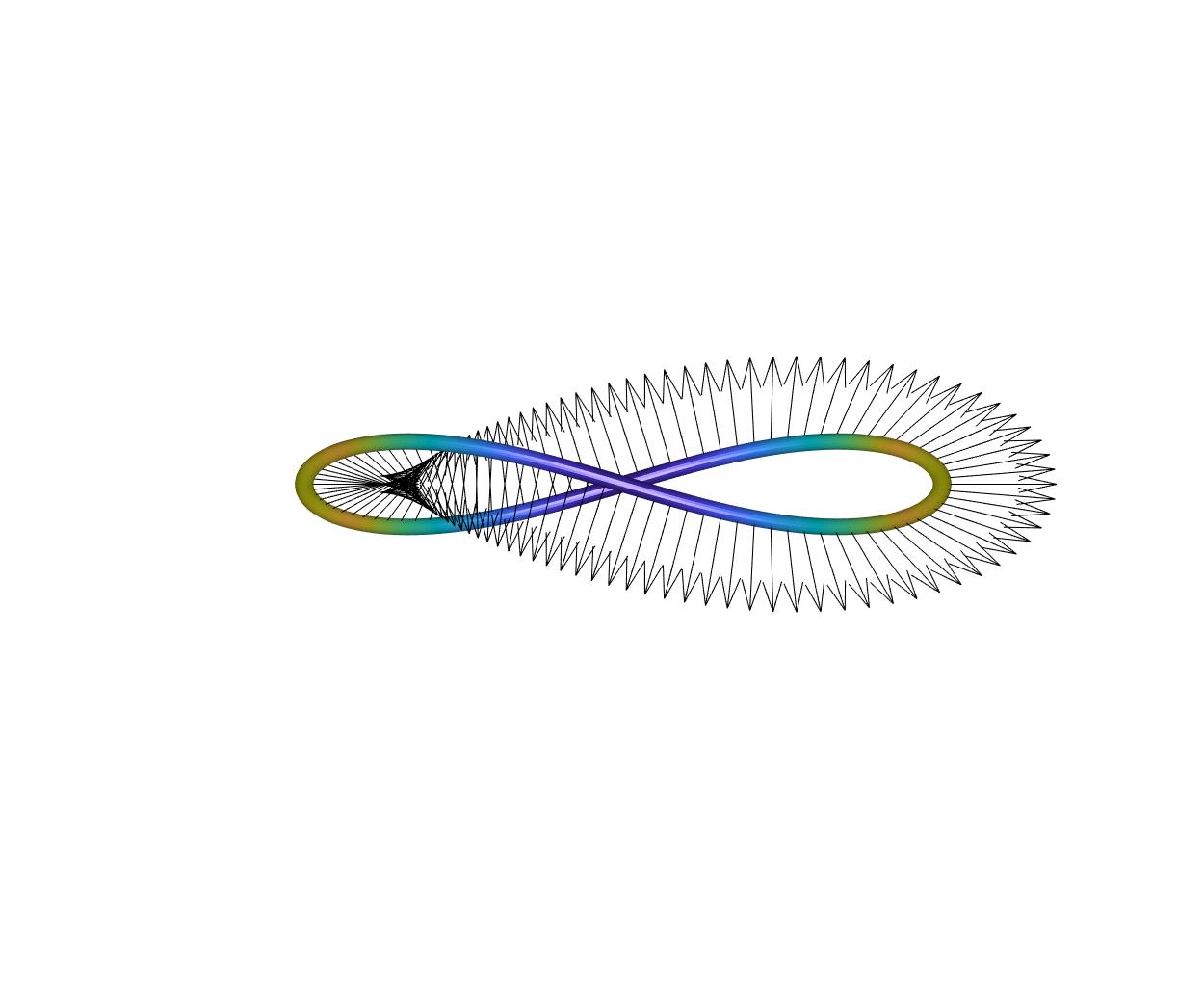}

\bild{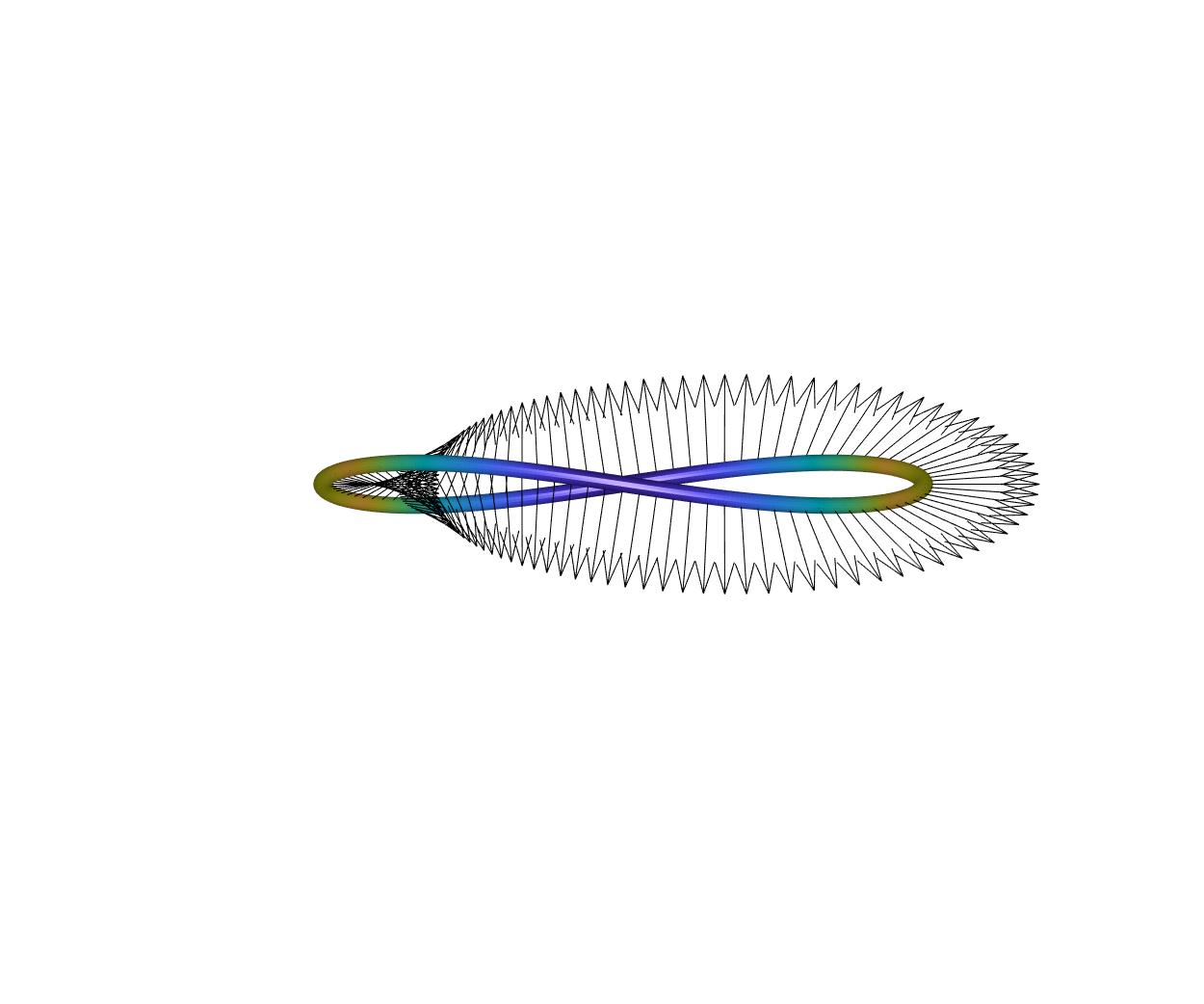}
\bild{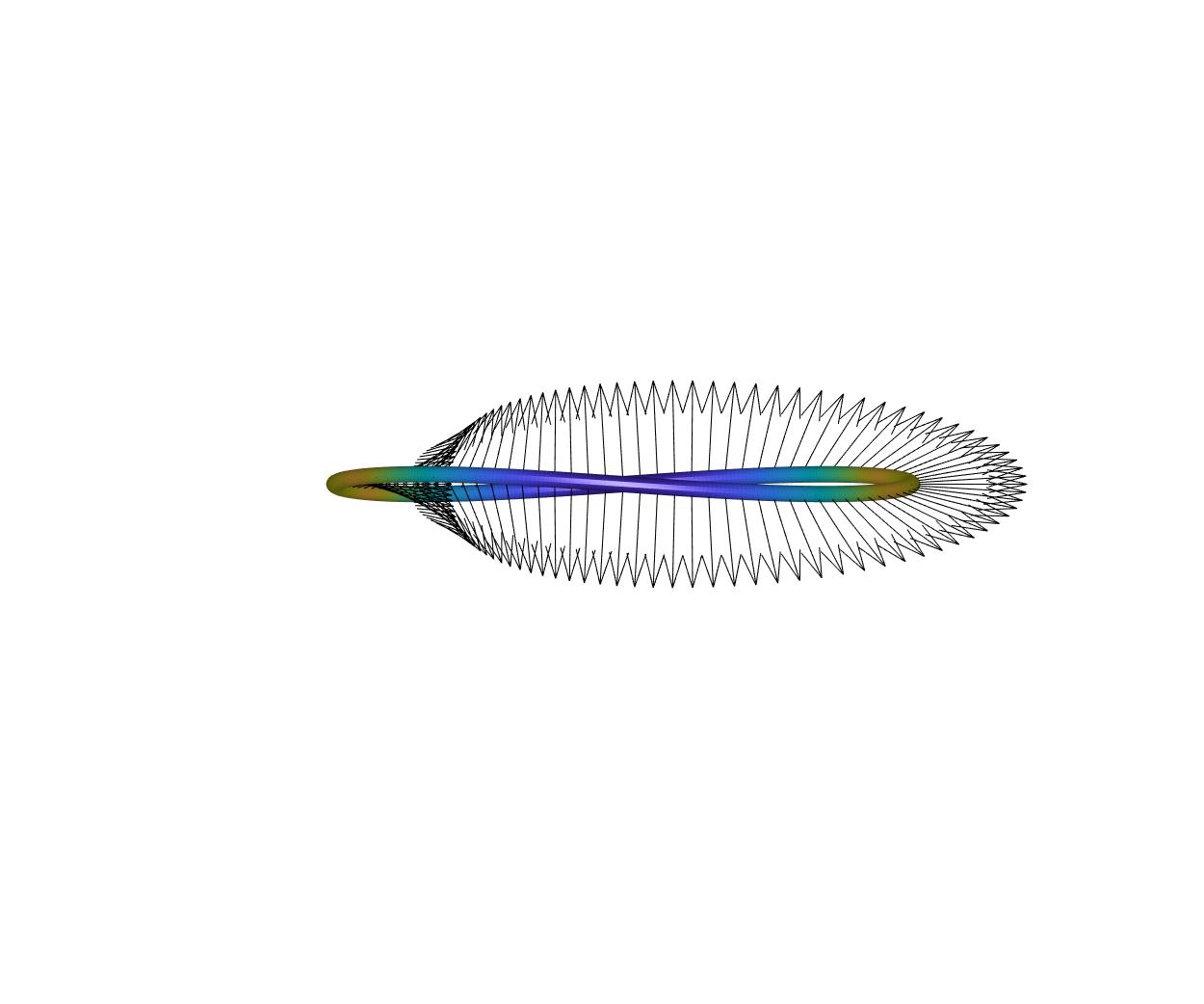}
\bild{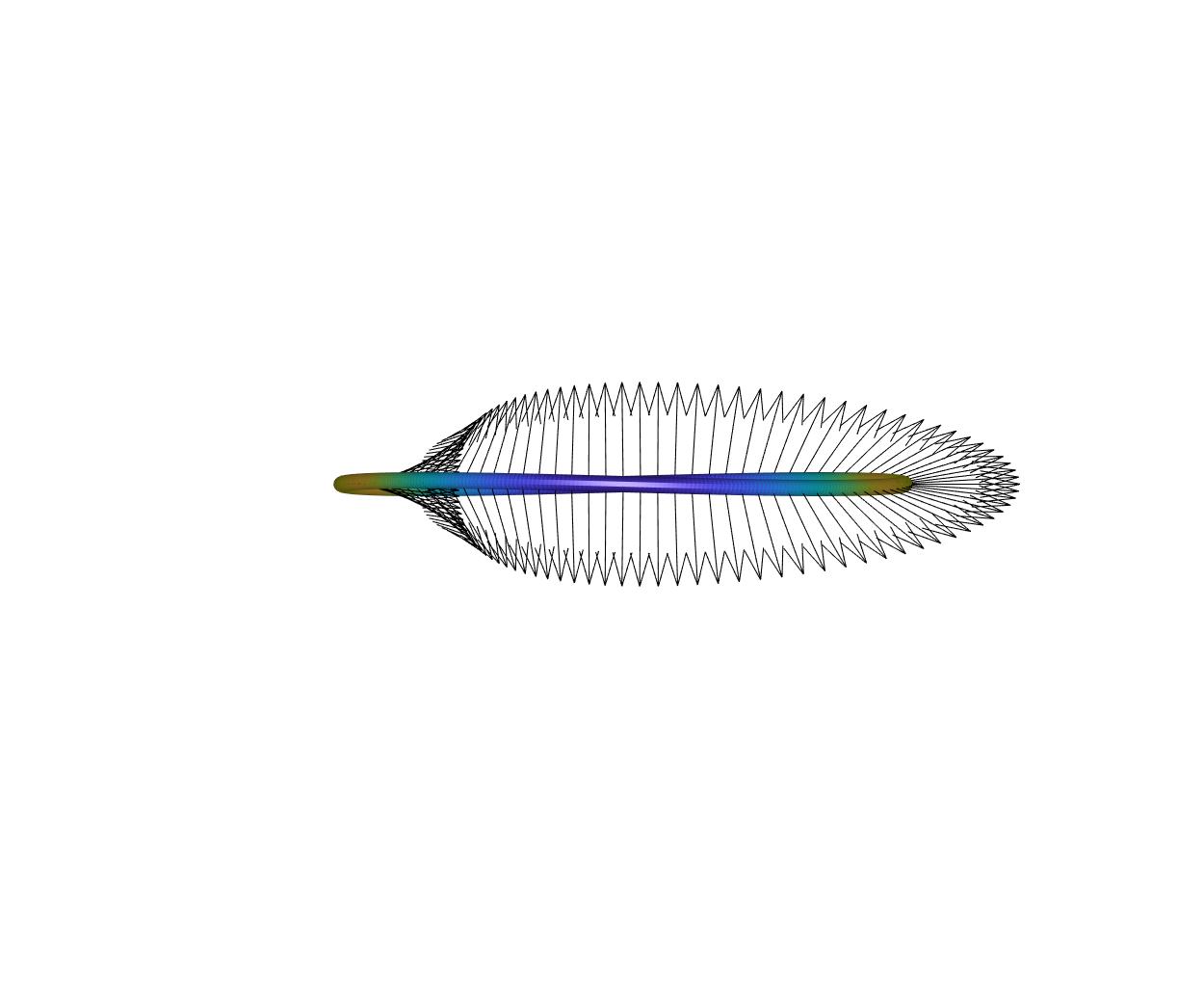}
\bild{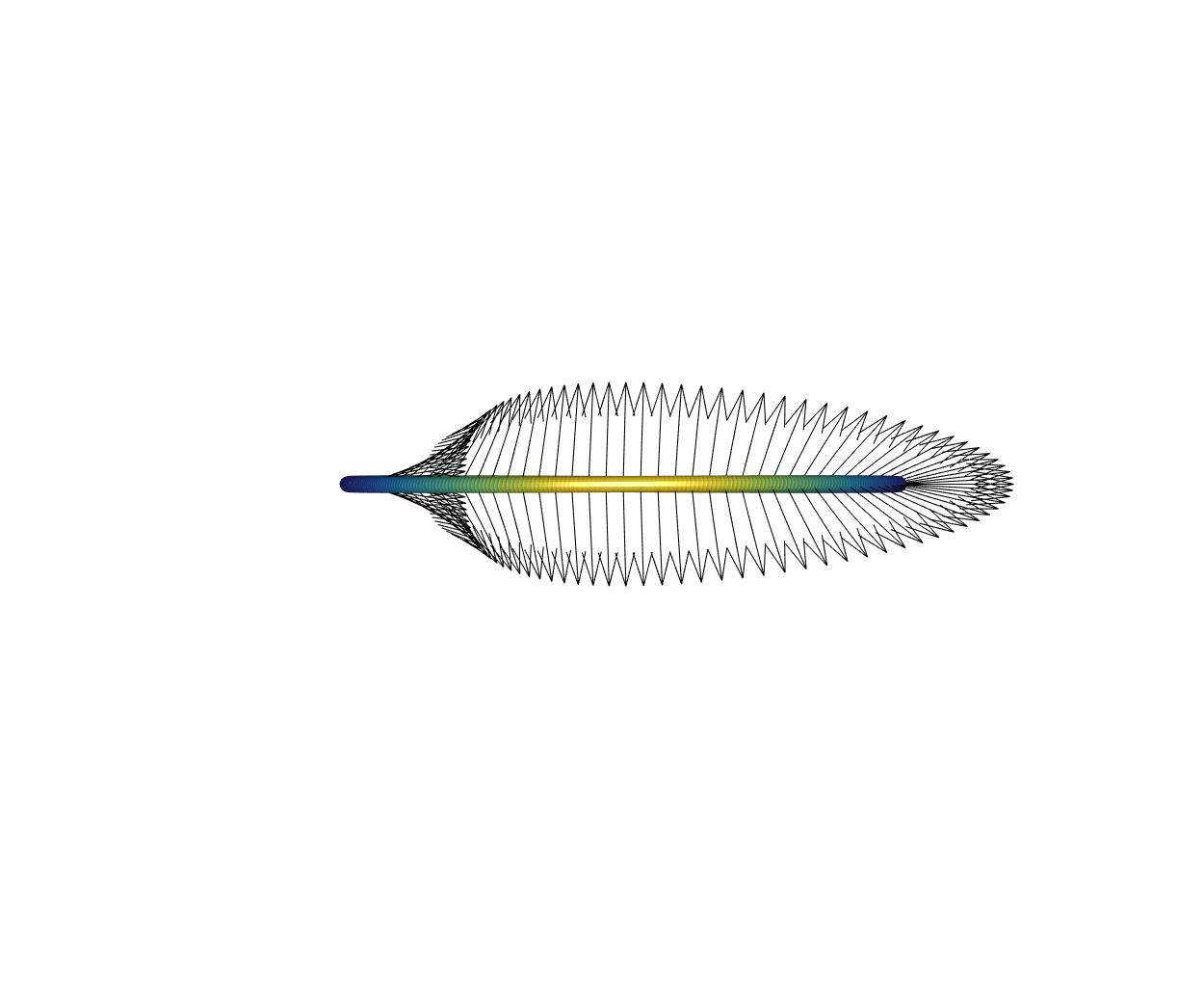}
\caption{Evolution of a {twist-free} planar elastic figure eight
{from Experiment~\ref{sect:f8} with $\k=0.7$}.
The initial curve is (almost) planar and evolves to a circle
in a plane which {seems to be} perpendicular to the initial configuration.
{The corresponding frame performs one full rotation.
Obviously the bending forces dominate in this case.
The energy plot is shown in Figure~\ref{fig:f8-plot} (right).}}
\label{fig:f8}
\end{figure}

\begin{figure}
 \raisebox{2ex}{\includegraphics[scale=.35]{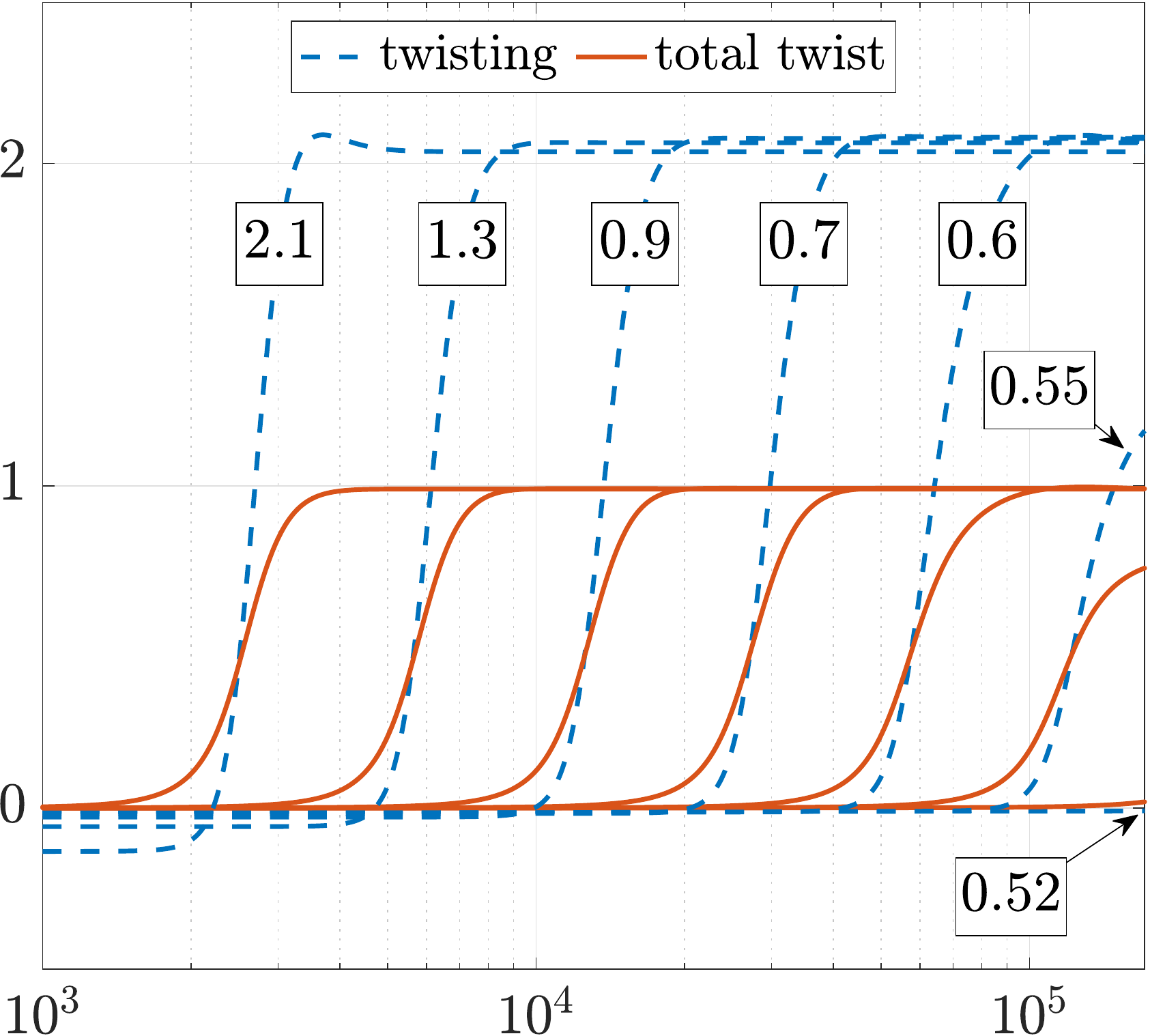}}\hfill
 \includegraphics[scale=.35]{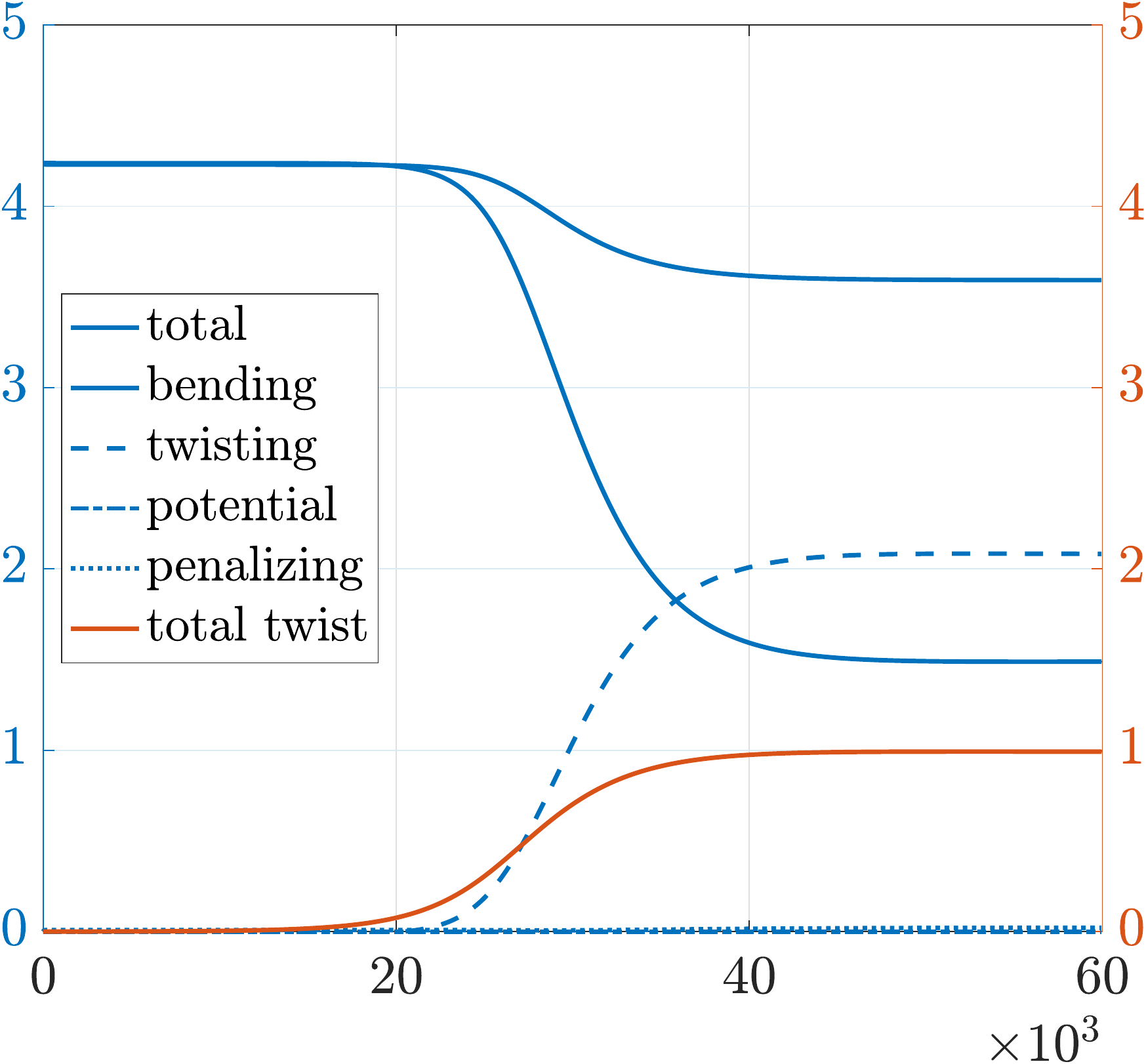}
 \caption{{(Experiment~\ref{sect:f8})
 \newline\emph{Left}:
 Aiming at experimentally confirming the threshold $\kappa=\tfrac12$,
 we study (semi-logarithmically scaled) profiles of twisting energy and total twist
 for several evolutions of the figure-eight curve.
 The smallest iteration step at which the total twist is above	
$\tfrac12$ serves as a threshold for the speed of the evolution.
 \newline\emph{Right}:
 \protect\scalebox{.95}[1]{Energy plot for the evolution in Figure~\ref{fig:f8} where $\k=0.7$.}}}
 \label{fig:f8-plot}
\end{figure}

Any closed \emph{planar} elastica (i.e., a critical point of the bending energy)
is either a circle or a planar figure-eight curve,
possibly several times covered, cf.\@ Sachkov~\cite{sachkov}.

Explicit formulae for elastica based on special functions have
been computed in the 19th century, see the references in Levien~\cite{levien}.
Here we make use of an arclength parametrization given by
Dall'Acqua and Pluda~\cite{DAP} which relies on earlier work by
Djondjorov et al.~\cite{Djondjorov}, 
namely
\[ {y(s) = }
\begin{pmatrix}
2E(\am(s,m),m)-s \\ 2\sqrt m \cn(s,m)
\end{pmatrix}, \qquad s\in\R/4K(m)\Z. \]
Here $E$ denotes the incomplete elliptic integral of the second kind
and $K$ the complete elliptic integral of the first kind
while $\am$ is the Jacobi amplitude function and $\cn$
the elliptic cosine function, cf.~\cite{DAP}.
The (signed) curvature amounts to $s\mapsto-2\sqrt m \cn(s,m)$.
The figure-eight curve corresponds to
$m\approx0.82611$ which is the uniquely defined number in $(0,1)$
with $2E(\tfrac\pi2,m)=K(m){{}\approx2.321}$.

In order to break the symmetry
which may result in an unstably stationary configuration,
a slight perturbation similarly to~\eqref{eq:perturb}
{has been added to the initial curve}, namely
$s\mapsto \tfrac1{1000}\sin\br{7\cdot\tfrac{2\pi}{4K(m)} s}$,
perpendicularly to the plane (i.e., in $z$-direction);
the frame is corrected accordingly.

According to Ivey and Singer~\cite[Sect.~7]{ivey-singer}
the twist-free planar figure-eight is stable if
$\kappa=\cb/\ct<\tfrac12$ and unstable if $\kappa>\tfrac12$.

We performed several evolutions whose energy plots
are depicted in Figure~\ref{fig:f8-plot} (left).
As in {Experiment}~\ref{sect:michell}, the evolutions are very similar and essentially only differ in speed.
In each case the total twist is raised from zero to one
which is still in accordance with the bound $\abs\Tw\le1$ in {Section}~\ref{sect:mintt}.
 Mind the semi-logarithmic scaling of the horizontal axis.
Negative values of the twisting energy are due to discretization errors
and tend to zero when choosing a larger number of nodes.

Snapshots of the evolution for $\kappa=0.7$ can be found
in Figure~\ref{fig:f8}. We observe an evolution to a round circle with one full twist.
For the same reason as in {Experiment}~\ref{sect:overtwist} we face integer values of $\Tw$ at beginning and end of the evolution.
The full energy plot is depicted in Figure~\ref{fig:f8-plot} (right).

The {parameters} $\kappa$
of the profiles shown in Figure~\ref{fig:f8-plot} (left)
amount to
{$\kappa=\tfrac12+\tfrac{2^{\ell}}{10}$ for $\ell=-2,-1,0,1,2,3,4$.}
The red solid line corresponding to {$\k=0.52$} is just
about to lift at the right margin.
The speed of the evolution seems to reciprocally depend on $\kappa-\tfrac12$,
suggesting that the evolutions for $\kappa\le\tfrac12$ will be stationary.
This confirms the threshold by Ivey and Singer as desired.

\subsection{Open clamped rods}
\label{sect:clamped}

\renewcommand{\bild}[1]{\fbox{%
\scalebox{-1}[1]{%
\includegraphics[scale=.13,trim=260 360 160 300,clip]{e#1.jpg}}}\makebox[0ex][r]{\makebox[0ex][r]{\tiny#1}\hspace{0ex} }\,\ignorespaces}

\begin{figure}

\bild{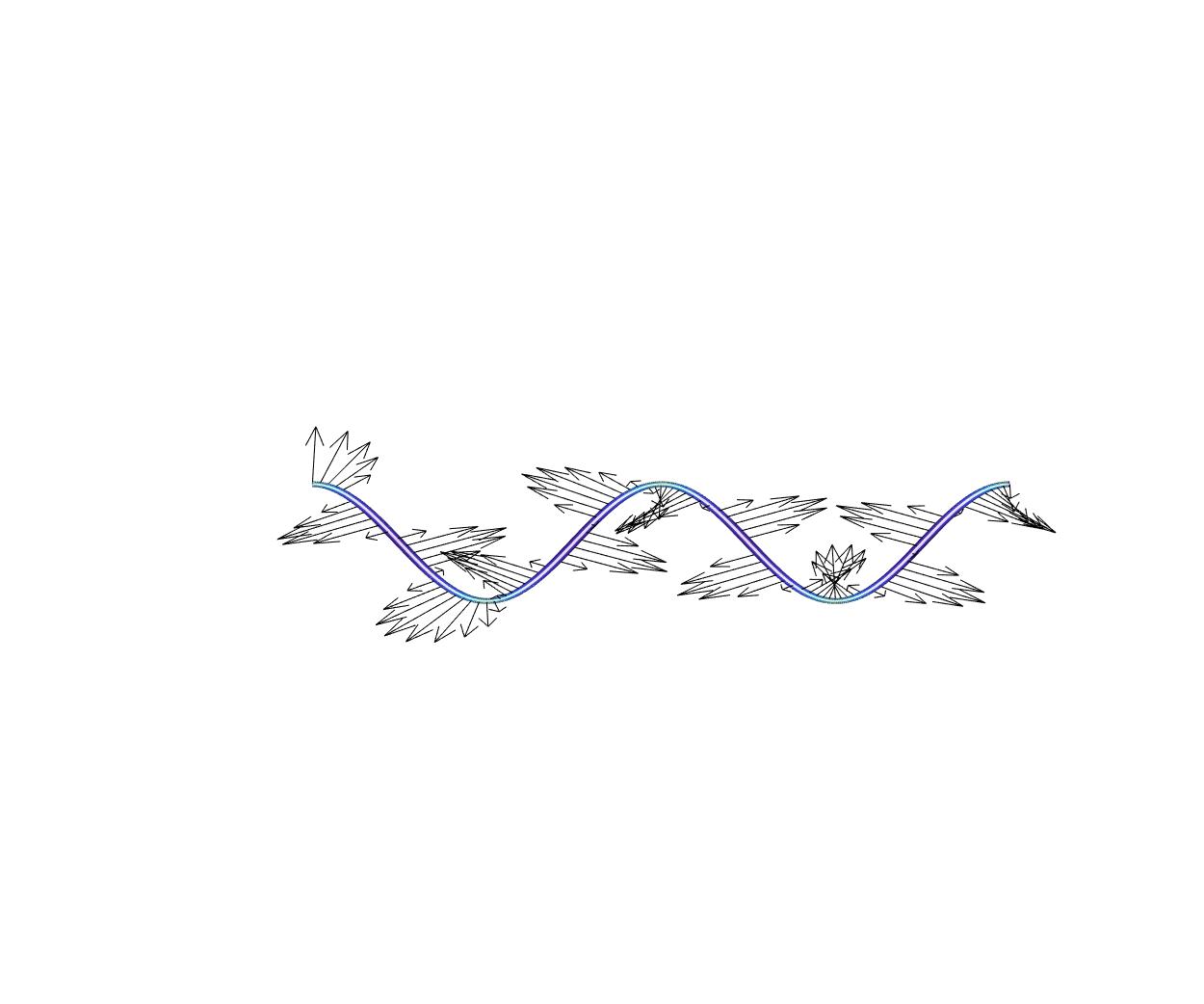}	\bild{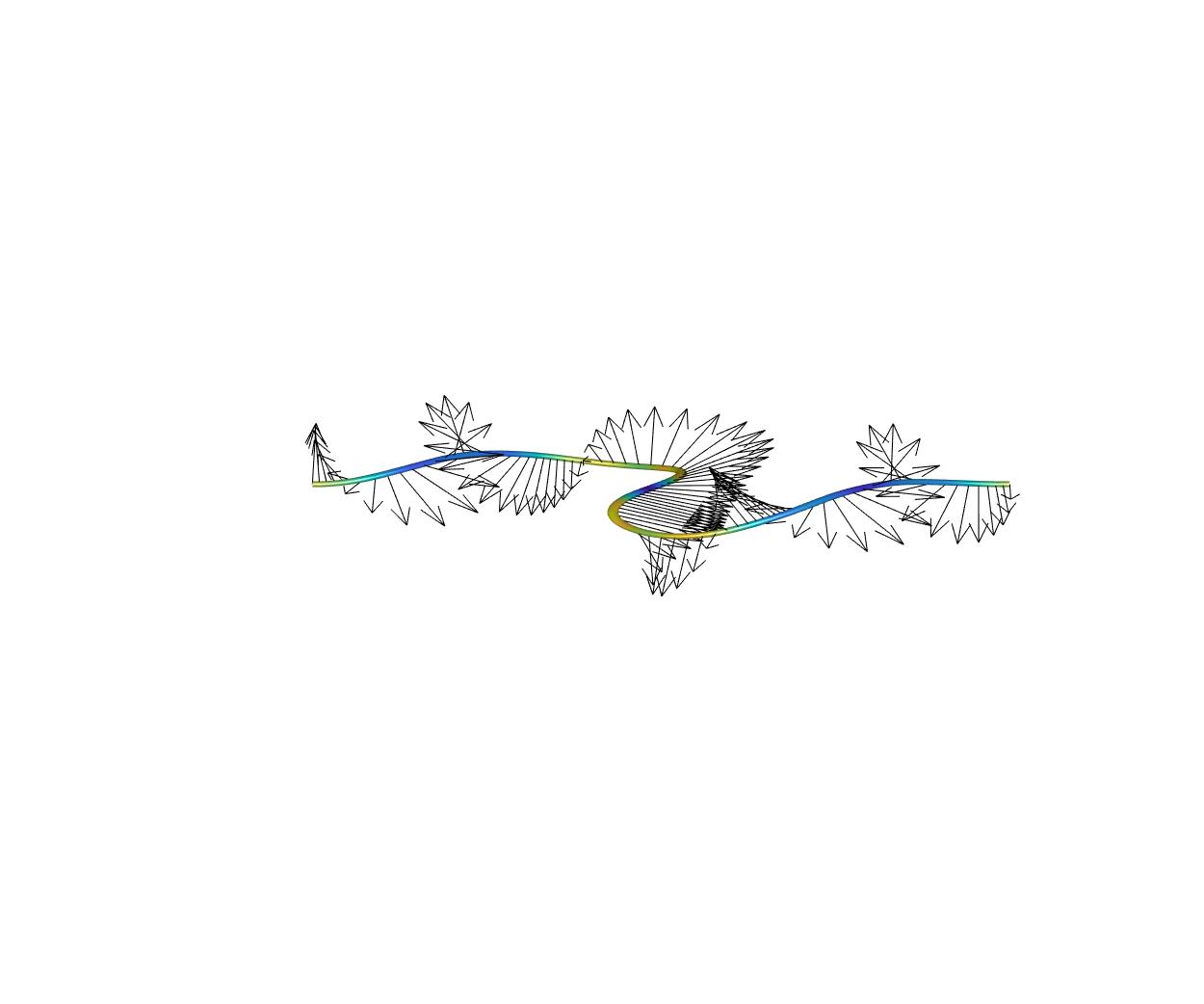}		\bild{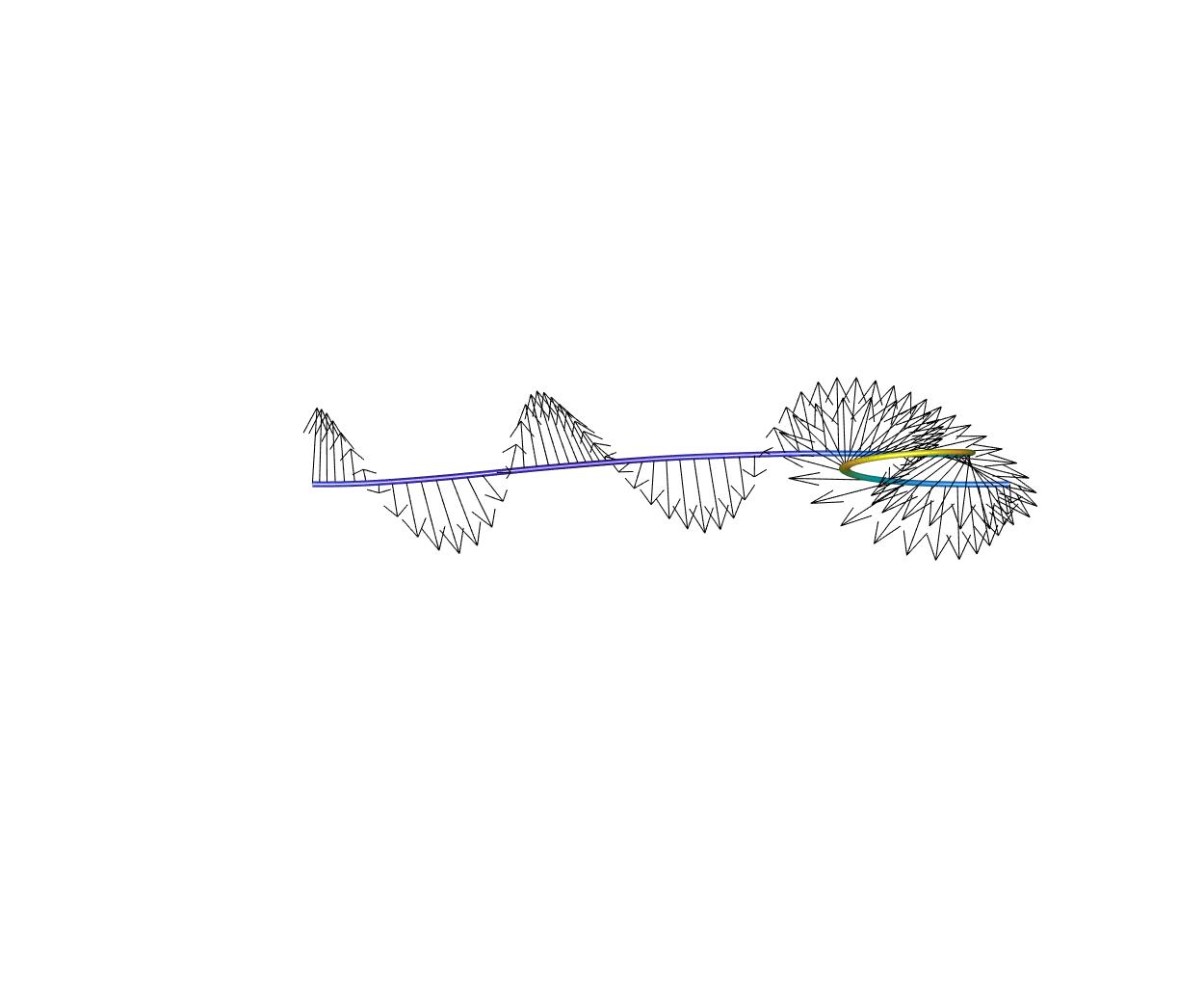}

\bild{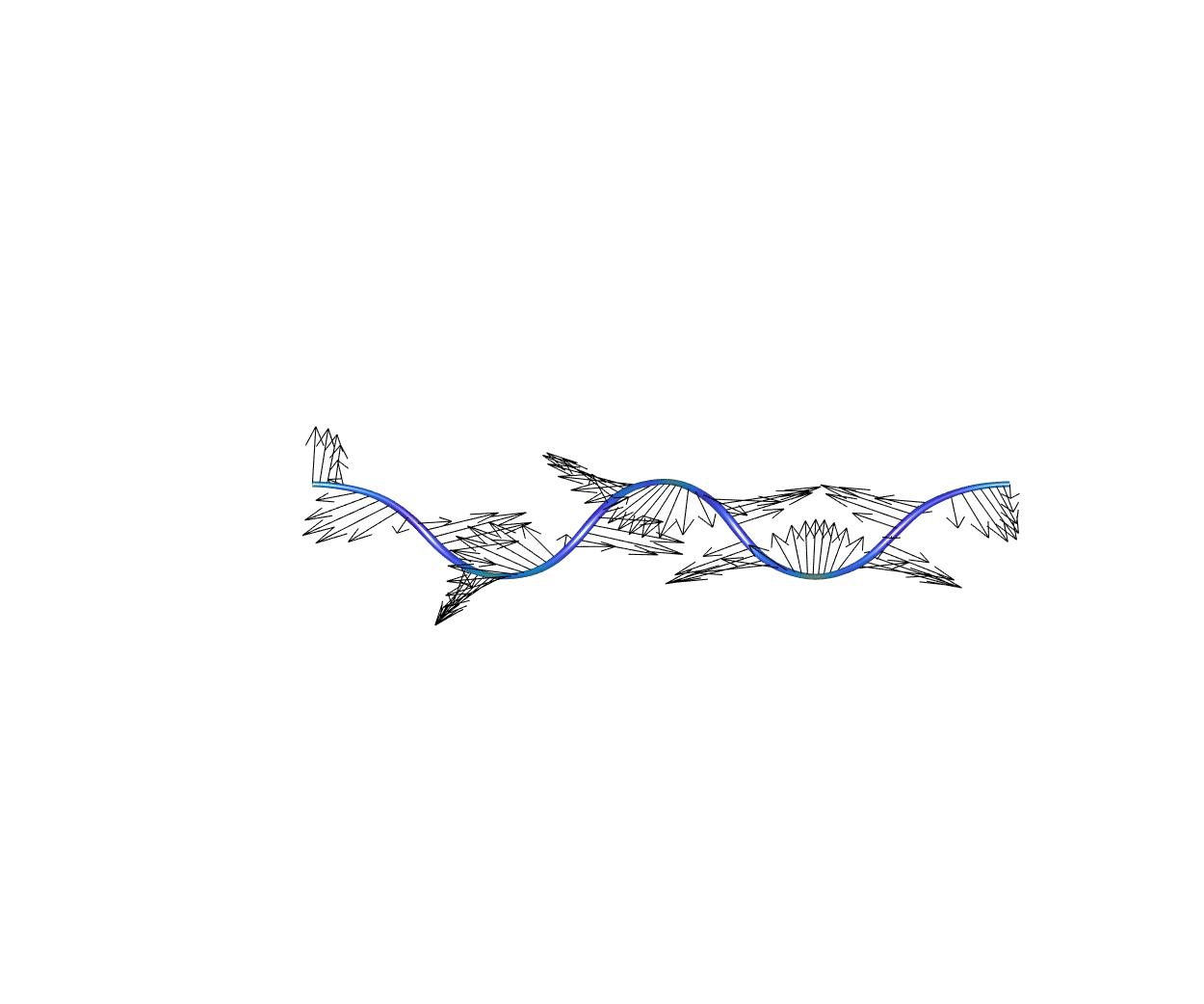}	\bild{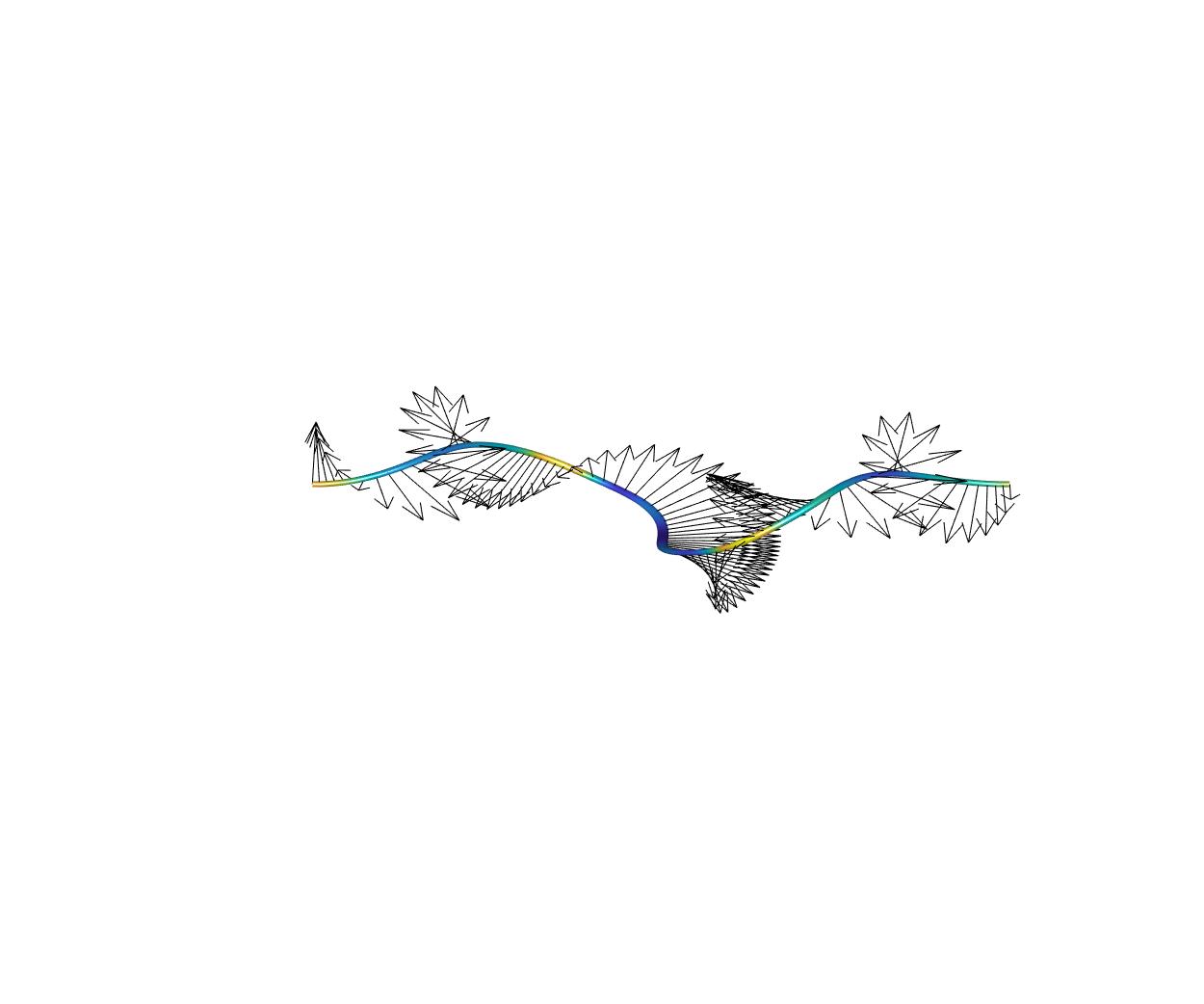}		\bild{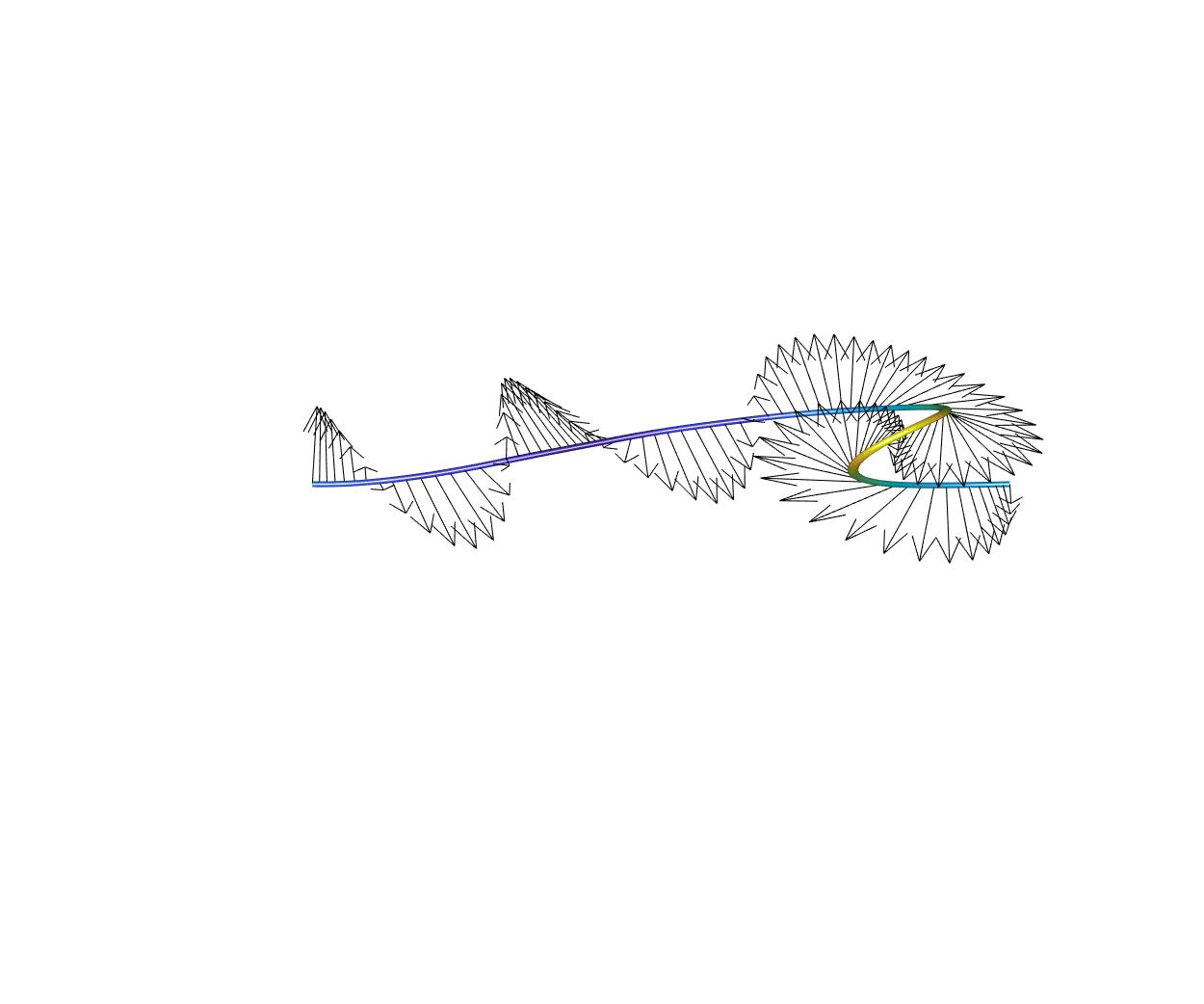}

\bild{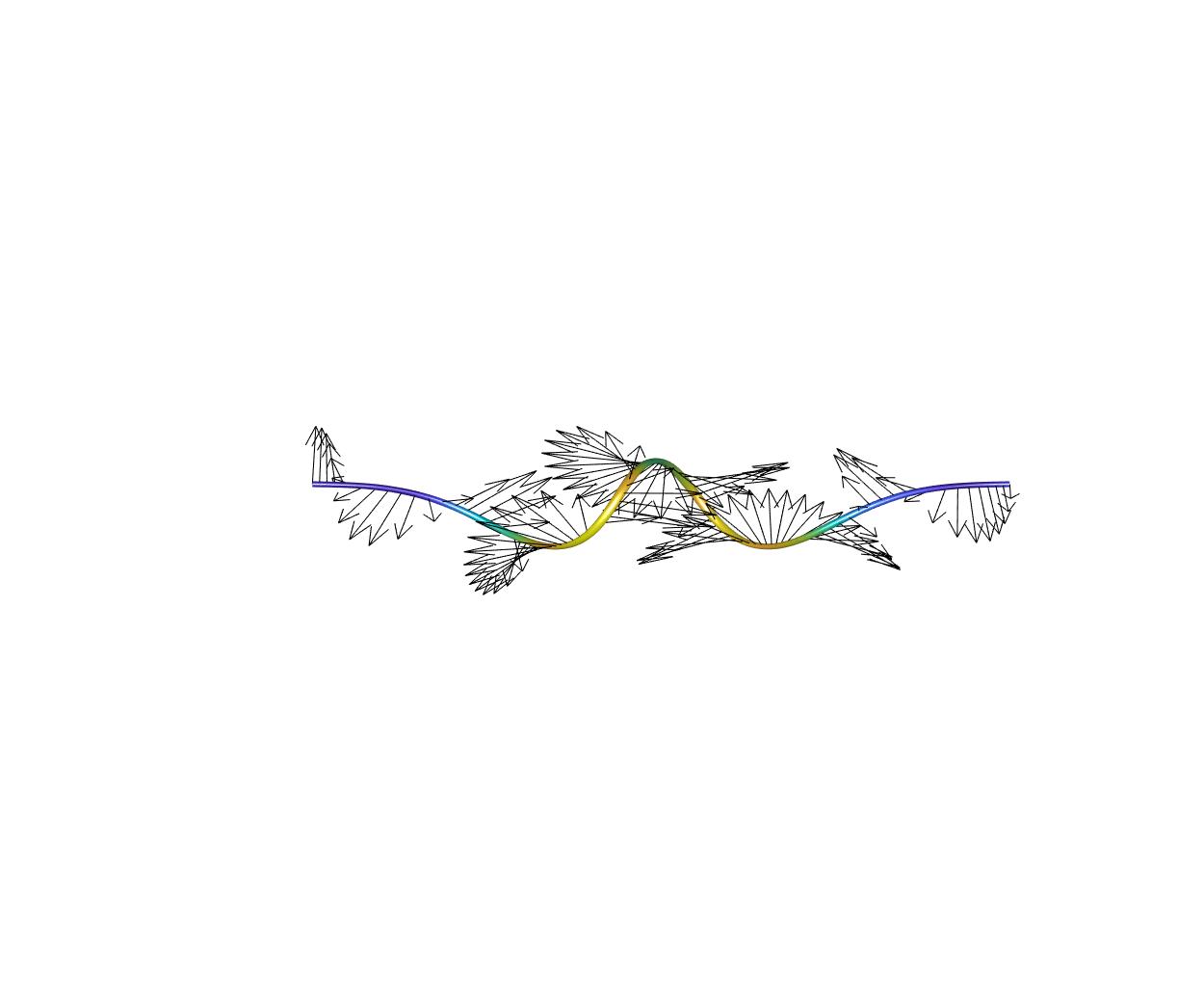}	\bild{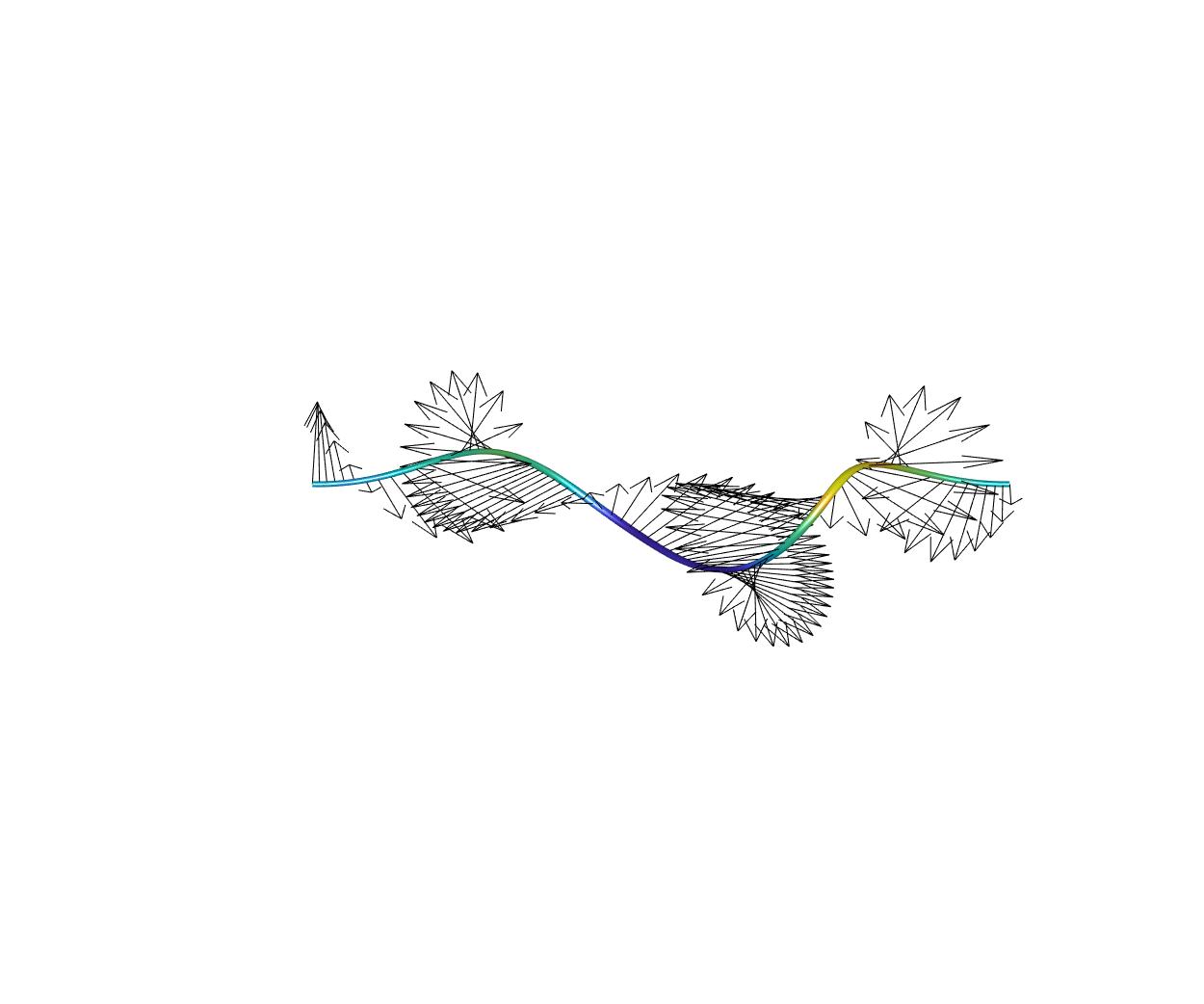}		\bild{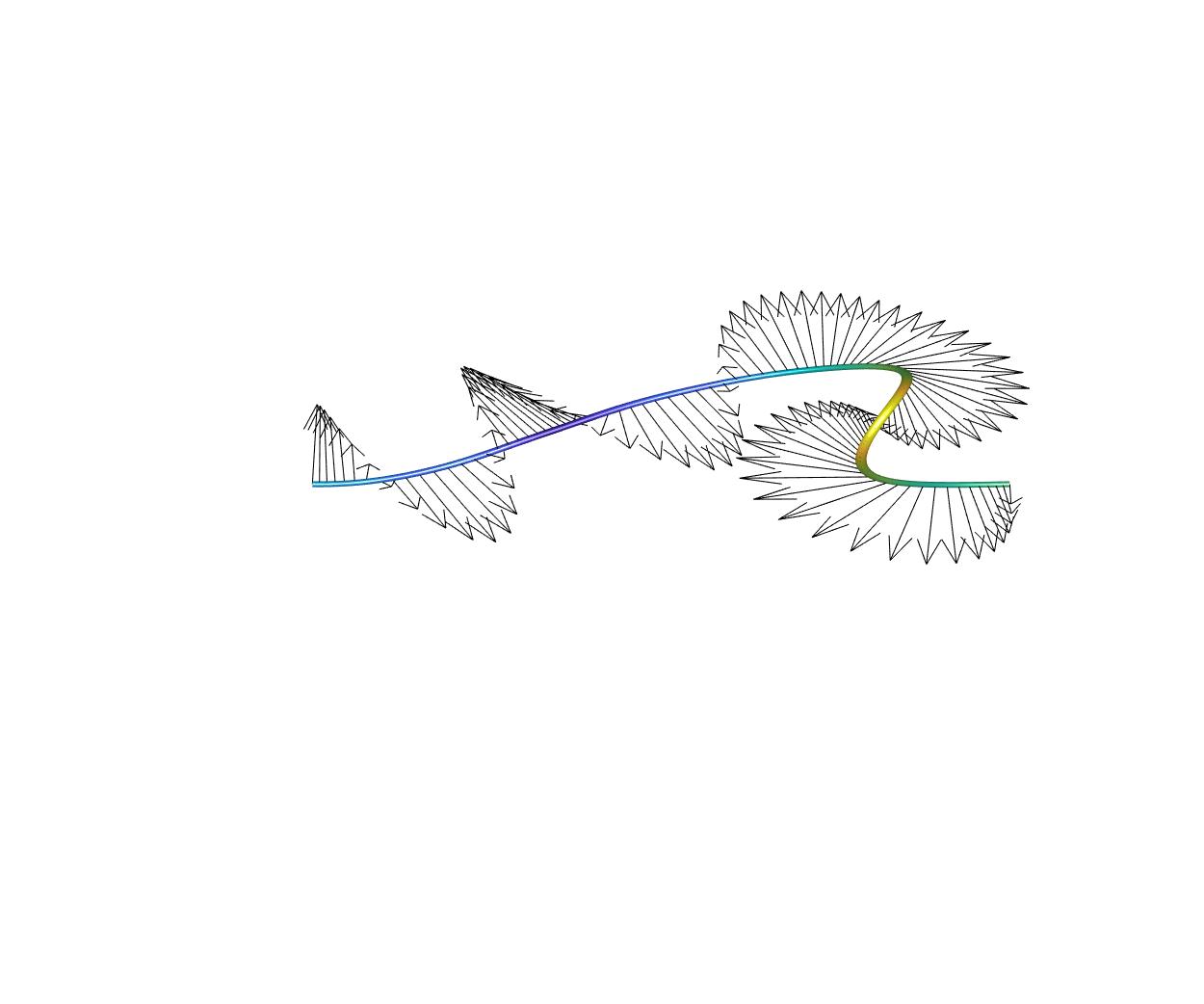}

\renewcommand{\bild}[1]{\fbox{%
\scalebox{-1}[1]{%
\includegraphics[scale=.13,trim=260 360 160 220,clip]{e#1.jpg}}}\makebox[0ex][r]{\makebox[0ex][r]{\tiny#1}\hspace{0ex} }\,\ignorespaces}

\bild{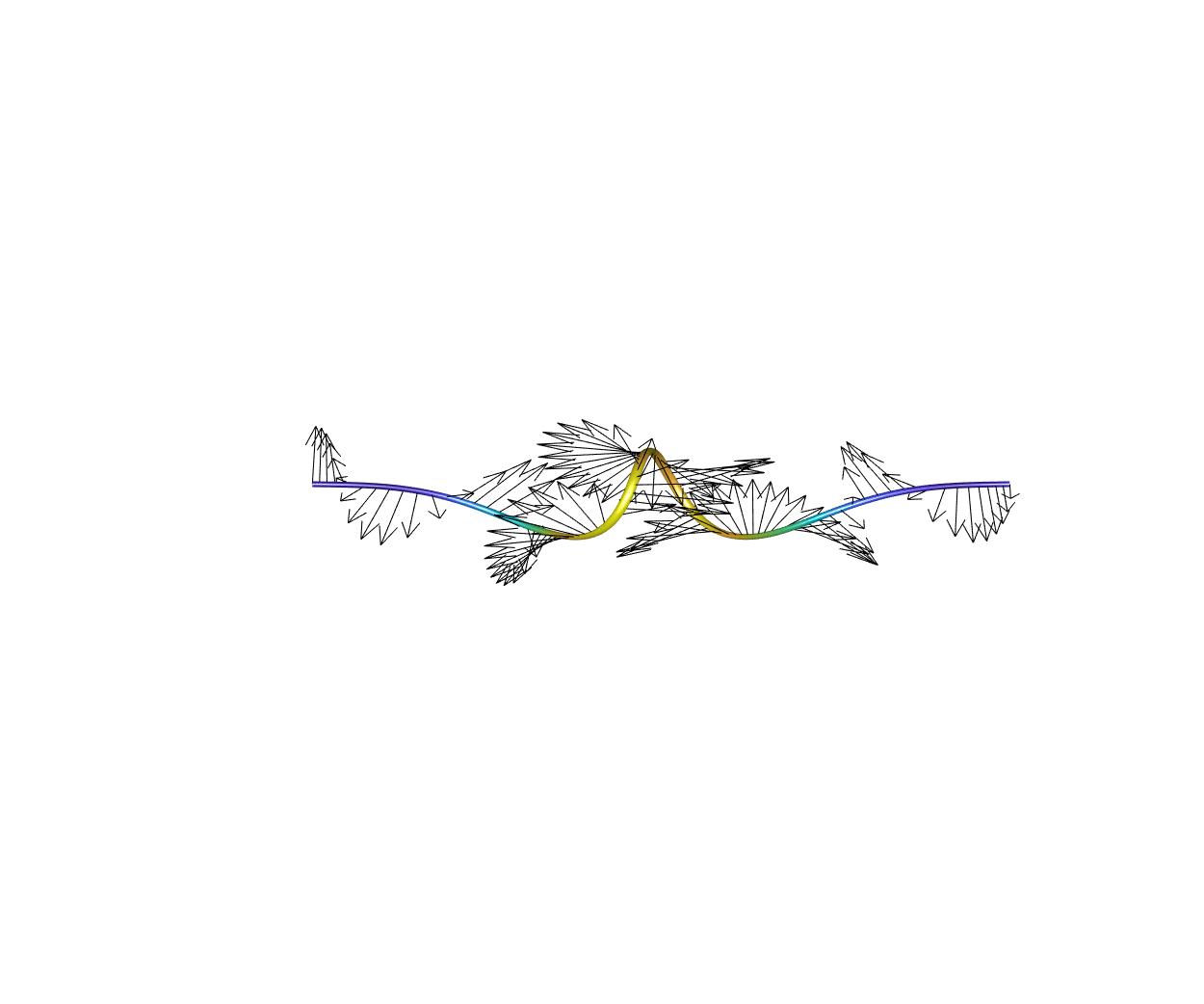}	\bild{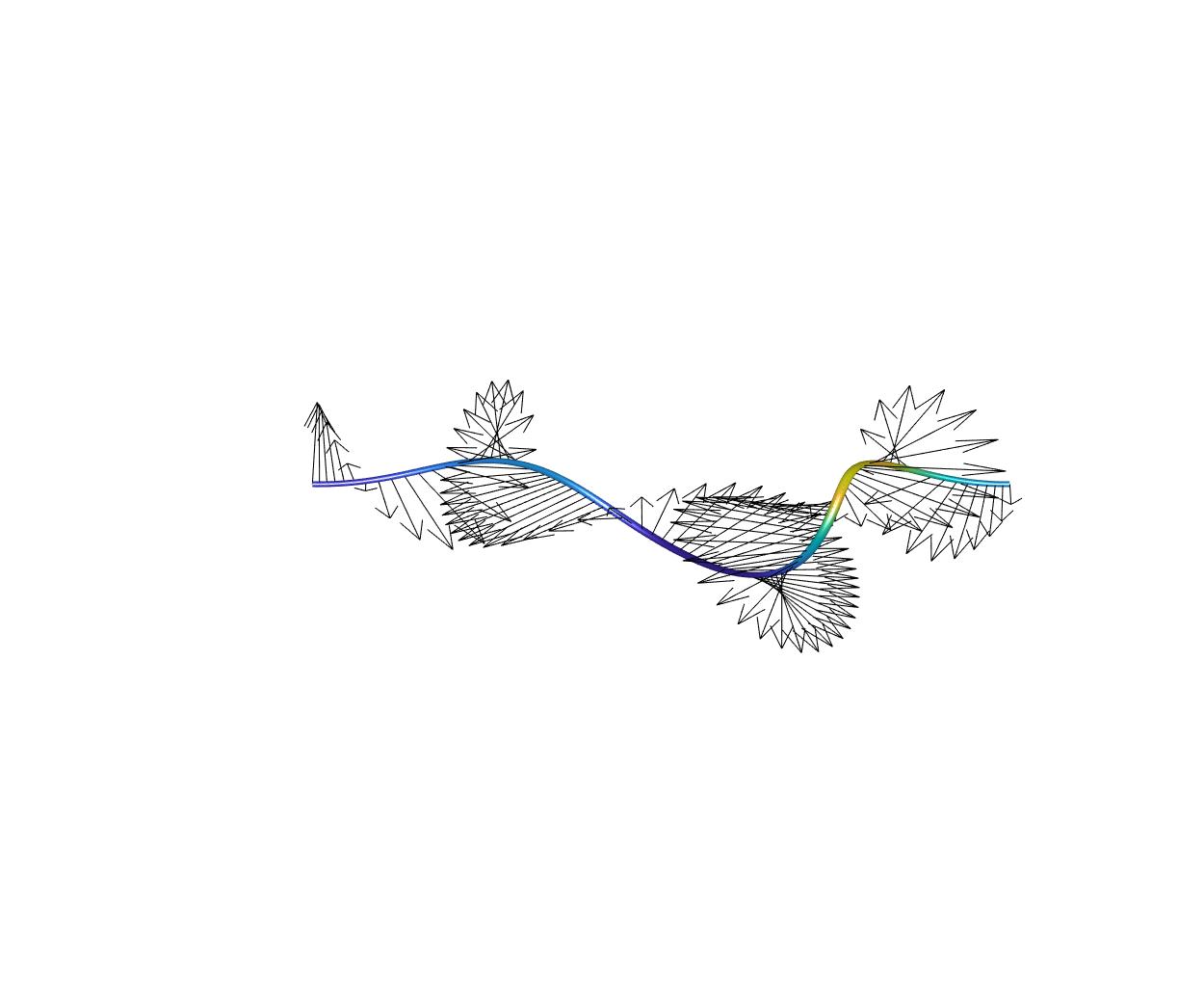}	\bild{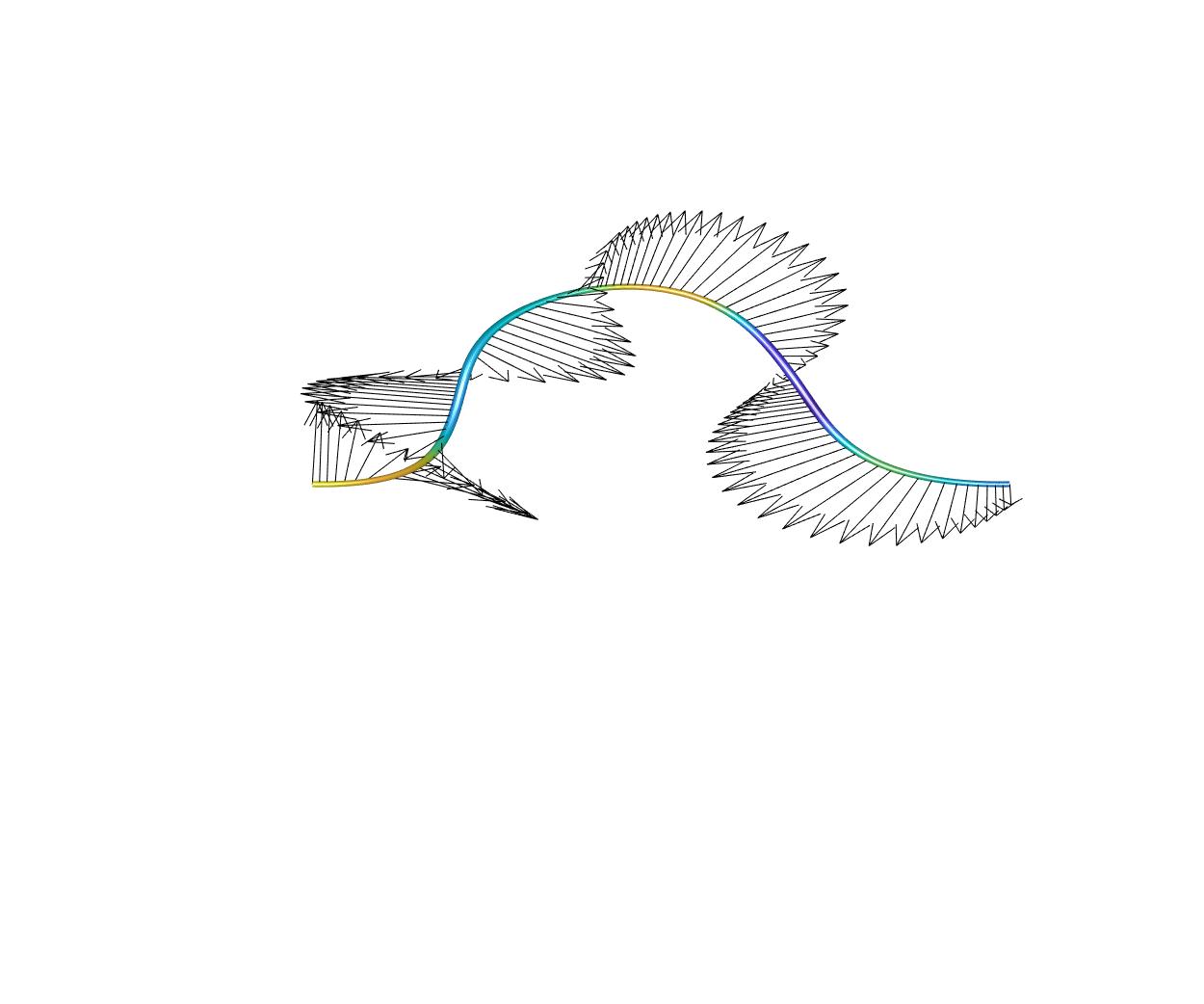}

\renewcommand{\bild}[1]{\fbox{%
\scalebox{-1}[1]{%
\includegraphics[scale=.13,trim=260 360 160 340,clip]{e#1.jpg}}}\makebox[0ex][r]{\makebox[0ex][r]{\tiny#1}\hspace{0ex} }\,\ignorespaces}

\bild{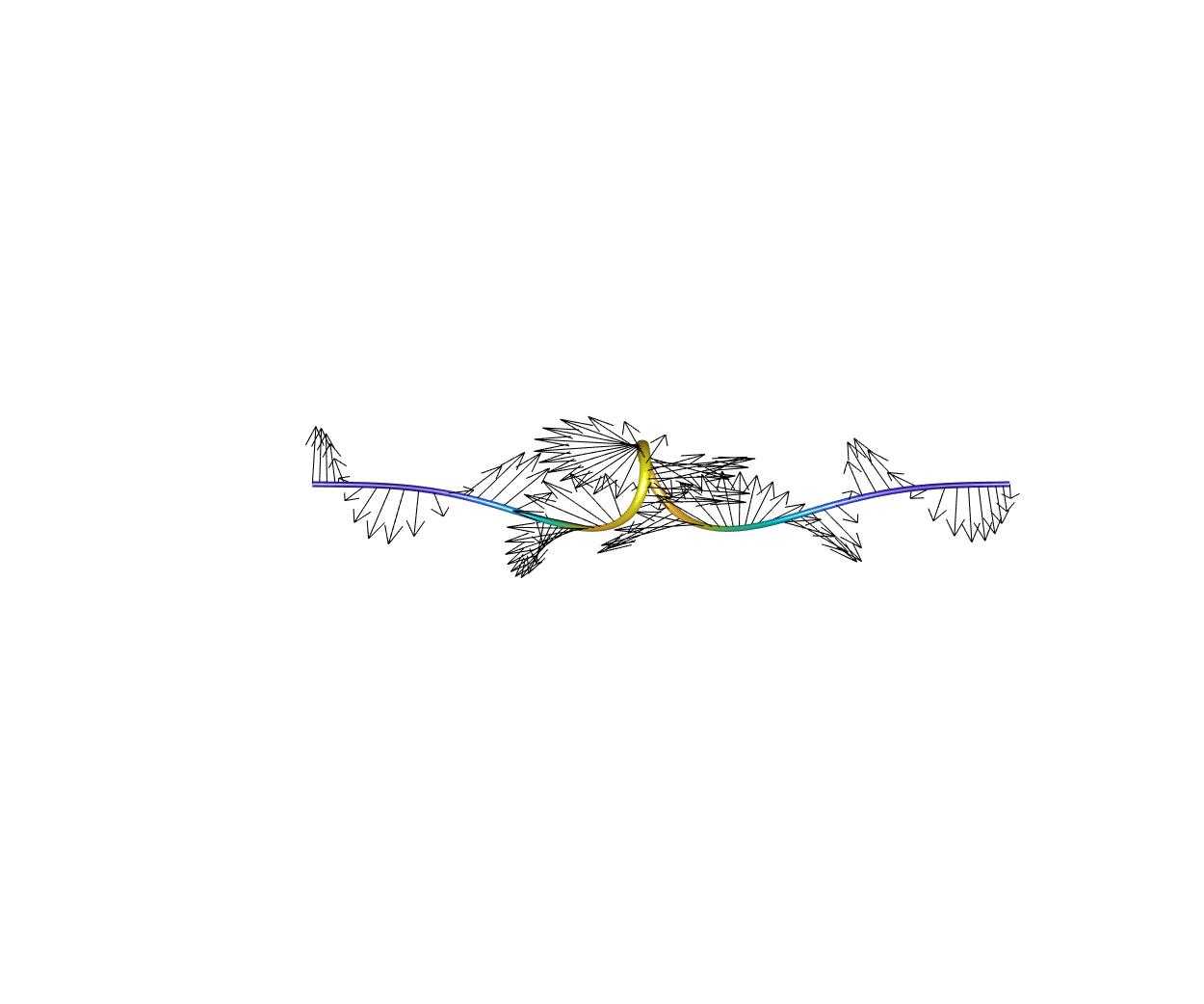}	\bild{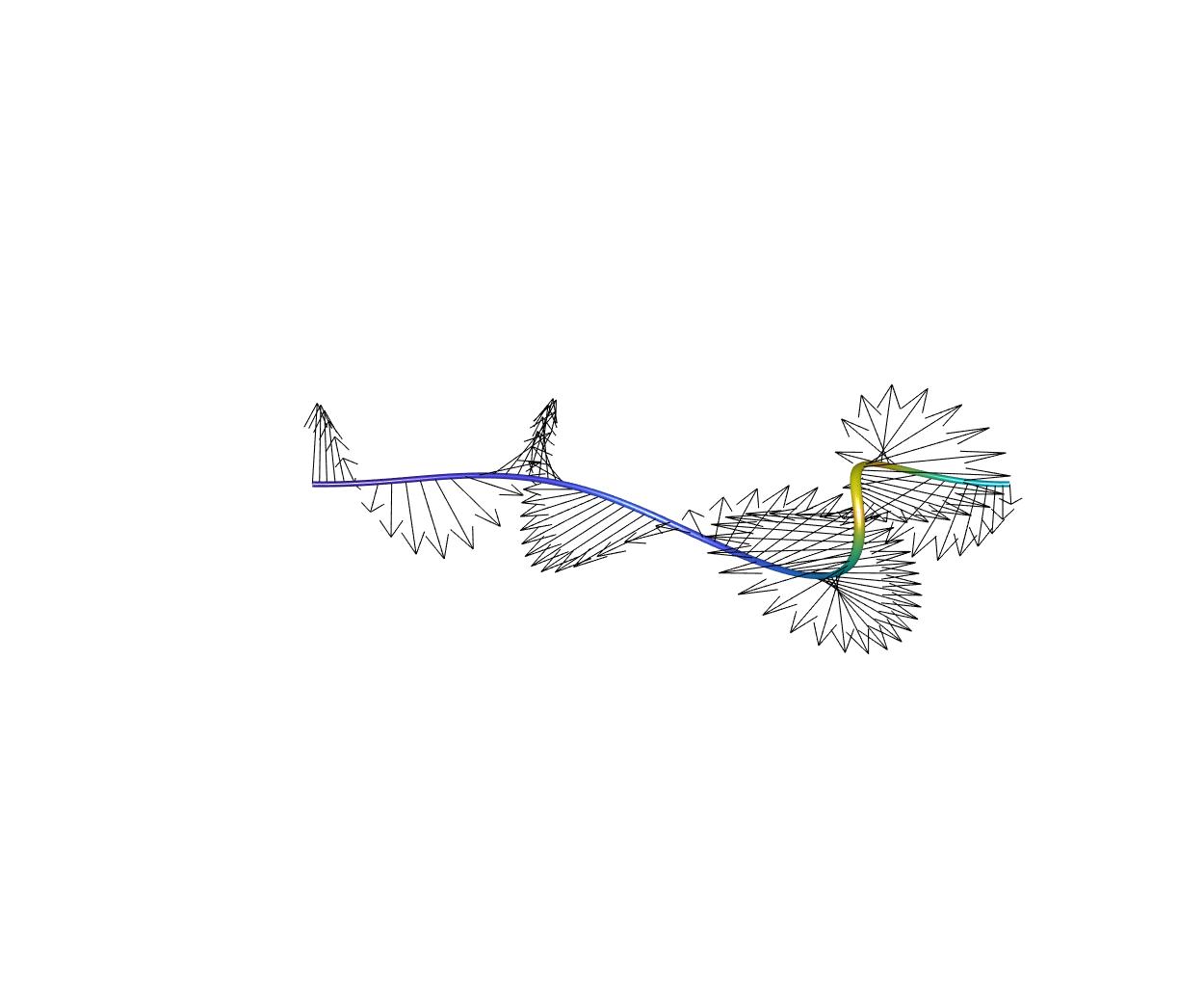}	\bild{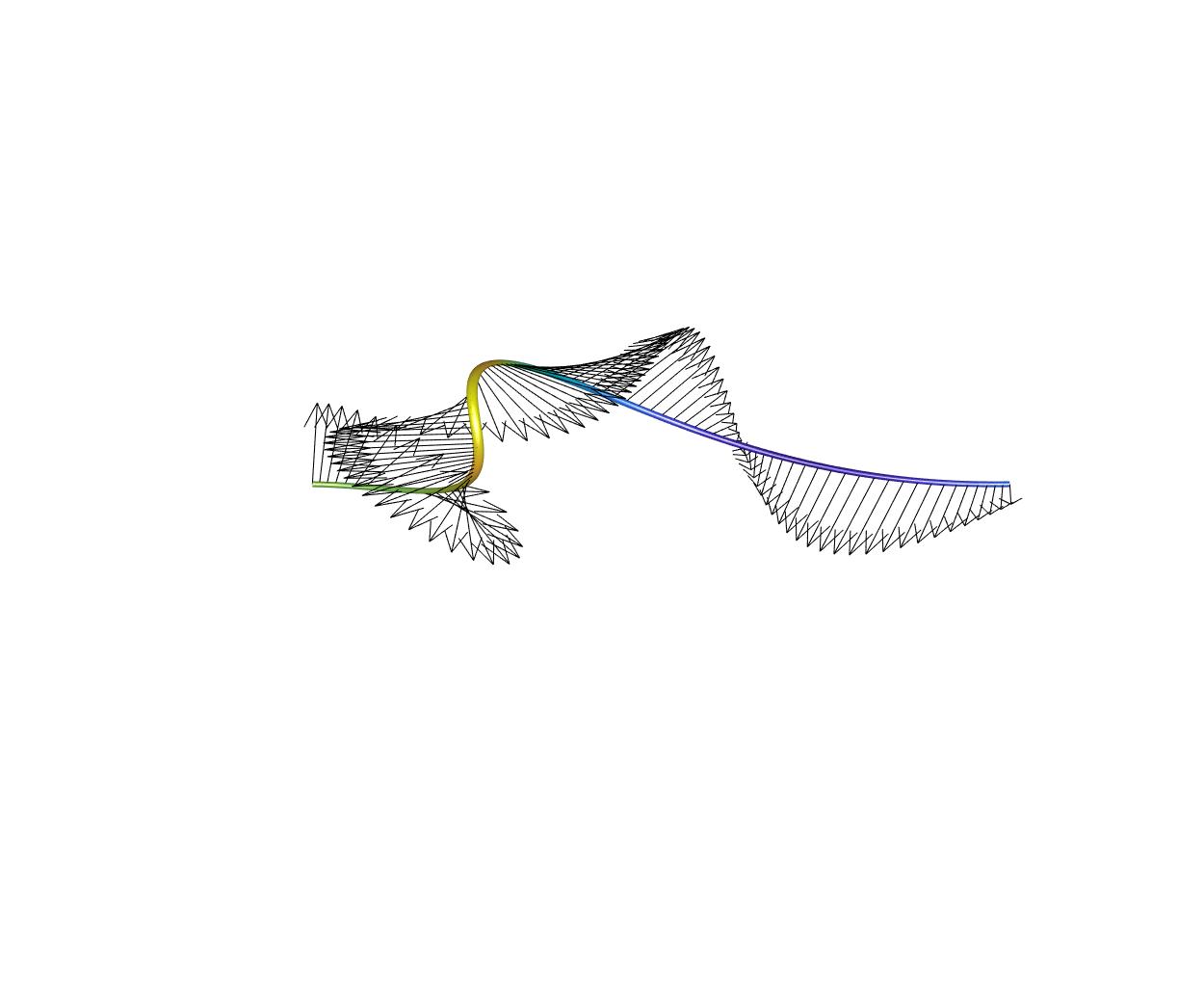}

\bild{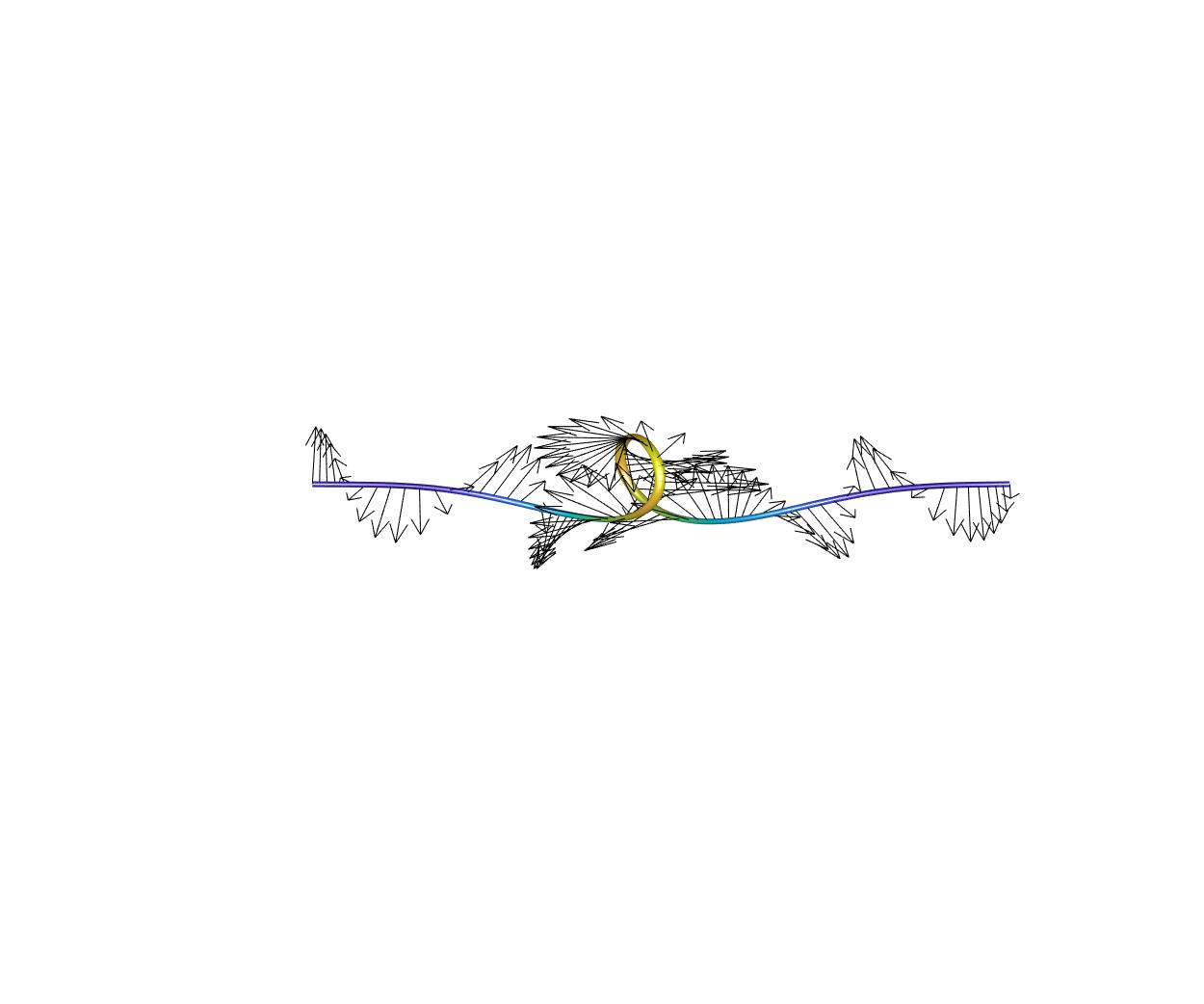}	\bild{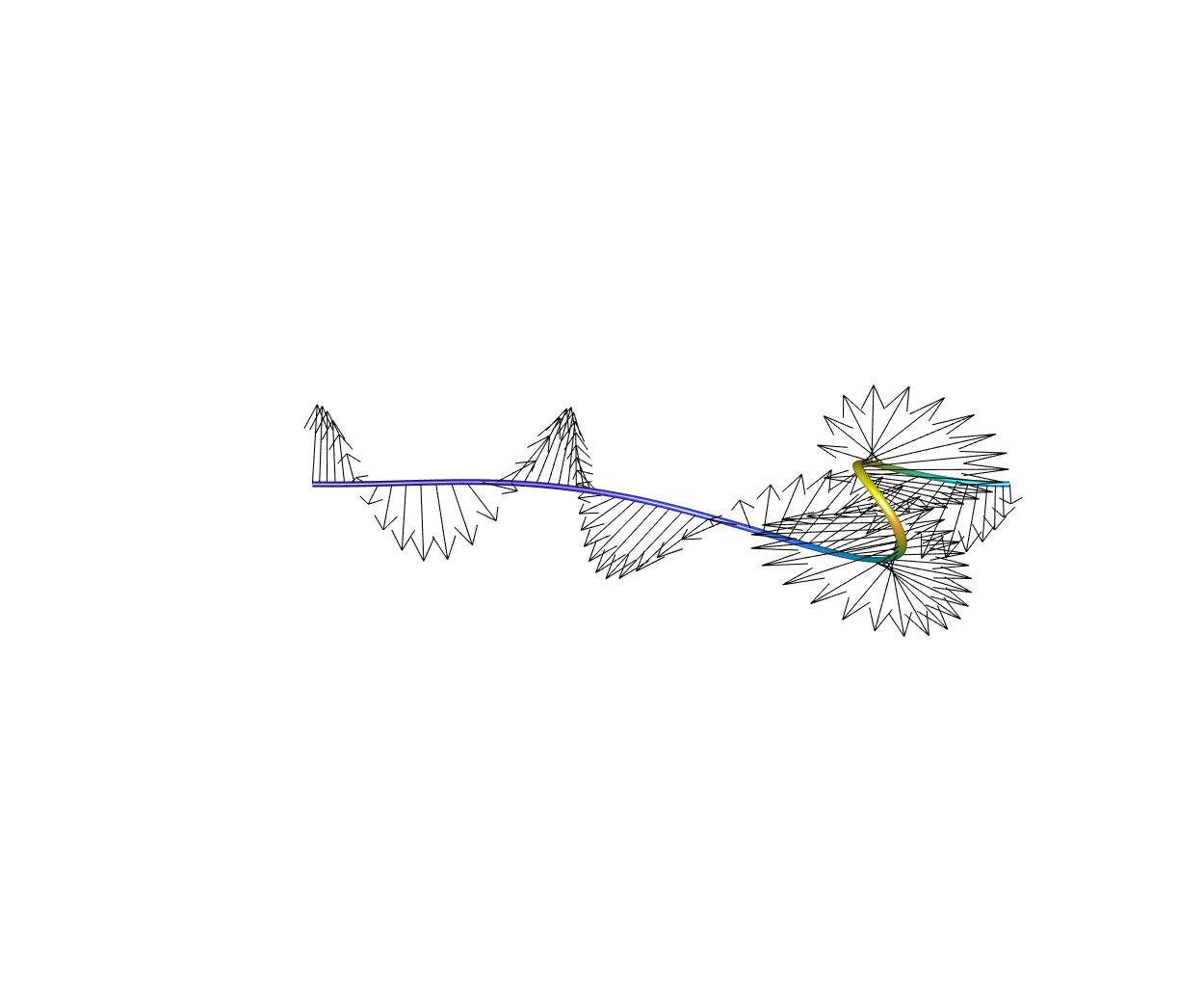}	\bild{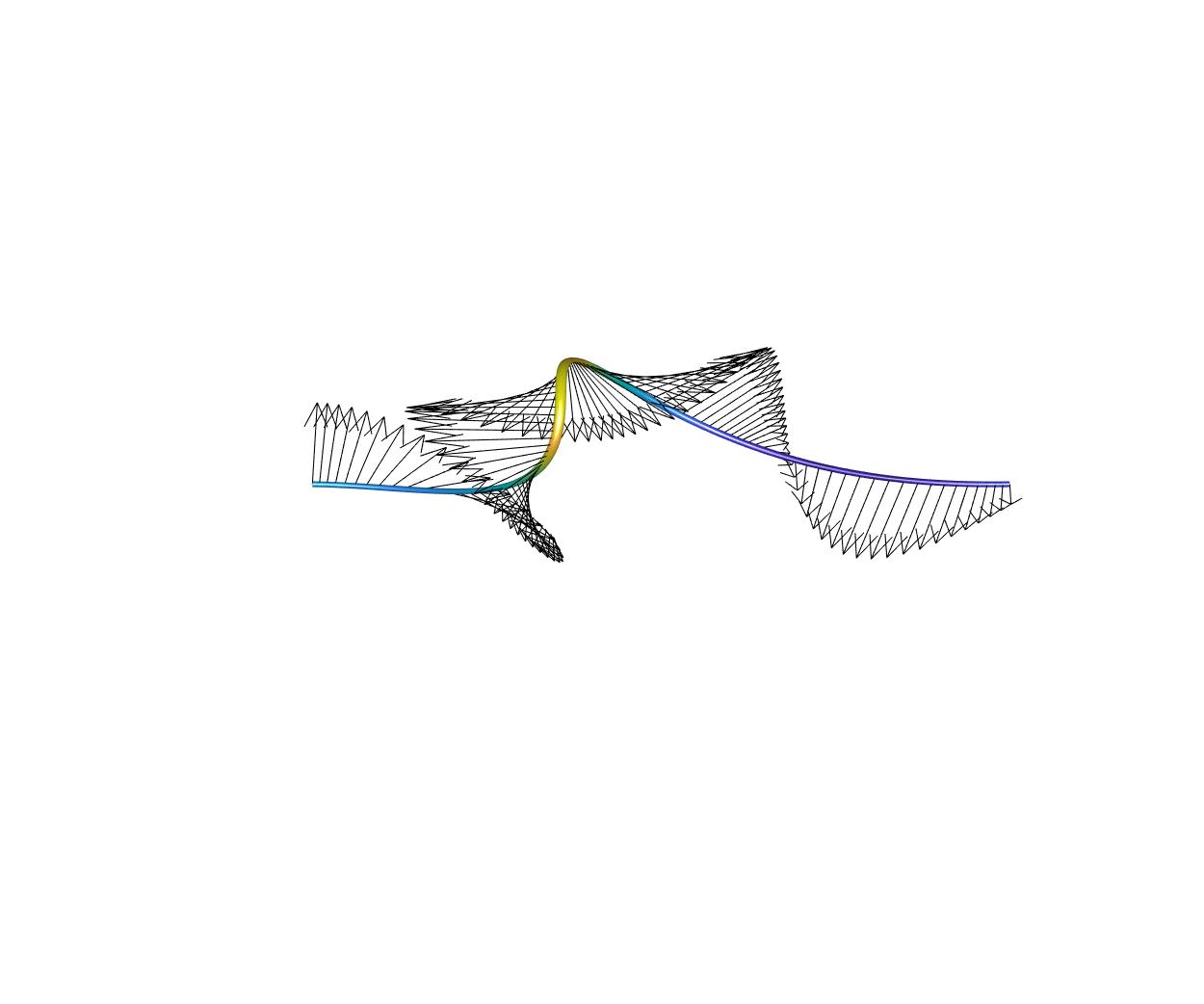}

\bild{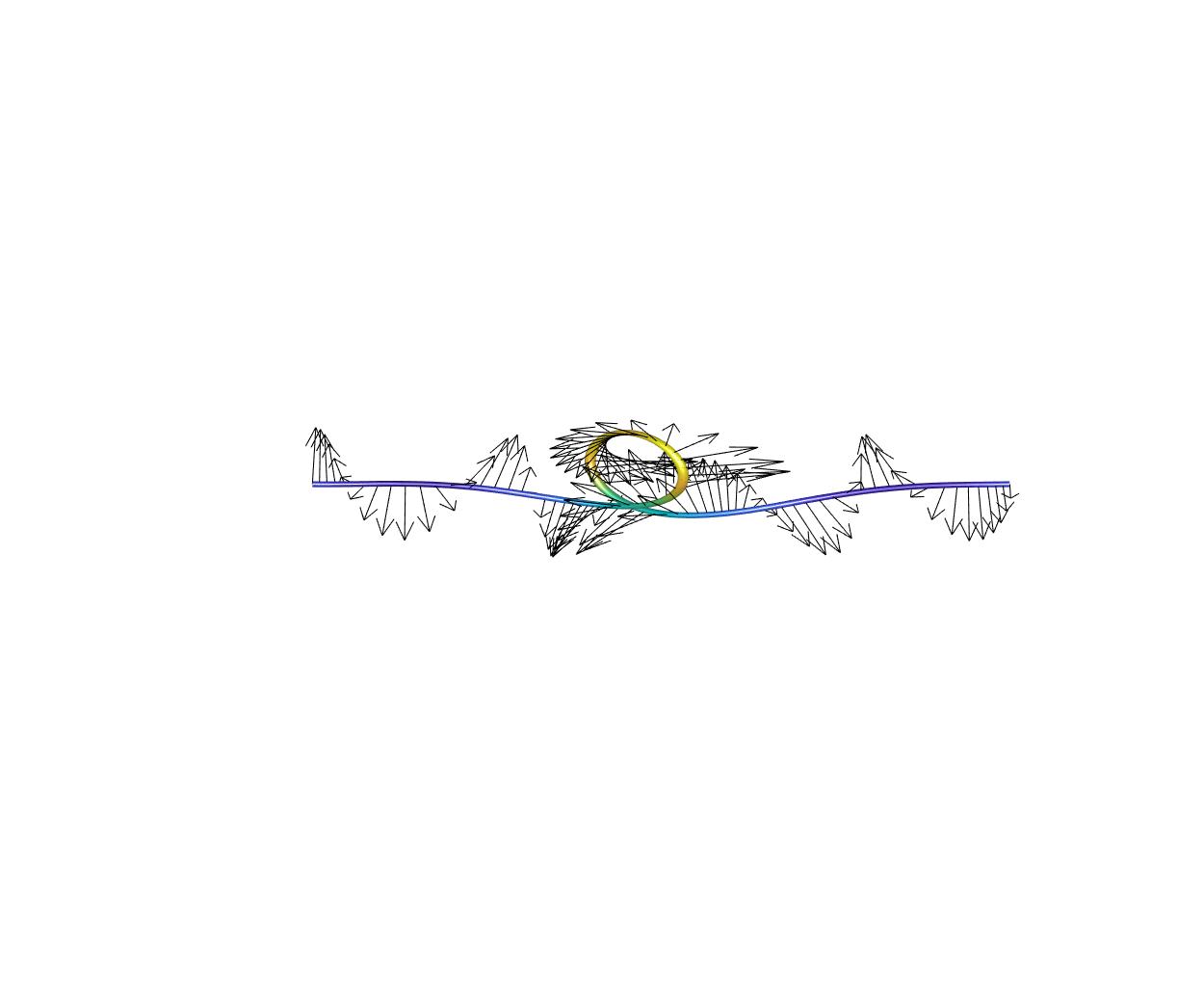}	\bild{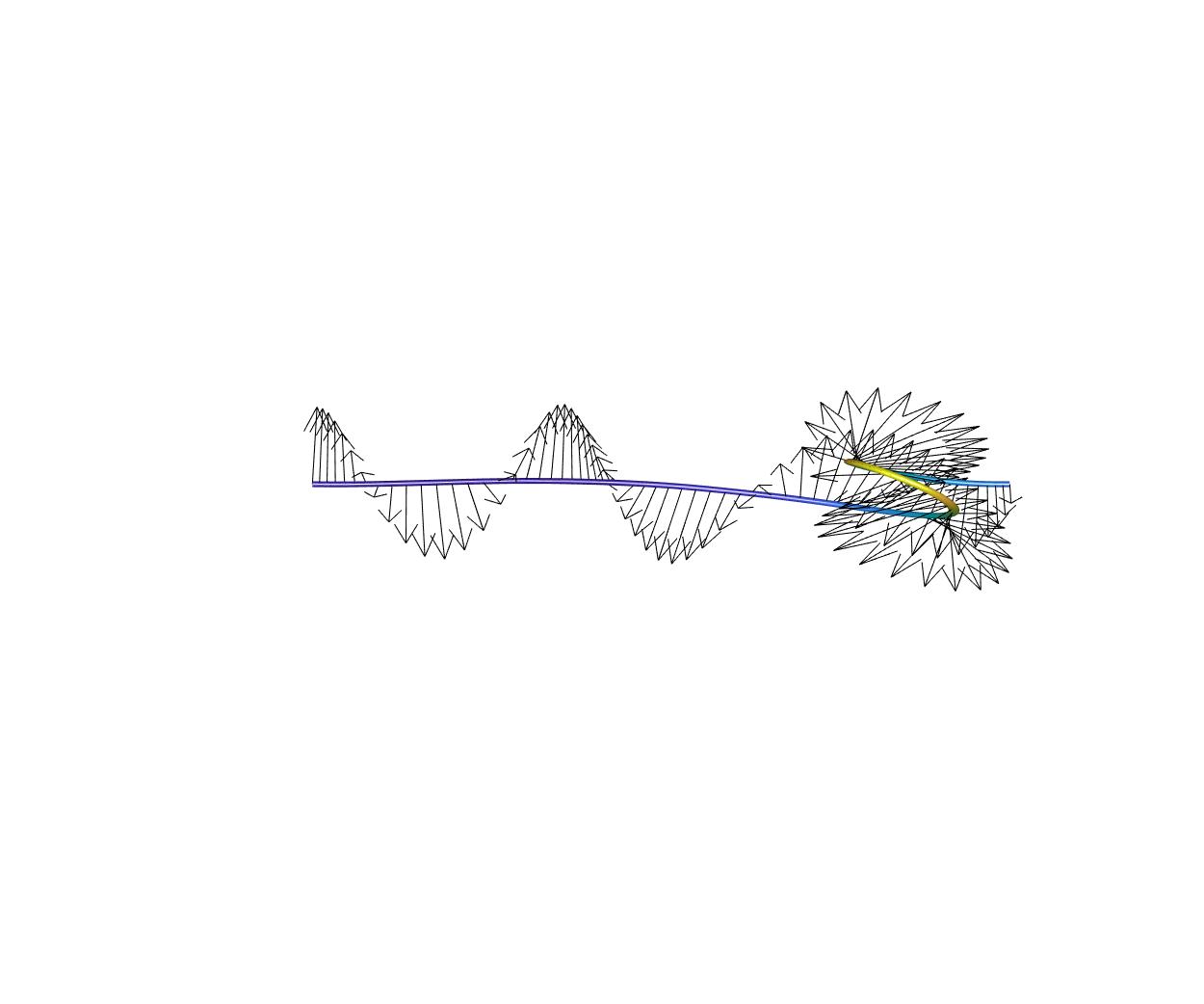}	\bild{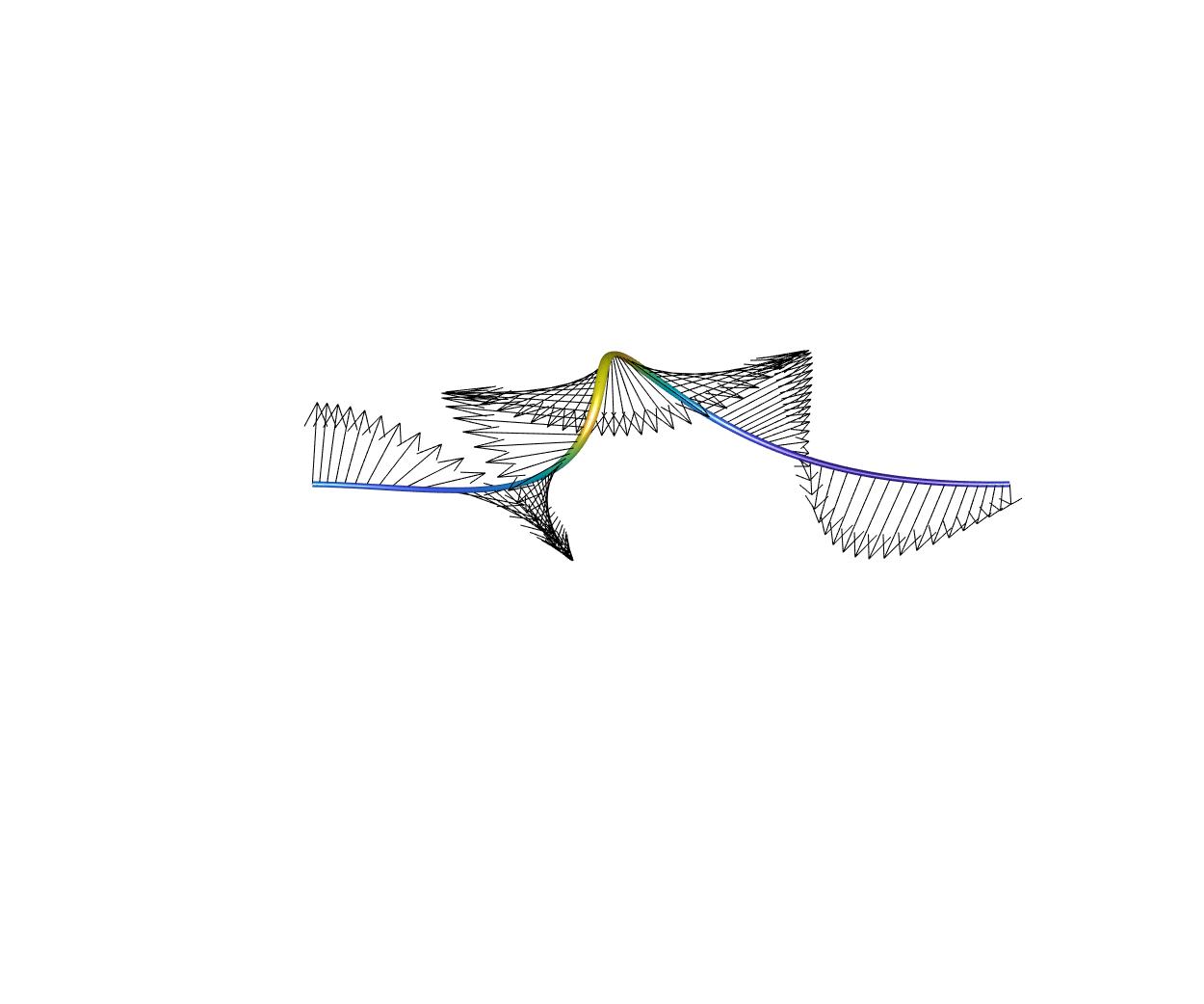}

\bild{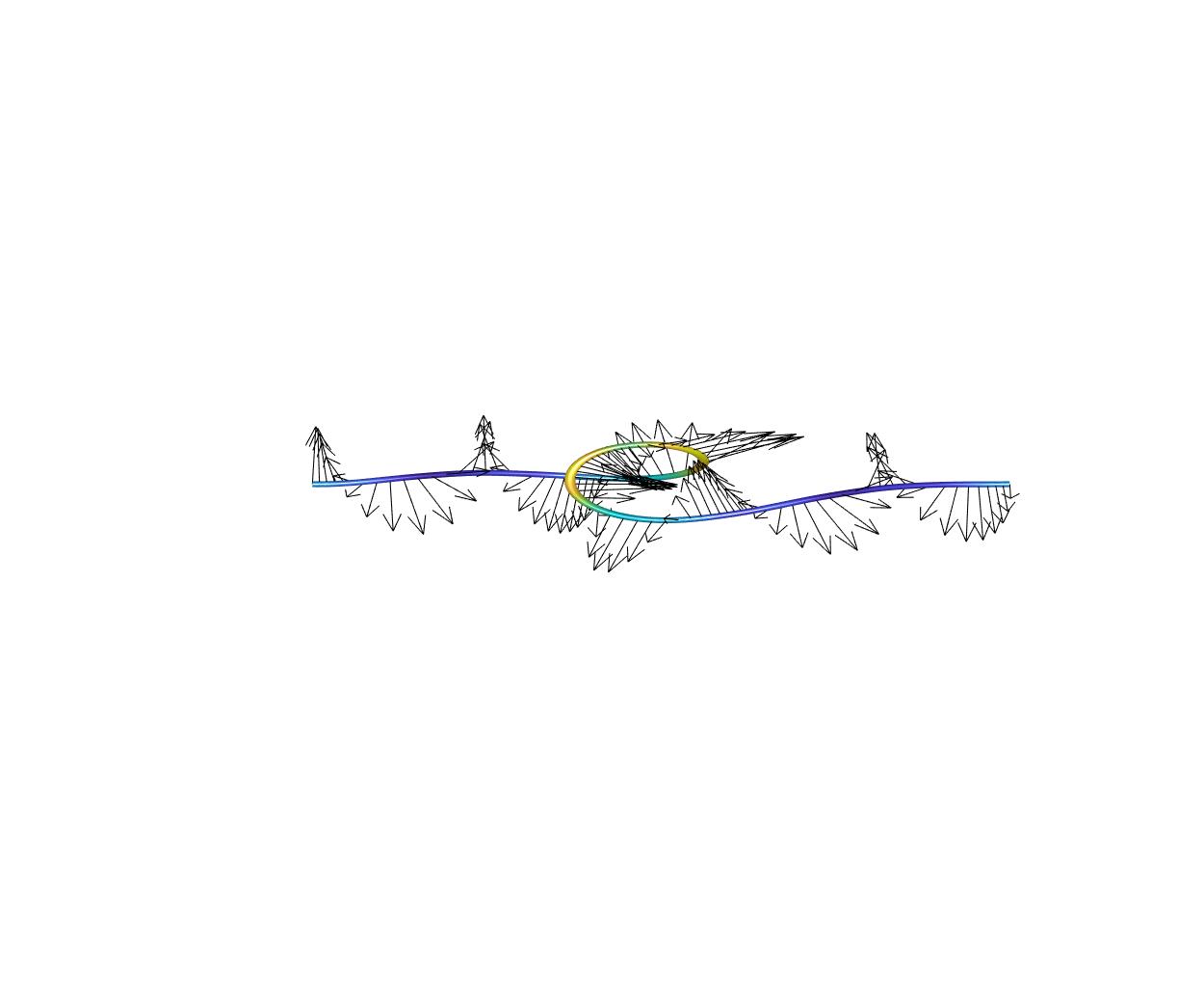}	\bild{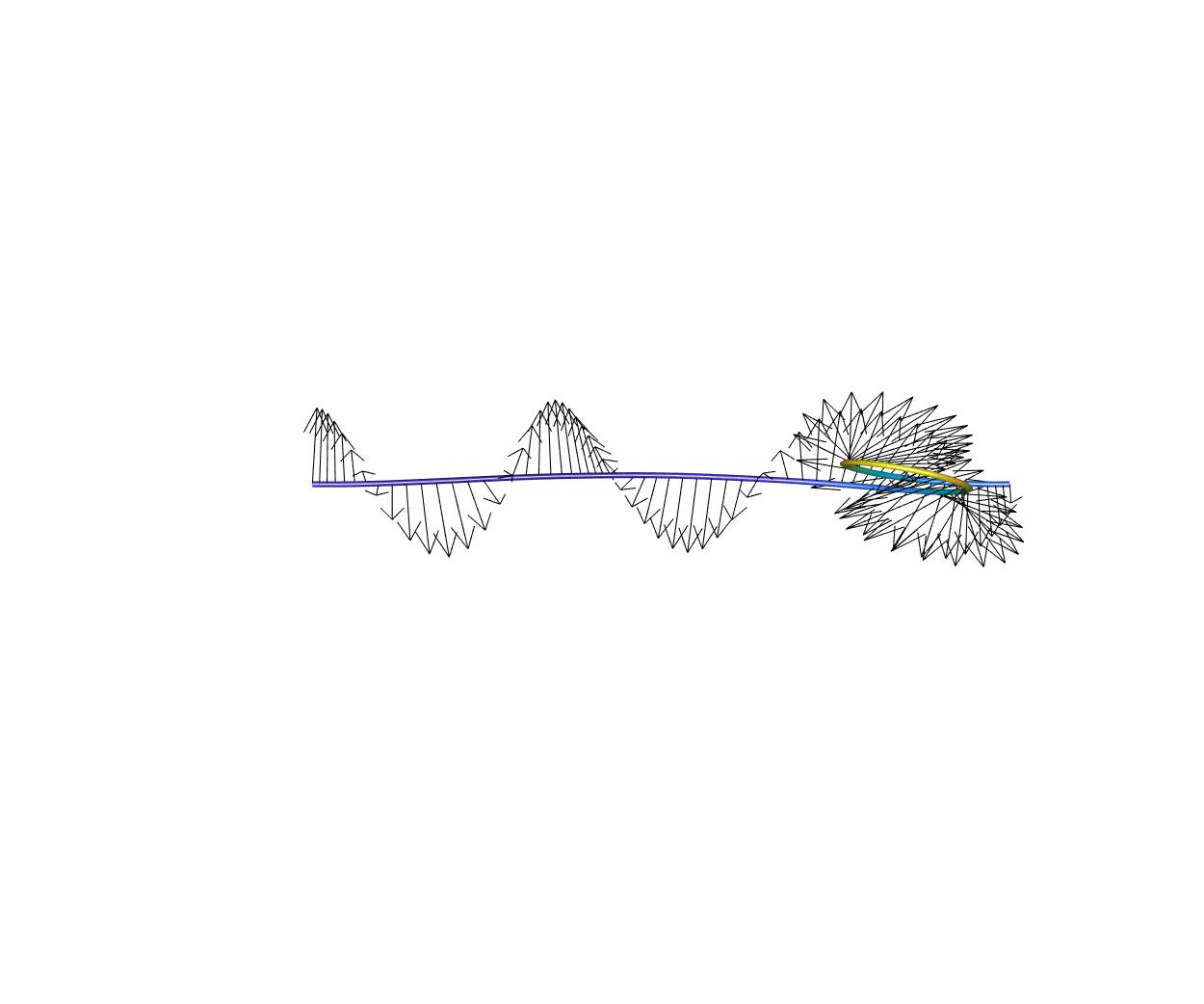}	\bild{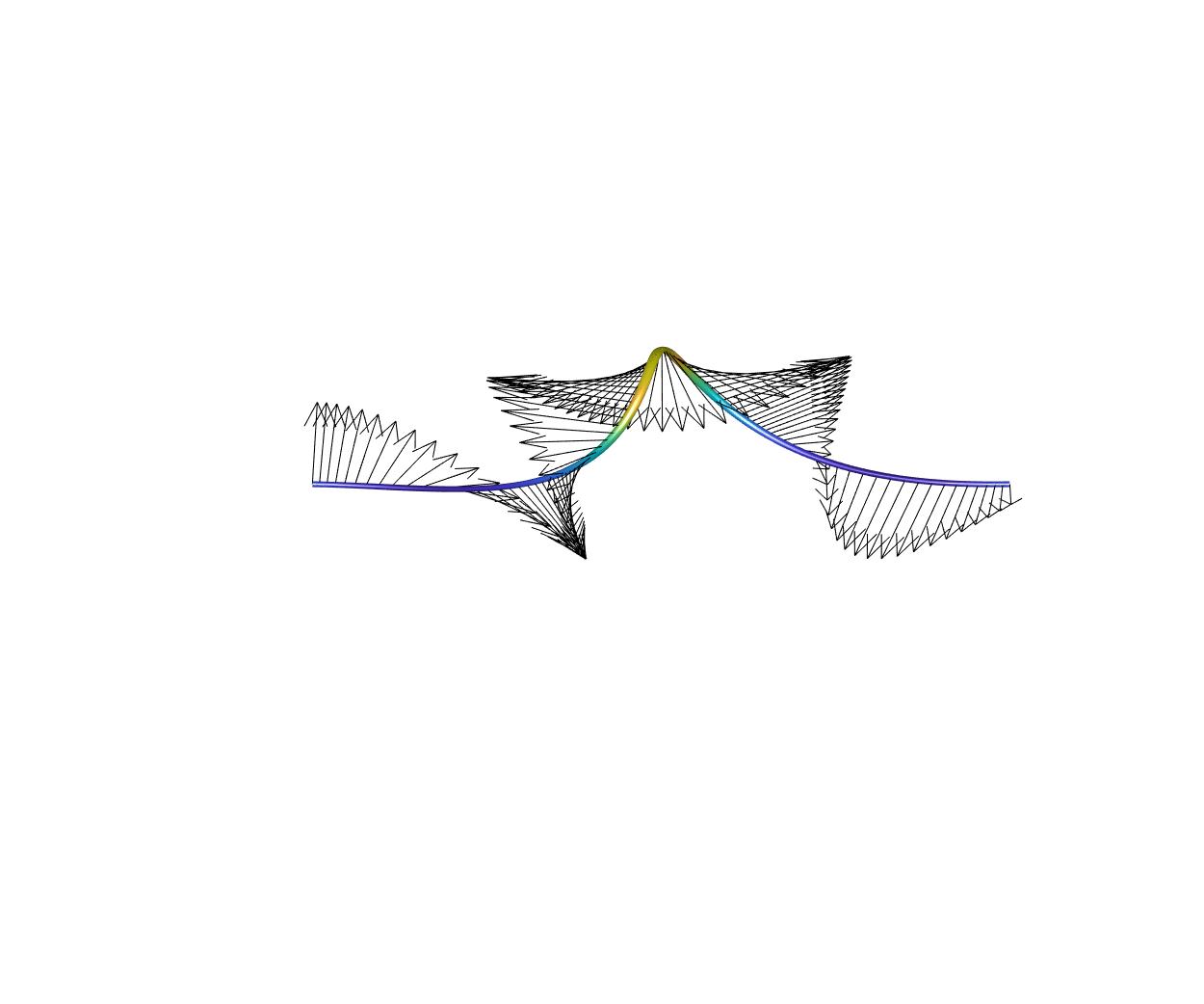}

\caption{Evolution of an open clamped rod
{from Experiment~\ref{sect:clamped}.
The total twist amounts to~$8$ initially
and is then reduced to about~$2$ by two self-penetrations
which occur around iteration steps $6{,}000$ and $16{,}500$.
The energy plot is shown in Figure~\ref{fig:clamped-plot} below.}}
\label{fig:clamped}
\end{figure}

\begin{figure}
 \includegraphics[scale=.35]{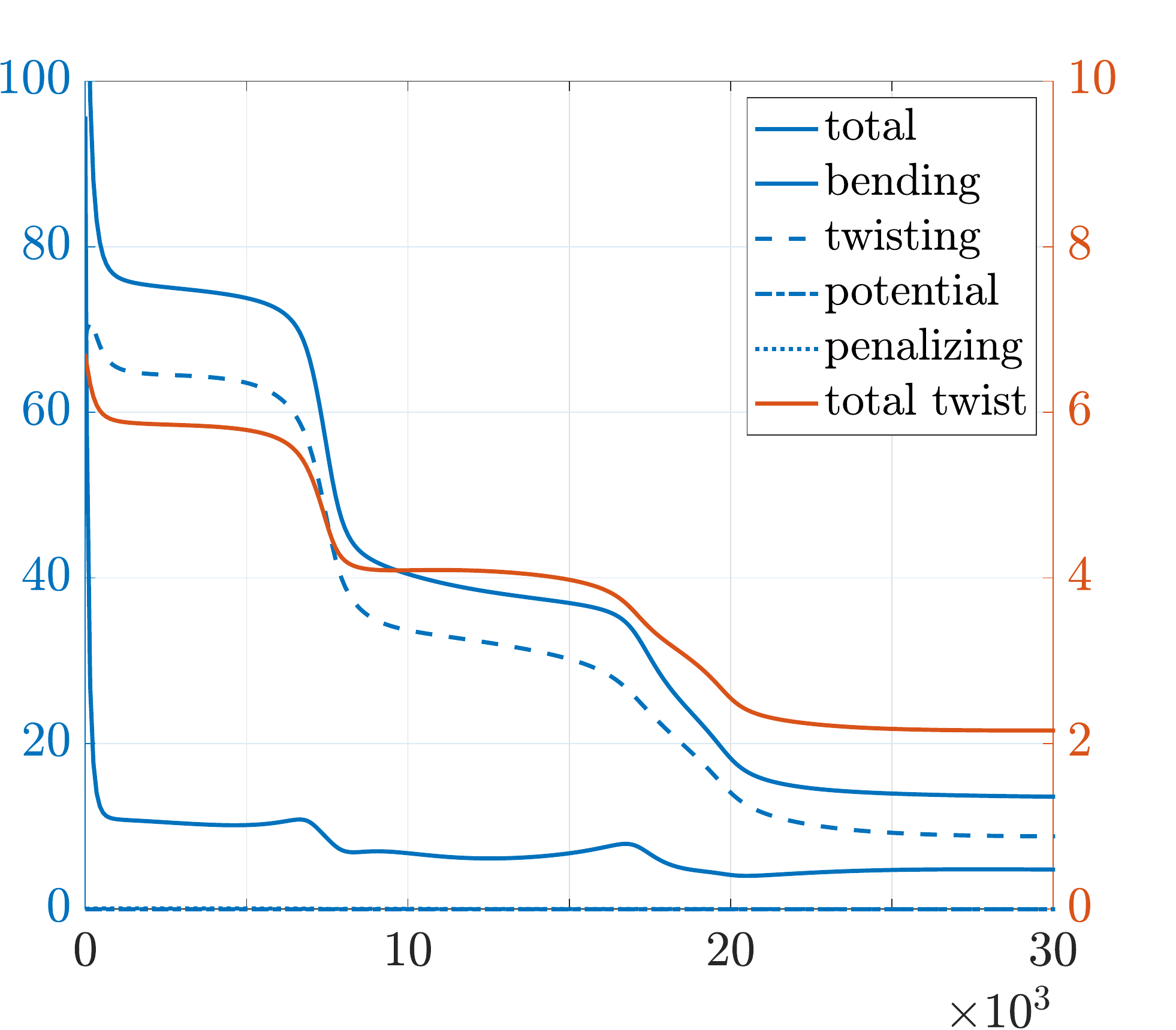}
 \caption{Energy plot {for the evolution from Experiment~\ref{sect:clamped}.
 Reduction of twist occurs following the self-penetrations of the curve.
 Total twist attains the values of~$5$ and~$3$
 around iteration steps $7{,}200$ and $18{,}800$ respectively.}
 The energy profile is virtually constant for the iteration steps $\ge30{,}000$.}
 \label{fig:clamped-plot}
\end{figure}

We can also simulate the evolution of \emph{open} rods.
Our initial curve is planar, namely
\[ {y(s) = }
\begin{pmatrix}
\tfrac 12s \\ \cos s-1
\end{pmatrix}, \qquad s\in[0,4\pi]. \]

This curve is not parametrized by arclength,
with the notation of elliptic integrals introduced in {Experiment}~\ref{sect:f8}
we have $L=4\sqrt 2E(\tfrac\pi2,-2)\approx12.357$.
{No perturbation of the initial rod was used.}

Choosing $\beta_{\ini}=4$ gives an initial total twist of $8$
according to {Section}~\ref{sect:tot-twist}.

The evolution is depicted in Figure~\ref{fig:clamped}, the corresponding
energy plot can be found in Figure~\ref{fig:clamped-plot}.
It seems to become stationary after $30{,}000$ steps
although the total twist could be further reduced, see {Section}~\ref{sect:mintt}.

\subsection{Implementing impermeability}
\label{sect:imper}

\renewcommand{\bild}[1]{\fbox{\includegraphics[scale=.14,trim=350 330 340 300,clip]{f#1.jpg}}\makebox[0ex][r]{\makebox[0ex][r]{\tiny#1}\hspace{0ex} }\,\ignorespaces}

\begin{figure}

\bild{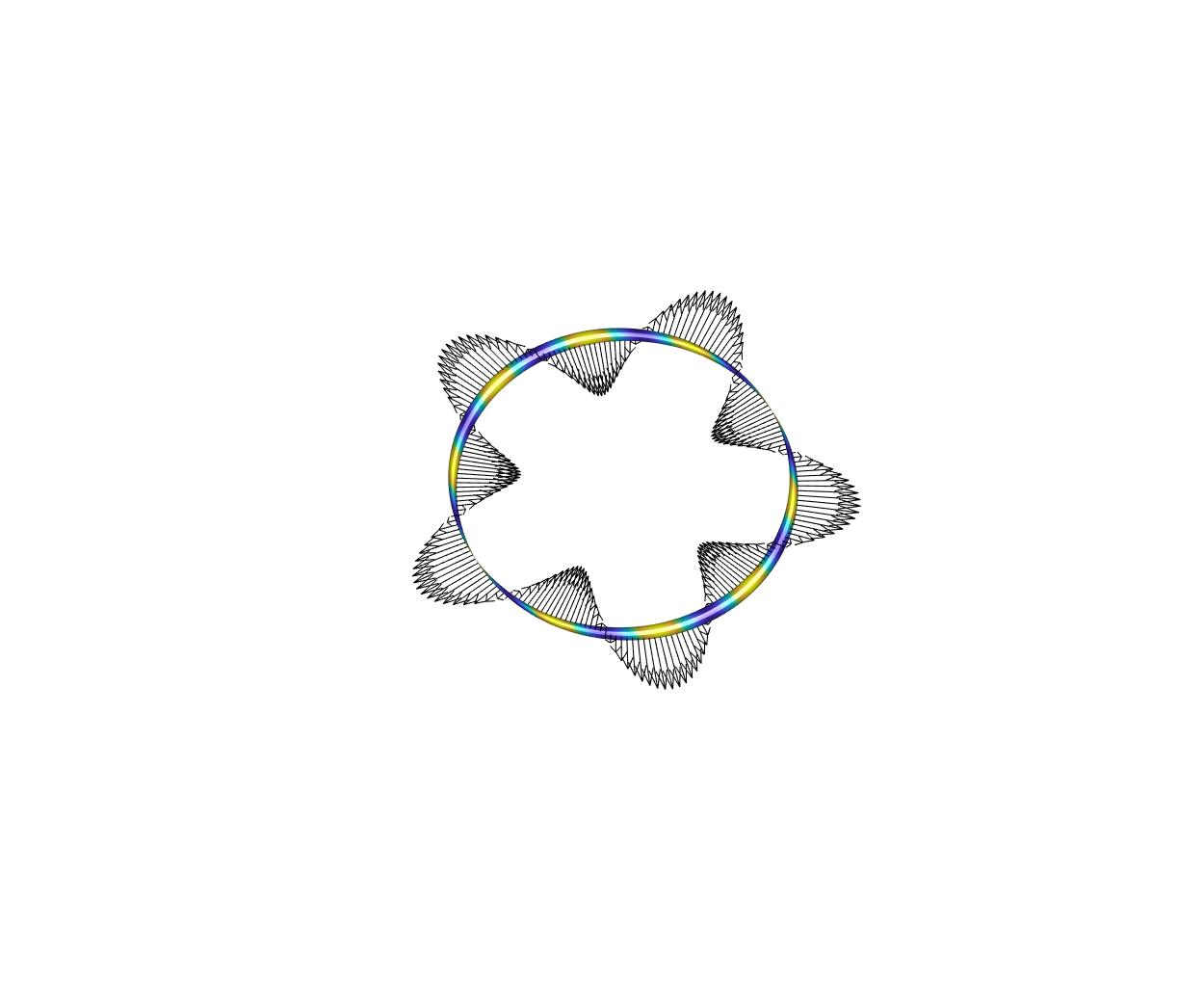}
\bild{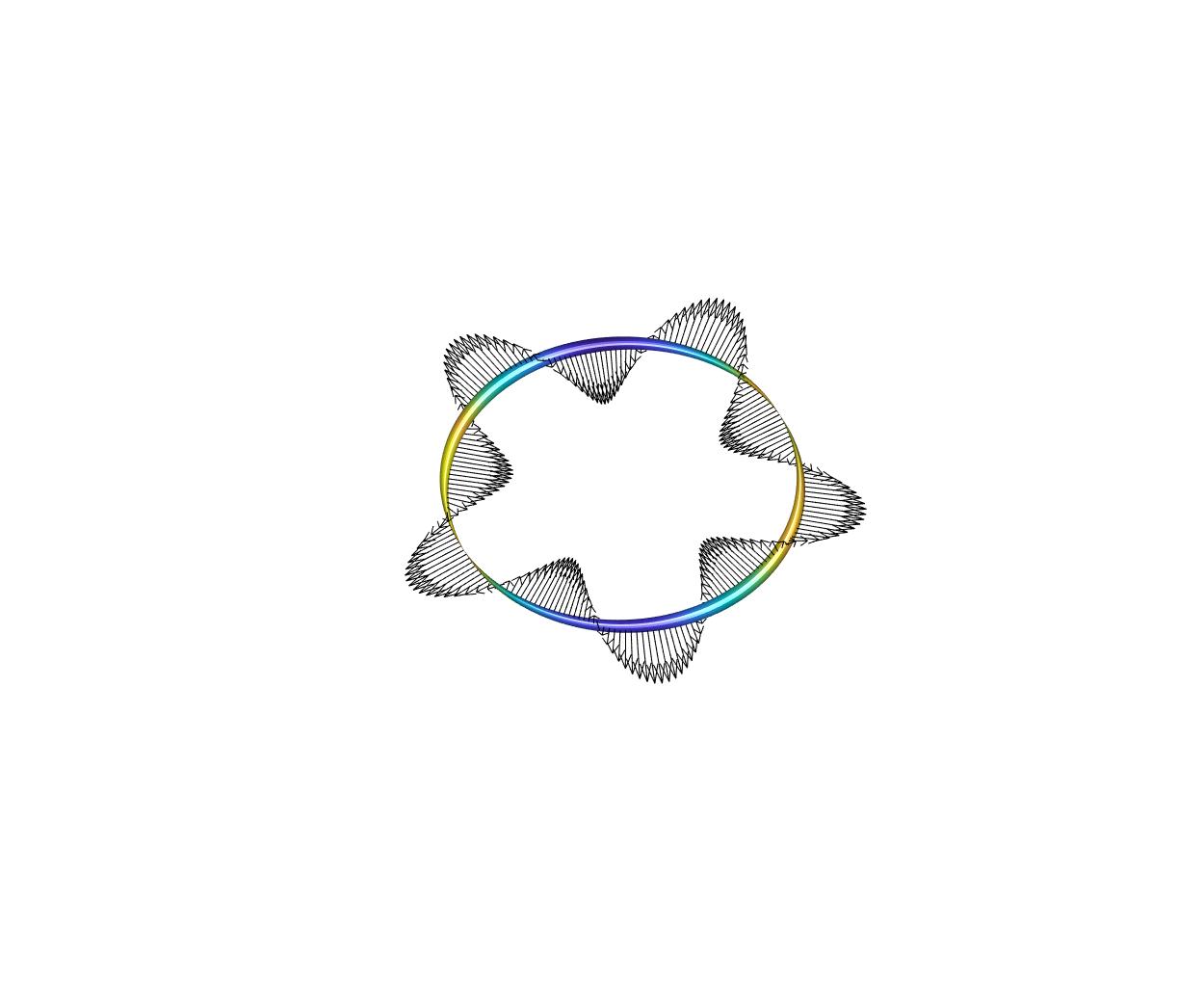}
\bild{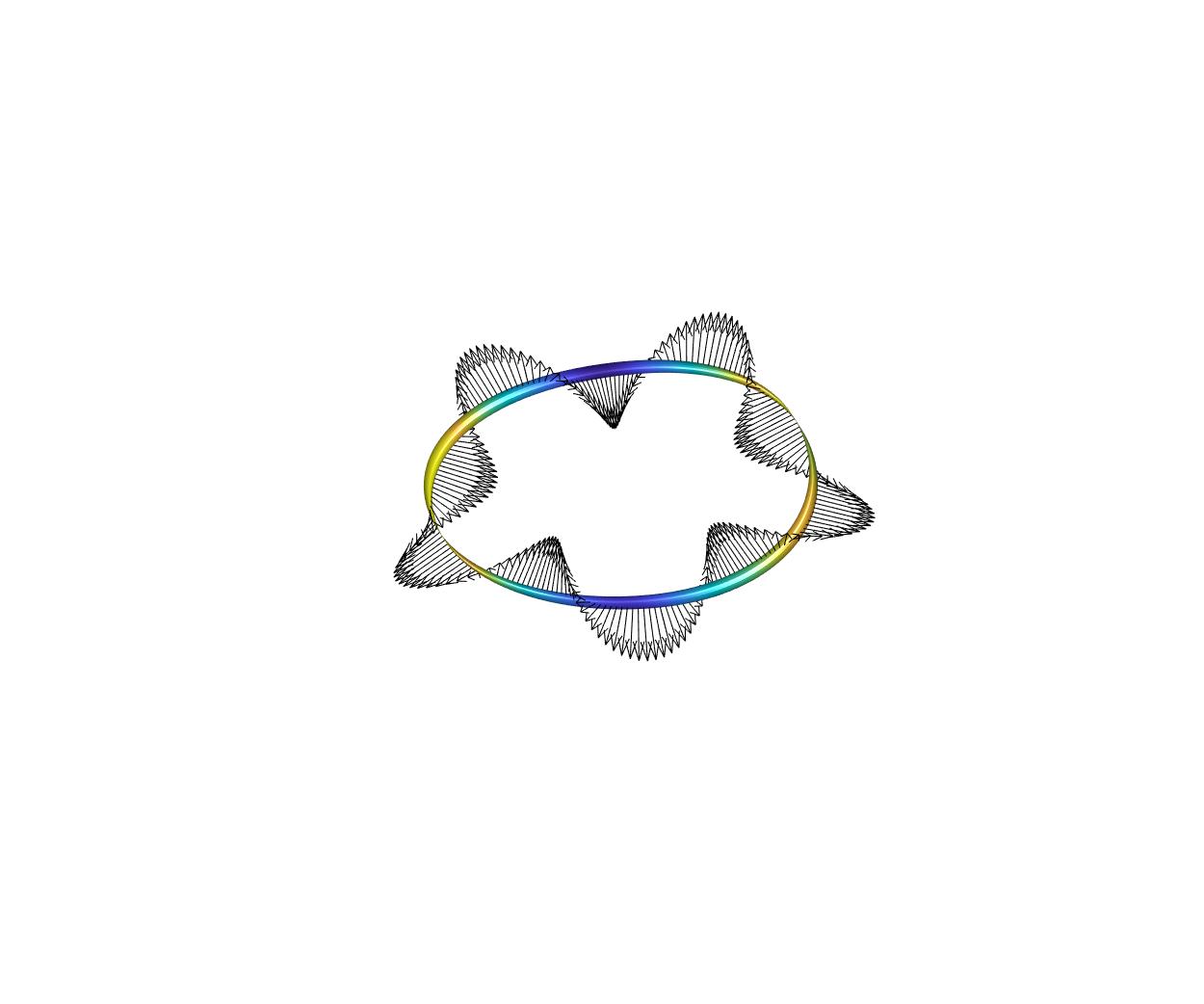}
\bild{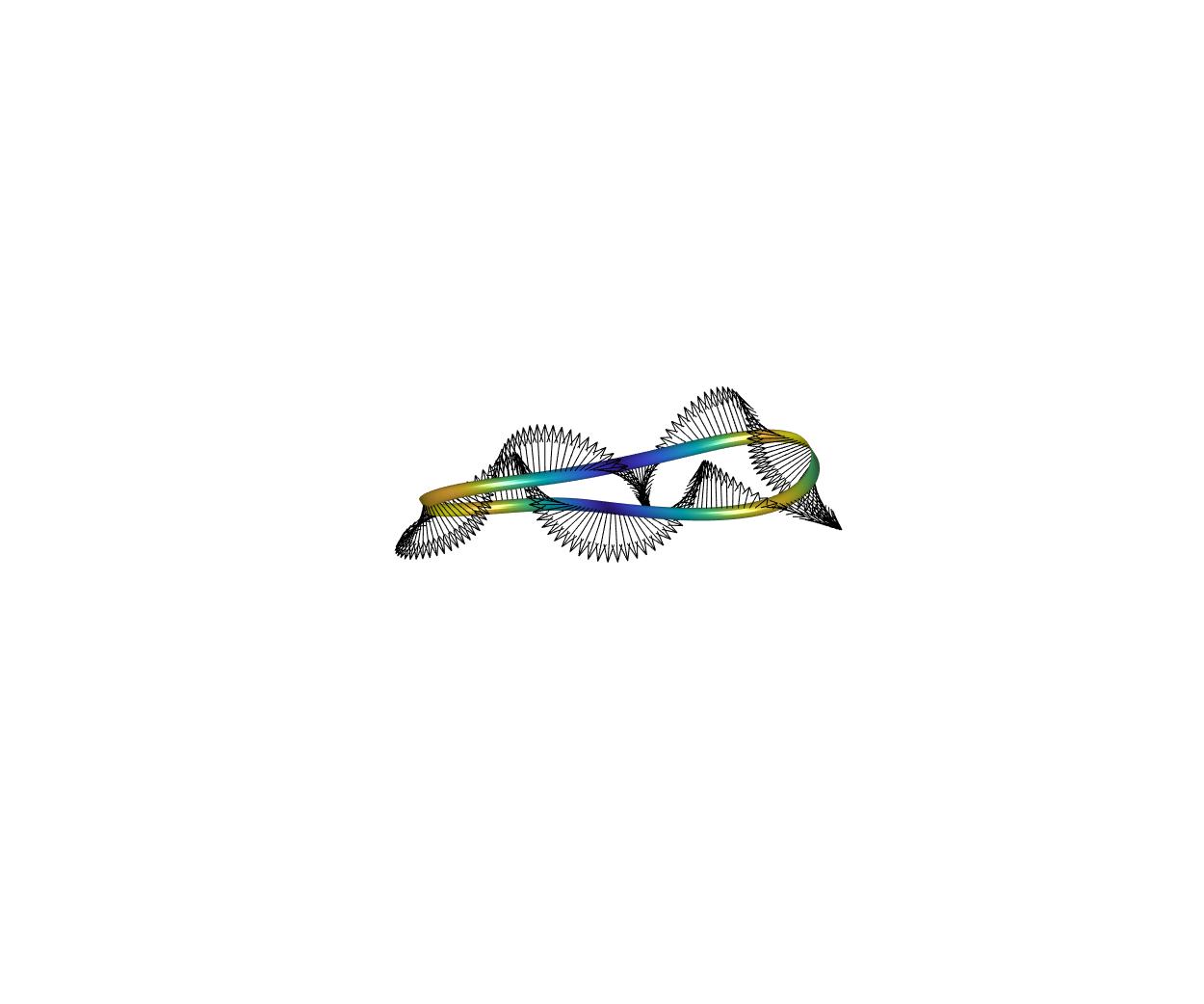}

\bild{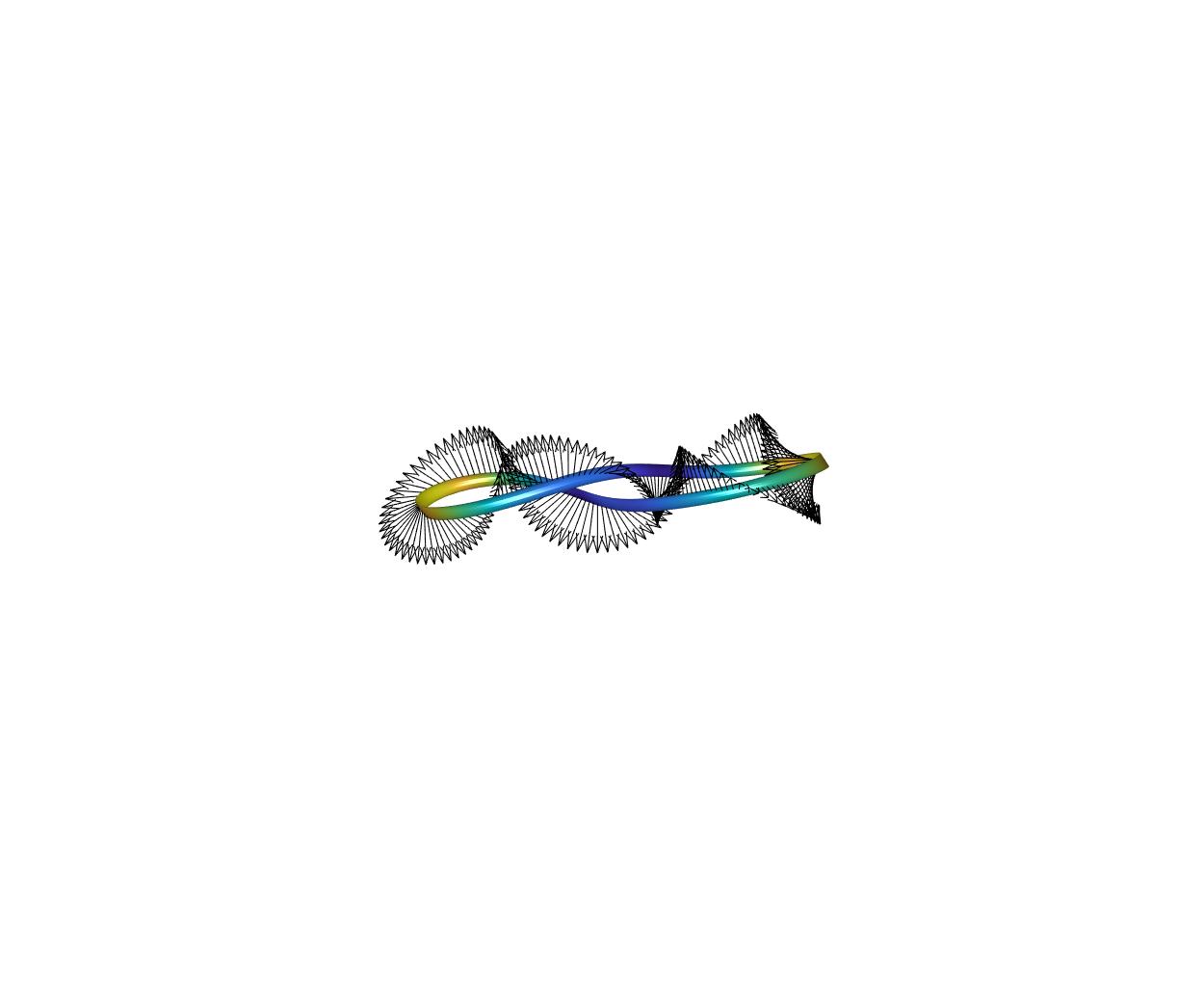}
\bild{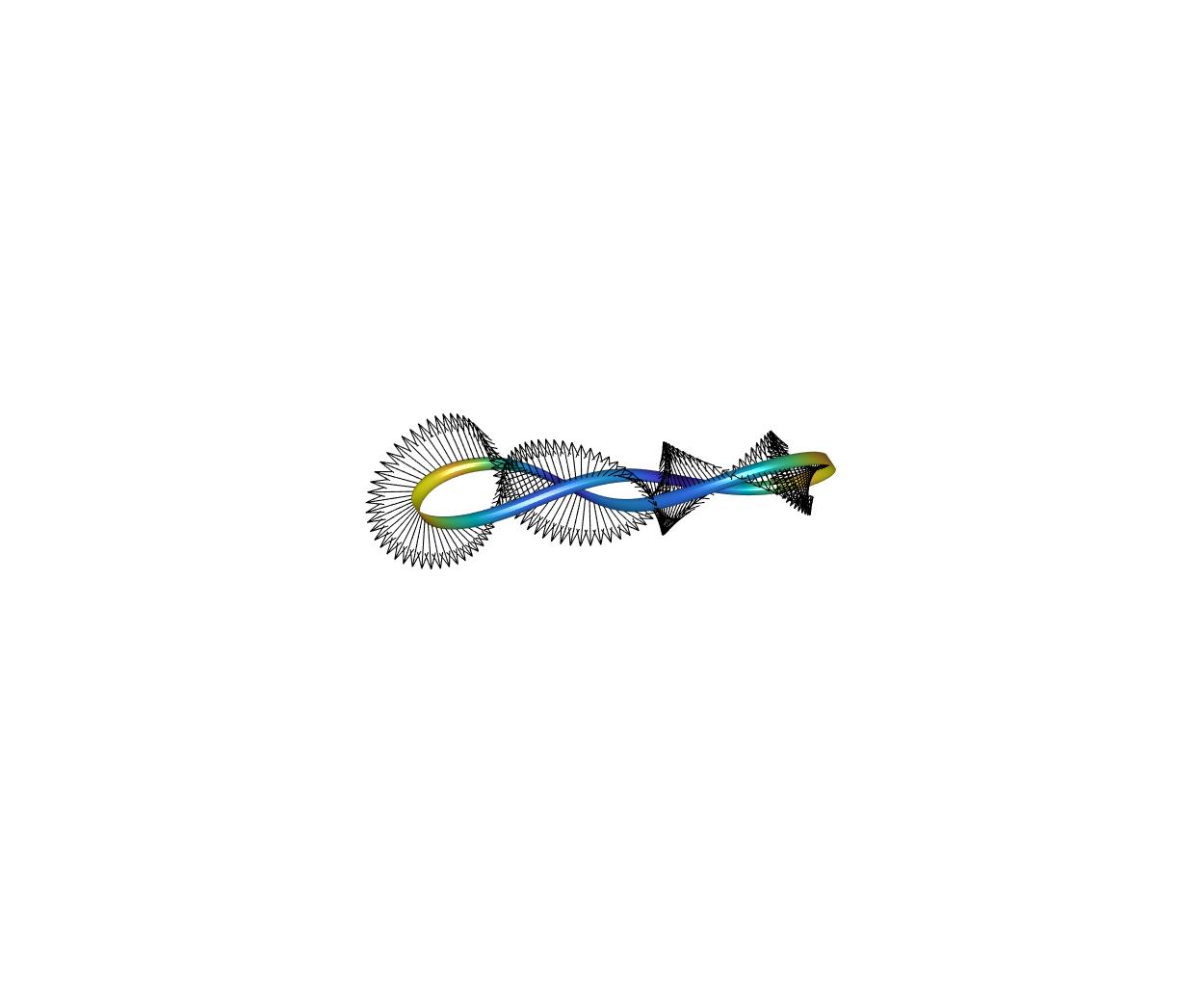}
\bild{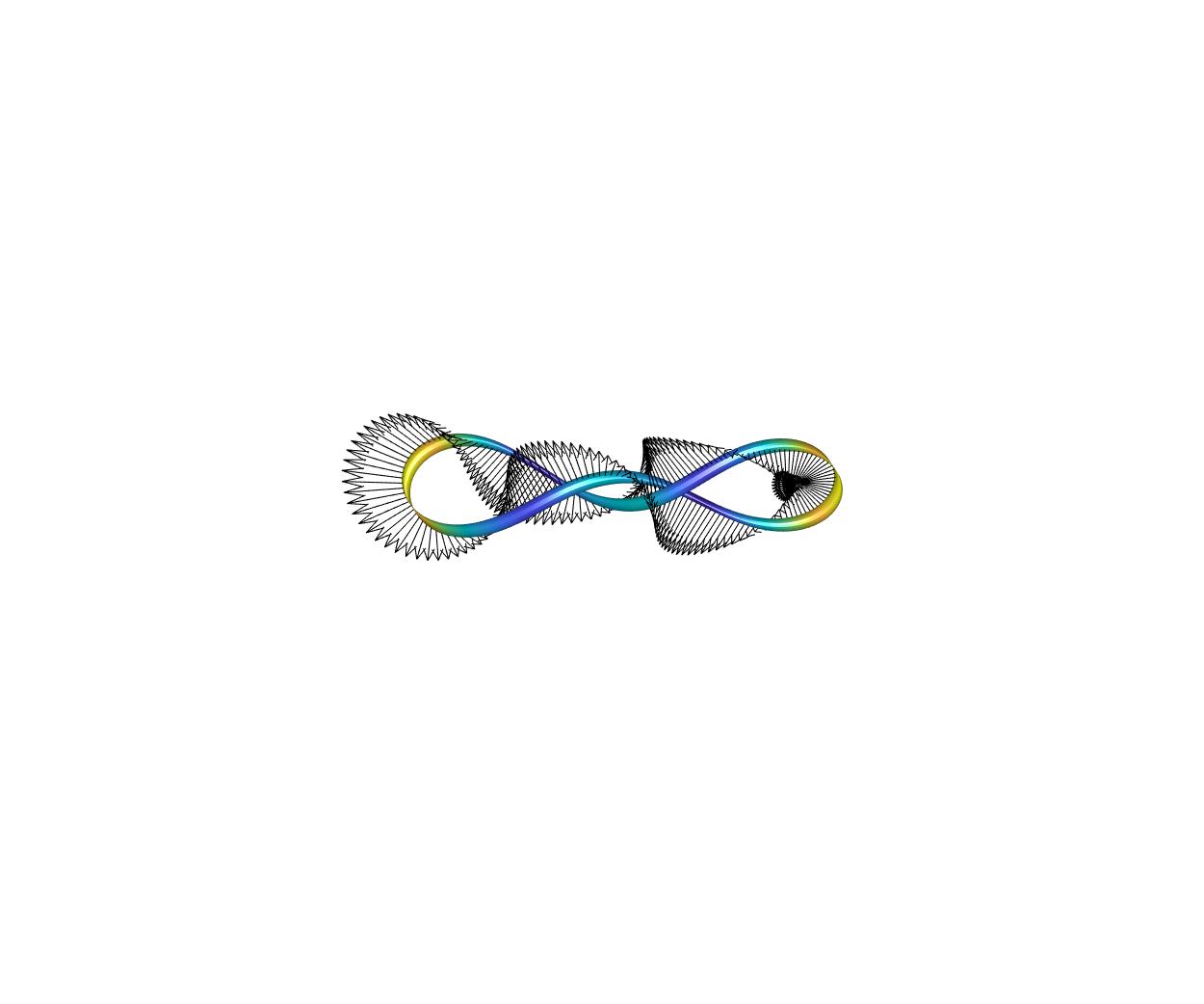}
\bild{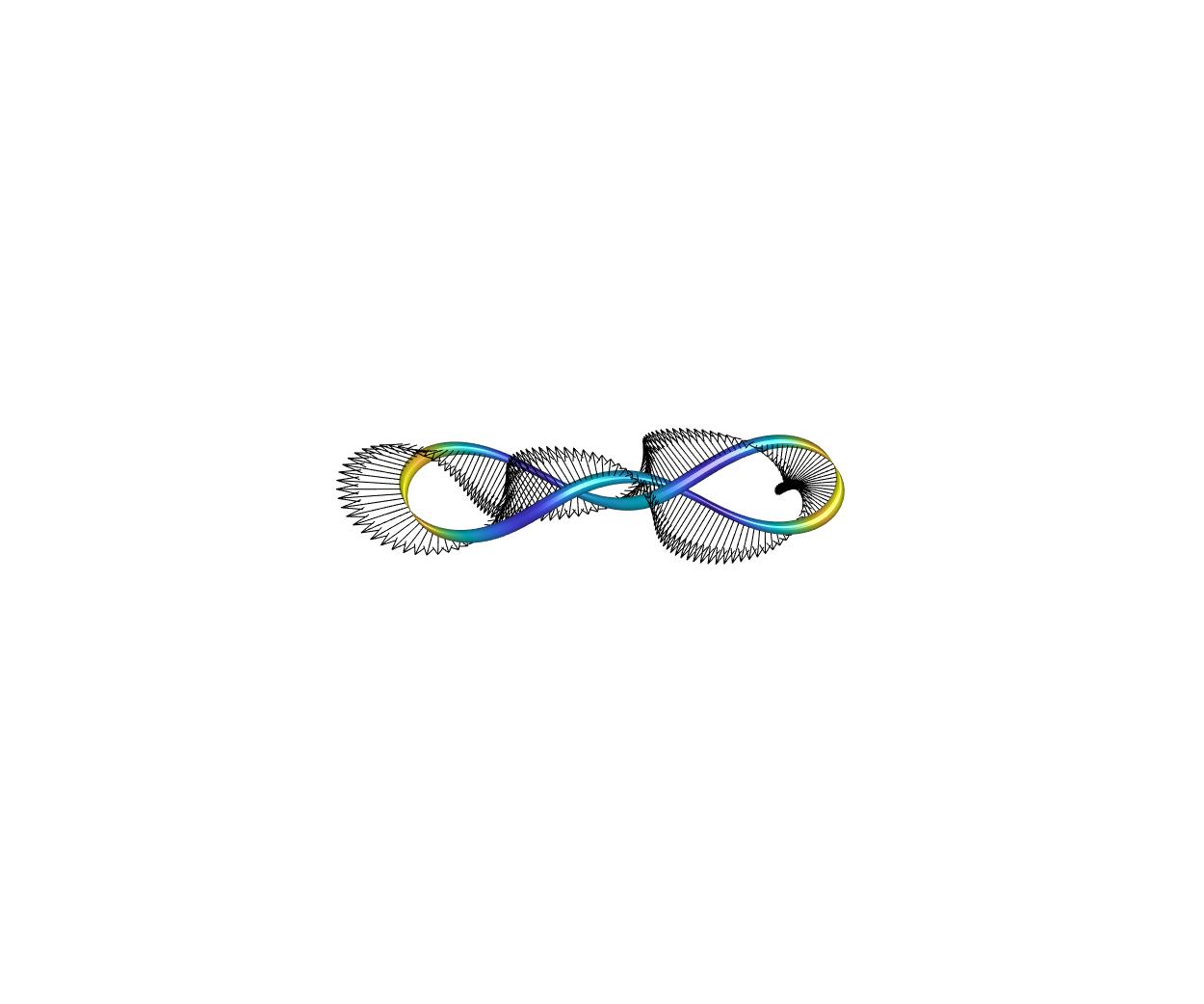}

\caption{{In the first part of Experiment~\ref{sect:imper}
we repeat Experiment~\ref{sect:overtwist}
in the presence of impermeability.
As self-penetrations are excluded, we observe the formation of coilings.
The energy plot is shown in Figure~\ref{fig:imper-plot} (left).}}
\label{fig:imper-overtwist}
\end{figure}

\renewcommand{\bild}[1]{\fbox{%
\scalebox{-1}[1]{%
\includegraphics[scale=.13,trim=280 420 160 320,clip]{g#1.jpg}}}\makebox[0ex][r]{\makebox[0ex][r]{\tiny#1}\hspace{0ex} }\,\ignorespaces}

\begin{figure}

\bild{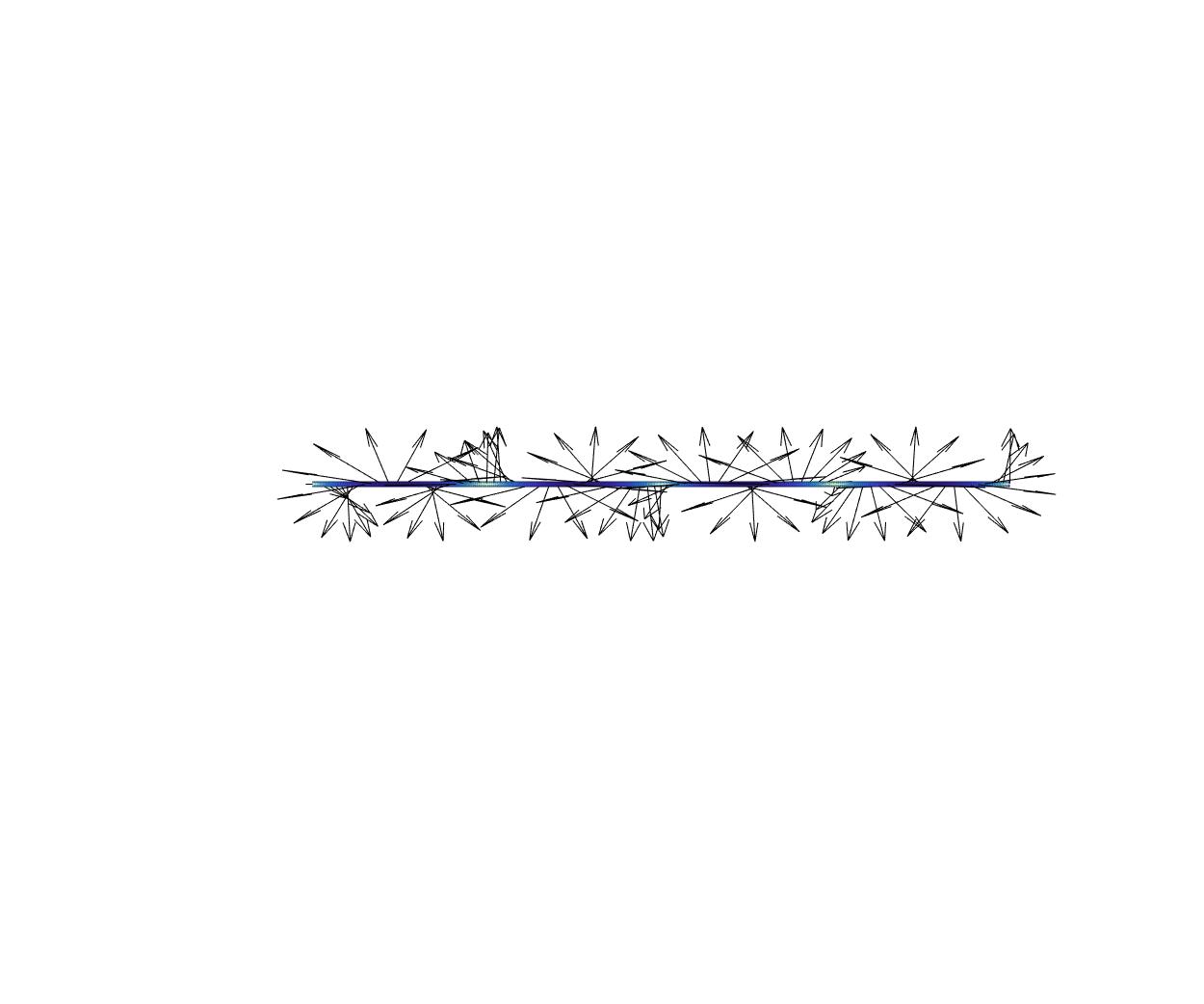}	\bild{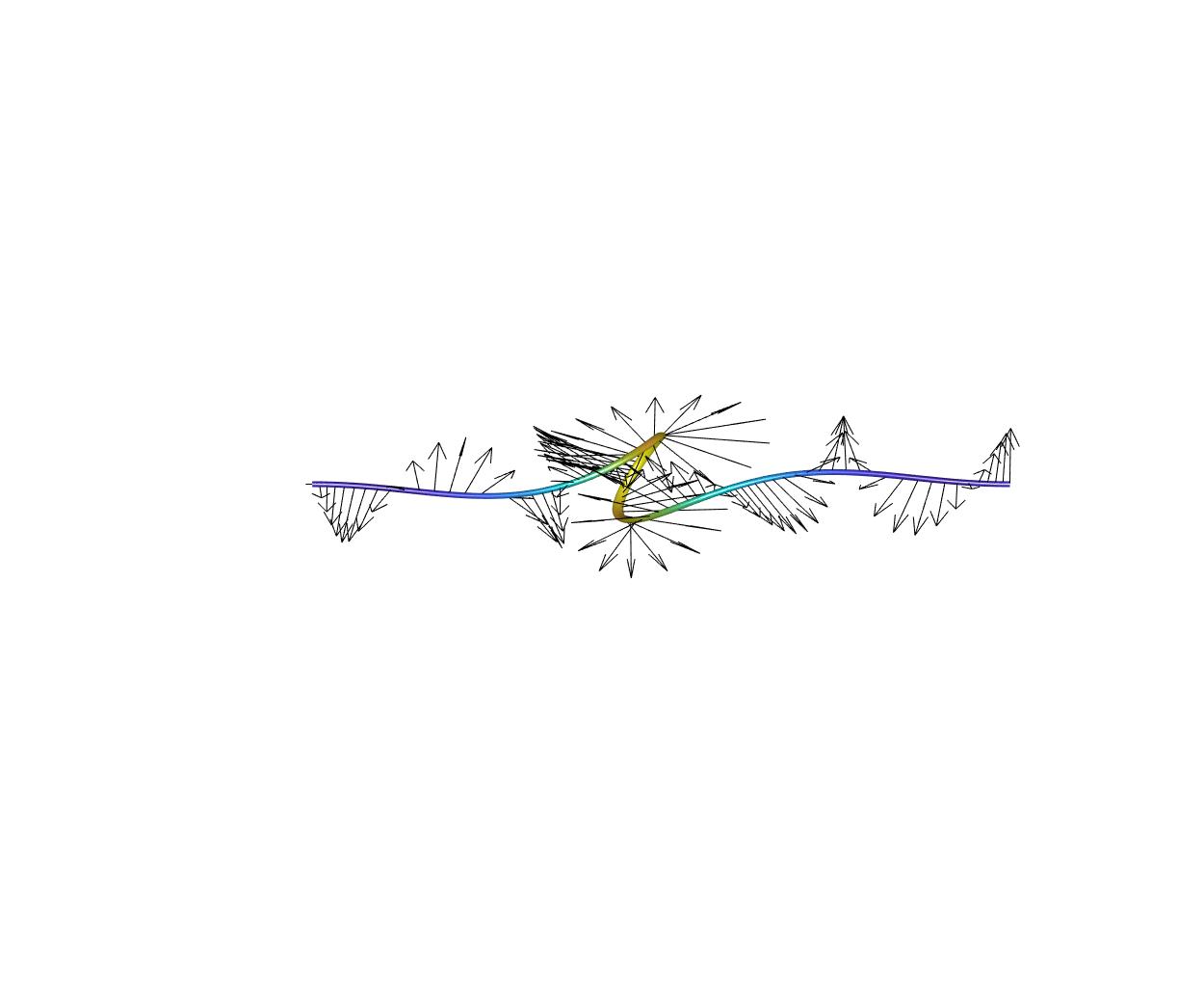}		\bild{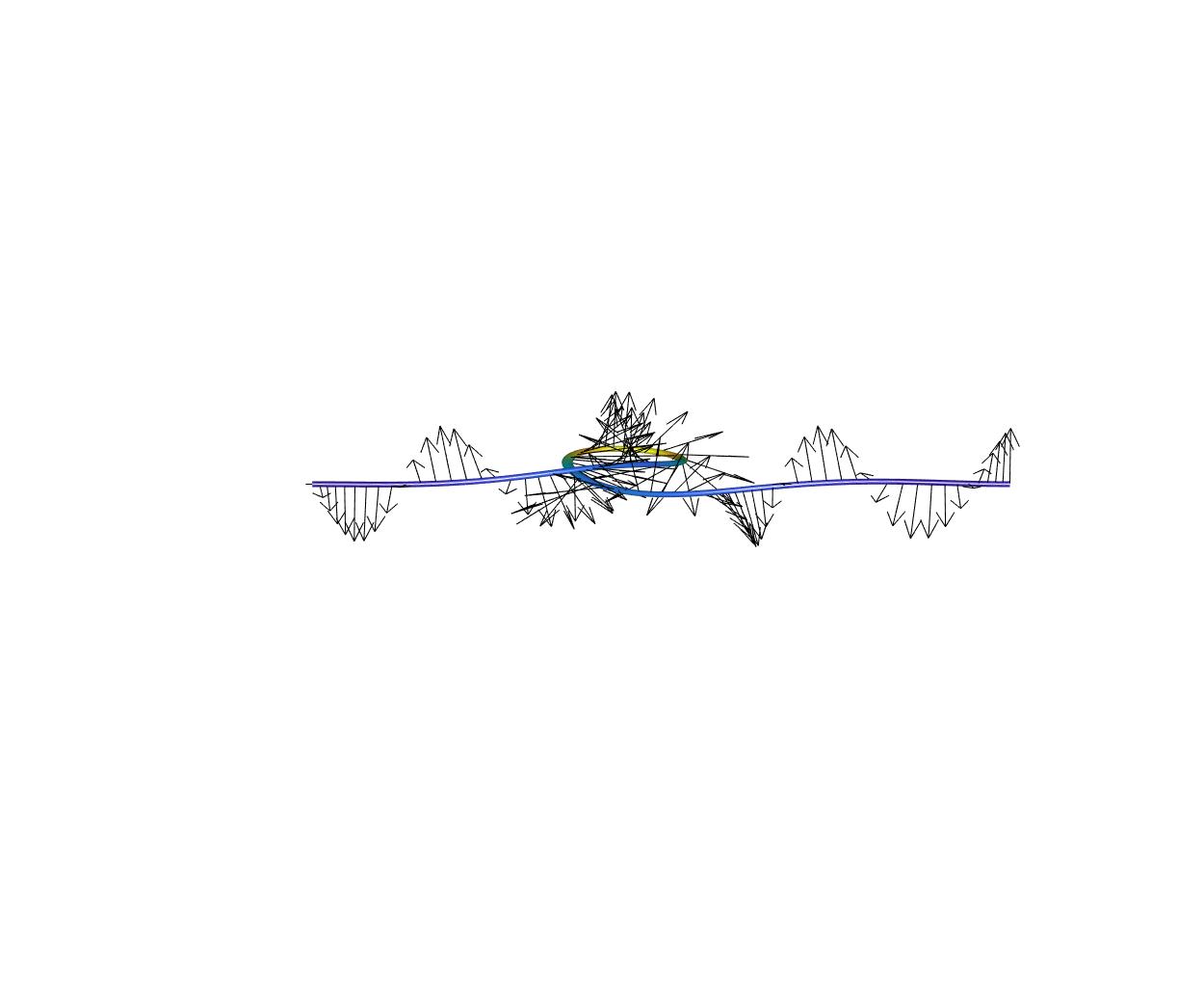}

\bild{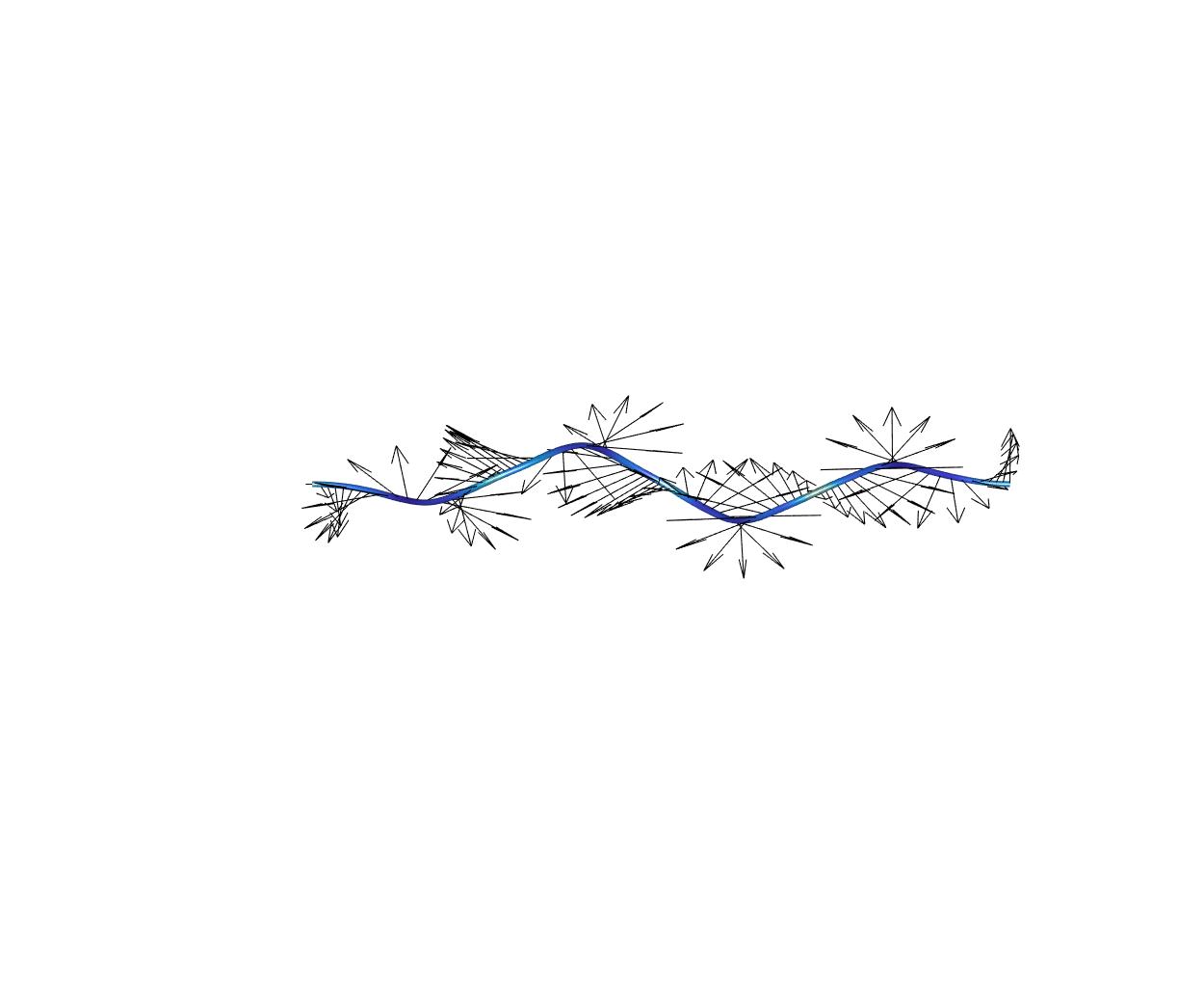}	\bild{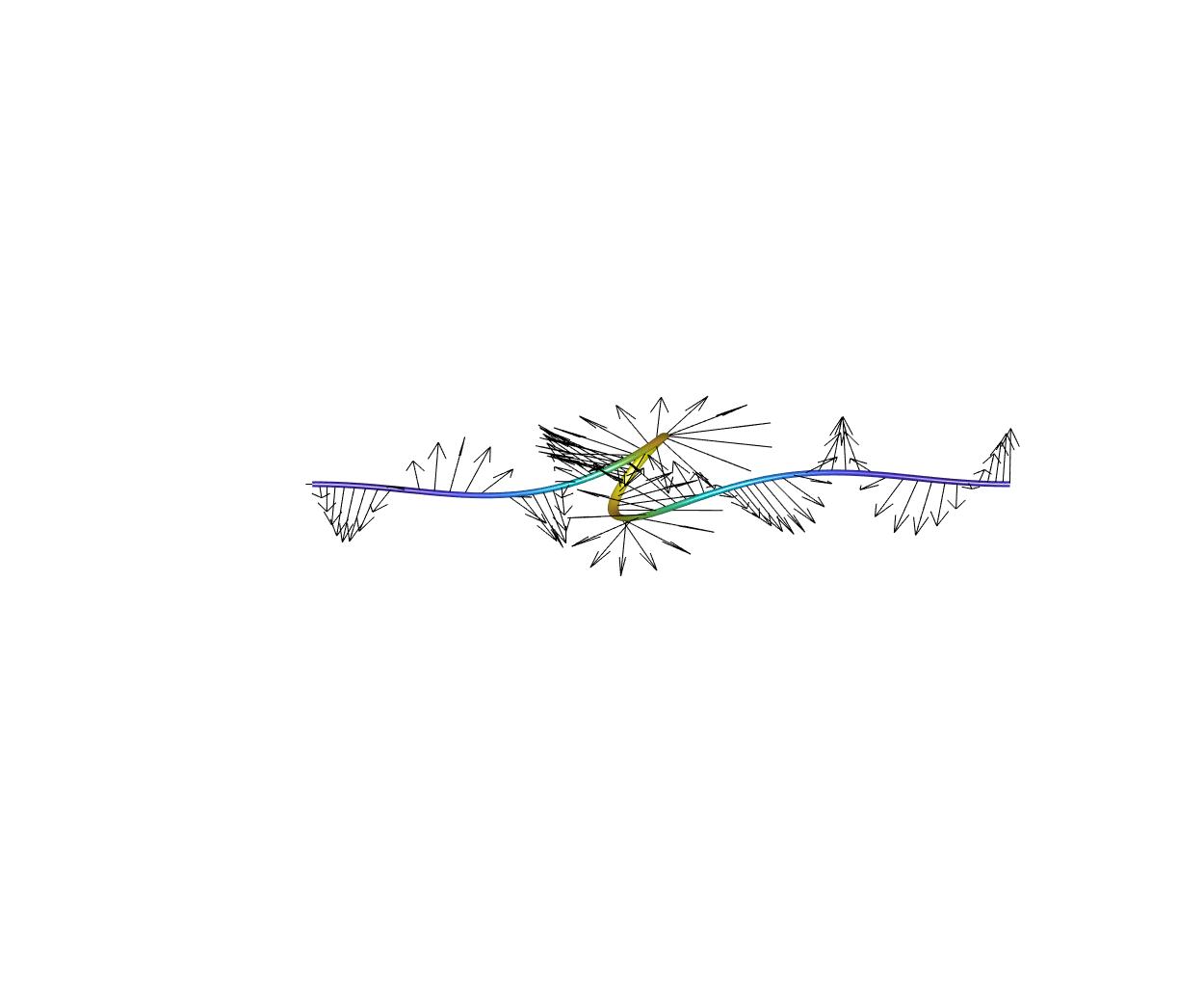}		\bild{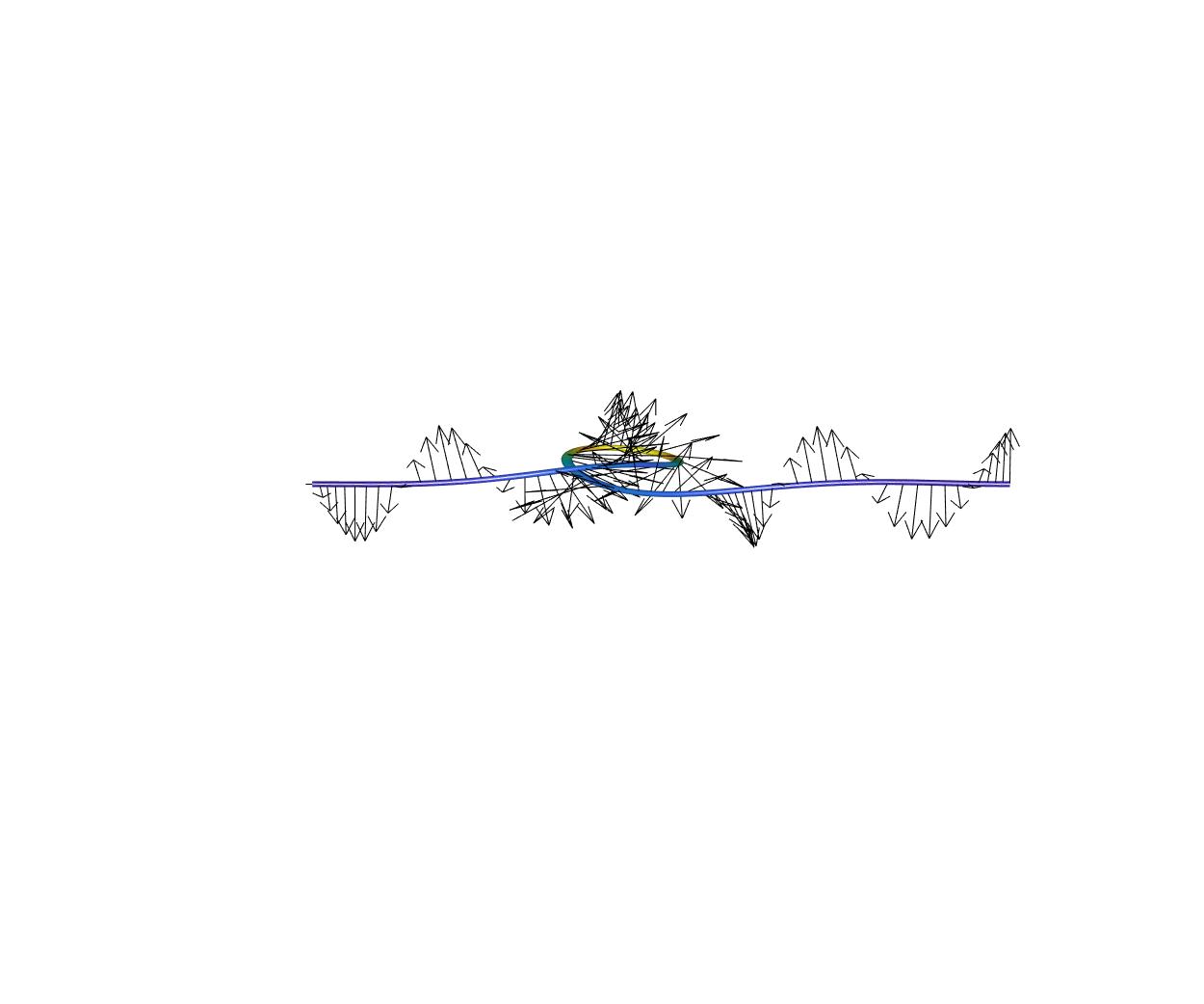}

\bild{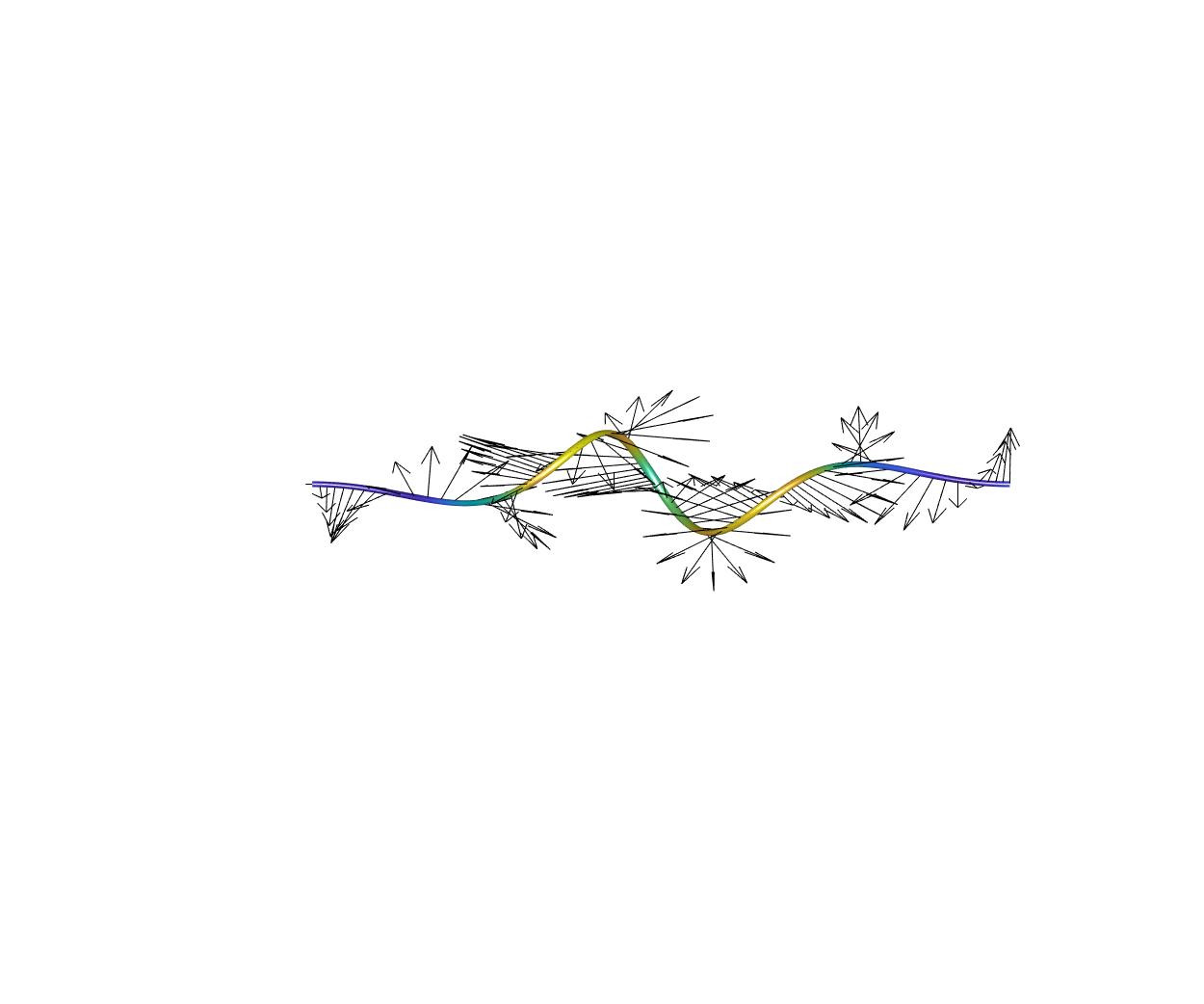}	\bild{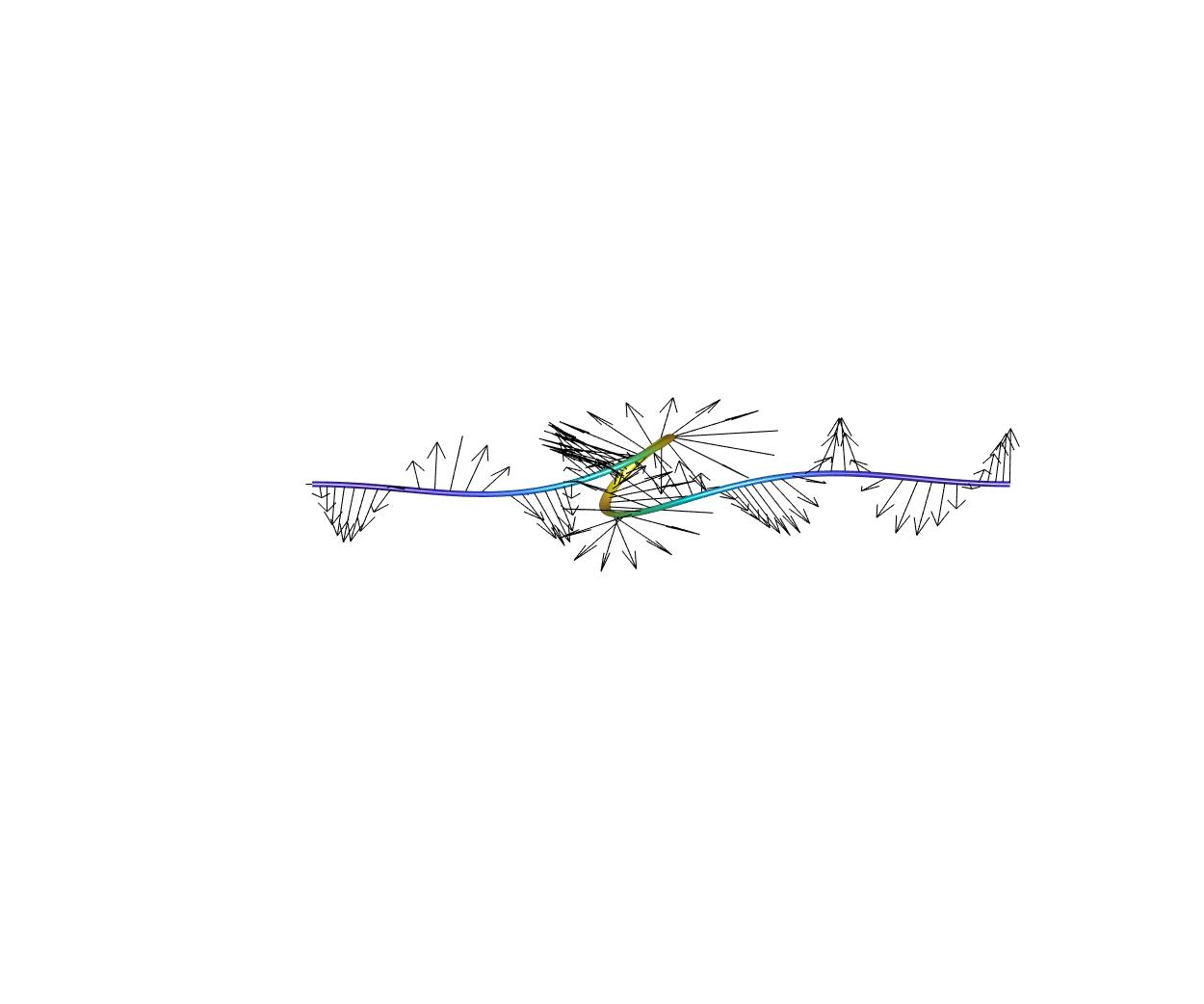}		\bild{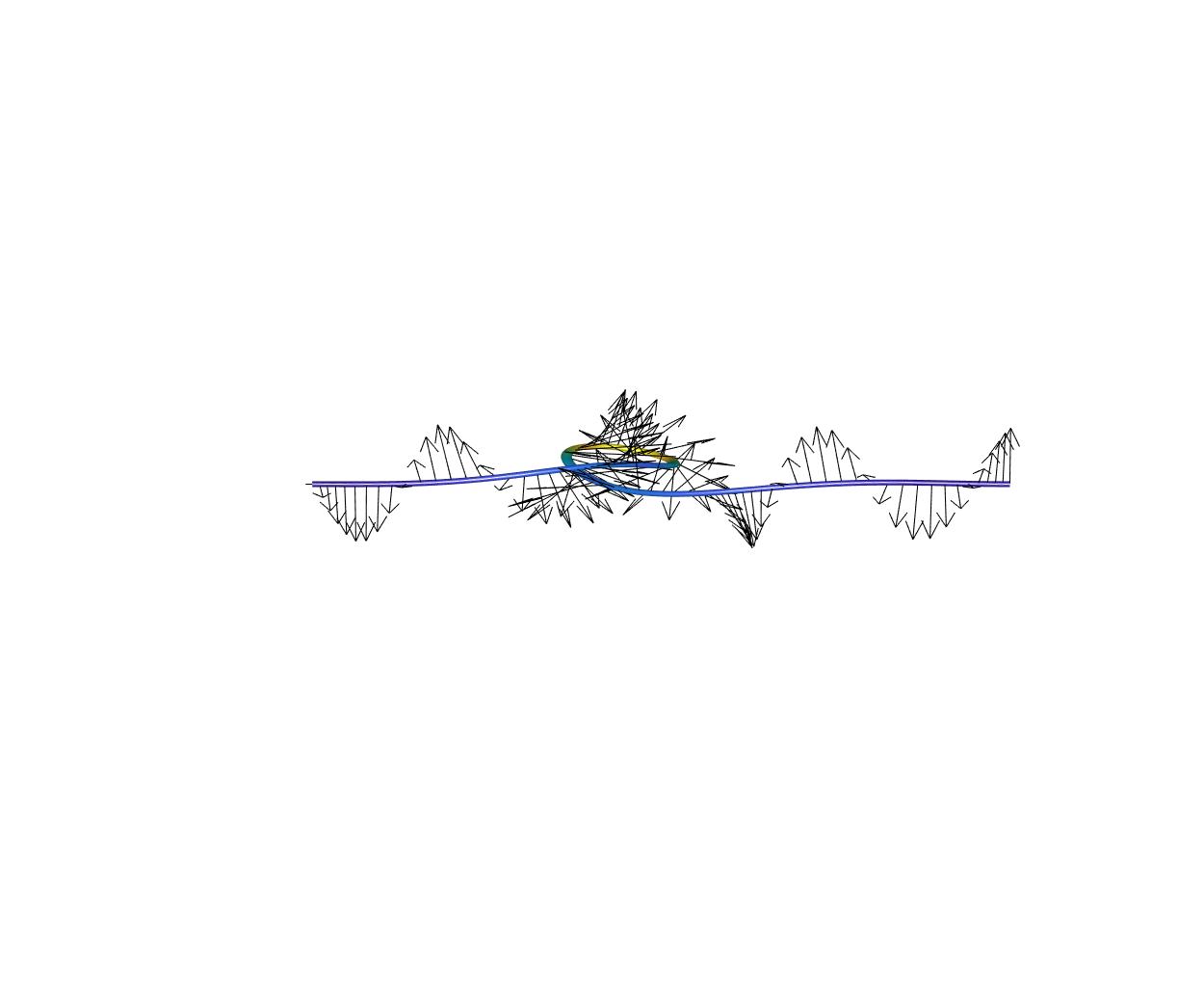}

\bild{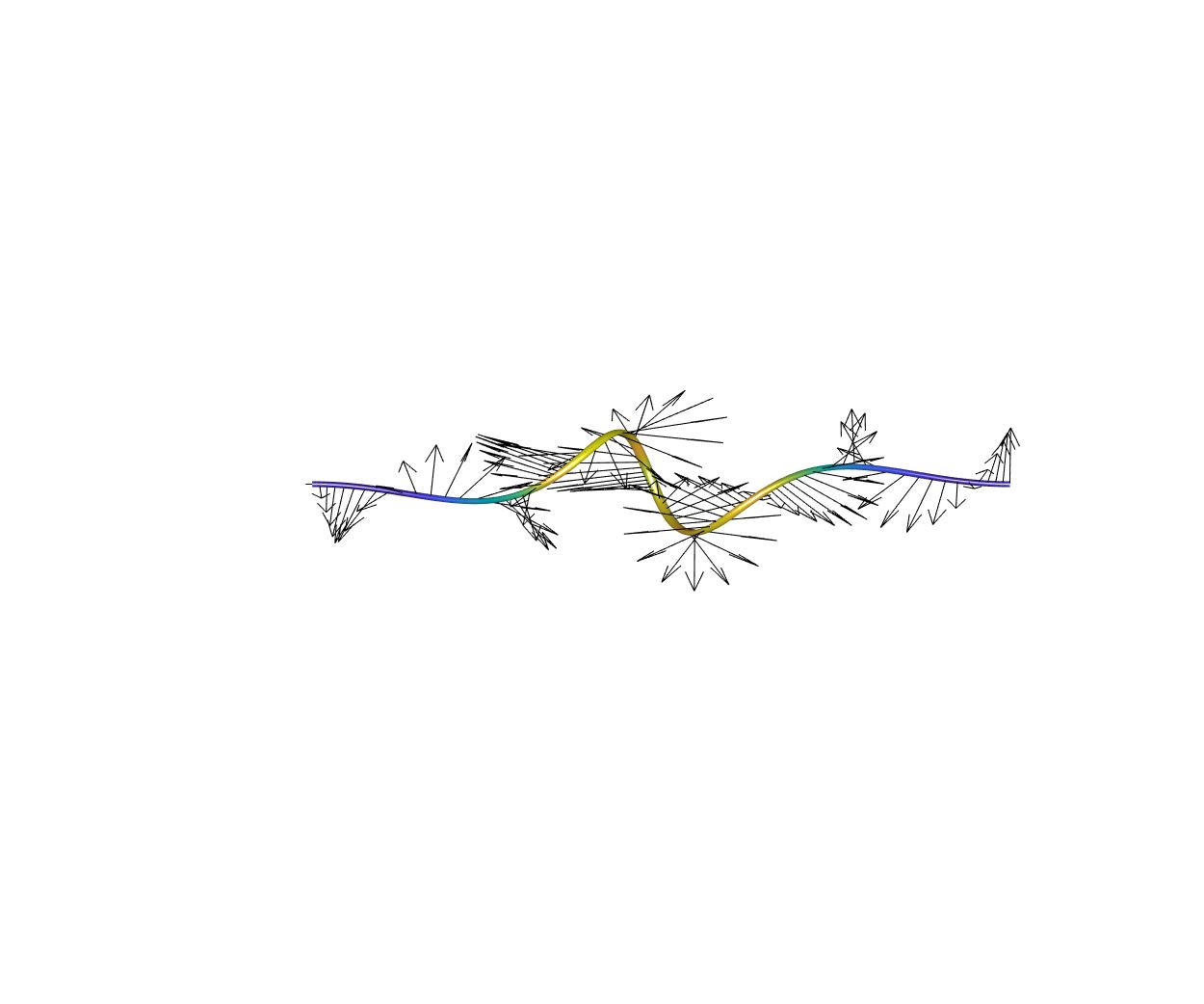}	\bild{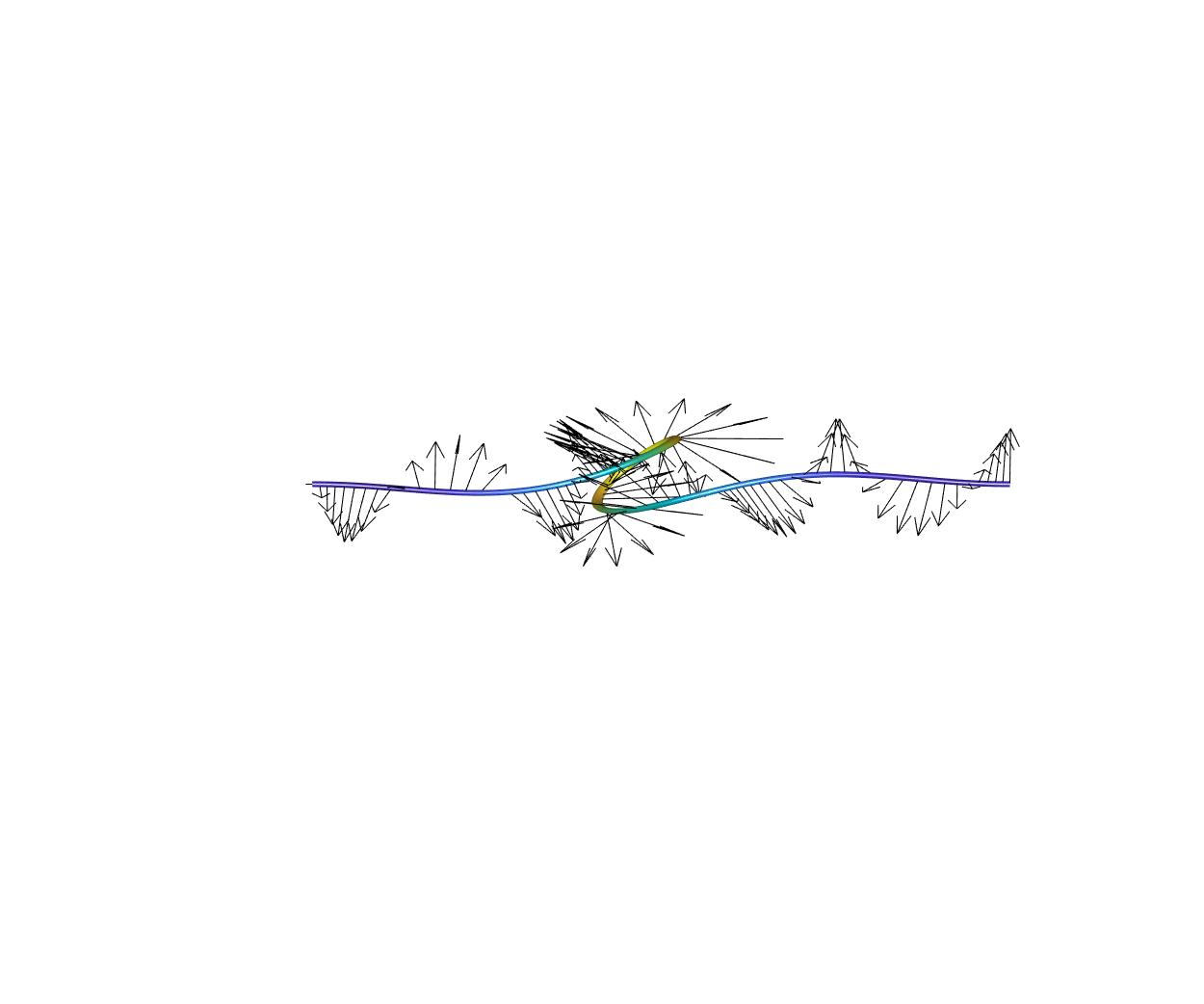}	\bild{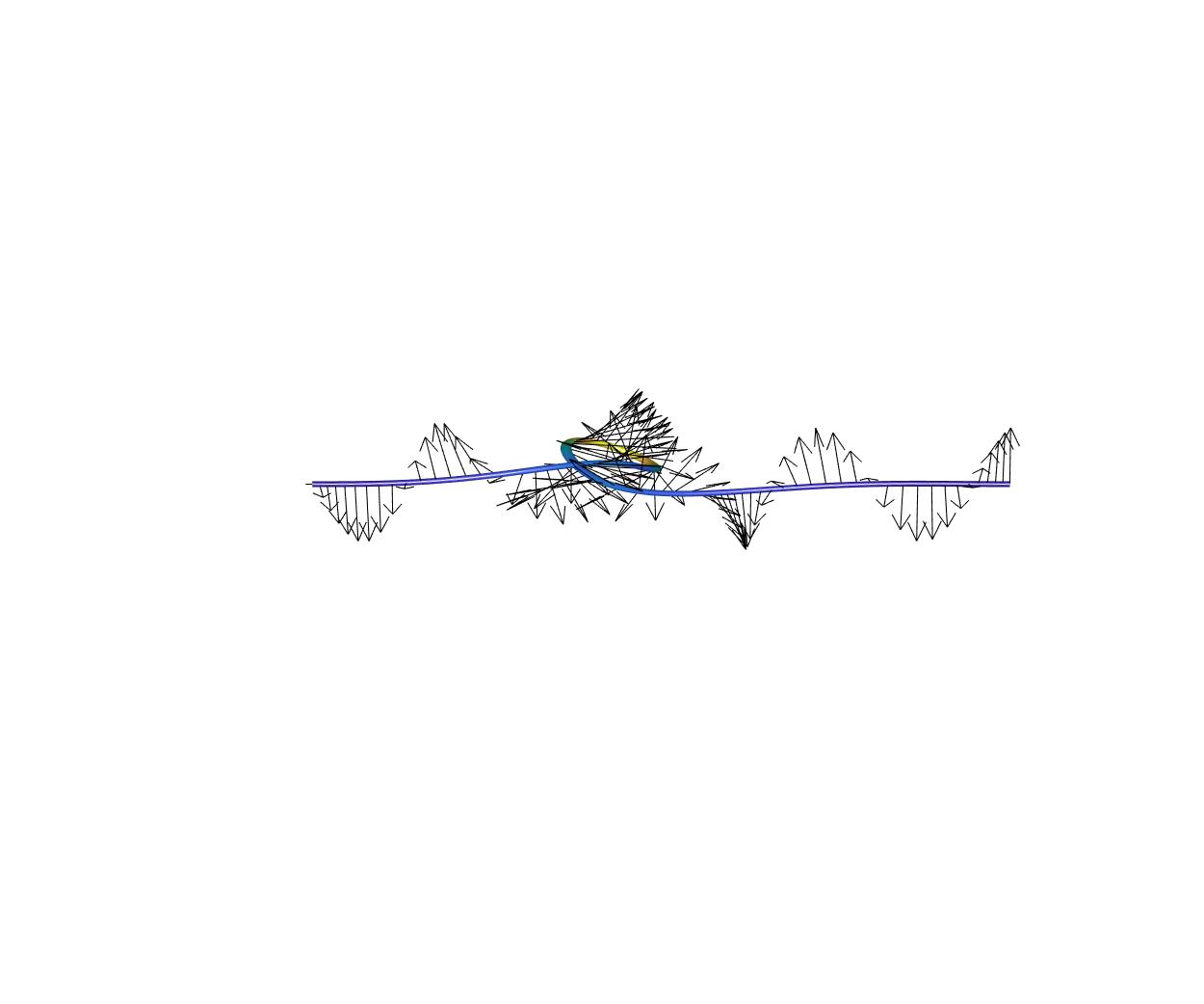}

\bild{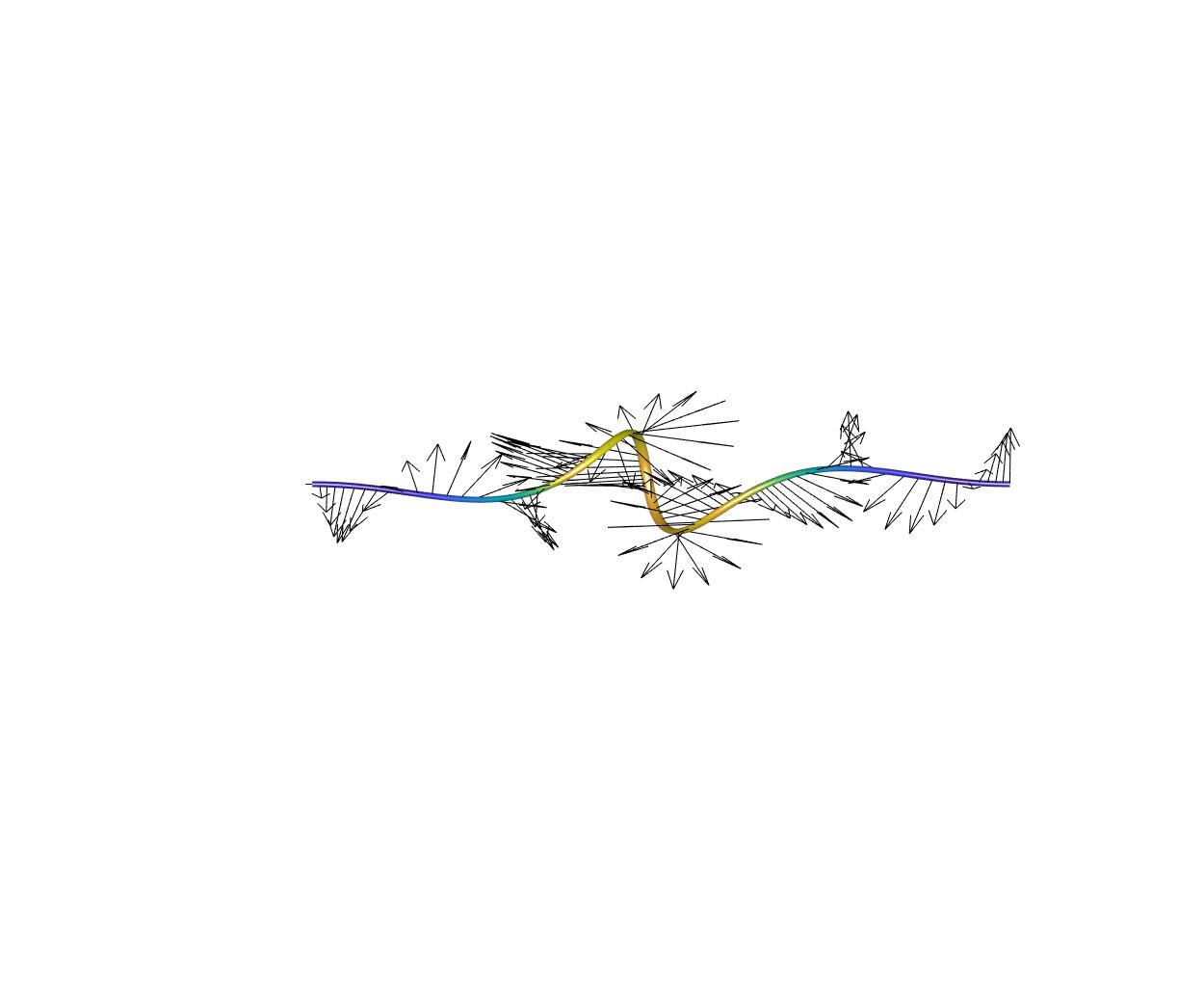}	\bild{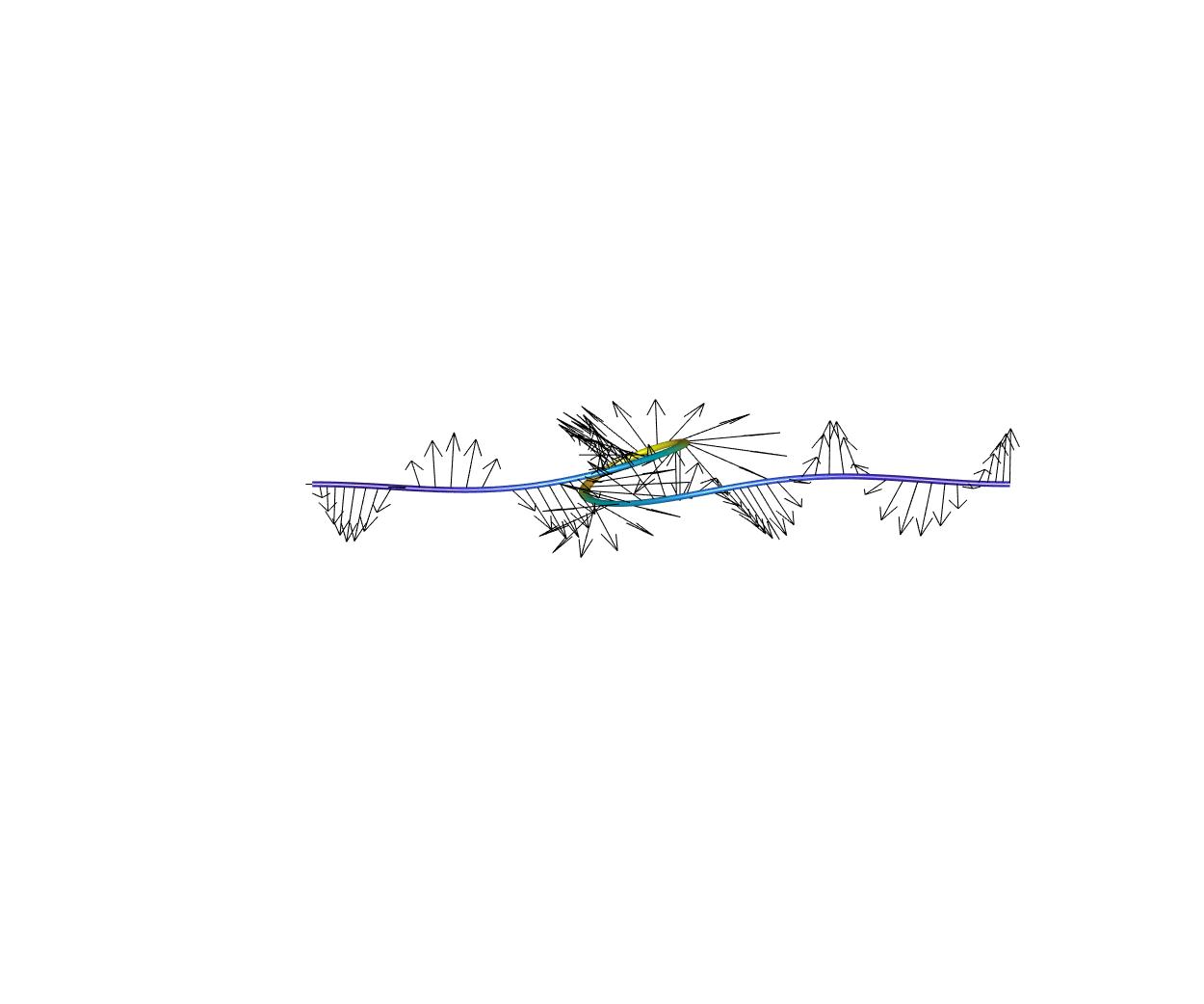}	\bild{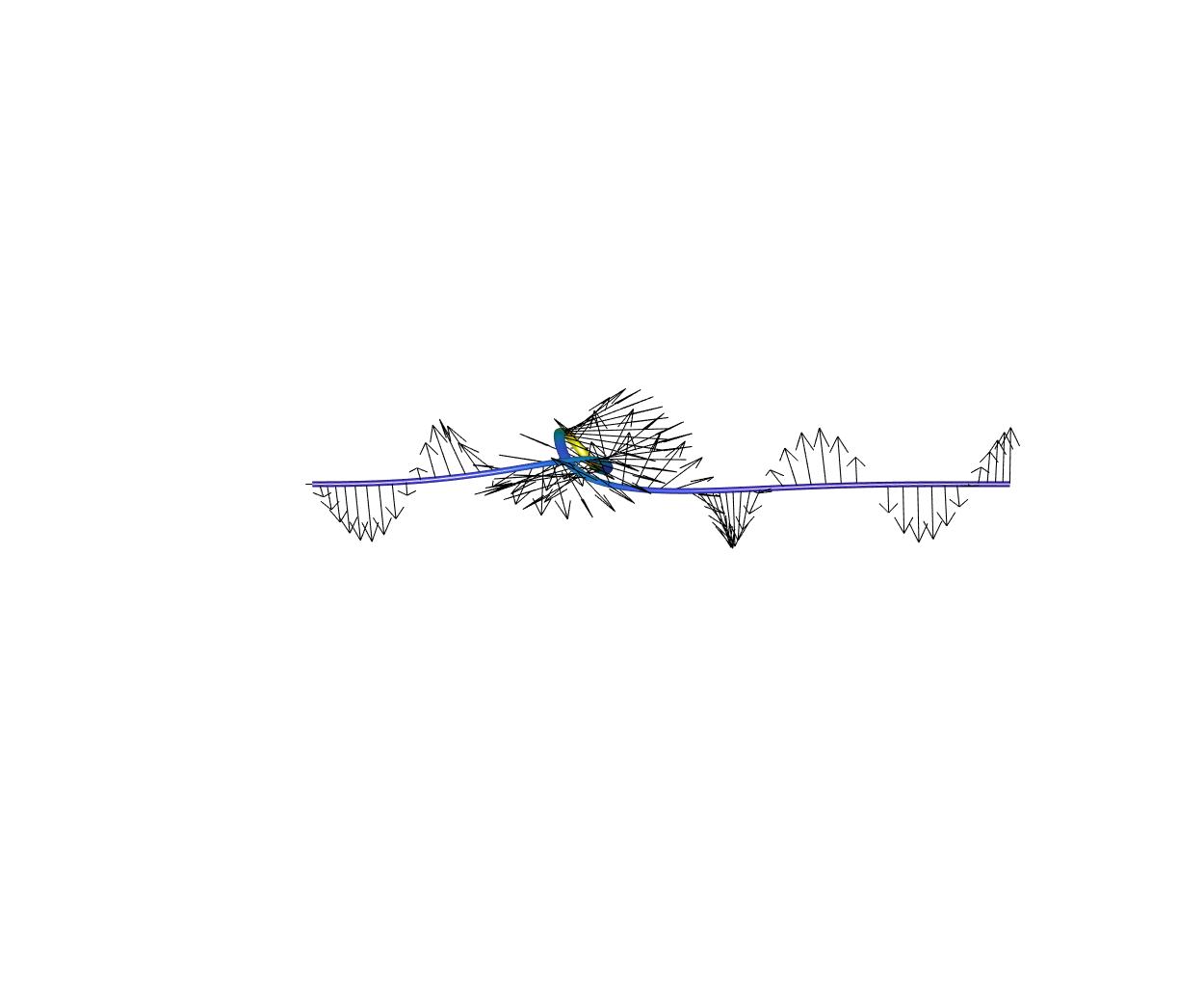}

\bild{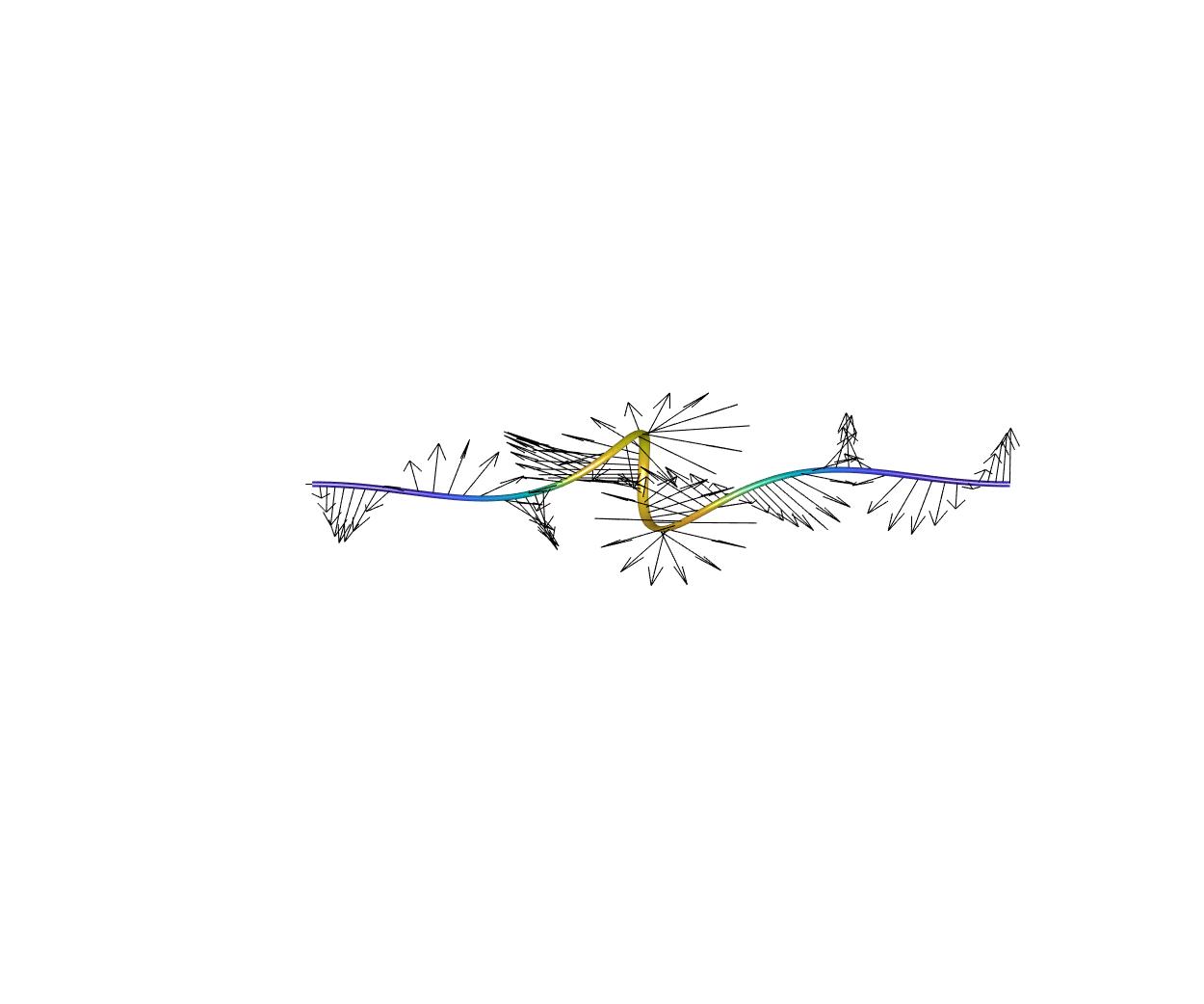}	\bild{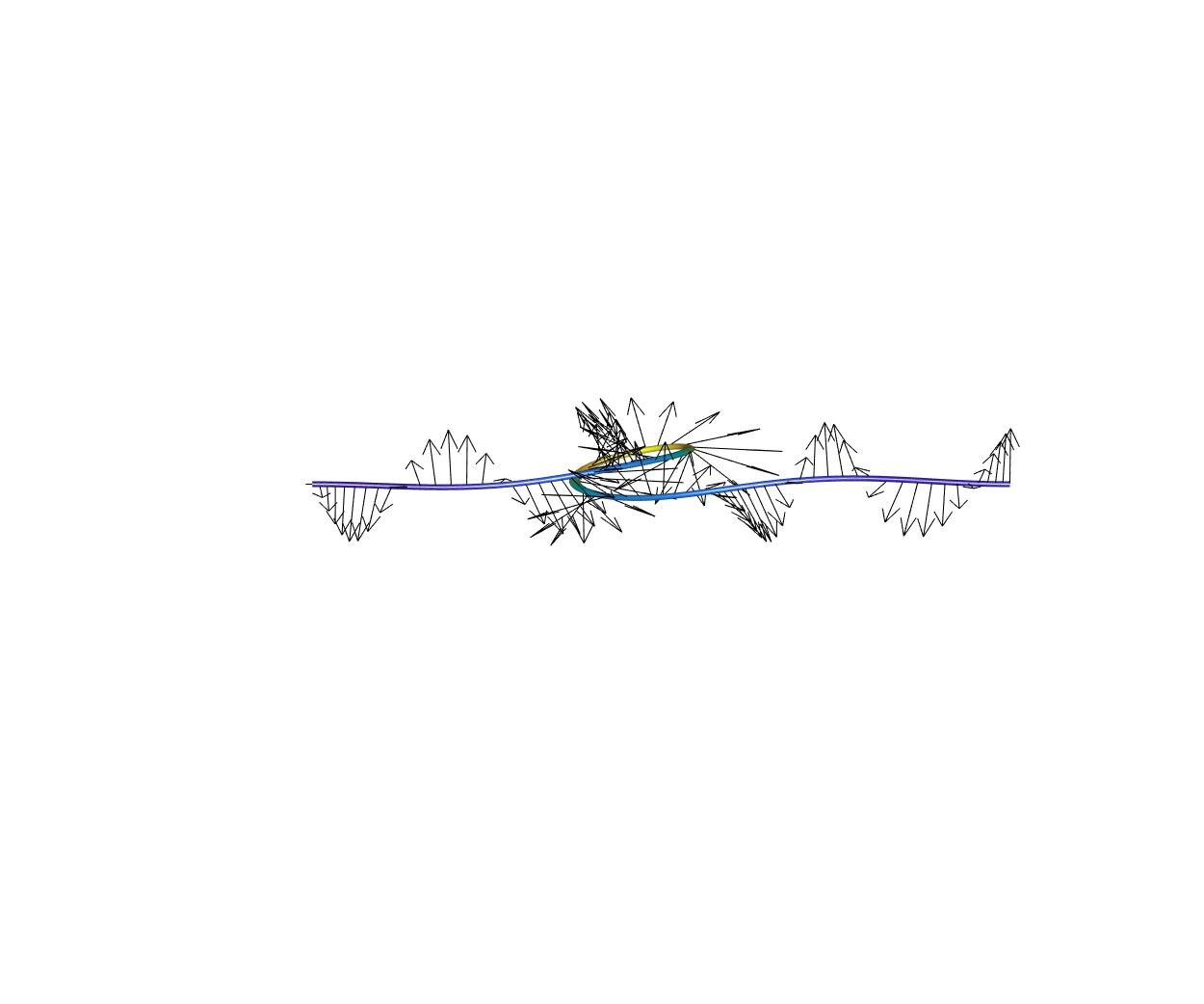}	\bild{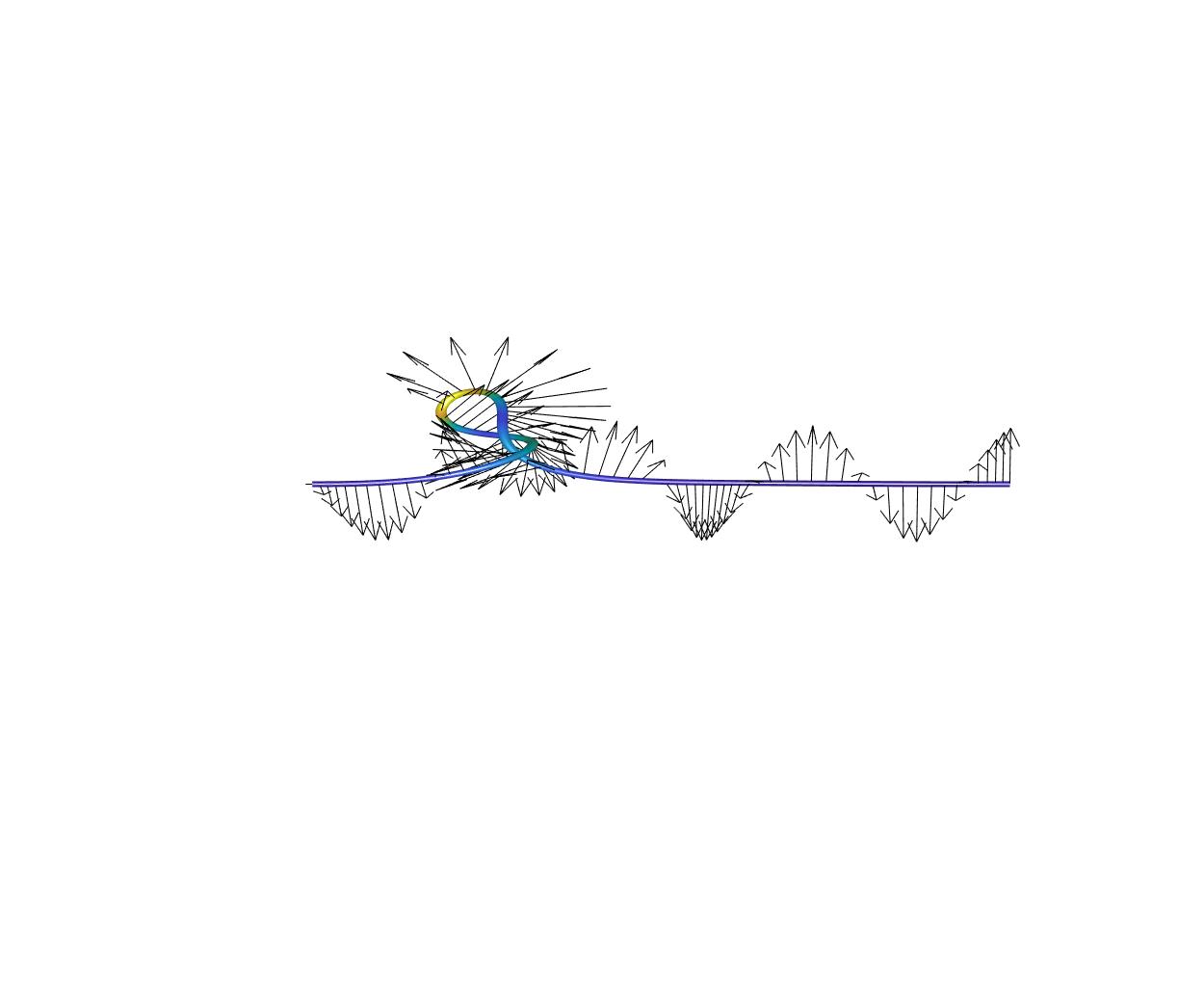}

\bild{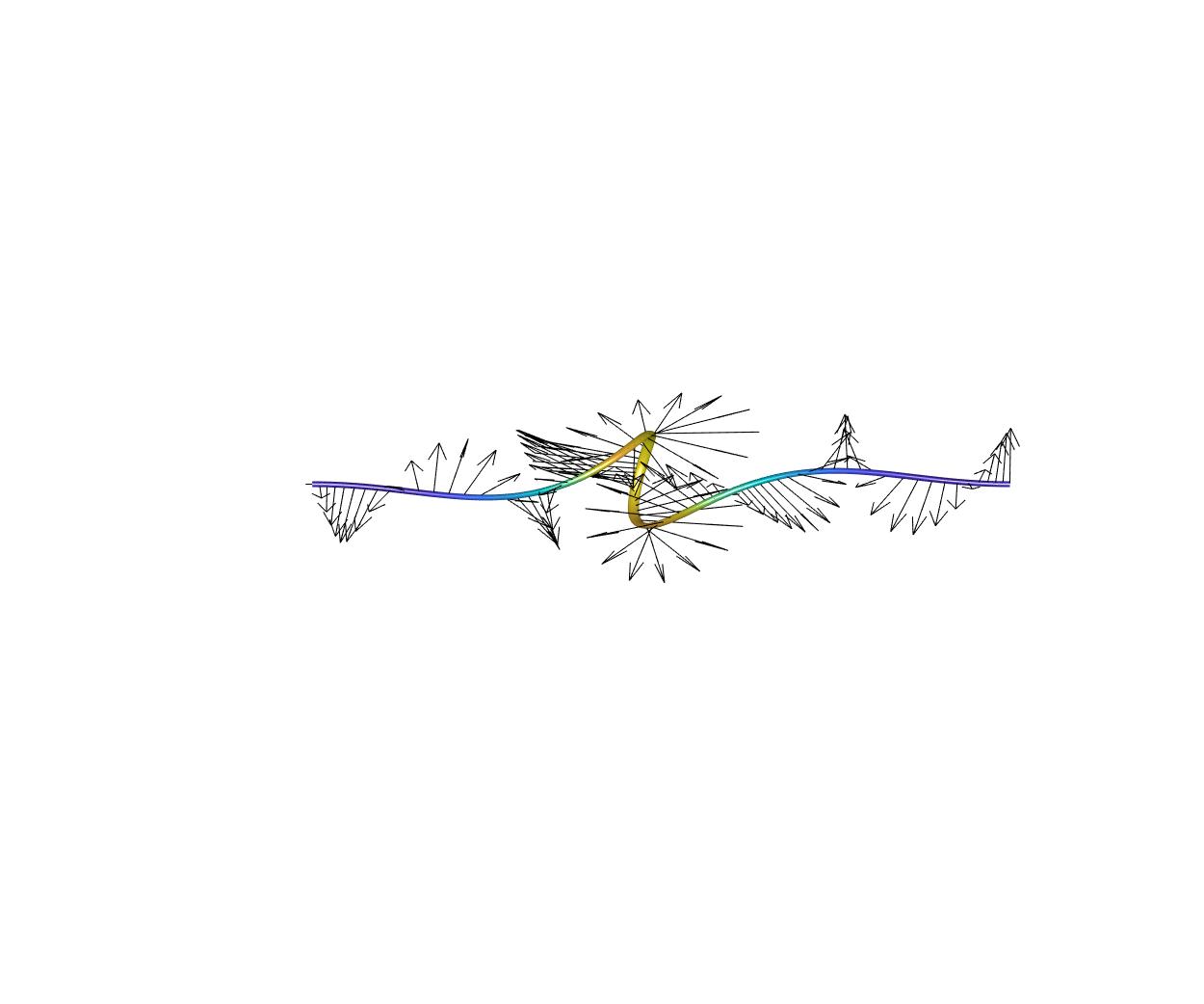}	\bild{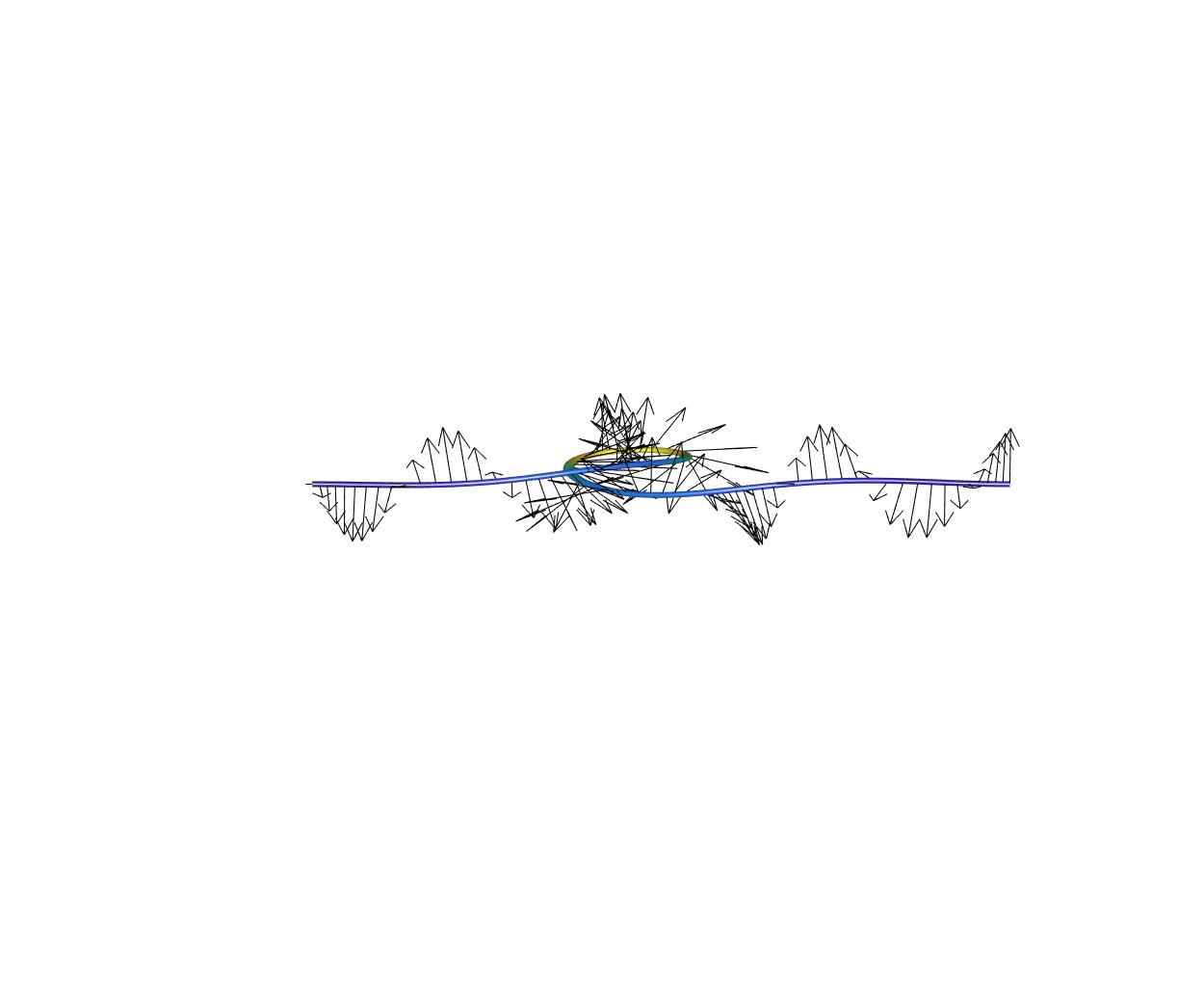}	\bild{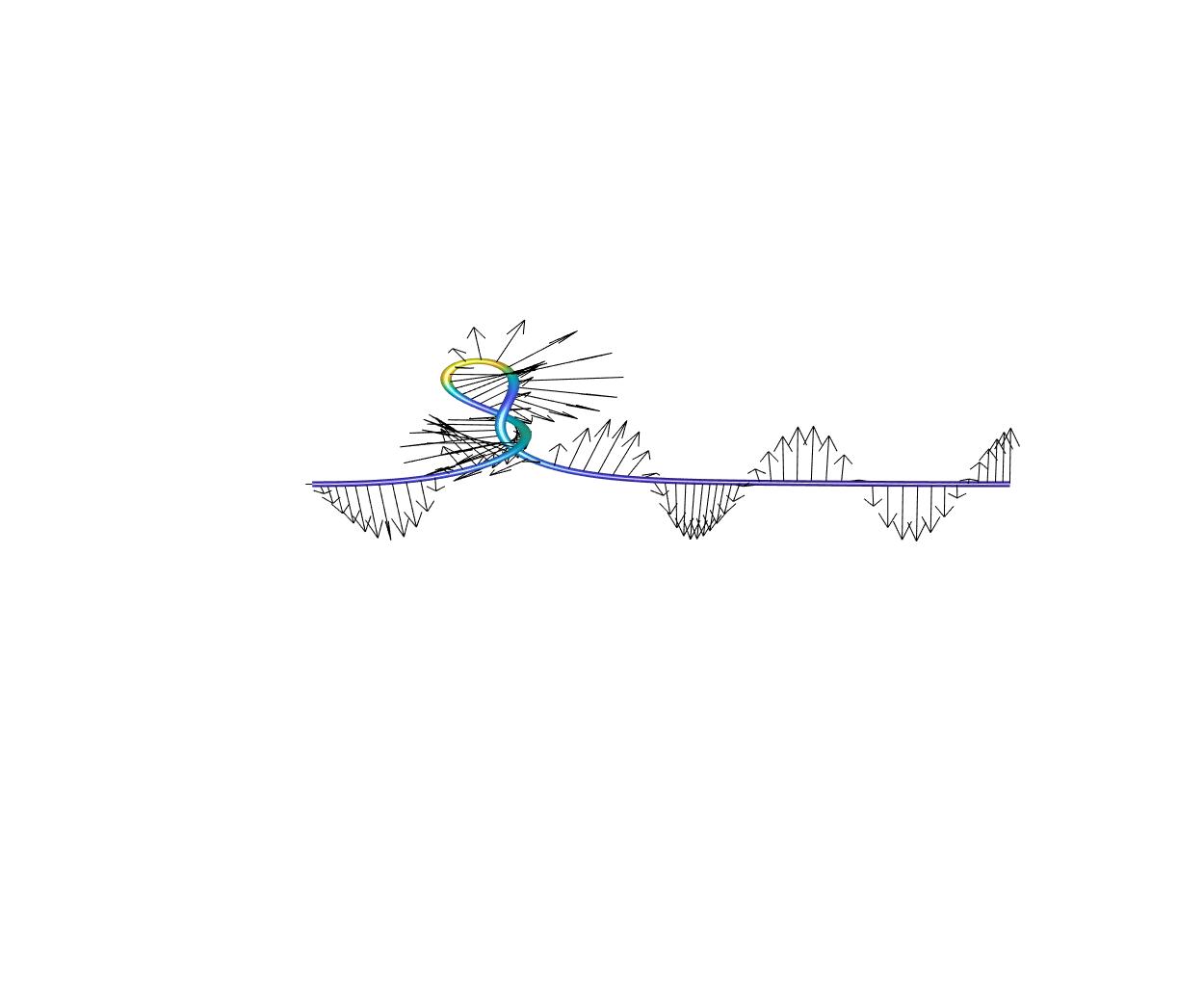}

\bild{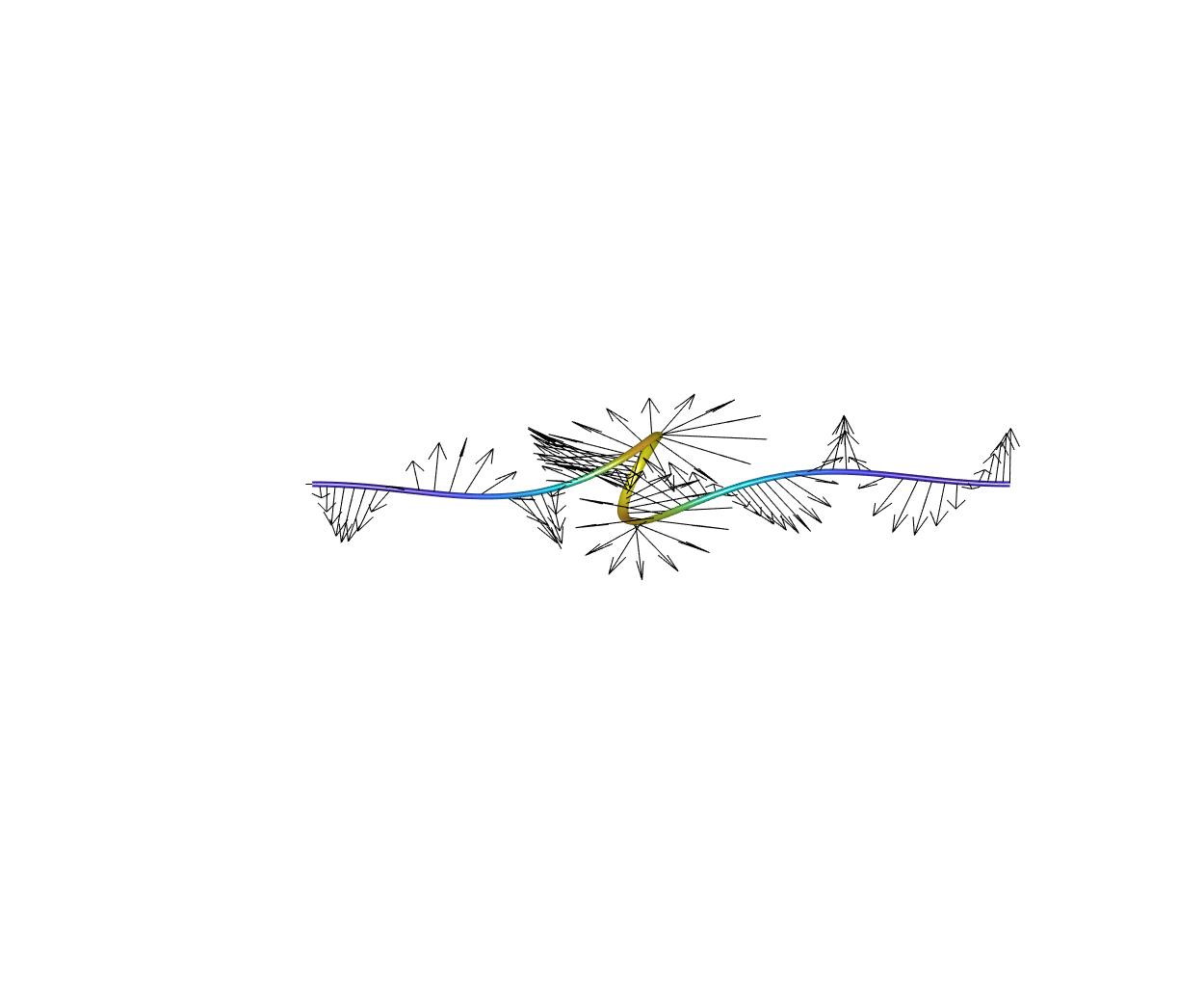}	\bild{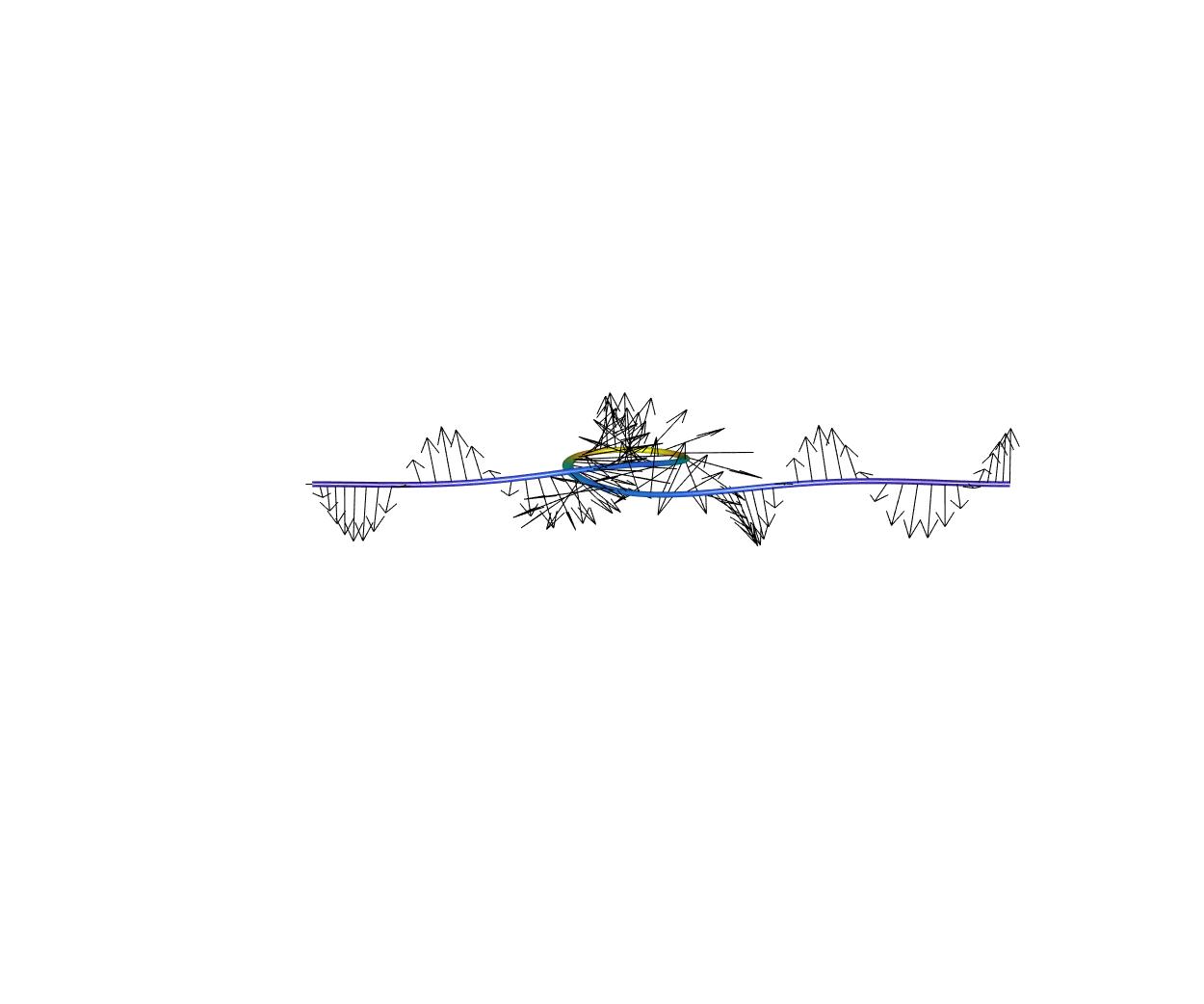}	\bild{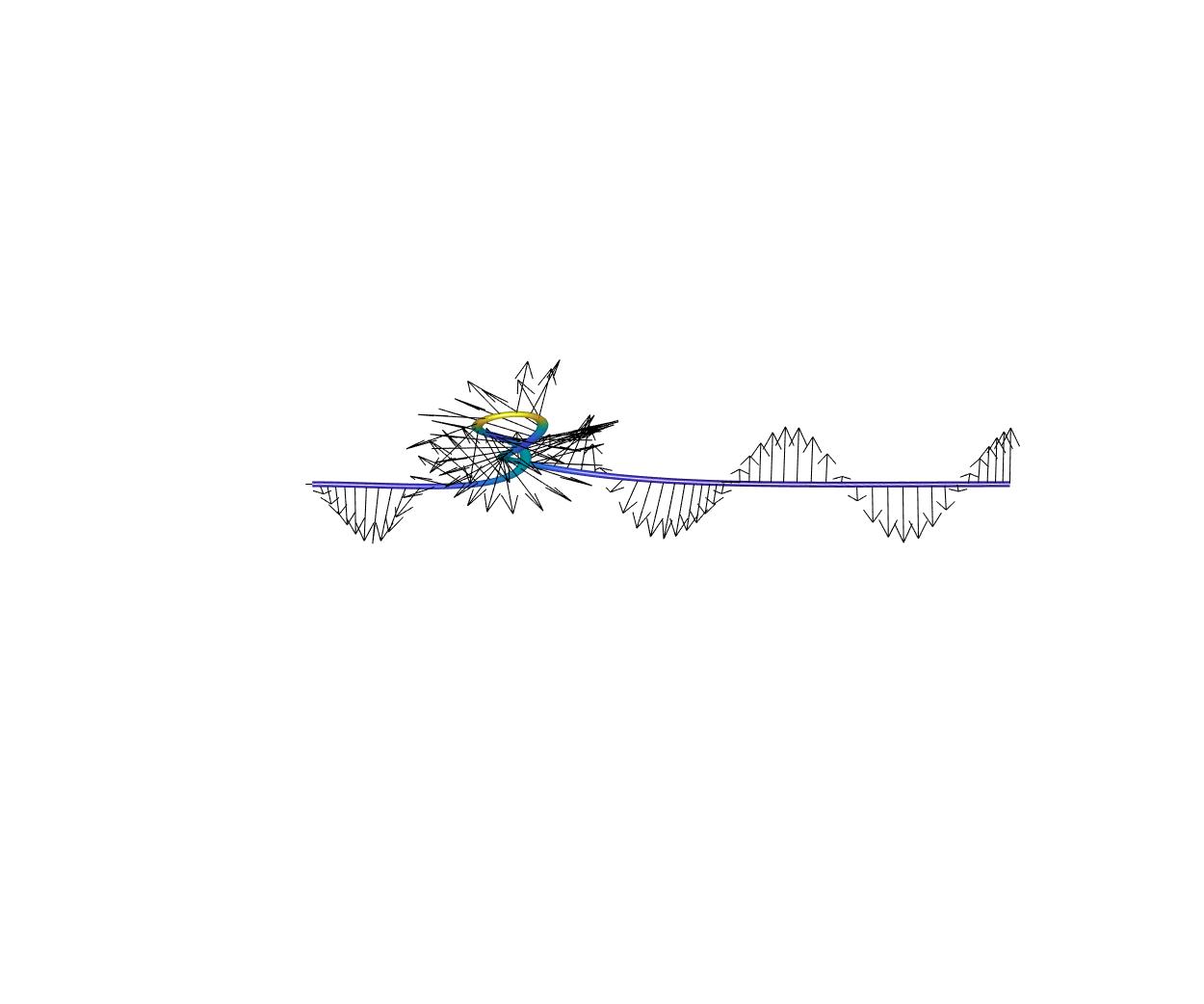}

\caption{{In the second part of Experiment~\ref{sect:imper}
we repeat Experiment~\ref{sect:clamped}
in the presence of impermeability.
As self-penetrations are excluded, we observe the formation of a plectoneme.
The viewer's position has been rotated by $90$ degrees.
The energy plot is shown in Figure~\ref{fig:imper-plot} (right).}}
\label{fig:imper-clamped}
\end{figure}

\begin{figure}
 \includegraphics[scale=.33]{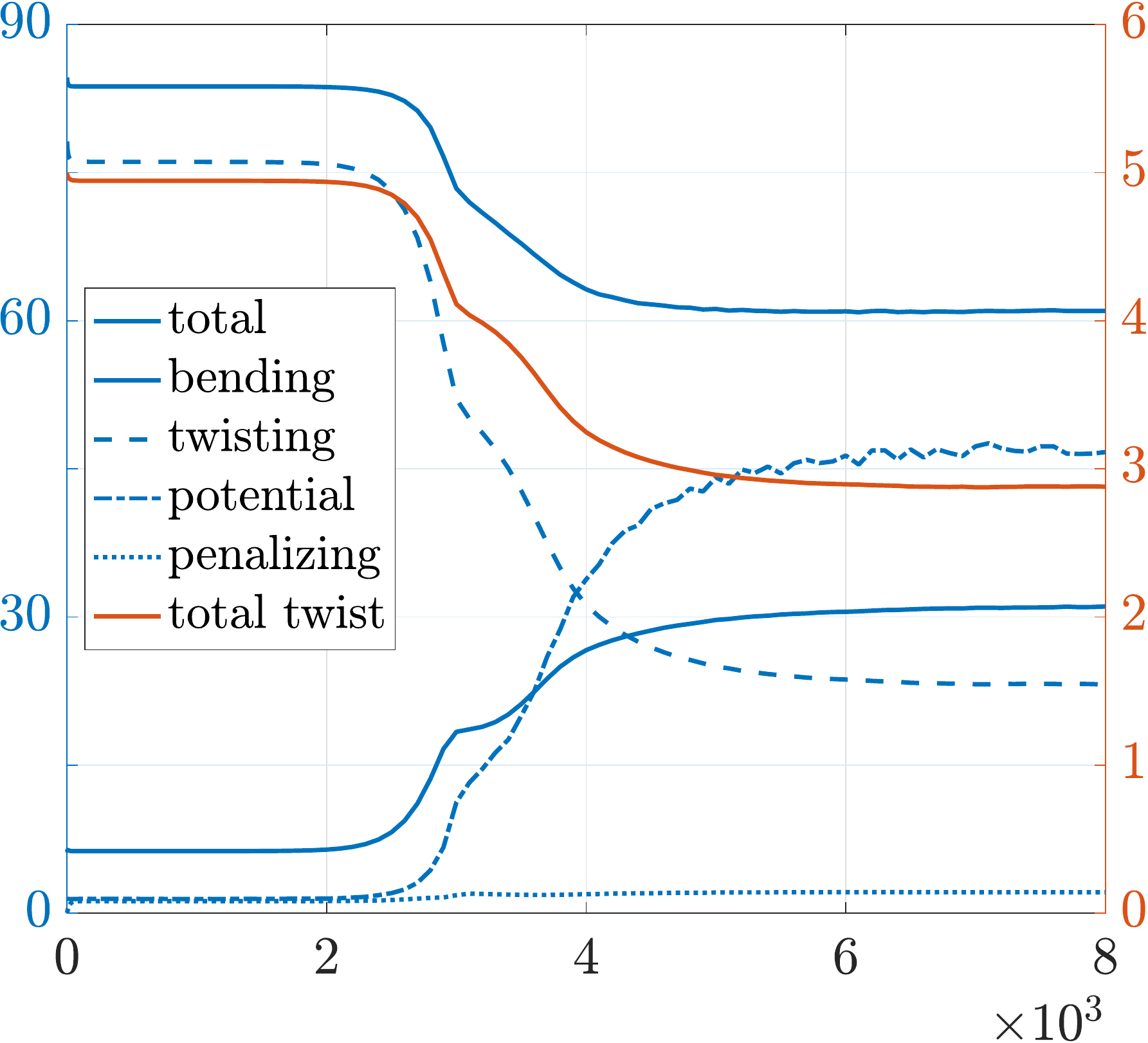}
\hfill
 \includegraphics[scale=.33]{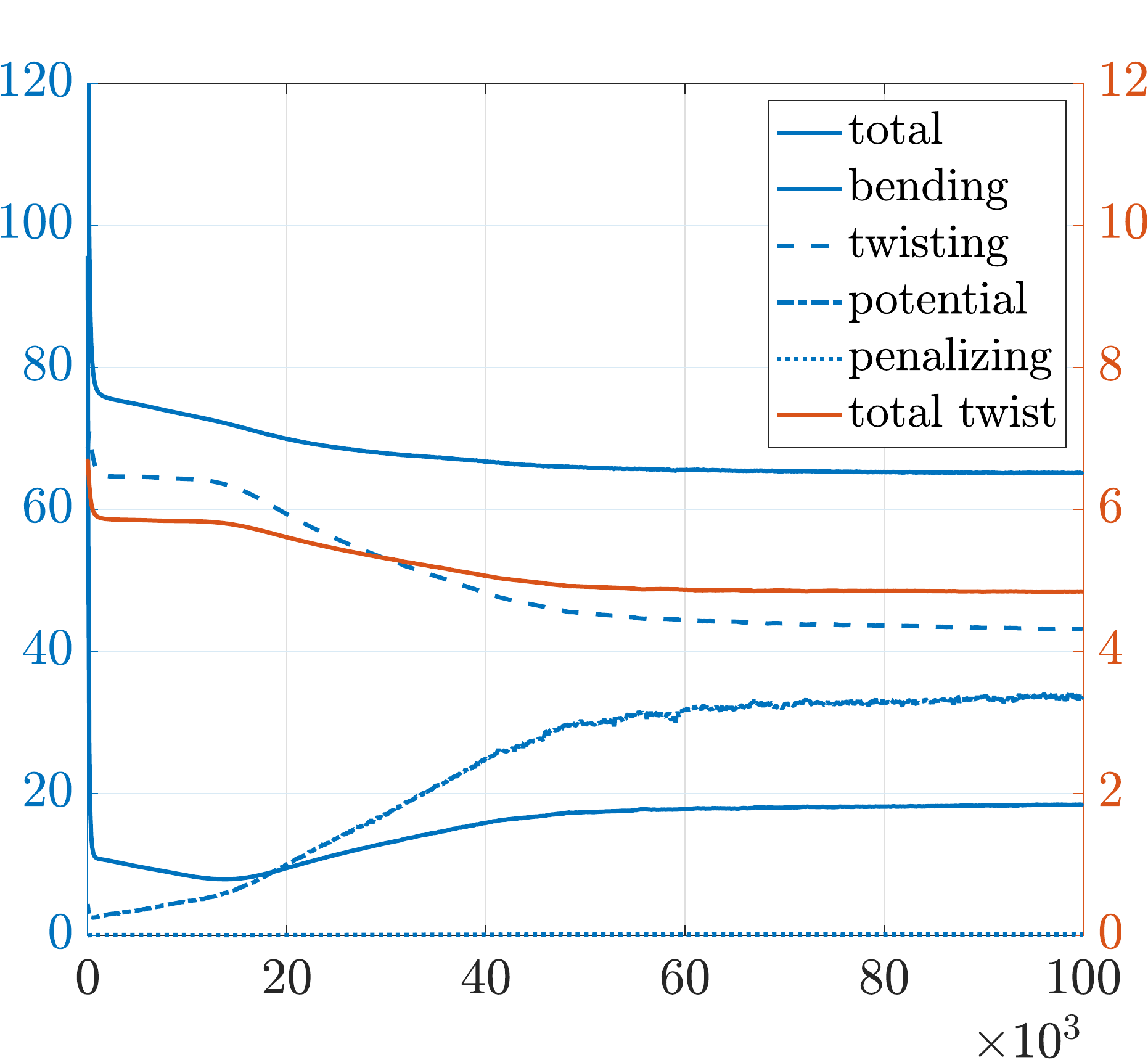}
 \caption{Energy plots for {Experiment~\ref{sect:imper}} {(a) and (b)}.}
 \label{fig:imper-plot}
\end{figure}

In order to preclude rods from self-penetrations we consider
the modified total energy $\energy+\varrho\TP$
where $\varrho\ge0$ and $\TP$ denote the \emph{tangent-point functional}
\begin{equation}\label{eq:tp}
 u\mapsto\frac1{2^{q}q}\int_{0}^{L}\int_{0}^{L}\frac{\d s\d\tilde s}{r(\curve(\tilde s),\curve(s))^{q}},
 \qquad q>2.
\end{equation}
Here $s,\tilde s$ denotes arclength parameters,
and $r(\curve(\tilde s), \curve(s))$ is the radius of the circle that is tangent to $\curve$ at the point $\curve(\tilde s)$ and that intersects with $\curve$ in~$\curve(s)$.

As many so-called knot energies~\cite{oha:book},
the tangent-point energies provide
a monotonic uniform bound on the bi-Lipschitz constant.
This implies in particular that the energy values of a sequence of
embedded curves converging
to a curve with a self-intersection will necessarily blow up.

The tangent-point energies have been proposed by
Gonzalez and Maddocks~\cite{GM99};
the scale invariant case $q=2$ has already been introduced by Buck and Orloff~\cite{BuckOrloff}.
They are defined on
(smooth) embedded curves $\curve:[0,L]\to\R^{n}$
and take values in~$[0,+\infty]$,
see Strzelecki and von der Mosel~\cite{StrMos12} and references therein.
Blatt~\cite{Blatt13} has characterized the energy spaces
in terms of Sobolev--Slobodecki{\u\i} spaces;
regularity aspects are discussed in~\cite{BR15}.

More information on the discretization of the tangent-point functional
can be found in~\cite{bartels-reiter,BRR}.
We cut out a strip of radius $2h_{\max}$ off the diagonal in $[0,L]^{2}$.

We repeat Experiments~\ref{sect:overtwist} and~\ref{sect:clamped}
in the presence of self-avoidance.
{No perturbation was added}
to the initial curves. The parameters are chosen similarly,
see Table~\ref{tab:param}.

 Note that for closed curves
with a closed frame the linking number is preserved throughout
the evolution.
Changes of the total twist will be entirely compensated by
the writhe functional.

%% file: twist.bbl
\begin{thebibliography}{10}

\bibitem{antman}
S.~S. Antman.
\newblock \href {https://doi.org/10.1007/0-387-27649-1} {{\em Nonlinear
  problems of elasticity}}, volume 107 of {\em Applied Mathematical Sciences}.
\newblock Springer, New York, second edition, 2005.

\bibitem{BaGaNu12}
J.~W. Barrett, H.~Garcke, and R.~N{\"u}rnberg.
\newblock \href {http://dx.doi.org/10.1007/s00211-011-0416-x} {Parametric
  approximation of isotropic and anisotropic elastic flow for closed and open
  curves}.
\newblock {\em Numer. Math.}, 120(3):489--542, 2012.

\bibitem{Bar13}
S.~Bartels.
\newblock \href {https://doi.org/10.1093/imanum/drs041} {A simple scheme for
  the approximation of the elastic flow of inextensible curves}.
\newblock {\em IMA J. Numer. Anal.}, 33(4):1115--1125, 2013.

\bibitem{Bar19}
S.~Bartels.
\newblock \href {http://dx.doi.org/https://doi.org/10.1016/bs.hna.2019.06.003}
  {Finite element simulation of nonlinear bending models for thin elastic rods
  and plates}.
\newblock Handbook of Numerical Analysis. Elsevier, 2019.

\bibitem{bartels-reiter}
S.~{Bartels} and {\relax Ph}.~{Reiter}.
\newblock \href {http://arxiv.org/abs/1804.02206} {{Stability of a simple
  scheme for the approximation of elastic knots and self-avoiding inextensible
  curves}}.
\newblock {\em arXiv e-prints}, page arXiv:1804.02206, Apr 2018.

\bibitem{BRR}
S.~Bartels, {\relax Ph}.~Reiter, and J.~Riege.
\newblock \href {http://dx.doi.org/10.1093/imanum/drx021} {A simple scheme for
  the approximation of self-avoiding inextensible curves}.
\newblock {\em IMA Journal of Numerical Analysis}, page drx021, 2017.

\bibitem{BWRAG}
M.~Bergou, M.~Wardetzky, S.~Robinson, B.~Audoly, and E.~Grinspun.
\newblock \href {http://dx.doi.org/10.1145/1360612.1360662} {Discrete elastic
  rods}.
\newblock {\em ACM Trans. Graph.}, 27(3):63:1--63:12, 2008.

\bibitem{Blatt13}
S.~Blatt.
\newblock \href {https://doi.org/10.1142/S1793525313500131} {The energy spaces
  of the tangent point energies}.
\newblock {\em J. Topol. Anal.}, 5(3):261--270, 2013.

\bibitem{BR15}
S.~Blatt and {\relax Ph}.~Reiter.
\newblock \href {https://doi.org/10.1515/acv-2013-0020} {Regularity theory for
  tangent-point energies: the non-degenerate sub-critical case}.
\newblock {\em Adv. Calc. Var.}, 8(2):93--116, 2015.

\bibitem{BuckOrloff}
G.~Buck and J.~Orloff.
\newblock \href {https://doi.org/10.1016/0166-8641(94)00024-W} {A simple energy
  function for knots}.
\newblock {\em Topology Appl.}, 61(3):205--214, 1995.

\bibitem{calu59}
G.~C{\u{a}}lug{\u{a}}reanu.
\newblock L'int\'egrale de {G}auss et l'analyse des n\oe uds tridimensionnels.
\newblock {\em Rev. Math. Pures Appl.}, 4:5--20, 1959.

\bibitem{calu61}
G.~C{\u{a}}lug{\u{a}}reanu.
\newblock Sur les classes d'isotopie des n\oe uds tridimensionnels et leurs
  invariants.
\newblock {\em Czechoslovak Math. J.}, 11 (86):588--625, 1961.

\bibitem{CAN}
N.~Clauvelin, B.~Audoly, and S.~Neukirch.
\newblock \href {http://dx.doi.org/10.1016/j.jmps.2009.05.004} {Matched
  asymptotic expansions for twisted elastic knots: a self-contact problem with
  non-trivial contact topology}.
\newblock {\em J. Mech. Phys. Solids}, 57(9):1623--1656, 2009.

\bibitem{coleman92}
B.~D. Coleman, E.~H. Dill, M.~Lembo, Z.~Lu, and I.~Tobias.
\newblock \href {http://dx.doi.org/10.1007/BF00375625} {On the dynamics of rods
  in the theory of {K}irchhoff and {C}lebsch}.
\newblock {\em Arch. Rational Mech. Anal.}, 121(4):339--359, 1992.

\bibitem{CS1}
B.~D. Coleman and D.~Swigon.
\newblock \href {http://dx.doi.org/10.1023/A:1010911113919} {Theory of
  supercoiled elastic rings with self-contact and its application to {DNA}
  plasmids}.
\newblock {\em J. Elasticity}, 60(3):173--221 (2001), 2000.

\bibitem{CS2}
B.~D. Coleman and D.~Swigon.
\newblock \href {http://dx.doi.org/10.1098/rsta.2004.1393} {Theory of
  self-contact in {K}irchhoff rods with applications to supercoiling of knotted
  and unknotted {DNA} plasmids}.
\newblock {\em Philos. Trans. R. Soc. Lond. Ser. A Math. Phys. Eng. Sci.},
  362(1820):1281--1299, 2004.

\bibitem{CST}
B.~D. Coleman, D.~Swigon, and I.~Tobias.
\newblock \href {http://dx.doi.org/10.1103/PhysRevE.61.759} {Elastic stability
  of {DNA} configurations. {II}. {S}upercoiled plasmids with self-contact}.
\newblock {\em Phys. Rev. E (3)}, 61(1):759--770, 2000.

\bibitem{DaLiPo14}
A.~Dall'Acqua, C.-C. Lin, and P.~Pozzi.
\newblock \href {http://dx.doi.org/10.1515/anly-2014-1249} {Evolution of open
  elastic curves in {$\Bbb{R}^n$} subject to fixed length and natural boundary
  conditions}.
\newblock {\em Analysis (Berlin)}, 34(2):209--222, 2014.

\bibitem{DAP}
A.~Dall'Acqua and A.~Pluda.
\newblock \href {http://dx.doi.org/10.1515/geofl-2017-0005} {Some minimization
  problems for planar networks of elastic curves}.
\newblock {\em Geom. Flows}, 2(1):105--124, 2017.

\bibitem{DeDz09}
K.~Deckelnick and G.~Dziuk.
\newblock \href {https://doi.org/10.1090/S0025-5718-08-02176-5} {Error analysis
  for the elastic flow of parametrized curves}.
\newblock {\em Math. Comp.}, 78(266):645--671, 2009.

\bibitem{dennis-hannay}
M.~R. Dennis and J.~H. Hannay.
\newblock \href {http://dx.doi.org/10.1098/rspa.2005.1527} {Geometry of
  {C}\u{a}lug\u{a}reanu's theorem}.
\newblock {\em Proc. R. Soc. Lond. Ser. A Math. Phys. Eng. Sci.},
  461(2062):3245--3254, 2005.

\bibitem{DLM}
D.~J. Dichmann, Y.~Li, and J.~H. Maddocks.
\newblock \href {http://dx.doi.org/10.1007/978-1-4612-4066-2_6} {Hamiltonian
  formulations and symmetries in rod mechanics}.
\newblock In {\em Mathematical approaches to biomolecular structure and
  dynamics ({M}inneapolis, {MN}, 1994)}, volume~82 of {\em IMA Vol. Math.
  Appl.}, pages 71--113. Springer, New York, 1996.

\bibitem{Djondjorov}
P.~A. Djondjorov, M.~T. Hadzhilazova, I.~M. Mladenov, and V.~M. Vassilev.
\newblock Explicit parameterization of {E}uler's elastica.
\newblock In {\em Geometry, integrability and quantization}, pages 175--186.
  Softex, Sofia, 2008.

\bibitem{DzKuSc02}
G.~Dziuk, E.~Kuwert, and R.~Sch{\"a}tzle.
\newblock \href {http://dx.doi.org/10.1137/S0036141001383709} {Evolution of
  elastic curves in {$\Bbb R^n$}: existence and computation}.
\newblock {\em SIAM J. Math. Anal.}, 33(5):1228--1245, 2002.

\bibitem{GRvdM}
H.~Gerlach, {\relax Ph}.~Reiter, and H.~von~der Mosel.
\newblock \href {https://doi.org/10.1007/s00205-017-1100-9} {The elastic
  trefoil is the doubly covered circle}.
\newblock {\em Arch. Ration. Mech. Anal.}, 225(1):89--139, 2017.

\bibitem{GM99}
O.~Gonzalez and J.~H. Maddocks.
\newblock \href {https://doi.org/10.1073/pnas.96.9.4769} {Global curvature,
  thickness, and the ideal shapes of knots}.
\newblock {\em Proc. Natl. Acad. Sci. USA}, 96(9):4769--4773, 1999.

\bibitem{gmsm}
O.~Gonzalez, J.~H. Maddocks, F.~Schuricht, and H.~von~der Mosel.
\newblock \href {http://dx.doi.org/10.1007/s005260100089} {Global curvature and
  self-contact of nonlinearly elastic curves and rods}.
\newblock {\em Calc. Var. Partial Differential Equations}, 14(1):29--68, 2002.

\bibitem{goriely}
A.~Goriely.
\newblock \href {http://dx.doi.org/10.1007/s10659-006-9055-3} {Twisted elastic
  rings and the rediscoveries of {M}ichell's instability}.
\newblock {\em J. Elasticity}, 84(3):281--299, 2006.

\bibitem{GT1}
A.~Goriely and M.~Tabor.
\newblock \href {http://dx.doi.org/10.1016/S0167-2789(96)00290-4} {Nonlinear
  dynamics of filaments. {I}. {D}ynamical instabilities}.
\newblock {\em Phys. D}, 105(1-3):20--44, 1997.

\bibitem{GT2}
A.~Goriely and M.~Tabor.
\newblock \href {http://dx.doi.org/10.1016/S0167-2789(97)83389-1} {Nonlinear
  dynamics of filaments. {II}. {N}onlinear analysis}.
\newblock {\em Phys. D}, 105(1-3):45--61, 1997.

\bibitem{GPL2}
S.~Goyal, N.~Perkins, and C.~Lee.
\newblock \href {http://dx.doi.org/10.1016/j.ijnonlinmec.2007.10.004}
  {Non-linear dynamic intertwining of rods with self-contact}.
\newblock {\em International Journal of Non-Linear Mechanics}, 43(1):65--73,
  2008.
\newblock cited By 34.

\bibitem{GPL1}
S.~Goyal, N.~C. Perkins, and C.~L. Lee.
\newblock \href {http://dx.doi.org/10.1016/j.jcp.2005.03.027} {Nonlinear
  dynamics and loop formation in {K}irchhoff rods with implications to the
  mechanics of {DNA} and cables}.
\newblock {\em J. Comput. Phys.}, 209(1):371--389, 2005.

\bibitem{HS2}
K.~A. Hoffman and T.~I. Seidman.
\newblock \href {http://dx.doi.org/10.1016/j.na.2011.05.022} {A variational
  characterization of a hyperelastic rod with hard self-contact}.
\newblock {\em Nonlinear Anal.}, 74(16):5388--5401, 2011.

\bibitem{HS1}
K.~A. Hoffman and T.~I. Seidman.
\newblock \href {http://dx.doi.org/10.1007/s00205-010-0368-9} {A variational
  rod model with a singular nonlocal potential}.
\newblock {\em Arch. Ration. Mech. Anal.}, 200(1):255--284, 2011.

\bibitem{hu}
K.~Hu.
\newblock \href {http://dx.doi.org/10.1016/S0893-9659(03)80031-9} {Buckling of
  some isotropic, intrinsically curved elasticas induced by a terminal twist}.
\newblock {\em Appl. Math. Lett.}, 16(2):193--197, 2003.

\bibitem{ivey-singer}
T.~A. Ivey and D.~A. Singer.
\newblock \href {http://dx.doi.org/10.1112/S0024611599011983} {Knot types,
  homotopies and stability of closed elastic rods}.
\newblock {\em Proc. London Math. Soc. (3)}, 79(2):429--450, 1999.

\bibitem{KM}
S.~Kehrbaum and J.~H. Maddocks.
\newblock \href {http://dx.doi.org/10.1098/rsta.1997.0113} {Elastic rods, rigid
  bodies, quaternions and the last quadrature}.
\newblock {\em Philos. Trans. Roy. Soc. London Ser. A}, 355(1732):2117--2136,
  1997.

\bibitem{kroemer-valdman}
S.~Kr\"{o}mer and J.~Valdman.
\newblock \href {http://dx.doi.org/10.1177/1081286519851554} {Global
  injectivity in second-gradient nonlinear elasticity and its approximation
  with penalty terms}.
\newblock {\em Math. Mech. Solids}, 24(11):3644--3673, 2019.

\bibitem{langer-singer:cs}
J.~Langer and D.~A. Singer.
\newblock \href {http://dx.doi.org/10.1016/0040-9383(85)90046-1} {Curve
  straightening and a minimax argument for closed elastic curves}.
\newblock {\em Topology}, 24(1):75--88, 1985.

\bibitem{langer-singer}
J.~Langer and D.~A. Singer.
\newblock \href {http://dx.doi.org/10.1137/S0036144593253290} {Lagrangian
  aspects of the {K}irchhoff elastic rod}.
\newblock {\em SIAM Rev.}, 38(4):605--618, 1996.

\bibitem{levien}
R.~Levien.
\newblock \href
  {http://www2.eecs.berkeley.edu/Pubs/TechRpts/2008/EECS-2008-103.html} {The
  elastica: a mathematical history}.
\newblock Technical Report UCB/EECS-2008-103, EECS Department, University of
  California, Berkeley, 2008.

\bibitem{LinSchw05}
C.-C. Lin and H.~R. Schwetlick.
\newblock \href {http://dx.doi.org/10.1137/S0036139903431713} {On the geometric
  flow of {K}irchhoff elastic rods}.
\newblock {\em SIAM J. Appl. Math.}, 65(2):720--736, 2004/05.

\bibitem{maddocks}
J.~H. Maddocks.
\newblock Bifurcation theory, symmetry breaking and homogenization in continuum
  mechanics descriptions of {DNA}. {M}athematical modelling of the physics of
  the double helix.
\newblock In {\em A celebration of mathematical modeling}, pages 113--136.
  Kluwer Acad. Publ., Dordrecht, 2004.

\bibitem{MOSS15}
A.~Manhart, D.~Oelz, C.~Schmeiser, and N.~Sfakianakis.
\newblock \href {http://dx.doi.org/10.1016/j.jtbi.2015.06.044} {An extended
  filament based lamellipodium model produces various moving cell shapes in the
  presence of chemotactic signals}.
\newblock {\em J. Theoret. Biol.}, 382:244--258, 2015.

\bibitem{MM}
R.~S. Manning and J.~H. Maddocks.
\newblock \href {http://dx.doi.org/10.1016/S0045-7825(98)00200-X} {Symmetry
  breaking and the twisted elastic ring}.
\newblock {\em Comput. Methods Appl. Mech. Engrg.}, 170(3-4):313--330, 1999.
\newblock Computational methods and bifurcation theory with applications.

\bibitem{MRM}
R.~S. Manning, K.~A. Rogers, and J.~H. Maddocks.
\newblock \href {http://dx.doi.org/10.1098/rspa.1998.0291} {Isoperimetric
  conjugate points with application to the stability of {DNA} minicircles}.
\newblock {\em R. Soc. Lond. Proc. Ser. A Math. Phys. Eng. Sci.},
  454(1980):3047--3074, 1998.

\bibitem{michell}
J.~H. Michell.
\newblock On the stability of a bent and twisted wire.
\newblock {\em Messenger of Math.}, 11:181--184, 1889--1890.
\newblock Reprinted in~\cite{goriely}.

\bibitem{moffatt-ricca}
H.~K. Moffatt and R.~L. Ricca.
\newblock \href {http://dx.doi.org/10.1098/rspa.1992.0159} {Helicity and the
  {C}\u{a}lug\u{a}reanu invariant}.
\newblock {\em Proc. Roy. Soc. London Ser. A}, 439(1906):411--429, 1992.

\bibitem{MorMul03}
M.~G. Mora and S.~M{\"u}ller.
\newblock \href {http://dx.doi.org/10.1007/s00526-003-0204-2} {Derivation of
  the nonlinear bending-torsion theory for inextensible rods by
  {$\Gamma$}-convergence}.
\newblock {\em Calc. Var. Partial Differential Equations}, 18(3):287--305,
  2003.

\bibitem{needham}
T.~Needham.
\newblock \href {http://dx.doi.org/10.1007/s10455-018-9595-3} {K\"{a}hler
  structures on spaces of framed curves}.
\newblock {\em Ann. Global Anal. Geom.}, 54(1):123--153, 2018.

\bibitem{NH}
S.~Neukirch and M.~E. Henderson.
\newblock \href {http://dx.doi.org/10.1023/A:1026064603932} {Classification of
  the spatial equilibria of the clamped elastica: symmetries and zoology of
  solutions}.
\newblock {\em J. Elasticity}, 68(1-3):95--121 (2003), 2002.

\bibitem{oha:en-fam}
J.~O'Hara.
\newblock \href {https://doi.org/10.1016/0166-8641(92)90023-S} {Family of
  energy functionals of knots}.
\newblock {\em Topology Appl.}, 48(2):147--161, 1992.

\bibitem{oha:book}
J.~O'Hara.
\newblock \href {https://doi.org/10.1142/9789812795304} {{\em Energy of knots
  and conformal geometry}}, volume~33 of {\em Series on Knots and Everything}.
\newblock World Scientific Publishing Co., Inc., River Edge, NJ, 2003.

\bibitem{PozSti17}
P.~Pozzi and B.~Stinner.
\newblock \href {https://doi.org/10.1007/s00211-016-0828-8} {Curve shortening
  flow coupled to lateral diffusion}.
\newblock {\em Numer. Math.}, 135(4):1171--1205, 2017.

\bibitem{reiter}
{\relax Ph}.~Reiter.
\newblock \href {http://dx.doi.org/10.1002/mana.201000090} {Repulsive knot
  energies and pseudodifferential calculus for {O}'{H}ara's knot energy family
  {$E^{(\alpha)},\alpha\in[2,3)$}}.
\newblock {\em Math. Nachr.}, 285(7):889--913, 2012.

\bibitem{RN}
R.~L. Ricca and B.~Nipoti.
\newblock \href {https://doi.org/10.1142/S0218216511009261} {Gauss' linking
  number revisited}.
\newblock {\em J. Knot Theory Ramifications}, 20(10):1325--1343, 2011.

\bibitem{sachkov}
Y.~L. Sachkov.
\newblock \href {http://dx.doi.org/10.1134/S0081543812060211} {Closed euler
  elasticae}.
\newblock {\em Proceedings of the Steklov Institute of Mathematics},
  278(1):218--232, 2012.

\bibitem{SSW}
S.~{Scholtes}, H.~{Schumacher}, and M.~{Wardetzky}.
\newblock \href {http://arxiv.org/abs/1901.02228} {{Variational Convergence of
  Discrete Elasticae}}.
\newblock {\em arXiv e-prints}, page arXiv:1901.02228, Jan 2019.

\bibitem{heiko2}
F.~Schuricht and H.~von~der Mosel.
\newblock \href {http://dx.doi.org/10.1007/s00205-003-0253-x} {Euler-{L}agrange
  equations for nonlinearly elastic rods with self-contact}.
\newblock {\em Arch. Ration. Mech. Anal.}, 168(1):35--82, 2003.

\bibitem{ST}
J.~Spillmann and M.~Teschner.
\newblock \href {http://dx.doi.org/10.1111/j.1467-8659.2008.01147.x} {An
  adaptive contact model for the robust simulation of knots}.
\newblock {\em Computer Graphics Forum}, 27(2):497--506, 2008.

\bibitem{starostin}
E.~L. Starostin.
\newblock \href {http://dx.doi.org/10.1098/rsta.2004.1388} {Symmetric
  equilibria of a thin elastic rod with self-contacts}.
\newblock {\em Philos. Trans. R. Soc. Lond. Ser. A Math. Phys. Eng. Sci.},
  362(1820):1317--1334, 2004.

\bibitem{SvdH}
E.~L. Starostin and G.~H.~M. van~der Heijden.
\newblock \href {http://dx.doi.org/10.1016/j.jmps.2013.10.014} {Theory of
  equilibria of elastic 2-braids with interstrand interaction}.
\newblock {\em J. Mech. Phys. Solids}, 64:83--132, 2014.

\bibitem{SvdH2}
E.~L. Starostin and G.~H.~M. van~der Heijden.
\newblock \href {http://stacks.iop.org/1742-6596/544/i=1/a=012007} {Tightening
  elastic $(n , 2)$-torus knots}.
\newblock {\em Journal of Physics: Conference Series}, 544(1):012007, 2014.

\bibitem{SvdH18}
E.~L. Starostin and G.~H.~M. van~der Heijden.
\newblock Equilibria of elastic cable knots and links.
\newblock In S.~Blatt, {\relax Ph}.~Reiter, and A.~Schikorra, editors, {\em New
  directions in geometric and applied knot theory}, pages 258--275. De Gruyter,
  Berlin, 2018.

\bibitem{StrMos12}
P.~Strzelecki and H.~von~der Mosel.
\newblock \href {https://doi.org/10.1142/S0218216511009960} {Tangent-point
  self-avoidance energies for curves}.
\newblock {\em J. Knot Theory Ramifications}, 21(5):1250044, 28, 2012.

\bibitem{coleman96}
I.~Tobias, B.~D. Coleman, and M.~Lembo.
\newblock \href {http://dx.doi.org/10.1063/1.472040} {A class of exact
  dynamical solutions in the elastic rod model of dna with implications for the
  theory of fluctuations in the torsional motion of plasmids}.
\newblock {\em The Journal of Chemical Physics}, 105(6):2517--2526, 1996.

\bibitem{olson}
I.~Tobias and W.~K. Olson.
\newblock \href {http://dx.doi.org/10.1002/bip.360330413} {The effect of
  intrinsic curvature on supercoiling: Predictions of elasticity theory}.
\newblock {\em Biopolymers}, 33(4):639--646, 1993.

\bibitem{TSC}
I.~Tobias, D.~Swigon, and B.~D. Coleman.
\newblock \href {http://dx.doi.org/10.1103/PhysRevE.61.747} {Elastic stability
  of {DNA} configurations. {I}. {G}eneral theory}.
\newblock {\em Phys. Rev. E (3)}, 61(1):747--758, 2000.

\bibitem{vdHNGT}
G.~van~der Heijden, S.~Neukirch, V.~Goss, and J.~Thompson.
\newblock \href
  {http://dx.doi.org/http://dx.doi.org/10.1016/S0020-7403(02)00183-2}
  {Instability and self-contact phenomena in the writhing of clamped rods}.
\newblock {\em International Journal of Mechanical Sciences}, 45(1):161 -- 196,
  2003.

\bibitem{vdM:meek}
H.~von~der Mosel.
\newblock \href
  {https://content.iospress.com/articles/asymptotic-analysis/asy314}
  {Minimizing the elastic energy of knots}.
\newblock {\em Asymptot. Anal.}, 18(1--2):49--65, 1998.

\bibitem{zajac}
E.~E. Zajac.
\newblock Stability of two planar loop elasticas.
\newblock {\em Trans. ASME Ser. E. J. Appl. Mech.}, 29:136--142, 1962.

\end{thebibliography}
